\documentclass [leqno, english]{smfbook}
\usepackage {amsmath, amssymb, amscd, mathrsfs, mathtools, enumerate, url, color, comment}
\usepackage[english, francais]{babel}
\usepackage[all,cmtip,knot]{xy}
\theoremstyle{plain} 
 \usepackage[T1]{fontenc}

\newtheorem {theorem}{Theorem}[section]
\newtheorem {lemma}[theorem]{Lemma}
\newtheorem {proposition}[theorem]{Proposition}
\newtheorem {corollary}[theorem]{Corollary}

\newtheorem {definition}[theorem]{Definition}

\newtheorem {fact}[theorem]{Fact}
\newtheorem {assumption}[theorem]{Assumption}
\theoremstyle{remark}
\newtheorem {remark}[theorem]{Remark}
\newtheorem {example}[theorem]{Example}

\DeclareFontFamily{U}{mathx}{\hyphenchar\font45}
\DeclareFontShape{U}{mathx}{m}{n}{
      <5> <6> <7> <8> <9> <10>
      <10.95> <12> <14.4> <17.28> <20.74> <24.88>
      mathx10
      }{}
\DeclareSymbolFont{mathx}{U}{mathx}{m}{n}
\DeclareFontSubstitution{U}{mathx}{m}{n}
\DeclareMathAccent{\widecheck}{0}{mathx}{"71}
 
\def\zz {{\mathbb{Z}}}
\def\rr {{\mathbb{R}}}
\def\R {\rr}
\def\bR{\overline{\R}}
\def\cc {{\mathbb{C}}}
\def\qq {{\mathbb{Q}}}
\def\pp{\mathrm{p}}

\def\q {\mathfrak{q}}
\def\del {\partial}
\def\L{\mathscr{L}}
\def\C{\mathcal{C}}
\def\Q {\mathbb{Q}}
\def\qhat {\hat{\mathfrak{q}}}

\def\L {\mathcal{L}}
\def\G {\mathcal{G}}
\def\Go {\G^{\circ}}

\def\B {\mathcal{B}}
\def\T {\mathcal{T}}
\def\F {\mathcal{F}}
\def\J {\mathcal{J}}
\def\Jo{\J^{\circ}}
\def\K {\mathcal{K}}

\def\P {\mathcal{P}}
\def\V {\mathcal{V}}
\def\O{\mathcal{O}}
\def\N {\mathcal{N}}
\def\calE{\mathcal{E}}

\def\tL{\tilde{L}}
\def\fb {\operatorname{fb}}
\def\ext{\operatorname{ext}}
\def\D {\mathcal{D}}
\def\Dg {\mathcal{D}^{\tilde g}}
\def\dgs {D^{\tilde g, \sigma}}
\def\gts {\tilde g^{\sigma}}
\def\Dgs {\mathcal{D}^{\tilde g, {\sigma}}}
\def\U{\mathcal{U}}

\def\Spin{\mathbb{S}}
\def\Slice{\mathcal{S}}
\def\Crit {\mathfrak{C}}
\def\is{\mathscr{S}}
\def\vol{\operatorname{vol}}

\def\q {\mathfrak{q}}
\def\Lq{\L_{\q}}
\def\psis {\boldsymbol {\psi}}
\def\sect {\zeta}
\def\prel{\operatorname{prel}}
\def\qhatsigma{{\hat{\mathfrak{q}}^\sigma}}
\def\qhatzerosigma {\hat{\mathfrak{q}}^{\sigma,0}}
\def\qhatonesigma {\hat{\mathfrak{q}}^{\sigma,1}}
\def\qtilde {{\tilde{\mathfrak{q}}}}

\def\Ue{\mathscr{U}}
\def\Ne{\mathscr{N}}

\def\dd{\mathbf{d}}
\def\Zs {\mathscr{Z}}

\def\Ell{\mathbb{L}}
\def\E{\mathcal{E}}
\def\Rhat{\widehat{R}}

\DeclareMathOperator{\grad}{\operatorname{grad}}
\DeclareMathOperator{\red}{\operatorname{red}}

\def\Qhat {\widehat{Q}}
\def\coker {\operatorname{coker}}
\def\spc{\operatorname{Spin}^c}
\def\End{\operatorname{End}}
\def\Inv{\operatorname{Inv}}
\def\index {\operatorname{ind}}

\def\sslash{/ \!/}
\def\x{x}
\def\y{y}
\def\t{\mathfrak{t}}

\def\gCoul{\operatorname{gC}}
\def\lCoul{\operatorname{lC}}
\def\elCoul{\operatorname{elC}}
\def\agCoul{\operatorname{agC}}
\def\Ke{\mathcal{K}^{\operatorname{e}}}
\def\Kesigma{\mathcal{K}^{\operatorname{e}, \sigma}}
\def\Re{\operatorname{Re}}
\def\SWF{\operatorname{SWF}}
\def\Hess{\operatorname{Hess}}
\def\Hom{\operatorname{Hom}}
\def\pin{\operatorname{Pin}(2)}
\def\gr{\operatorname{gr}}
\def\ind{\operatorname{ind}}

\def\cmto{\widecheck{\mathit{CM}}}
\def\hmto{\widecheck{\mathit{HM}}}

\def\hmtilde{\widetilde{\mathit{HM}}}
\def\HFhat{\widehat{\mathit{HF}}}
\def\HFminus{\mathit{HF}^-}
\def\HFplus{\mathit{HF}^+}
\def\HFinfty{\mathit{HF}^{\infty}}

\def\tH{\widetilde{H}}
\def\im {\operatorname{im}}
\def\inte{\operatorname{int}}
\def\swf{\SWF}
\def\spinc{\mathfrak{s}}
\def\llambda{\lambda^{\bullet}}
\def\s{\spinc}

\def\tC{\widetilde{\C}}
\def\tB{\widetilde{\B}}

\def\tW{\widetilde{W}}

\def\lambdanenergy{\Lambda_{\q^{\lambda_n}}}
\def\pml{p^\lambda}
\def\pmlprel{p^{\lambda}_{\prel}}
\def\pmln{p^{\lambda_n}}

\def\tpgml{\tilde p^{\lambda}_{\tilde g}}
\def\pgmlx{p^{\lambda}_{\tilde g(x)}}

\def\vml{W^\lambda}

\def\vmln{W^{\lambda_n}}
\def\crit{\mathfrak{c}}
\def\qml{\q^\lambda}
\def\qmln{\q^{\lambda_n}}

\def\Btaug{\mathcal{B}^{\gCoul, \tau}}
\def\tBtaug{\widetilde{\mathcal{B}}^{\gCoul, \tau}}
\def\Fsigma {\mathcal{F}^\sigma}

\def\Fgc{\F^{\gCoul}}
\def\Fgctau{\F^{\gCoul, \tau}}
\def\Fqtau{\Fq^\tau}

\def\Fqml{\mathcal{F}_{\qml}}

\def\etaqml{\eta_{\q}^{\lambda}}
\def\Fq{\mathcal{F}_{\q}}
\def\Rcal{\mathcal{R}}
\newcommand\gammas{\boldsymbol\gamma}
\def\cgammas{\breve{\gammas}}
\def\cgamma{\breve{\gamma}}
\def\Mbreve {\breve{M}}
\def\Orb{\mathfrak{O}}
\def\sp{\operatorname{sp}}

\def\X{\mathcal{X}}
\def\Xq{\mathcal{X}_{\q}}
\def\Xgc{\X^{\gCoul}}
\def\Xqgc{\Xq^{\gCoul}}
\def\Xqogcsigma{\X_{\q_0}^{\gCoul,\sigma}}
\def\Xqgcsigma{\Xq^{\gCoul,\sigma}}

\def\Xqsigma{\mathcal{X}^\sigma_{\q}}
\def\Xqagcsigma{\X_{\q}^{\agCoul, \sigma}}

\def\Xqmlgc{\mathcal{X}^{\gCoul}_{\q^{\lambda}}} 

\def\Xqmlgctau{\mathcal{X}^{\gCoul, \tau}_{\qml}}

\def\Xqmlgcsigma{\mathcal{X}^{\gCoul, \sigma}_{\qml}}
\def\Xqomlgcsigma{\mathcal{X}^{\gCoul, \sigma}_{\q_0^{\lambda}}}
\def\Xqmlngc{\mathcal{X}^{\gCoul}_{\q^{\lambda_n}}}
\def\Xqmlngcsigma{\mathcal{X}_{\q^{\lambda_n}}^{\gCoul, \sigma}}
\def\Xqmlagcsigma{\X_{\qml}^{\agCoul, \sigma}}
\def\Xqmlnagcsigma{\X_{\q^{\lambda_n}}^{\agCoul, \sigma}}

\def\Fqgctau{\F_{\q}^{\gCoul, \tau}}
\def\Fqagctau{\F_{\q}^{\gCoul, \tau}}
\def\Fqmlgctau{\Fqml^{\gCoul, \tau}}

\def\etaq{\eta_{\q}}

\def\sC{\mathcal{C}}
\def\SF{\operatorname{SF}}
\def\Hx{\mathcal{H}^\sigma_x}
\def\Span{\operatorname{Span}}
\def\What{\widehat{W}}

\title[The equivalence of two Seiberg-Witten Floer homologies]{The equivalence of two Seiberg-Witten Floer homologies}

\begin{document}
\frontmatter

\author[Tye Lidman]{Tye Lidman}
\address {Department of Mathematics, North Carolina State University\\  Raleigh, NC 27607, USA}
\email {tlid@math.ncsu.edu}
\urladdr{http://www4.ncsu.edu/~tlidman/}

\author[Ciprian Manolescu]{Ciprian Manolescu}
\address {Department of Mathematics, UCLA, 520 Portola Plaza\\ Los Angeles, CA 90095, USA}
\email {cm@math.ucla.edu}
\urladdr{http://www.math.ucla.edu/~cm/}

\begin{abstract}
We show that monopole Floer homology (as defined by Kronheimer and Mrowka) is isomorphic to the $S^1$-equivariant homology of the Seiberg-Witten Floer spectrum constructed by the second author. 
\end {abstract}

\begin{altabstract}
Dans ce volume nous montrons que l'homologie de Floer des monopoles (telle que d\'efinie par Kronheimer et Mrowka) est isomorphe \`a l'homologie $S^1$-\'equivariante du spectre de Seiberg-Witten Floer construit par le second auteur.
\end{altabstract}

\subjclass{57R58}
\keywords{Floer homology, Seiberg-Witten equations, monopoles, 3-manifolds, Morse homology, Conley index, Morse-Smale, Coulomb gauge}

\thanks{The first author was partially supported by NSF grants DMS-1128155 and DMS-1148490. The second author was partially supported by NSF grants DMS-1104406 and DMS-1402914.}

\maketitle

\tableofcontents

\mainmatter
\chapter {Introduction}

\section{Background}
The Seiberg-Witten (monopole) equations \cite{SW1, SW2} are an important tool for understanding the topology of smooth four-dimensional manifolds. A signed count of the solutions of these equations on a closed four-manifold yields the Seiberg-Witten invariant \cite{Witten}. On a four-manifold with boundary, instead of a numerical invariant one can define an element in a group associated to the boundary, called the Seiberg-Witten Floer homology. There are several different constructions of Seiberg-Witten Floer homology in the literature, \cite{MarcolliWang, Spectrum,KMbook, FroyshovSW}. The goal of this monograph is to prove that, for rational homology spheres, the definitions given by Kronheimer-Mrowka in \cite{KMbook} and by the second author in \cite{Spectrum} are equivalent.

The construction in \cite{KMbook} applies to an arbitrary three-manifold $Y$, equipped with a $\spc$ structure $\s$. Given this data, Kronheimer and Mrowka define an infinite dimensional analog of the Morse complex, with the underlying space being the blow-up of the configuration space of $\spc$ connections and spinors (modulo gauge). The role of gradient flow lines is played by solutions to generic perturbations of the Seiberg-Witten equations on $\R \times Y$. Their invariant, monopole Floer homology, is the homology of the resulting complex. This complex (and hence also its homology) comes with a $\zz[U]$-module structure. The applications of monopole Floer homology include the surgery characterization of the unknot \cite{KMOS} and Taubes' proof of the Weinstein conjecture in three dimensions \cite{TaubesWeinstein}.

In fact, there are three different versions of monopole Floer homology defined in \cite{KMbook}; they are denoted $\hmto, \widehat{\mathit{HM}},$ and $\overline{\mathit{HM}}$. Yet another version, $\hmtilde$, was constructed by Bloom in \cite{Bloom}: To define $\hmtilde$, one considers the cone of the $U$ map on the complex that defines $\hmtilde$, and then takes homology.

Compared with \cite{KMbook}, the construction in \cite{Spectrum} was originally done only for rational homology spheres; on the other hand, it yields something more than a homology group. By using finite dimensional approximation of the Seiberg-Witten equations, combined with Conley index theory, one obtains an invariant in the form of an equivariant suspension spectrum. Specifically, given a rational homology sphere $Y$ equipped with a $\spc$ structure $\s$, one can associate to it an $S^1$-equivariant spectrum $\SWF(Y, \s)$. (See also \cite{PFP, KhandhawitThesis, Sasahira} for extensions of this construction to the case $b_1 > 0.$)

The $S^1$-equivariant homology of $\SWF(Y, \s)$ can be viewed as a definition of Seiberg-Witten Floer homology. The advantage of having a Floer spectrum is that one can also apply other (equivariant) generalized homology functors to it. For example, by adding the conjugation symmetry, one can define a $\pin$-equivariant Seiberg-Witten Floer homology; this was instrumental in the disproof of the triangulation conjecture by the second author \cite{Triangulation}. For other applications of the Floer spectrum, see \cite{GluingBF, kg, LinKO, Covers}.
 
 \section{Results} \label{sec:results} The bulk of this monograph is devoted to proving:

\begin{theorem}\label{thm:Main}
Let $Y$ be a rational homology sphere with a $\spc$ structure $\spinc$. There is an isomorphism of  absolutely-graded $\mathbb{Z}[U]$-modules:
$$\hmto_*(Y,\spinc)  \cong \widetilde{H}^{S^1}_*(\SWF(Y,\spinc)),$$
where $\hmto$ is the ``to'' version of monopole Floer homology defined in \cite{KMbook}, and $\tH^{S^1}_*$ denotes reduced equivariant (Borel) homology.\footnote{In this book, we grade Borel homology so that we simply have $\tH^{S^1}_*(X) = \tH_*(X \wedge_{S^1} ES^1_+)$. This differs from the grading conventions in \cite{GreenleesMay} or \cite{PFP} by one.}
\end{theorem}

From Theorem~\ref{thm:Main} we deduce that Bloom's homology $\hmtilde$ can be identified with the ordinary (non-equivariant) homology of the Floer spectrum $\SWF$:

\begin{corollary}\label{cor:MainHat}
Let $Y$ be a rational homology sphere equipped with a $\spc$ structure $\spinc$. Then, $\hmtilde_*(Y,\spinc) \cong \widetilde{H}_*(\SWF(Y,\spinc))$ as absolutely graded abelian groups.
\end{corollary}

From the absolute grading on monopole Floer homology one can extract a $\mathbb{Q}$-valued  invariant, called the Fr{\o}yshov invariant; see \cite{FroyshovSW} or \cite[Section 39.1]{KMbook}. A similar numerical invariant, called $\delta$, was defined in \cite[Section 3.7]{Triangulation} using the Floer spectrum $\SWF$.  (The definition there was only given for Spin structures, but it extends to the Spin$^c$ setting.)  An immediate consequence of Theorem~\ref{thm:Main} is
\begin{corollary}
\label{cor:delta}
Let $Y$ be a rational homology sphere equipped with a Spin$^c$ structure $\spinc$.  Then, $\delta(Y,\spinc) =  - h(Y,\spinc)$, where $h$ is the Fr{\o}yshov invariant as defined in \cite[Section 39.1]{KMbook}.  
\end{corollary}

Furthermore, we can combine Theorem~\ref{thm:Main} with the equivalence between monopole Floer homology and Heegaard Floer homology, established in work of Kutluhan-Lee-Taubes  \cite{KLT1, KLT2, KLT3, KLT4, KLT5} and of Colin-Ghiggini-Honda \cite{CGH1, CGH2, CGH3} and Taubes \cite{Taubes12345}. In this way we obtain a relationship between the spectrum $\SWF(Y, \spinc)$ and Heegaard Floer theory. Precisely, we confirm a conjecture from \cite{PFP}, that the different flavors of Heegaard Floer homology are different equivariant homology theories applied to $\SWF(Y, \spinc)$:
\begin{corollary}
\label{cor:hfswf}
If $Y$ is a rational homology sphere and $\spinc$ is a $\spc$ structure on $Y$, we have the following  isomorphisms of relatively graded $\zz[U]$-modules:
\begin{align*}
\HFplus_*(Y, \spinc) & \cong \tH_*^{S^1}(\swf(Y, \spinc)), \hskip1cm \HFhat_*(Y, \spinc)  \cong \tH_*(\swf(Y, \spinc)), \\
\HFminus_*(Y, \spinc) & \cong c\tH_{*}(\swf(Y, \spinc)), \hskip1.07cm \HFinfty_*(Y, \spinc) \cong t\tH_*(\swf(Y, \spinc)),
\end{align*}
where $c\tH^*$ and $t\tH^*$ denote co-Borel and Tate homology, respectively, as defined in \cite{GreenleesMay}.
\end{corollary}

The results above can be applied in two directions. On the one hand, there are numerous computational techniques available for Heegaard Floer homology, whereas the class of manifolds $Y$ for which one can compute $\SWF(Y, \s)$ is rather small. By making use of the isomorphisms in Corollary~\ref{cor:hfswf}, one can at least understand the (equivariant or non-equivariant) homologies of $\SWF(Y, \s)$. In particular:
\begin{itemize}
\item If $\HFhat_*(Y, \spinc) \cong \zz$, this suffices to determine the non-equivariant stable homotopy type of $\SWF(Y, \s)$: by the Hurewicz and Whitehead theorems, it has to be that of a (de-)suspension of the sphere spectrum. (However, it is unclear whether the equivariant stable homotopy type of $\SWF(Y, \s)$ is determined by this information.) 
\item In the case when $\s$ is a Spin structure, the Seiberg-Witten equations on $(Y, \s)$ have a $\operatorname{Pin}(2)$ symmetry, and one can define the $\operatorname{Pin}(2)$-equivariant homology of $\SWF(Y, \s)$. (This homology was used in \cite{Triangulation} to disprove the triangulation conjecture in high dimensions.) When $Y$ is Seifert fibered, the Heegaard Floer homology was calculated in \cite{Plumbed}, so Corollary~\ref{cor:hfswf} can tell the $S^1$-equivariant homology of $\SWF(Y, \s)$. Together with knowledge of the monopoles (generators of the monopole Floer complex) from \cite{MOY}, this suffices to determine the $\operatorname{Pin}(2)$-equivariant homology of $\SWF(Y, \s)$; see \cite{Stoffregen} for details.
\end{itemize}

In the reverse direction, there are certain results that are easier to prove with Floer spectra, and one can use Theorem~\ref{thm:Main} and Corollary~\ref{cor:hfswf} to translate them into the settings of monopole Floer or Heegaard Floer homology. This is the case with the Smith inequality for the Floer homology of coverings, which is the subject of the sequel to this book \cite{Covers}. By combining the Smith inequality with the knot surgery formula in Heegaard Floer homology \cite{RatSurg}, we obtain restrictions on surgeries on a knot being regular covers over other surgeries on that knot, and over surgeries on other knots; see \cite{Covers}.

\section{Outline and organization of the book} Theorem~\ref{thm:Main} asserts that the monopole Floer homology of Kronheimer-Mrowka is isomorphic to the equivariant homology of the Conley index used for $\SWF(Y, \s)$. Let us describe the basic strategy for the proof. The monopole Floer complex is  defined from the {\em perturbed} Seiberg-Witten equations in infinite dimensions, whereas the Conley index is a space associated to an {\em approximation} of the Seiberg-Witten equations in finite dimensions. On the Conley index side, we can also use an approximation of the perturbed Seiberg-Witten equations; by the continuation properties of the Conley index, this yields a space homotopy equivalent to the one from the unperturbed case. Thus, we start by fixing a suitable perturbation, so that we have regularity in infinite dimensions, and hence we can define the monopole Floer complex. We then show that for a large enough approximation, we also have regularity in finite dimensions, and hence the homology of the Conley index is given by a Morse complex. Further, by applying the inverse function theorem (for stationary points and trajectories), we conclude that the Morse complex in finite dimensions is isomorphic to the original monopole Floer complex, and this will complete the proof.

There are several technical difficulties that must be overcome to implement this strategy. To start with, the version of Morse homology that we need to use in finite dimensions is different from the standard one in three ways: we are on a non-compact manifold (a vector space), we have to define $S^1$-equivariant rather than ordinary homology, and the approximate Seiberg-Witten flow is not a gradient flow. Non-compactness is taken care of using the notion of Conley index; this goes back to the work of Floer \cite{FloerMorse}. Equivariance was dealt with by Kronheimer and Mrowka in \cite{KMbook}, using the real blow-up construction. One intriguing aspect is that the flow is not a gradient. This leads us to introduce the weaker notion of {\em quasi-gradient}, which is a particular case of the Morse-Smale flows that appear in the study of stability for dynamical systems. We will show that one can define Morse homology from quasi-gradients. All this is done in Chapter~\ref{sec:finite}; there, we recall the various notions from Morse theory and Conley index theory, the relation between the two theories, and then proceed to define (in several steps) equivariant Morse homology for quasi-gradients on non-compact manifolds.

In Chapter~\ref{sec:spectrum} we outline the construction of the Seiberg-Witten Floer spectrum from \cite{Spectrum}. The main idea is to approximate the Seiberg-Witten equations (in global Coulomb gauge) by a flow in finite dimensions, and then to take the Conley index associated to that flow.

In Chapter~\ref{sec:HM} we sketch the construction of monopole Floer homology by Kronheimer and Mrowka, following their book \cite{KMbook}. In particular, we describe the class of admissible perturbations of the Seiberg-Witten equations that can be used to define the Floer complexes.

In Chapter~\ref{sec:coulombgauge} we explain how the Kronheimer-Mrowka construction can be rephrased in terms of configurations in global Coulomb gauge. This is the first step in bringing it closer to the construction of the Floer spectrum. The section is rather lengthy, because there are several aspects of the theory that have to be translated into the new setting: gauge conditions (in particular, for four-dimensional configurations we introduce the concept of {\em pseudo-temporal gauge}), admissibility for perturbations, Hessians, regularity for stationary points, spaces of paths, regularity for trajectories, gradings, orientations, and the $\zz[U]$-module action.

Another step in relating the two theories is taken in Chapter~\ref{sec:finiteapproximations}. There, we show that the Floer spectrum can be defined from finite dimensional approximation of the {\em perturbed} Seiberg-Witten equations. We use an admissible perturbation of the kind considered in \cite{KMbook}.

In Chapter~\ref{sec:criticalpoints} we use the inverse function theorem to identify the stationary points in infinite dimensions (i.e., the generators of the monopole Floer complex) with the stationary points of the gradient flow in the finite dimensional approximations. (This identification is limited to a certain grading range.) Further, we show that if the stationary points are non-degenerate in infinite dimensions, then the corresponding stationary points in sufficiently large approximations are non-degenerate as well. Note that the approximations are described by eigenvalue cut-offs, and these range over the interval $(0, \infty]$, with $\infty$ giving the original equations (in infinite dimensions). When applying the inverse function theorem, one thing to be careful about is how to give a manifold-with-boundary structure to the interval $(0, \infty]$. We do so by identifying $(0, \infty]$ with $[0,1)$, via a homeomorphism that depends on the growth rate of the eigenvalues.

In Chapter~\ref{sec:quasigradient} we show that the approximate Seiberg-Witten flow in finite dimensions is a quasi-gradient. The main difficulty there is to perturb the Chern-Simons-Dirac functional in such a way so that it decreases along flow lines. 

In Chapter~\ref{sec:gradings}, we study the gradings of the stationary points in the finite dimensional approximations and show that the correspondence with stationary points in infinite dimensions established in Chapter~\ref{sec:criticalpoints} preserves gradings.  

In Chapter~\ref{sec:MorseSmale}, we show how to arrange so that the Morse-Smale condition is satisfied for the approximate flow.

Chapter~\ref{sec:appendix} contains the construction of some diffeomorphisms $\Xi_{\lambda}$ of the configuration space. These diffeomorphisms take a stationary point of the approximate Seiberg-Witten flow to the corresponding stationary point of the original flow. They allow us to identify the corresponding path spaces, so that we may directly relate infinite dimensional trajectories with approximate trajectories in Chapters~\ref{sec:trajectories1} and \ref{sec:trajectories2}.  

In Chapter~\ref{sec:trajectories1} we prove various convergence results for trajectories of the approximate Seiberg-Witten flow, in the limit as the eigenvalue cut-off goes to infinity.

In Chapter~\ref{sec:trajectories2} we use the inverse function theorem to identify the flow trajectories  in infinite dimensions with those in the approximations. Furthermore, we show that these identifications preserve orientations.

Finally, in Chapter~\ref{sec:equivalence}, we bring these results together to prove Theorem~\ref{thm:Main}, about the equivalence between monopole Floer homology and the equivariant homology of $\SWF(Y, \s)$. We also prove the three corollaries stated in Chapter~\ref{sec:results}.

\section{Conventions.} Throughout the book, when we talk about the flow associated to a vector field $v$, we mean the reverse flow generated by $v$, i.e, the flow whose trajectories satisfy:
$$\frac{d}{dt}\gamma(t) + v(\gamma(t)) =0.$$

Also, we depart from the terminology in \cite{KMbook}, where critical points of the CSD functional are also referred to as critical points of the Seiberg-Witten vector field.  Here, we sometimes need to think of vector fields as maps to the tangent space, and the critical points of a map are the zeros of its derivative. To prevent confusion, we will use the term {\em stationary point} to refer to the zero of a vector field.

\section{Acknowledgements.} We thank Peter Kronheimer, Robert Lipshitz, Max Lipyanskiy, Tim Perutz, and Matthew Stoffregen for helpful conversations.

\chapter{Morse homology in finite dimensions}
\label{sec:finite} 
In this chapter we describe several versions of Morse homology (all in finite dimensions), building up to the construction that will be needed in this monograph.

\section{The standard construction} \label{sec:standard}
Let $X$ be a closed, oriented Riemannian manifold equipped with a Morse function $f$. The Morse-Smale condition requires that for any two critical points $\x$ and $\y$, the unstable manifold $W^u_{\x}$ and the stable manifold $W^s_{\y}$ intersect transversely. Their intersection is then a smooth manifold $M(\x,\y)$, the moduli space of parameterized flows between $\x$ and $\y$, which has dimension $\index(\x) - \index(\y)$. 

With this information, we can build the {\em Morse complex} \cite{WittenMorse, BottMorse, FloerMorse}. The chain groups in degree $i$, denoted $C_i$, are freely generated over $\zz$ by the critical points of $f$ of index $i$. We define the differential $\partial$ as follows.  If $\x$ and $\y$ are critical points  of index $i$ and $i-1$ respectively, then $\breve{M}(\x,\y) := M(\x,\y)/\mathbb{R}$ is a compact, oriented $0$-manifold. (See Section~\ref{sec:or1} below for the construction of orientations.) We let $n(\x,\y)$ denote the number of unparameterized flow lines of $\nabla f$ from $\x$ to $\y$, counted with sign. The Morse differential is
\begin{equation}
\label{eq:delx}
 \del \x = \sum_{\{\y \mid \index(\y)=\index(\x)-1\} } n(\x,\y) \cdot \y.
 \end{equation}
 It turns out that $\partial^2 = 0$ and one can take the homology of this complex. The Morse homology is isomorphic to the singular homology, $H_*(X)$.  
 
For future reference, we mention a (well-known) alternative way of expressing the Morse-Smale condition. Let $f: M \to \R$ be Morse, and $x$, $y$ be two critical points. Define $\P(x, y)$ to be the space of smooth paths $\gamma: \R \to M$ with $\lim_{t \to -\infty} \gamma(t)=x$ and $\lim_{t \to +\infty} \gamma(t)=y$. The space $\P(x, y)$ has a well-defined $L^2_k$ completion $\P_k(x, y)$, which is a Banach manifold. (See, for example, \cite[Section 2.1]{Schwarz} for the case $k=1$.) For $0 \leq j \leq k$, we let $T_{j, \gamma}\P(x, y)$ be the $L^2_j$ completion of the tangent space to $\P_k(x, y)$ at $\gamma$. (If we fix $j$, these spaces are naturally identified for all $k \geq j$, so we do not include $k$ in the notation.) The moduli space $M(x, y)$ is the zero set of the section
$$ F: \P_k(x, y) \to T_{k-1}\P(x,y), \ \ \gamma \mapsto \frac{d\gamma}{dt} + (\nabla f)(\gamma).$$

Given a path $\gamma \in \P_k(x, y)$, we get the linearized operator
\begin{equation}
\label{eq:lgamma}
L_{\gamma}=(dF)_\gamma : T_{j, \gamma}\P(x, y) \to T_{j-1, \gamma}\P(x, y),
\end{equation}
for all $j$ with $1 \leq j \leq k$.

The definition of $L_{\gamma}$ involves taking derivatives of vector fields, i.e., covariant derivatives. To do this, we need to choose a connection on $TM$ in a neighborhood $U$ of the image of $\gamma$. This can be the Levi-Civita connection coming from the Riemannian metric or, as in \cite{AudinDamian}, the trivial connection induced from a trivialization of $TM$ over $U$. In either case, if we denote by $D$ the covariant derivative, and $w \in T_{j, \gamma}\P(x, y)$ is a vector field along $\gamma$, we have
\begin{equation}
\label{eq:Lgamma}
 L_{\gamma}(w) = \frac{Dw}{dt} + D(\nabla f)(w).
 \end{equation}
Note that if $D$ is the Levi-Civita connection, then the second term $D(\nabla f)(w)$ is the Hessian of $f$ applied to $w$ at $\gamma(t)$.

One can check that $L_{\gamma}$ is a Fredholm operator. When $\gamma \in M(x, y)$, the index of $L_{\gamma}$ is the expected dimension of the moduli space $M(x, y)$.

\begin{lemma}
\label{lem:Connections}
Suppose $k \geq 2$ and $1 \leq j \leq k$. Let $D$ and $D'$ be two connections on $TM$ in an open set $U$ containing the image of $\gamma \in \P_k(x, y)$.  Write $L_{\gamma}$ and $L'_{\gamma}$ for the operators \eqref{eq:lgamma} that correspond to $D$, resp. $D'$. Then:
\begin{enumerate}[(i)]
\item
The difference $L_{\gamma} - L'_{\gamma}$ is compact, so the operators $L_{\gamma}$ and $L'_{\gamma}$ have the same Fredholm index;
\item If $\gamma \in M(x, y)$, then $L_{\gamma} = L'_{\gamma}$.
\end{enumerate}
\end{lemma}

\begin{proof}
Let $A=D - D' \in \Omega^1(U; \End(TM))$. Then
\begin{equation}
\label{eq:lgdiff}
L_{\gamma} - L'_{\gamma} =  A\bigl (\frac{d\gamma}{dt} + (\nabla f)(\gamma) \bigr).
\end{equation}
Since $k \geq 2$ and $j \leq k$, we have a Sobolev multiplication $L^2_{j-1} \times L^2_{k-1} \to L^2_{j-1}$. From here we see that the difference $L_{\gamma} - L'_{\gamma}$ takes $T_{j-1, \gamma} \P(x, y)$ to itself. 

Consider a sequence of smooth bump functions $\beta_n: \R \to [0,1]$ such that $\beta_n$ is supported on $[-n-1, n+1]$ and identically $1$ on $[-n, n]$. The truncations
$$ \beta_n \cdot (L_{\gamma} - L'_{\gamma}) : T_{j, \gamma}\P(x, y) \to T_{j, \gamma}\P(x, y)$$
converge to $L_{\gamma} - L'_{\gamma}$ in operator norm. Furthermore, when precomposing these truncations with the inclusion of $T_{j, \gamma}\P(x, y)$ into $T_{j-1, \gamma}\P(x, y)$, we obtain compact operators. (Indeed, we can apply Rellich's lemma, because we restrict attention to a fixed compact interval.) Since the limit of compact operators is compact, we conclude that $L_{\gamma} - L'_{\gamma}$, as an operator from $T_{j, \gamma}\P(x, y)$ to $T_{j-1, \gamma}\P(x, y)$, is compact.

Part (ii) follows immediately from \eqref{eq:lgdiff}.
\end{proof}

\begin{lemma} \label{lem:MSlu}
Fix $j \geq 1$. Then, the function $f$ is Morse-Smale if and only if for any two critical points $x$ and $y$ of $f$, and for any gradient trajectory $\gamma \in M(x, y)$, the operator $L_{\gamma}$ from \eqref{eq:lgamma} is surjective.
\end{lemma}

\begin{proof}
See \cite[Theorem 10.1.5]{AudinDamian} for the case $j=1$. The main idea is that surjectivity of $L_{\gamma}$ is equivalent to injectivity of the formal adjoint $L^*_{\gamma}$ and, for any $t\in \R$, one can identify $\ker(L^*_{\gamma})$ with the orthogonal complement $(T_{\gamma(t)} W^u_x + T_{\gamma(t)}W^s_y)^{\perp}.$ Thus, $W^u_x$ and $W^s_y$ are transverse at $\gamma(t)$ if and only if $L^*_{\gamma}$ is injective. 

The same proof works for any $j \geq 1$. In fact, the kernel of $L^*_{\gamma} : T_{j, \gamma}\P(x, y) \to T_{j-1, \gamma}\P(x, y)$ consists of smooth configurations (as can be seen by a standard bootstrapping argument), and hence the kernel is  independent of $j$. 
\end{proof}

\section{Orientations}
\label{sec:or1}
We now discuss the construction of orientations in Morse homology. The traditional way is to choose orientations for $W^u_x$ for all critical points $x$. These induce orientations on $W^s_x$, and hence on the moduli spaces $M(x, y)$. When $\index(x)-\index(y)=1$, a trajectory $u \in M(x, y)$ is counted in the differential $\del$ with a sign $\pm 1$ depending on whether the orientation of $M(x, y)$ at $u$ coincides with the canonical orientation of $\R$.

An alternative way of constructing orientations, closer to Floer theory, is described in \cite[Chapter 3]{Schwarz}. Let us sketch this construction here. First, note that to every Fredholm operator $L$ we can associate a one-dimensional vector space, the determinant line
$$ \det(L) =  \Lambda^{\max} (\ker L) \otimes \Lambda^{\max}(\coker L)^*.$$
From the proof of Lemma~\ref{lem:MSlu} we see that the tangent bundle to $M(x, y)$ can be identified with the kernel of the operator $L_{\gamma}$ from \eqref{eq:Lgamma}, whereas the cokernel of $L_{\gamma}$ is trivial. Thus, orienting $M(x, y)$ is equivalent to orienting the determinant line bundle $\det(L) \to M(x, y)$, with fibers $\det(L_{\gamma})$ over $\gamma \in M(x, y)$.

Following \cite[Definition 2.1]{Schwarz}, let us equip $\bR = \R \cup \{\pm \infty \}$ with the structure of  a smooth manifold with boundary by requiring that the map $\bR \to [-1, 1], t \mapsto t/\sqrt{1+t^2}$ is a diffeomorphism. For every $x, y \in X$ (not necessarily critical points), we consider the space of smooth curves
$$ C^{\infty}_{x, y} = \{ \gamma \in C^{\infty}(\bR, X) \mid \gamma(-\infty)=x, \ \gamma(+\infty) = y\}.$$
For $\gamma \in C^{\infty}_{x, y}$, we let $\Sigma_{\gamma^*TX}$ be the space of Fredholm operators of the form
\begin{equation}
\label{eq:Kop}
 K = \frac{D}{dt} + A(t) : T_{1,\gamma} \P(x, y) \to T_{0,\gamma} \P(x, y),
 \end{equation}
where $A \in C^0(\rr; \End(\gamma^*TX))$ is a path of endomorphisms such that $K^{\pm} := A(\pm \infty)$ is non-degenerate and conjugated self-adjoint\footnote{In \cite{Schwarz}, for a real vector space $V$, an operator $A \in \End(V)$ is called conjugated self-adjoint if it is self-adjoint with respect to some scalar product. This condition is equivalent to $A$ being diagonalizable over $\R$.}. In particular, when $x$ and $y$ are critical points, the operator $L_{\gamma}$ from \eqref{eq:Lgamma} is of this form.

Consider the set of pairs $(\gamma, K)$, where $\gamma \in C^{\infty}(\bR, X)$ and $K \in \Sigma_{\gamma^*TX}$. Following \cite[Definition 3.7]{Schwarz}, two such pairs $(\gamma, K)$ and $(\zeta, L)$ are said to be equivalent if we have the asymptotical identities
$$ \gamma(\pm \infty) = \zeta(\pm \infty) \ \ \text{and}  \ \ K^{\pm} = L^{\pm}.$$
Let $\calE$ be the set of such equivalence classes. Thus, an element of $\calE$ is determined by two points $x, y \in X$, together with two non-degenerate and conjugated self-adjoint operators $K^- \in \End(T_xX)$ and $K^+ \in \End(T_yY)$.

It is proved in \cite[Section 3.2.1]{Schwarz} that, if $(\gamma, K)$ is equivalent to $(\zeta, L)$, there is a natural identification of the determinant lines $\det(K)$ and $\det(L)$. Hence, we can write $\det([\gamma, K])$ for a class $[\gamma, K] \in \calE$. As in \cite[Definition 3.15]{Schwarz}, we define a {\em coherent orientation} to be a map $\sigma$ that associates an orientation of $\det([\gamma, K])$ to any $[\gamma, K] \in \calE$, in a way compatible with concatenation of paths and operators; that is,
\begin{equation}
\label{eq:gluingOr}
 \sigma[\gamma, K] \# \sigma[\zeta, L] = \sigma[\gamma \# \zeta, K \# L],
 \end{equation}
whenever $\gamma(+\infty) = \zeta(-\infty)$ and $K^+ = L^-$. We denote by $\#$ the natural concatenation operator.

Proposition 3.16 in \cite{Schwarz} shows that coherent orientations exist. Concretely, a coherent orientation can be specified by choosing a basepoint $x_0 \in X$, some non-degenerate and conjugated self-adjoint operator $A \in \End(T_{x_0}X)$, choosing the trivial orientation for the class $[x_0, K_0]$, where $x_0$ is the constant path at $x_0$ and $K_0 = \frac{d}{dt} + A$, and finally choosing arbitrary orientations $\sigma[\gamma, K]$ for all classes $[\gamma, K] \in \calE$ with $[\gamma, K] \neq [x_0, K_0]$ and $\gamma(-\infty) = x_0$. Once this set of data is fixed, the orientations for the other classes in $\calE$ are determined by the gluing condition \eqref{eq:gluingOr}.

In particular, a coherent orientation gives orientations of the determinant line bundles $\det(L) \to M(x, y)$, and hence of the moduli spaces $M(x, y)$ themselves. These orientations can be used to define the Morse differential as in \eqref{eq:delx}. Furthermore, it is proved in \cite[Appendix B]{Schwarz} that the coherent orientation can be chosen so that the orientations on $M(x, y)$ coincide with those coming from orienting the unstable manifolds. Therefore, the resulting Morse complex is the same as in the classical definition.

We now introduce a third way of defining orientations for the Morse complex. This is a variant of Schwarz's construction, but closer to what Kronheimer and Mrowka do for monopole Floer homology in \cite[Section 20]{KMbook}. Rather than considering all Fredholm operators of the form $\frac{D}{dt} + A(t)$, we focus on the operators $L_{\gamma}$ as in \eqref{eq:Lgamma}, but with $x, y \in X$ and $\gamma \in \P(x, y)$ arbitrary. However, note that $L_{\gamma}$ is only Fredholm when the Hessian $\Hess(f)=D(\nabla f)$ is non-degenerate at its endpoints $x, y$. This is not true in general, so let us instead consider compact intervals $I=[t_1, t_2] \subset \R$, spaces of paths 
$$\P^I(x, y) = \{\gamma \in C^{\infty}(I, X) \mid \gamma(t_1)=x,\ \gamma(t_2)=y \},$$
and their $L^2_k$ completions $\P_k^I(x, y)$. At each $x \in X$, the Hessian $\Hess(f)_x$ is self-adjoint, and gives a decomposition of $T_x X$ into the span of the nonpositive and positive eigenspaces:
$$ T_x X = H^-_x \oplus H^+_x.$$
In other words, we are using the spectral decomposition of the non-degenerate self-adjoint operator $\Hess(f)_x - \epsilon$, for $\epsilon > 0$ small. 

Let $\Pi^-_x$ and $\Pi^+_x$ be the orthogonal projections onto $H^-_x$ and $H^+_x$, respectively.
For any $\gamma \in \P^I_k(x, y)$, we define a Fredholm operator
\begin{align}
\label{eq:tLgamma}
 \tL_{\gamma} &: T_{1,\gamma}\P^I(x, y) \to T_{0,\gamma} \P^I(x, y) \oplus H^+_x \oplus H^-_y \\
w &\mapsto \Bigl ( \frac{Dw}{dt} + \Hess(f)(w), -\Pi_x^+(w(t_1)), \Pi_y^+(w(t_2)) \Bigr).
\notag 
\end{align}

One can turn $\tL_{\gamma}$ into an operator of the form \eqref{eq:Kop} as follows. Pick a smooth map $f: I \to I$ such that $f' \geq 0$ and 
$$ f(t_1+ s) = t_1, \ \ f(t_2 - s) = t_2$$
for all $s \in [0, \delta]$, for some small $\delta > 0$. Then, consider the extended path
$$ \gamma^{\ext}: \bR \to X, \ \ \gamma^{\ext}(t) = \begin{cases} 
x & \text{for} \ t < t_1,\\
\gamma(f(t)) & \text{for} \ t \in I,\\
y & \text{for} \ t > t_2.
\end{cases} $$

Define a Fredholm operator $\tL^{\ext}_{\gamma} \in \Sigma_{(\gamma^{\ext})^*TX}$ by
$$ \tL^{\ext}_{\gamma}= \frac{D}{dt} + \Hess(f)_{\gamma^{\ext}(t)} - \epsilon, $$
for $\epsilon > 0$ small. Standard deformation and concatenation arguments show that the operators $\tL_{\gamma}$ and $\tL^{\ext}_{\gamma}$ have the same index, and that orienting $\det(\tL_{\gamma})$ is equivalent to orienting $\det(\tL^{\ext}_{\gamma})$.

Let $\Lambda_{\gamma}(x, y)$ be the two-element set consisting of the orientations of $\det(\tL_{\gamma})$ or, equivalently, of $\det(\tL^{\ext}_{\gamma})$. For fixed $x, y$ but different $\gamma \in \P_k(x, y)$, the pairs $(\gamma^{\ext}, \tL^{\ext}_{\gamma}) \in \calE$ are equivalent. Hence, by the results of \cite[Section 3.2.1]{Schwarz}, we have a canonical identification between the different sets $\Lambda_{\gamma}(x, y)$. Consequently, we can drop $\gamma$ from the notation and write $\Lambda(x, y)$ for any $\Lambda_{\gamma}(x, y)$.

By considering concatenation of paths, we obtain a natural composition map
\begin{equation}
\label{eq:composeorientations}
\Lambda(x, y) \times \Lambda(y, z) \to \Lambda(x, z).
\end{equation}

\begin{definition}
\label{def:sco}
A {\em specialized coherent orientation} $o$ consists of choices of elements $o_{x, y} \in \Lambda(x, y)$, one for each $x, y \in X$, such that 
$$o_{x, y} \cdot o_{y, z} = o_{x, z}$$
for all $x, y, z \in X$, where the multiplication is with respect to \eqref{eq:composeorientations}.
\end{definition} 

A specialized coherent orientation can be constructed as follows. We choose a critical point $x_0$ of $f$, and consider the operator $\tL_{x_0}$ for the constant path at $x_0$. There is a canonical identification $\det(\tL_{x_0}) \cong \R$, and hence from $o_{x_0, x_0} \cong \{\pm 1\}$ we can choose the element $+1$. Then, we choose arbitrary elements of $o_{x_0, x}$ for all other $x \in X$, and we obtain elements in any $o_{x, y}$ using the concatenation rule.

Using the identification between $\det(\tL_{\gamma})$ and $\det(\tL^{\ext}_{\gamma})$, we see that a coherent orientation gives rise to a specialized coherent orientation. Conversely, any specialized coherent orientation can be extended to a coherent orientation. To check this last statement, consider the construction of a specialized coherent orientation from the choices of $x_0$ and elements in $o_{x_0, x}$, as above. This is a subset of the data needed to give a coherent orientation. Indeed, letting $K_0 = \frac{d}{dt} + \Hess(f)_{x_0}$, we see that from the elements of $o_{x_0, x}$ we get orientations $\sigma[\gamma, K]$ for paths $\gamma$ starting at $x_0$ and ending at $x$, and for $K$ with $K^+=H^+_x$. Let us now also choose orientations $\sigma[\gamma, K]$ for  $\gamma \in \P_k(x_0,x)$ but $K^+ \neq H^+_x$. This gives the necessary data to construct a coherent orientation.

Finally, note that a specialized coherent orientation gives rise to orientations on the moduli spaces $M(x, y)$ in the following way. Given $\gamma \in M(x, y)$, pick $I=[t_1, t_2] \subset \R$ large enough so that the operators $\Hess(f)_{\gamma(t)}$ are non-degenerate for all $t \in (-\infty, t_1] \cup [t_2, \infty)$. Consider the restriction of $\gamma$ to the interval $I$, and the corresponding operator $\tL_{\gamma|_I}$. The orientation on $\det(\tL_{\gamma|_I})$ gives an orientation on $\det(L_{\gamma})$, by a concatenation process similar to the one in \cite[Section 20.4]{KMbook}. In turn, this orients the tangent space $T_{\gamma}M(x, y)$. 

We conclude that a specialized coherent orientation $o$ can be used to give signs in the Morse complex. When $o$ is induced by a coherent orientation $\sigma$, the signs coming from $o$ are the same as those coming from $\sigma$. Thus, the three ways of producing orientations, which we have described in this section, all produce the same Morse complex.

\section{Quasi-gradient vector fields}
\label{sec:quasi}
The purpose of this subsection is to extend Morse-Smale theory to a class of vector fields that are not gradients. Note that among all smooth vector fields on a manifold, gradient vector fields are rather special---even when allowing both the metric and the function to vary. For example, the differential of a gradient vector field at a critical point has only real eigenvalues (because it is a symmetric operator with respect to the given metric). This is also true for the gradient-like (also called pseudo-gradient) vector fields that are sometimes used in Morse theory, e.g., in \cite{MilnorHCob, AudinDamian}. We want to relax this condition, and allow for the stationary points of our vector field to only be hyperbolic, in the following sense:

\begin{definition}
Let $v$ be a smooth vector field on a manifold $X$. A stationary point $x$ of $v$ is called {\em hyperbolic} if (the complexification of) the derivative $(dv)_x: T_xX \to T_xX$ has no purely imaginary eigenvalues.
\end{definition}

We now introduce the class of vector fields we want to work with.

\begin{definition}
\label{def:qg}
Let $X$ be a smooth manifold. A smooth vector field $v$ on $X$ is called {\em Morse quasi-gradient} if the following conditions are satisfied: 
\begin{enumerate}[(a)]
\item All stationary points of $v$ are hyperbolic.
\item There exists a smooth function $f: X \to \R$ such that $df(v) \geq 0$ at all $x \in X$, with equality holding if and only if $x$ is a stationary point of $v$.
 \end{enumerate}
\end{definition}

\begin{example}
The gradient vector field of a Morse function $f$ has hyperbolic stationary points, since the eigenvalues of $(dv)_x$ are real and non-zero; further, the Morse function $f$ can be used in condition (b).
\end{example}

For future reference, it is worth pointing out the following.

\begin{lemma}
\label{lem:qgCritical}
Let $v$ be a Morse quasi-gradient vector field on a smooth manifold $X$, and let $f$ be a function as in part (b) of Definition~\ref{def:qg}. Then, any stationary point of $v$ is a critical point of $f$.
\end{lemma}

\begin{proof}
Let $x$ be a stationary point of $v$. Since $x$ is a minimum of the function $df(v)$, we have $d(df(v))_x = 0$. For $y \in T_xX,$ we get 
$$ 0= d(df(v))_x (y) = (d^2f)_x(v_x, y) + (df)_x((dv)_x(y)).$$
The first term in the last expression is zero, because $v_x = 0$. Further, since $x$ is hyperbolic, we have that $(dv)_x$ is an automorphism of $T_xX$, hence $(dv)_x(y)=z$ can take any value in $T_x X$. We conclude that $(df)_x(z)=0$ for all $z \in T_x X $, so $x$ is a critical point of $f$.
\end{proof}

For the rest of this subsection we will assume that $X$ is closed and oriented. Let $v$ be a Morse quasi-gradient vector field on $X$. We seek to establish several properties that $v$ has in common with gradients of Morse functions.
 
Let $\Crit$ be the set of stationary points of $v$. If $x \in \Crit$, hyperbolicity implies that we have a decomposition
$$ T_xX = T_x^sX \oplus T_x^uX,$$
where $T_x^sX$ and $T_x^uX$ are invariant subspaces for $L=(dv)_x$, the eigenvalues of $L|_{T_x^sX}$ have negative real part, and the eigenvalues of $L|_{T_x^uX}$ have positive real part. (See \cite[Proposition 2.8]{PalisMelo} for a proof.) We define the {\em index} of $x$ to be the dimension of $T_x^uX$.

The local structure of flows near hyperbolic singularities is well-understood; a good reference is \cite[Chapter 2]{PalisMelo}. Theorem 4.10 in \cite{PalisMelo} says that around every $x \in \Crit$ there is a local coordinate chart $\N_x$ such that the restriction of $v$ to $\N_x$ is conjugate to a linear flow (under a homeomorphism). This implies that $x$ is the only stationary point inside $\N_x$. Since $X$ is compact, we deduce that the set $\Crit$ is finite.

We can arrange so that the neighborhoods $\N_x$ are disjoint for different $x$. By part (b) of Definition~\ref{def:qg}, there is an $\epsilon > 0$ such that 
\begin{equation}
\label{eq:eeps}
df(v)(p) \geq \epsilon, \ \ \forall \ p \in X - \bigcup_{x \in \Crit} \N_x.
\end{equation}

Next, consider the flow on $X$ associated to $v$. Let $\gamma: \R \to X$ be a flow line, i.e., $d\gamma/dt + v(\gamma(t))=0.$ We have:
\begin{equation}
\label{eq:fgamma}
 \frac{d}{dt}f(\gamma(t))= -df_{\gamma(t)} \circ v(\gamma(t)) \leq 0,
 \end{equation}
with equality only if $\gamma(t)$ is stationary. Therefore, $f$ decreases (strictly) along non-stationary flow lines.

\begin{lemma}
\label{lem:ao}
Let $v$ be a Morse quasi-gradient vector field on a closed, oriented, smooth manifold $X$. Let $\gamma: \R \to X$ be a flow line of $v$. Then, $\lim_{t\to -\infty} \gamma(t)$ and $\lim_{t\to +\infty} \gamma(t)$ exist, and they are both stationary points of $v$. 
\end{lemma}

\begin{proof}
This is similar to the Morse gradient case; cf. \cite[Proposition 3.19]{BanyagaHurtubise}. Let $\N_x$ be the neighborhoods of $x \in \Crit$ constructed above, and $\epsilon > 0$ be such that \eqref{eq:eeps} holds.

Compactness of $X$ implies that $\gamma(t)$ is defined for all $t \in \R$, and $f \circ \gamma : \R \to \R$ is bounded. (Here, $f$ is the function in part (b) of Definition~\ref{def:qg}.) From Equation~\eqref{eq:fgamma} we see that
$$ \lim_{t \to \pm \infty} -df(v)(\gamma(t))= \lim_{t \to \pm \infty}  \frac{d}{dt}f(\gamma(t))=0.$$
In view of \eqref{eq:eeps}, this shows that for $t \ll 0$, the flow line $\gamma(t)$ is contained in the chart $\N_x$ for some $x$. Further, if $t_n \to -\infty$, then at any accumulation point of the sequence $\{\gamma(t_n)\} \subseteq X$ we must have $df(v)=0$. Hence, the only accumulation point (which exists by the compactness of $X$) must be the stationary point $x$. We deduce that $\lim_{t\to -\infty} \gamma(t)=x$. A similar argument applies to $\gamma(t)$ as $t\to +\infty$.
\end{proof}

Hyperbolicity of stationary points implies the existence of stable and unstable manifolds $W^s_x, W^u_x$, cf. the Stable Manifold Theorem \cite[Theorem 6.2]{PalisMelo}. These are the images of injective immersions
$$ E^s : T_x^sX \to W^s_x \subseteq X,$$
$$ E^u : T_x^sX \to W^u_x \subseteq X.$$
If we did not have condition (b) in Definition~\ref{def:qg}, the maps $E^s$ and $E^u$ may not be embeddings, and $W^s_x, W^u_x$ may not be manifolds. (See \cite[p.74]{PalisMelo} for an example.) However, for Morse quasi-gradients, $E^s$ and $E^u$ are homeomorphisms onto their images, and hence embeddings. The proof of this fact can be taken verbatim from the Morse gradient case; cf. Lemma 4.10 in \cite{BanyagaHurtubise}. Indeed, the main input in that proof is the existence of limits of flow lines at $\pm \infty$; this was established for quasi-gradients in Lemma~\ref{lem:ao}.

Thus, for every $x \in \Crit$, we have stable and unstable submanifolds $W^s_x, W^u_x \subseteq M$, of dimensions $\dim(X)-\ind(x)$ and $\ind(x)$, respectively. We can then formulate the Morse-Smale condition as in the gradient case: 

\begin{definition}
\label{def:Mqv}
A Morse quasi-gradient vector field $v$ is called {\em Morse-Smale} if for all stationary points $x$ and $y$, the unstable manifold $W^u_{\x}$ and the stable manifold $W^s_{\y}$ intersect transversely. 
\end{definition}

Morse-Smale quasi-gradients are a particular example of the {\em Morse-Smale vector fields}, which play an important role in the theory of dynamical systems. We recall their definition below. See \cite[Chapter 5]{PalisMelo} for an introduction to the subject. 

\begin{definition}[\cite{PalisMelo}, p.118-119]
\label{def:MS}
Let $X$ be a closed, smooth manifold. A smooth vector field $v$ on $X$ is called {\em Morse-Smale} if the following conditions are satisfied:
\begin{enumerate}[(a)]
\item $v$ has a finite number of critical elements (stationary points and closed orbits), all of which are hyperbolic.\footnote{See \cite[p.95]{PalisMelo} for the definition of hyperbolicity for closed orbits. Stable and unstable manifolds for closed orbits are defined in the obvious way; see \cite[p.98]{PalisMelo}.}
\item If $\sigma_1$ and $\sigma_2$ are critical elements of $v$, then the unstable manifold of $\sigma_1$ is transverse to the stable manifold of $\sigma_2$;
\item The set of nonwandering points for $v$ is equal to the union of the critical elements of $v$.
\end{enumerate}
Here, $x \in X$ is called a {\em wandering point} for $v$ if there exists a neighborhood $U$ of $x$ and a number $t_0 > 0$ such that $\Phi_t(U) \cap U = \emptyset$ for $|t| > t_0$, where $\{\Phi_t\}_{t \in \R}$ is the flow generated by $v$. Otherwise, we say that $x$ is {\em nonwandering}.
\end{definition}

Morse-Smale vector fields satisfy structural stability, that is, a small perturbation of a Morse-Smale dynamical system is conjugate to the original system via a homeomorphism. 

\begin{lemma}
\label{lem:isMS}
A Morse-Smale quasi-gradient vector field, in the sense of Definition~\ref{def:Mqv}, is a Morse-Smale vector field in the sense of Definition~\ref{def:MS} and, furthermore, it has no closed orbits.
\end{lemma}

\begin{proof}
Closed orbits do not exist because $f$ decreases along flow lines. Parts (a) and (b) of Definition~\ref{def:MS} are immediate consequences of Definition~\ref{def:Mqv}. Part (c), the fact that nonstationary points are wandering, follows since $f$ decreases along flow lines.
\end{proof} 
 
Suppose $v$ is a Morse-Smale quasi-gradient field. For every $x, y \in \Crit$, the intersection 
$$M(x,y):= W^u_x \cap W^s_y$$ is transverse, and hence a manifold of dimension $\index(x) - \index(y)$. We can construct a Morse complex $(C_*, \del)$ as in the gradient case. The groups $C_*$ are generated by the stationary points, the grading is determined by the index, and the differential $\del$ is given by the formula \eqref{eq:delx}.

\begin{proposition} 
\label{prop:MSquasi}
Let $v$ be a Morse-Smale quasi-gradient vector field on a closed, oriented, smooth manifold $X$. Let $(C_*, \del)$ be the corresponding Morse complex. Then $\del^2=0$, and the homology $H_*(C)$ is isomorphic to $H_*(M)$.
\end{proposition}

\begin{proof}
In the gradient case, one way to prove that Morse homology recovers singular homology is by using the Conley index; cf. \cite{FloerMorse, SalamonMorse, Franzosa, McCord} or \cite[Chapter 7]{BanyagaHurtubise}. (See Section~\ref{subsec:ConleyMorse} below for the definition of the Conley index.) The basic idea is to find a Morse decomposition of the flow into attractor-repeller pairs. This produces a filtration of $X$ by index pairs, such that the quotients are the Conley indices of the stationary points $x \in \Crit$. These quotients are homotopy equivalent to spheres of dimensions $\index(x)$, and the connecting homomorphisms between their homology groups are given by the signed count of flow lines.

This proof can be adapted to the setting of Morse-Smale quasi-gradients. Indeed, one can construct Morse decompositions and connection matrices for general flows, as in the work of Franzosa \cite{Franzosa}. Further, Theorem 3.1 in \cite{McCord} describes the neighborhood of a flow line between two hyperbolic points in a Morse-Smale (not necessarily gradient) flow. This shows that the connection  matrices are given by counts of flow lines, i.e. that each gradient flow line contributes plus or minus $1$ to the boundary operator. Orientations can be fixed as in \cite{FloerMorse, SalamonMorse}.
\end{proof}

The rest of the discussion in Section~\ref{sec:standard} extends to quasi-gradients as well. The Morse-Smale condition can be phrased in terms of the surjectivity of the operator
$$  L_{\gamma}(w) = \frac{Dw}{dt} + D(v)(w), $$
as in Lemma~\ref{lem:MSlu}. Furthermore, the discussion from Section~\ref{sec:or1} carries over to this setting, so the moduli spaces $M(x,y)$ can be oriented using coherent orientations or specialized coherent orientations. The only caveat is that in the definition of coherent orientations, when we define the sets $\Sigma_{\gamma^*TX}$, we should allow the operators $K^{\pm}$ to be hyperbolic, rather than conjugated self-adjoint.

\section{Morse homology for isolated invariant sets}  
\label{subsec:ConleyMorse}
Suppose $X$ is a (possibly non-compact) smooth, oriented manifold equipped with a Morse quasi-gradient vector field $v$. In general, there is no way of defining a Morse complex. Even if we assume the Morse-Smale condition on the unstable and stable manifolds, the resulting differential may not square to zero, because flow lines can escape to infinity.

Nevertheless, we have a Morse complex in a certain setting, described below. First, we recall the definitions of an isolated invariant set, index pair, and Conley index, following \cite{ConleyBook}. Fix a one-parameter subgroup $\{\Phi_t \}$ of diffeomorphisms of a manifold $X$.  Given a compact subset $A \subseteq X$, define the compact subset 
\[
\Inv (A, \Phi) = \{x \in A \mid \Phi_t (x) \in A \text{ for all } t \in \mathbb{R} \}.
\]

\begin{definition}
\label{def:iis}
A compact set $\is \subseteq X$ is called an {\em isolated invariant set} if there exists a compact $A \subseteq X$ with $\is = \Inv(A, \Phi) \subseteq \inte(A)$. Such a set $A$ is called an {\em isolating neighborhood} for $\is$.
\end{definition}

\begin{definition}
\label{def:indexpair}
An {\em index pair} $(N,L)$ for $\is$ consists of compact sets $L \subseteq N \subseteq X$ such that \\
a) $\Inv (N - L, \Phi) = \is \subset \inte(N - L)$, \\
b) for all $x \in N$, if there exists $t >0$ such that $\Phi_t(x)$ is not in $N$, there exists $0 \leq \tau < t$ with $\Phi_\tau(x) \in L$ $(L$ is an {\em exit set} for $N)$, \\
c) given $x \in L$, $t>0$, if $\Phi_s(x) \subseteq N$ for all $0 \leq s \leq t$, then $\Phi_s(x)$ is in $L$ for $0 \leq s \leq t$ $(L$ is {\em positively invariant in $N$)}. 
\end{definition}

Conley \cite{ConleyBook} showed that any isolated invariant set $\is$ admits an index pair. The {\em Conley index} for an isolated invariant set $\is$, denoted $I(\Phi, \is)$, is defined to be the pointed space $(N/L,[L])$.  Its pointed homotopy type is an invariant of the triple $(X,\Phi_t,\is)$. In fact, the Conley index is invariant under continuous deformations of the flow, as long as $\is$ remains isolated in a suitable sense.

Going back to the case of a Morse-Smale quasi-gradient flow on $X$, suppose $\is \subseteq X$ is an isolated invariant set. Let $\Crit[\is]$ be the set of stationary points of $v$ that lie in $\is$. The set $\Crit[\is]$ is finite, because $\is$ is compact. For any $\x, \y \in \Crit[\is]$, if a point of $\is$ is on a  flow line $\gamma \in \breve{M}(x,y)$, then all the points on that flow line must be in $\is$. Further, the subsets 
$$\breve{M}[\is](\x, \y) = \{\gamma \in \breve{M}(\x,\y) \mid \gamma \subset \is\}$$
are both open and closed in $\breve{M}(\x,\y)$. We have
$$ \is = \bigcup_{\substack{\x, \y \in \Crit[\is] \\  \gamma \in \breve{M}[\is](\x, \y)}} \gamma.$$

We can define a Morse complex $C[\is]$ to be freely generated by the points in $\Crit[\is]$, with the differential given by counting only the flow lines in $\breve{M}[\is](\x, \y)$. We still have $\del^2=0$, and we obtain a Morse homology $H_*(C[\is])$.

The following theorem was proved by Floer in \cite{FloerMorse} in the setting of Morse-Smale gradient flows. The extension to the quasi-gradient case goes along the same lines as the proof of Proposition~\ref{prop:MSquasi}. (In fact, the results of Franzosa and McCord used in that proof were originally phrased for more general flows.)

\begin{theorem}[cf. \cite{FloerMorse}]
Given an isolated invariant set $\is$ in a Morse-Smale quasi-gradient flow $\phi$ on a manifold $X$, the Morse homology $H_*(C[\is])$ is isomorphic to the reduced homology of the Conley index of $\is$.
\end{theorem}

Thus, if $(N, L)$ is an index pair for $\is$ and $L \subseteq N$ is a neighborhood deformation retract (which can easily be arranged), then the Morse homology is isomorphic to the relative homology $H_*(N, L)$.

\section{Morse homology for manifolds with boundary}
\label{subsec:morseboundary}
Let $X$ be a smooth, oriented, compact manifold with boundary.  We proceed to describe the construction of  a Morse complex, $(\check{C}(X),\check{\partial})$, whose homology is isomorphic to $H_*(X;\zz)$ (we will ignore the analogous constructions for $H_*(\del M; \zz)$ or $H_*(X,\partial X; \zz)$).  For gradient flows, this construction appeared in print in \cite[Section 2.4]{KMbook}, although it may have been known before. See also \cite{Laudenbach}, \cite{Ranicki} for related work. Our exposition follows \cite{KMbook} closely, except we use the more general quasi-gradient flows.

Fix a metric $g$ and a Morse quasi-gradient vector field $v$ on $X$ such that $g$ and $v$ are respectively the restrictions of a metric and a vector field on the double of $X$, which are invariant under the obvious involution.  Among the stationary points of $v$, there are those which occur in the interior of $X$; their set is denoted $\Crit^o$.  The others are stationary points on $\del X$. If $x \in \del X$ is such a point, note that the normal vector to the boundary $N_x$ is taken to its opposite under the involution, and the same holds for $dv(N_x)$. Hence, $N_x$ is an eigenvector of $dv$, so it must live in either $T_x^sX$ or $T_x^uX$. In the first case we say that $x$ is {\em boundary-stable}, and in the second that it is {\em boundary-unstable}. The set of boundary-stable stationary points is denoted $\Crit^s$, and the set of boundary-unstable stationary points is denoted $\Crit^u$.    

We require that $v$ satisfy the usual Morse-Smale condition, except in the so-called  {\em boundary-obstructed} case, when $\x$ is boundary-stable and $\y$ is boundary-unstable.  In the boundary-obstructed case, $W^u_\x$ and $W^s_\y$ are subsets of $\partial X$, so transversality in $X$ is impossible.  Therefore, we instead require that they intersect transversely in $\partial X$.  When this happens, the dimension of $M(\x,\y)$ is in fact $\index(\x) - \index(\y) + 1$.  

For any $\x,\y \in \Crit=\Crit^o \cup \Crit^s \cup \Crit^u$, we let $M^\partial(\x,\y) = M(\x,\y) \cap \partial X$ and $\breve{M}^\partial(\x,\y) = M^\partial(\x,\y)/\mathbb{R}$.    One can induce the same orientations on the moduli spaces as in the usual closed setting, except in the boundary-obstructed case, where we must also make use of an outward normal vector.  

We let $C^o$, $C^s$, and $C^u$ be the free Abelian groups generated by $\Crit^o$, $\Crit^s$, and $\Crit^u$ respectively.  The chain groups we work with will be $\check{C}(X) = C^o \oplus C^s$.  Although these chain groups do not incorporate the boundary-unstable stationary points, the differential does.  The way that the flows are counted varies from the usual construction.  We define two sets of maps 
\[
\partial^\theta_\varpi ,\bar{\partial}^\theta_\varpi: C^\theta \to C^\varpi
\]
for various pairs of $\theta, \varpi \in \{o,u,s\}$.  
If $(\theta, \varpi) \in \{(o,o), (o,s), (u,o),(u,s)\}$ and $\x \in \Crit^\theta$, define  
\[
\partial^\theta_\varpi \x  = \sum_{\y \in \Crit^\varpi } n(\x,\y) \y,
\]  
where the summation is again over those $\y$ with $\dim \breve{M}(\x,\y)=0$. 

If $\theta, \varpi \in \{s,u\}$ and $\x \in \Crit^\theta$, we instead define $\bar{\partial}^\theta_\varpi$ to count flows on $\partial X$: 
\[
\bar{\partial}^\theta_\varpi \x = \sum_{\y \in \Crit^\varpi} \bar{n}(\x,\y) \y,
\]
where $\bar{n}(\x,\y)$ is the signed count of points (unparameterized flows) in zero-dimensional moduli spaces $\breve{M}^\partial(\x,\y)$.  

On $\check{C}$, we define $\check{\partial}$ by 
\begin{equation}
\label{eq:delcheck}
\check{\partial} = \begin{bmatrix} \partial^o_o & - \partial^u_o \bar{\partial}^s_u \\ \partial^o_s & \bar{\partial}^s_s - \partial^u_s \bar{\partial}^s_u \end{bmatrix}.
\end{equation}

\begin {theorem}[Theorem 2.4.5 in \cite{KMbook}]
We still have $\check{\partial}^2 = 0$.  Furthermore, $H_*(\check{C},\check{\partial}) \cong H_*(X)$.  
\end{theorem}

\begin{remark}\label{rmk:noMred}
In this book, since $\bar{\partial}^u_s$ does not appear in $\check{\partial}$, we will not need the notation $M^\partial(x,y)$, since either $M^\partial(x,y)$ agrees with $M(x,y)$ or is empty, except when $x$ is boundary-unstable and $y$ is boundary-stable.  In \cite{KMbook}, the term $\bar{\partial}^u_s$ is needed for the proof that $\check{\partial}^2 = 0$.  We keep the notation and discussion for consistency with \cite{KMbook}.
\end{remark}

\section{Morse homology for manifolds with circle actions}
\label{subsec:CircleMorse}
A circle action is called {\em semifree} if it is free on the complement of the fixed points. For manifolds with semifree $S^1$-actions, Kronheimer and Mrowka showed that the methods in the previous subsection can be applied to obtain Morse theoretic approximations to the $S^1$-equivariant homology. We sketch their arguments here, following \cite[Sections 2.5--2.6]{KMbook}, but phrasing everything in terms of quasi-gradients rather than gradients.

Suppose that a closed Riemannian manifold $X$ has a smooth, semifree $S^1$-action by isometries, and let $Q$ denote the fixed point set. Let $N(Q)$ be the normal bundle of $Q$, and $N^1(Q)$ the unit normal bundle. Observe that the $S^1$-action gives $N(Q)$ the structure of a complex vector bundle. In order to make $X/S^1$ into a manifold, we resolve the singularity at $Q$ by considering the  blow-up 
\[
X^\sigma  = (X - Q) \cup \bigl( N^1(Q) \times [0,\epsilon) \bigr),
\]
where we have identified $N^1(Q) \times (0,\epsilon)$ with $N(Q) - Q \subset X$.  Note that $X^\sigma/S^1$ is a smooth manifold with boundary $N^1(Q)/S^1$.

A smooth, $S^1$-invariant vector field $\tilde v$ on $X$ induces a vector field $v$ on $(X - Q)/S^1$. In turn, this extends naturally to a smooth vector field $v^\sigma$ on  $X^\sigma/S^1$; cf. \cite[Lemma 2.5.2]{KMbook}. Note that $\tilde v$ must be tangent to $Q$, and hence $v^{\sigma}$ is tangent to the boundary.

Precisely, the dynamics of $v^{\sigma}$ on the boundary $N^1(Q)/S^1 = \del(X^\sigma/S^1)$ can be described as follows. A point on the boundary can be written as $(q,[\phi])$ with $q \in Q$ and $\phi \in N^1_q(Q)$. As discussed, $N_q(Q)$ has a complex structure; we let $\langle \phi \rangle^{\perp}$ be the complex orthogonal complement to $\phi$ in $N_q(Q)$. Decompose the tangent space to $X^{\sigma}/S^1$ at $q$ as
\begin{equation}
\label{eq:decomposes}
 T_qQ \oplus \langle \phi \rangle^{\perp} \oplus \R,
 \end{equation}
where $\R$ is the direction normal to the boundary. 

The covariant derivative $(\nabla \tilde v)_q : T_qX \to T_qX$ is $S^1$-equivariant, and hence takes the normal direction $N_qQ \subset T_qX$ to $N_qQ$. Let 
\begin{equation}
\label{eq:lq}
L_q :=  (\nabla \tilde v)|_{N_qQ}.
\end{equation}

With respect to the decomposition \eqref{eq:decomposes}, let us write
$$
 v^{\sigma}(q, [\phi])=( \tilde v(q), \Ell_q \phi, 0).
$$
Here, when $\tilde v(q)=0$, the second term equals
$$ \Ell_q \phi = L_q\phi -  \Re \langle \phi, L_q\phi\rangle\phi.$$

Thus, the stationary points of $v^{\sigma}$ on the boundary are the pairs $(q, [\phi])$, where $q$ is a zero of $\tilde v|_Q$ and $\phi$ is an eigenvector of $L_q$. 

Furthermore, the flow associated to $v^{\sigma}$ on the boundary is given by the equations:
\begin{align}
\label{eq:lq0}
\frac{dq}{dt} + \tilde v(q(t)) &= 0,\\
q^*(\nabla) + \bigl( L_{q(t)} -  \Re \langle \phi(t), L_{q(t)}\phi(t)\rangle\phi(t) \bigr) dt &=0, \label{eq:lqx}
\end{align}
where $|\phi(t)|=1$ for all $t$.

\begin{definition}
\label{def:eqgv}
A smooth, $S^1$-invariant vector field $\tilde v$ on $X$ is called a {\em Morse equivariant quasi-gradient} if the following conditions are satisfied: 
\begin{enumerate}[(a)]
\item All stationary points of $v$ on $(X - Q)/S^1$ are hyperbolic.
\item All stationary points of $\tilde v|_Q$ are hyperbolic.
\item At each stationary point $q$ of $\tilde v|_Q$, the operator $L_q: N_qQ \to N_qQ$ is self-adjoint, and admits a basis of eigenvectors $\phi_1(q), \phi_2(q), \dots, \phi_n(q)$ with corresponding eigenvalues $\lambda_1(q), \lambda_2(q), \dots, \lambda_n(q)$ such that
$$ \lambda_1(q) <  \lambda_2(q) < \dots <  \lambda_n(q)$$
and $\lambda_i(q) \neq 0$ for any $i$.
\item There exists a smooth, $S^1$-equivariant function $\tilde f: X \to \R$ such that $d\tilde f(\tilde v) \geq 0$ at all $x \in X$, with equality holding if and only if $\tilde v(x)=0$.
 \end{enumerate}
\end{definition}

\begin{lemma}
\label{lem:MEquiv}
Parts (a), (b) and (c) in Definition~\ref{def:eqgv}, taken together, are equivalent to asking for all the stationary points of $v^{\sigma}$ on $X^{\sigma}/S^1$ to be hyperbolic, and for those on the boundary to give rise to operators $L_q$ that are self-adjoint. 

Furthermore, if $\tilde v$ is as in Definition~\ref{def:eqgv}, and $(q, [\phi_i(q)])$ is a stationary point of $v^{\sigma}$ on the boundary, then its index is given by
\[
\index(q, [\phi_i(q)])  = 
\begin{cases}
\index_Q(q) + 2i-2 & \text{ if }  \lambda_i(q) > 0,\\
\index_Q(q) + 2i-1 & \text{ if }  \lambda_i(q) < 0.
\end{cases}
\]
\end{lemma}

\begin{proof}
Away from the boundary, a stationary point of $v^{\sigma}$ is just a stationary point of $v$, and the two notions of hyperbolicity (and index) correspond. With regard to stationary points on the boundary and their indices, see \cite[Proof of Lemma 2.5.5]{KMbook} for the case of gradient fields. The arguments there extend to the quasi-gradient setting without difficulty.
\end{proof}

\begin{remark}
In part (c) of Definition~\ref{def:eqgv}, the condition that $L_q$ is self-adjoint (and has real eigenvalues) is not strictly necessary. We could have only asked for $L_q$ to have a complex basis of eigenvectors $\phi_1(q), \phi_2(q), \dots, \phi_n(q)$ with corresponding eigenvalues $\lambda_1(q), \lambda_2(q), \dots, \lambda_n(q)$ such that
$$ \Re \lambda_1(q) < \Re \lambda_2(q) < \dots < \Re \lambda_n(q)$$
and $ \Re \lambda_i(q) \neq 0$ for any $i$. With this weaker requirement, the results of this section would still hold, but the proof of Lemma~\ref{lem:ao3} below would be more complicated. In the case of interest to us in this book, the $L_q$ are self-adjoint, so we decided to include this condition in the definition.
\end{remark}

\begin{lemma}
\label{lem:qgCritical2}
Let $\tilde v$ be a Morse equivariant quasi-gradient vector field, and let $f$ be a function as in part (d) of Definition~\ref{def:eqgv}. Then, any stationary point of $\tilde v$ is a critical point of $f$.
\end{lemma}

\begin{proof}
This is similar to the proof of Lemma~\ref{lem:qgCritical}.
\end{proof}

Given $x \in X$, we write $[x]$ for its $S^1$-equivalence class, i.e. its projection to the singular space $X/S^1$.

\begin{lemma}
\label{lem:ao2}
Let $\tilde v$ be a Morse equivariant quasi-gradient vector field on $X$, as in Definition~\ref{def:eqgv}. Let $\gamma: \R \to X$ be a flow line of $\tilde v$. Then, $\lim_{t\to -\infty} [\gamma(t)]$ and $\lim_{t\to +\infty} [\gamma(t)]$ exist in $X/S^1$, and they are both projections of stationary points of $\tilde v$. 
\end{lemma}

\begin{proof}
This is similar to the proof of Lemma~\ref{lem:ao}.
\end{proof}

Suppose $\tilde v$ is a Morse equivariant quasi-gradient vector field on $X$, with $\tilde f$ as in part (d) of Definition~\ref{def:eqgv}. Observe that $\tilde{f}$ induces a function $f:X/S^1 \to \mathbb{R}$ which is smooth except on $Q/S^1$. Using $f$ we see that $v^{\sigma}$ is a Morse quasi-gradient vector field on the interior of $X^{\sigma}/S^1$. 

If we let $f^{\sigma} : X^{\sigma}/S^1 \to \R$ be the composition of the blow-down map with $f$, then $f^{\sigma}$ is smooth, and we have
\begin{equation}
\label{eq:dfsigma}
df^{\sigma}(v^{\sigma}) \geq 0.
\end{equation}
However, $\tilde f$ is constant on the fibers $N_q^1(Q)/S^1$. Hence, equality happens in \eqref{eq:dfsigma} not only at the stationary points of $v^{\sigma}$, but also at all $(q, [\phi]) \in \del(X^{\sigma}/S^1)$ such that $\tilde v(q)=0$ (and $\phi$ does not have to be an eigenvector). Therefore, $v^{\sigma}$ is not naturally a Morse quasi-gradient on all of $X^\sigma/S^1$; or, at least, we cannot use $f^{\sigma}$ to argue that it is.

Nevertheless, we will still be able to do Morse homology using the flow of $v^{\sigma}$. To start with,  observe that there are no closed orbits in this flow: In view of \eqref{eq:dfsigma}, the only such orbits would have to be contained in a fiber $N_q^1(Q)/S^1$, where $\tilde v(q)=0$. From \eqref{eq:lqx} we see that the flow on the projective space $N_q^1(Q)/S^1$ is the projection of a linear flow on $N_qQ$, which is in fact a gradient flow. This shows that there are no closed orbits in that fiber.

Furthermore, we have the analogue of Lemmas~\ref{lem:ao} and \ref{lem:ao2}:
\begin{lemma}
\label{lem:ao3}
Let $\tilde v$ be a Morse equivariant quasi-gradient vector field on $X$, as in Definition~\ref{def:eqgv}. Let $\gamma: \R \to X^{\sigma}/S^1$ be a flow line of $v^{\sigma}$. Then, $\lim_{t\to -\infty} \gamma(t)$ and $\lim_{t\to +\infty} \gamma(t)$ exist in $X^{\sigma}/S^1$, and they are both stationary points of $v^{\sigma}$. 
\end{lemma}

\begin{proof}
By Lemma~\ref{lem:ao2}, we already know that the projection of $\gamma$ to the blow-down $X/S^1$ limits to two stationary points. The projection is one-to-one away from the boundary, so it suffices to study the case of a trajectory that (in the blow-down) limits to a stationary point $q_0$ in the fixed point set $Q$. The limit can be at either $-\infty$ and $+\infty$. Without loss of generality, we consider the case of $+\infty$, and focus on the half-trajectory
$$\gamma_+: [0, \infty) \to X^{\sigma}/S^1, \ \ \gamma_+(t) = \gamma(t).$$

The idea is that, in the blow-up, the flow of $v^{\sigma}$ near the boundary $N^1(Q)/S^1$ is approximated by the flow on the boundary, which is given by \eqref{eq:lq0}-\eqref{eq:lqx}, and whose behavior we understand. 

A suitably small neighborhood $V \subset X^{\sigma}/S^1$ of the boundary $\del (X^{\sigma}/S^1)=N^1(Q)/S^1$ can be identified with the normal bundle to the boundary. Hence, a point $v \in V$ can be written as a triple $(q, s, [\phi])$, with $q \in Q$, $s \geq 0$, and $\phi \in N^1_q(Q)$, normalized so that $|\phi|=1$. (Here, $s$ represents the distance to the boundary.) With $L_q$ as in \eqref{eq:lq}, we introduce the functional
$$ \Lambda: V \to \R, \ \ \Lambda(q, s, [\phi]) = \langle \phi, L_q \phi \rangle.$$
The stationary points of $v^{\sigma}$ on the boundary correspond to $s=0$ and $d\Lambda =0$.

Consider the fiber $F = N^1_{q_0}(Q)/S^1$ over $q_0$. The restriction of $v^{\sigma}$ to $F$ is the gradient of $\frac{1}{2} \Lambda|_F$. By part (c) of Definition~\ref{def:eqgv}, $v^{\sigma}|_F$ has $n$ stationary points, corresponding to the eigenvalues $\lambda_i = \lambda_i(q_0)$ of $L_q$, such that
$$ \lambda_1 < \lambda_2 < \dots  < \lambda_n.$$

From our assumption, we know that given any neighborhood $U$ of $F$, there exists $t_0$ such that $\gamma(t) \in U$ for $t > t_0$. In particular, all the accumulation points of the half-trajectory $\gamma_+$ must be contained in $F$. We seek to show that $\gamma_+$ has a unique accumulation point, and that point is a stationary point of $v^{\sigma}$. 

Note that if $x \in F$ is an accumulation point of $\gamma_+$, then all the points on the flow trajectory $\zeta$ through $x$ are also accumulation points. Indeed, if $\{\Phi_t\}$ denotes the flow of $v^{\sigma}$, and $\gamma(t_n) \to x = \zeta(0)$, then $\gamma(t_n + t) = \Phi_t(\gamma(t_n)) \to \Phi_t(x) = \zeta(t)$.  The trajectory $\zeta$ is contained in $F$, where the flow is a gradient flow, and therefore $\zeta$ limits to two stationary points $x=(q_0, 0, [\phi_i])$ and $y=(q_0, 0, [\phi_j])$ at $\mp \infty$. Moreover, $x$ and $y$ have to be accumulation points for $\gamma_+$ as well, and we have $x=y$ if and only if $\zeta$ is stationary.

Therefore, we are left to show that $\gamma_+$ cannot have two different stationary points $x, y \in F$ as accumulation points. Suppose that were the case, and let $\lambda_i > \lambda_j$ be the eigenvalues corresponding to $x$ and $y$. These are also the values of $\Lambda$ at $x$ and $y$.  Pick an intermediate value $\lambda \in (\lambda_j, \lambda_i)$, such that $\lambda \neq \lambda_k$ for any $k$. Since $v^{\sigma}$ restricts to be the gradient of $\frac{1}{2}\Lambda$ on $F$, we have
$$ d\Lambda(v^{\sigma}) > 0 \text{  on } \Lambda^{-1}(\lambda) \cap F.$$
Since $\Lambda^{-1}(\lambda) \cap F$ is compact, we can find a neighborhood $U\subset V$ of $F$ such that 
\begin{equation}
\label{eq:dLa}
 d\Lambda(v^{\sigma}) > 0 \text{  on } \Lambda^{-1}(\lambda) \cap U.
 \end{equation}
Because $x$ and $y$ are accumulation points of $\gamma_+$, it must be that the half-trajectory $\gamma_+$ intersects the intermediate level set $\Lambda^{-1}(\lambda)$ infinitely many times. Furthermore, we know that after a certain time $t_0$, the trajectory $\gamma_+$ is contained in $U$. However, this contradicts \eqref{eq:dLa}, which says that $\Lambda$ decreases every time the trajectory goes through $\Lambda^{-1}(\lambda) \cap U$. The conclusion follows.
\end{proof}

Since the stationary points of $v^{\sigma}$ are hyperbolic, they admit stable and unstable manifolds; cf. the discussion in Section~\ref{sec:quasi}. Further, we can separate the stationary points on the boundary into stable and unstable, and then define the notion of boundary-obstructed trajectories, as in Section~\ref{subsec:morseboundary}.

\begin{definition}
\label{def:eMSqg}
A Morse equivariant quasi-gradient vector field $\tilde v$ on $X$ is called {\em Morse-Smale} if the induced vector field $v^{\sigma}$ on $X^{\sigma}/S^1$ satisfies the Morse-Smale condition for boundary-unobstructed trajectories; and the Morse-Smale condition inside $\del(X^{\sigma}/S^1)$ for the boundary-obstructed trajectories.  
\end{definition}

If we have a Morse-Smale equivariant quasi-gradient vector field, the constructions in Section~\ref{subsec:morseboundary} carry through using $v^{\sigma}$ (even though $v^\sigma$ is not a quasi-gradient vector field in a natural way).  In particular, Lemma~\ref{lem:ao3} implies that $v^{\sigma}$ is a Morse-Smale vector field as in Definition~\ref{def:MS}; compare Lemma~\ref{lem:isMS}. Thus, we obtain a Morse complex $$(\check{C}(X^\sigma/S^1),\check{\partial})$$ which computes the homology of $X^\sigma/S^1$.

The space $X^\sigma/S^1$ can be viewed as an approximation to the homotopy quotient $X\sslash S^1 := X \times_{S^1} ES^1$. Precisely, let $n$ be the connectivity of the pair $(X, X-Q)$, that is, the largest $j$ such that $(X, X-Q)$ is $j$-connected. Note that if the real codimension of $Q$ in $X$ is $2c$, then $n \geq 2c-1$. Since $X^{\sigma}$ is $S^1$-equivariantly homotopy equivalent to $X - Q$, we have an isomorphism in homology
\begin{equation}
\label{eq:approx1}
 H_j(X^\sigma/S^1) \cong H_j(X \sslash S^1)
 \end{equation}
for all $j \leq n-1$. The homology $H_*(X \sslash S^1) \cong H_*^{S^1}(X)$ is the $S^1$-equivariant Borel homology of $X$.

\section{The $U$-action in Morse homology}
\label{subsec:UMorse}
We keep the same setting as in the previous subsection. The equivariant homology 
$H_*^{S^1}(X)$ admits a natural $\zz[U]$-module structure, given by cap products with the elements of $H^*_{S^1}(pt) \cong \zz[U]$. (The action of $U$ decreases degree by two.) Our goal here is to explain how this module structure can be approximated in terms of Morse theory. The discussion is modeled on the infinite-dimensional case presented in \cite[Section 4.11]{KMOS}.

The circle action on $X^{\sigma}$ produces a natural complex line bundle $E^{\sigma}$ on $X^{\sigma}/S^1$. Let $\sect$ be a generic, smooth section of $E^{\sigma}$, such that $\sect$ is transverse to the $0$-section and further, $\sect$ restricted to the boundary, $N^1(Q)/S^1$ is transverse to restriction of the $0$-section to the boundary.  Note that $\sect$ inherits a canonical orientation from $X^\sigma/S^1$.  Let $\Zs$ denote the zero set of $\sect$, which is a manifold with boundary $\del \Zs = \Zs \cap (N^1(Q)/S^1)$.  From the orientation of $\sect$ and $X^\sigma/S^1$, we obtain an orientation on $\Zs$.  

Notice that a trajectory $\gamma$ of $v^{\sigma}$ is determined by its value at time $t=0$. Thus, we can define cut-down moduli spaces by intersecting $M(\x, \y)$ with $\Zs$ at time $t=0$: 
$$ M(\x, \y) \cap \Zs := \{\gamma \in M(\x, \y) \mid \gamma(0) \in \Zs \}.$$

Given $(\theta, \varpi) \in \{(o,o), (o,s), (u,o),(u,s)\}$, we define maps of the form $m^\theta_\varpi : \check{C}^\theta_*(X^{\sigma}/S^1) \to \check{C}^\varpi_{*-2}(X^{\sigma}/S^1)$ by counting flows that intersect $\Zs$:
\[
m^\theta_\varpi(\x) = \sum_{\y \in \crit^\varpi} \# (M(\x,\y) \cap \Zs) \cdot \y.
\]
Note that the conditions on the gradings of $x$ and $y$ guarantee that the cut-down moduli spaces being counted are 0-dimensional.  

For $\theta, \varpi \in \{u, s\},$ we also define analogous maps that count intersections in the boundary of $X$:
\[
\bar{m}^\theta_\varpi \x = \sum_{\y \in \crit^\varpi}\# (M(\x,\y) \cap \del \Zs) \cdot \y,  
\]
where we only consider $\y$ where the relevant cut-down moduli space is $0$-dimensional.  Further, recall the shift in gradings by one in the boundary-obstructed case.  
We define a chain map $\check{m} : \check{C}_*(X^{\sigma}/S^1) \to \check{C}_{*-2}(X^{\sigma}/S^1)$ by 
\begin{equation}\label{eqn:morsecap}
\check{m} = \begin{bmatrix}  m^o_o & - m^u_o \bar{\partial}^s_u - \partial^u_o \bar{m}^s_u \\  m^o_s & \bar{m}^s_s - m^u_s \bar{\partial}^s_u - \partial^u_s \bar{m}^s_u  \end{bmatrix}.
\end{equation}

The map induced by $\check{m}$ on the homology $H_*(X^{\sigma}/S^1)$ is exactly the cap product with $c_1(E^{\sigma})$. This recovers the $\zz[U]$-module structure on $H_*(X^{\sigma}/S^1) = H_*^{S^1}(X^{\sigma})$. Note that the isomorphism \eqref{eq:approx1} discussed in the previous subsection,
\begin{equation}
\label{eq:approx2}
 H_{\leq n-1}(X^{\sigma}/S^1) \cong H_{\leq n-1}(X \sslash S^1),
 \end{equation}
is actually an isomorphism of $\zz[U]$-modules. Thus, Morse theory tells us the action of $\zz[U]$ on $H_*^{S^1}(X)$ in degrees up to $n-1$.

\section{Combined generalizations}
\label{sec:combinedMorse}
The versions of Morse homology described in Sections~\ref{subsec:ConleyMorse} and ~\ref{subsec:morseboundary} can be combined as follows. Let $X$ be a (possibly non-compact) manifold with boundary. Fix a metric $g$ and a Morse quasi-gradient vector field $v$ on $X$ as in Section~\ref{subsec:morseboundary}, and let $\is \subseteq X$ be an isolated invariant set of the resulting flow. Although our definitions in Section~\ref{subsec:ConleyMorse} were for flows on manifolds without boundary, Conley index theory easily extends to the boundary case, provided that the gradient vector field is tangent to the boundary; compare \cite[Section 3.1.2]{HellThesis}.  The only caveat is that in Definitions~\ref{def:iis} and ~\ref{def:indexpair}, the interior of a subset $A \subseteq X$ should be defined as in point-set topology, without regard to the structure of $X$ as a manifold-with-boundary. With this in mind, the Conley index of $\is$ is defined as before, to be the quotient $N/L$ of an index pair $(N,L)$ for $\is$. Alternatively, we could consider the double $D(X)$ of $X$ with its $\zz/2$-action, and appeal to the equivariant Conley index theory developed in \cite{FloerConley, Pruszko}. This guarantees the existence of a $\zz/2$-equivariant Conley index of $D(\is)$ on $D(X)$, well-defined up to $\zz/2$-equivariant homotopy equivalence. In particular, its quotient by $\zz/2$, which we take to be the Conley index of $\is$ on $X$, is well-defined up to homotopy equivalence. This is equivalent to \cite[Definition 3.1.32]{HellThesis}.

We now impose a Morse-Smale condition for the trajectories in $\is$ as in Section~\ref{subsec:morseboundary}, where in the boundary-obstructed case we only require transversality inside $\del X$. We then construct a complex $\check{C}[\is]$ using the same formula \eqref{eq:delcheck}, but involving only the stationary points and the flow trajectories in $\is$. The homology of $\check{C}[\is]$ will be the reduced homology of the Conley index associated to $\is$. 

Starting from this, we can also combine the construction in Section~\ref{subsec:ConleyMorse}  with those in Sections~\ref{subsec:CircleMorse}-\ref{subsec:UMorse}. Suppose that $X$ is a (possibly non-compact) Riemannian manifold $X$ with a smooth, semifree $S^1$-action by isometries. (In our applications, $X$ will be a vector space of the form $\rr^m \oplus \cc^n$, with the linear $S^1$-action.) Let $\tilde v$ be a smooth vector field on $X$, and $\is \subseteq X$ be an $S^1$-invariant, isolated invariant set in the flow of $\tilde v$. Let $Q, X^{\sigma}, v^{\sigma}, L_q$ be constructed as in Section~\ref{subsec:CircleMorse}.  Given a closed $S^1$-invariant subset $M \subset X$, we will denote by $M^\sigma$ the closure of the preimage of $M - Q$ in $X^\sigma$.  We seek to form a complex $\check{C}(X^\sigma/S^1)[\is]$ using only trajectories in $\is^\sigma/S^1$.  In order to do this, we require that we can find an $S^1$-invariant isolating neighborhood $A$ of $\is$ such that the restriction of $\tilde v$ to $A$ is a Morse-Smale equivariant quasi-gradient vector field in the sense of Definition~\ref{def:eMSqg}.

Under these assumptions, we obtain the desired Morse complex $\check{C}(X^\sigma/S^1)[\is]$ using trajectories in $A^\sigma/S^1$. Observe that $\is^{\sigma}/S^1$ is an isolated invariant set for the flow of $v^{\sigma}$ on $X^{\sigma}/S^1$; let $I(\is^\sigma/S^1)$ be its Conley index. We have:
\begin{equation}
\label{eq:ConleyMorse}
 H_*(\check{C}(X^\sigma/S^1)[\is]) \cong \tilde{H}_*(I(\is^\sigma/S^1)).
 \end{equation}

On the other hand, $\is$ is an isolated invariant set itself, and is fixed by the $S^1$-action. We may choose an index pair which is $S^1$-invariant \cite{FloerConley, Pruszko}, and we thus get an $S^1$-equivariant Conley index $I_{S^1}(\is)$. Let $(I_{S^1}(\is))^{S^1}$ be its fixed point set. The Morse homology in \eqref{eq:ConleyMorse} approximates the reduced equivariant homology of $I_{S^1}(\is)$, in the sense that:
\begin{equation}
\label{eq:EquivConleyMorse}
 H_{\leq n-1}(\check{C}(X^\sigma/S^1)[\is]) \cong \tilde{H}_{\leq n-1}^{S^1}(I_{S^1}(\is)),
\end {equation}
where $n$ is the connectivity of the pair $\bigl(I_{S^1}(\is), \bigl(I_{S^1}(\is) - (I_{S^1}(\is))^{S^1}\bigl ) \cup *\bigr)$, with $*$ denoting the basepoint in $I_{S^1}(\is)$; compare \eqref{eq:approx2}. Moreover, after choosing a suitable section $\sect$ of $E^{\sigma}$ as in Section~\ref{subsec:UMorse}, the isomorphism in \eqref{eq:EquivConleyMorse} becomes one of $\zz[U]$-modules.

Finally, recall that in Sections~\ref{sec:standard} and \ref{sec:or1} we described alternative ways of stating the Morse-Smale condition and of constructing orientations. These descriptions apply equally well to the more general settings discussed here.

\chapter{The Seiberg-Witten Floer spectrum}\label{sec:spectrum}
We review here the construction of the Seiberg-Witten Floer spectrum $\SWF(Y, \s)$, following \cite{Spectrum}.  

\section{The configuration space and the gauge group action}
\label{sec:configure}
We will be studying the Seiberg-Witten equations on a tuple $(Y,g,\spinc,\Spin)$, where $Y$ is a rational homology three-sphere, $g$ is a metric on $Y$, $\spinc$ is a $\spc$ structure on $Y$, and $\Spin$ is a spinor bundle for $\spinc$.   We choose a flat $\spc$ connection $A_0$ on $\Spin$ which gives an affine identification of $ \Omega^1(Y; i \R)$ with $\spc$ connections on $\Spin$.    

We will be doing analysis on the configuration space 
\[
\C(Y) = \Omega^1(Y; i\R) \oplus \Gamma(\Spin).
\]
Of course, $\C(Y)$ also depends on $\spinc$, but we omit it from the notation.  The gauge group $\G = \G(Y):=C^\infty(Y,S^1)$ acts on $\C(Y)$ by $u \cdot (a,\phi) = (a - u^{-1}du,u \cdot \phi)$. Since $b_1(Y)=0$, each $u\in \G$ can be written as $e^{f}$ for some $f: Y \to i\R$. We define the {\em normalized gauge group} $\Go$ to consist of those $u=e^{f} \in \G$ such that $\int_Y f = 0$.

For any integer $k$, following \cite{KMbook}, we let $H_{k}$ denote the completion of a subspace $H \subseteq \C(Y)$ with respect to the $L^2_k$ Sobolev norm. In particular, the completion of $W$ is denoted $W_k$. The Sobolev norm is defined in the standard way using the $L^2$ norms of iterated gradients $\nabla^j$, as in \cite[Section 5.1]{KMbook}.  For $j \leq k$, define $\T_{j}$ as the $L^2_j$ completion of $T\sC_k(Y)$. (In particular, $\T_k$ is $T\sC_k(Y)$.)   

\section{Coulomb slices}
\label{sec:coulombs}
We have a {\em global Coulomb slice}\footnote{The global Coulomb slice was denoted $V$ in \cite{Spectrum}. We switched to $W$ in order to avoid confusion with the spaces denoted $\V(Z)$ in \cite{KMbook}, which will also appear in this book.}: 
$$W = \ker d^* \oplus \Gamma(\Spin) \subset \C(Y),$$
where $d^*$ is meant to act on imaginary $1$-forms.   Given $(a,\phi) \in \C(Y)$, there is a unique element of $W$ which is obtained from $(a,\phi)$ by a normalized gauge transformation; this element is called the {\em global Coulomb projection} of $(a, \phi)$. Explicitly, the global Coulomb projection of $(a, \phi)$ is
\begin{equation}
\label{eq:gCoulomb}
 \Pi^{\gCoul}(a, \phi) = (a -  df, e^{f} \phi),
 \end{equation} 
where $f: Y \to i\rr$  is such that $d^*(a-df) =0$ and $\int_Y f= 0$; that is, $f = Gd^*a$, where $G$ is the Green's operator of $\Delta = d^*d$. For future reference, let us denote by $$\pi: \Omega^1(Y;i\R) \to \ker d^*$$ the $L^2$-orthogonal projection, given by $\pi(a) = a - df = a- dGd^*a$.

The derivative of $\Pi^{\gCoul}$ is called the {\em infinitesimal global Coulomb projection}, given by
$$( \Pi^{\gCoul}_*)_{(a, \phi)}(b, \psi) = \bigl (b - dGd^*b, e^{Gd^*a}(\psi + (Gd^*b) \phi) \bigr).$$
 In particular, if $(a, \phi)$ happens to be already in $W$, we have
\begin{equation}
\label{eq:icp}
( \Pi^{\gCoul}_*)_{(a, \phi)}(b, \psi)= (b- d\xi,  \psi +  \xi \phi) = (\pi(b), \psi + \xi \phi)\in T_{(a, \phi)} W,
\end{equation}
where $\xi=Gd^* b$.

Analogous to the definition of $\T_{j}$, let $\T^{\gCoul}_{j}$ be the $L^2_j$ completion of the tangent bundle to $W_k$, namely the trivial vector bundle with fiber $W_{j}$ over $W_{k}$. We keep the notation $\T^{\gCoul}_j$ (rather than just $W_k \times W_j$) to emphasize the bundle structure; this will be convenient when we discuss bundle decompositions.

Let us mention the following lemma, which will be of use to us later: 
  
\begin{lemma}
\label{lem:igc}
Let $k \geq 2$. View the infinitesimal global Coulomb projection as a section of the bundle $\Hom(T\C(Y), TW)$ over $W$, i.e., a map from $W$ to $\Hom(\C(Y), W)$, given by
$$ (a,\phi) \mapsto \bigl( (b, \psi) \mapsto (\Pi^{\gCoul}_*)_{(a,\phi)}(b,\psi) \bigr).$$ 
Then, this map extends to smooth maps between the Sobolev completions 
$$ W_{k} \to  \Hom(\C_j(Y), W_{j})$$
for all $-k \leq j \leq k$.
\end{lemma}

\begin{proof}
Recall the formula~\eqref{eq:icp}:
$$ (\Pi^{\gCoul}_*)_{(a, \phi)} (b, \psi) = (b - dGd^*b, \psi + (Gd^*b)\phi).$$
Observe that $d$ and $d^*$ decrease Sobolev coefficients by one, the Green operator $G$ increases them by $2$.  Because Sobolev multiplication $L^2_{k} \times L^2_j \to L^2_j$ induces a smooth map $L^2_{k} \to \Hom(L^2_j, L^2_j)$, the desired map is smooth.
\end{proof}

Infinitesimally, we can consider a different slice to the gauge action, the one perpendicular to the orbits in the $L^2$ metric. This is the {\em local Coulomb slice} at $(a, \phi) \in \C(Y)$, denoted $\K_{(a, \phi)}$, and consisting of those tangent vectors $(b, \psi) \in T_{(a, \phi)} \C(Y)$ such that
\begin{equation}
\label{eq:localCoulomb}
- d^* b + i \Re \langle i\phi, \psi \rangle = 0.
\end{equation}
Away from the reducibles, we see that $\T_{k}$ splits into a direct sum of two bundles:
$$ \T_k = \J_{k} \oplus \K_{k},$$ where $\J_{k}$ consists of the vectors tangent to the $\G_{k+1}$ orbits, and $\K_{k}$ is the completion of the local Coulomb slice.  Note that the local Coulomb slice does not form a bundle over the entire configuration space, since the local Coulomb slice is ``bigger'' at reducibles.

Given any $(b, \psi) \in T_{(a, \phi)} \C(Y)$, we define the (infinitesimal) {\em local Coulomb projection} of $(b, \psi)$ to be 
$$\Pi^{\lCoul}_{(a, \phi)} (b, \psi) := (b -  d\zeta, \psi + \zeta \phi),$$ where $\zeta: Y \to i\R$ is, for $\phi \neq 0$, the unique function  such that 
\begin{equation}
\label{eq:zetaf}
 -d^*(b-d\zeta) + i\Re\langle i\phi , \psi +  \zeta \phi \rangle = 0.
\end{equation}
The existence and uniqueness of such a $\zeta$ follow from \cite[Proposition 9.3.4]{KMbook}.  When $\phi = 0$, we again ask for \eqref{eq:zetaf} to be satisfied, but to guarantee uniqueness we also impose the condition $\int_Y \zeta=0$.

Note that both types of Coulomb slices are also mentioned by Kronheimer and Mrowka; see \cite[Sections 9.3 and 9.6]{KMbook}.  

For our purposes, we will also need the {\em enlarged local Coulomb slice}, which consists of vectors that are only required to be perpendicular to the orbits of the normalized gauge group action.
We denote this by $\Ke_{(a, \phi)}$. A vector $(b, \psi)\in T_{(a, \phi)} \C(Y)$ is in $\Ke_{(a, \phi)}$ if and only if $-d^*b + i \Re \langle i\phi, \psi \rangle$ is a constant function. Equivalently, we can write this condition as
\begin{equation}
\label{eq:elocalCoulomb}
- d^* b + i \Re \langle i\phi, \psi \rangle^{\circ} = 0.
\end{equation}
Here, and later in the book, given a smooth function $f: Y \to \cc$, we denote by $\mu_Y(f) $ the average value of $f$ over the $3$-manifold $Y$:
\begin{equation}
\label{eq:muy}
\mu_Y(f)= \frac{1}{\vol(Y)}\int_Y f 
\end{equation}
and set
\begin{equation}
\label{eq:intzero}
f^{\circ} = f - \mu_Y(f),
\end{equation}
so that $\int_Y f^{\circ} = 0$.

We remark that the reason why we did not add the superscript $\circ$ to $d^*b$ in \eqref{eq:elocalCoulomb} is because $\int_Y d^*b = 0$,  so $(d^*b)^{\circ} = d^*b$.  Note also that unlike the local Coulomb slice, the enlarged local Coulomb slices do produce a bundle over the configuration space.

Given any $(b, \psi) \in T_{(a, \phi)} \C(Y)$, we define the {\em enlarged local Coulomb projection} of $(b, \psi)$ to be 
\begin{equation}
\label{eq:Pielc}
\Pi^{\elCoul}_{(a, \phi)} (b, \psi) := (b -  d\zeta, \psi + \zeta \phi),
\end{equation}
 where $\zeta: Y \to i\R$ is such that $\int_Y \zeta=0$ and 
\begin{equation}
\label{eq:zetafirst}
 -d^*(b-d\zeta) + i\Re\langle i\phi , \psi +  \zeta \phi \rangle^{\circ} = 0.
\end{equation}

The fact that $\Pi^{\elCoul}$ is well-defined is established by the following lemma. 
\begin{lemma}
\label{lem:elcUniqueness}
(a) Fix $k, j \in \zz$ with $k \geq 2$ and $-k < j \leq k$. Then, for any $x=(a, \phi) \in \C_k(Y)$ and $(b, \psi) \in \T_{j,x}$, there is a unique $\zeta \in L^2_{j+1}(Y; i\R)$ satisfying $\int_Y \zeta =0$ and \eqref{eq:zetafirst}. Further, if $d^*b=0$, then $\zeta \in L^2_{j+2}(Y; i\R)$.

(b) If $d^*b=0$ and the $L^2_k$ norm of $\phi$ is bounded above by a constant $R$, $k \geq 3$ and $|j| \leq k-1$, then there exists a constant $C(R) > 0$ such that
$$ \| \zeta\|_{L^2_{j+2}} \leq C(R) \cdot \|\psi \|_{L^2_j}.$$
\end{lemma}

\begin{proof}
Consider the direct sum decomposition
\begin{equation}
\label{eq:deco}
 C^{\infty}(Y; i\R) = (\im d^*) \oplus \R
 \end{equation}
into functions that integrate to zero and constant functions. This induces a similar decomposition on the $L^2_j$ Sobolev completions.

Let $\Delta=d^*d$ denote the (geometer's) Laplacian on imaginary-valued functions. Consider the linear operator between Sobolev completions
 $$E_{\phi}: L^2_{j+1}(Y; i\R) \to L^2_{j-1}(Y; i\R) $$
given by
$$ E_{\phi}(\zeta) = \bigl( \Delta \zeta + (|\phi|^2 \zeta)^{\circ} \bigr) + \int_Y \zeta.$$
With respect to the decomposition \eqref{eq:deco} for $L^2_{j-1}$, note that the expression $\Delta \zeta + (|\phi|^2 \zeta)^{\circ}$ lands in the first summand, and $\int_Y \zeta$ in the second summand.

Equation~\eqref{eq:zetafirst} together with the condition $\int_Y \zeta = 0$ can be written as
$$ E_{\phi}(\zeta)=\bigl( d^*b - i\Re\langle i\phi , \psi \rangle + i\mu_Y( \Re \langle i\phi, \psi \rangle) \bigr) + 0.$$

Observe that the right hand side lives in $L^2_{j-1}(Y; i\rr)$ in general, and in $L^2_j(Y: i\rr)$ when $d^*b=0$.

We need to show that $E_{\phi}$ is invertible. Observe that $E_{\phi}$ is a compact deformation of the operator $E_0 = \Delta + \int_Y$. The latter is invertible, and in particular Fredholm of index zero. Hence, $E_{\phi}$ is also Fredholm of index zero. To show that it is invertible, it suffices to show that it has no kernel. Indeed, suppose $\zeta \in \ker(E_{\phi})$. Then, $\int_Y \zeta=0$ and
\begin{align}
0 &=\int_Y \langle \Delta \zeta + |\phi|^2 \zeta - \mu_Y(|\phi|^2 \zeta) , \zeta \rangle \\
&=  \int_Y  |d\zeta|^2 + \int_Y  |\phi|^2 |\zeta|^2 - 0.
\end{align}
This implies $d\zeta=0$, and since $\int_Y \zeta=0$, we get $\zeta=0$. We conclude that $E_{\phi}$ is injective and hence invertible. This proves part (a).

For part (b), note that if $k \geq 3$, the map
$$ L^2_{k-1}(Y; \Spin) \to \Hom( L^2_{j+2}(Y; i\R), L^2_{j}(Y; i\R) ), \ \ \ \phi \mapsto E_{\phi}$$
is continuous (and lands in invertible operators) for $|j| \leq k-1$. Hence, the map $\phi \mapsto E_{\phi}^{-1}$ is also continuous. Since the ball of radius $R$ in $L^2_k$ is precompact in $L^2_{k-1}$, the resulting collection of $E_{\phi}^{-1}$ is also precompact, and hence bounded, in $\Hom( L^2_{j+2}(Y; i\R), L^2_{j}(Y; i\R) )$.  This gives the desired bounds.  
\end{proof}

Observe that, for $(a, \phi) \in W$, we can view the (restrictions of the) projections $(\Pi^{\gCoul}_*)_{(a, \phi)}$ and $\Pi^{\elCoul}_{(a, \phi)}$ as inverse maps relating the enlarged local Coulomb slice to the global Coulomb slice:
\begin{equation}
\label{eq:backandforth}
\xymatrixcolsep{5pc}\xymatrix{\Ke_{(a, \phi)} \ar@/^/[r]^{(\Pi^{\gCoul}_*)_{(a,\phi)}}
 & \T^{\gCoul}_{(a, \phi)}. \ar@/^/[l]^{\Pi^{\elCoul}_{(a, \phi)}}}
\end{equation}
To see that they are inverse to each other, it suffices to note that both are given by adding a uniquely determined vector that is tangent to the $\Go$-orbit.

Here is the analogue of Lemma~\ref{lem:igc} for $\Pi^{\elCoul}$ instead of $(\Pi^{\gCoul}_*)$. 

\begin{lemma}
\label{lem:elc}
Let $k \geq 2$. View the enlarged local Coulomb projection as a section of the bundle $\Hom(TW, T\C(Y))$ over $W$, i.e., a map from $W$ to $\Hom(W, \C(Y))$, given by
$$ (a,\phi) \mapsto \bigl( (b, \psi) \mapsto \Pi^{\elCoul}_{(a,\phi)}(b,\psi) \bigr).$$ 
Then, this map extends to smooth maps between the Sobolev completions 
$$ W_{k} \to  \Hom(W_j, \C_j(Y))$$
for all $-k \leq j \leq k$.
\end{lemma}

\begin{proof} The fact that the extension to Sobolev completions is well-defined was established in Lemma~\ref{lem:elcUniqueness}. Smoothness can be deduced from Lemma~\ref{lem:igc}, using the fact that $\Pi^{\elCoul}$ (restricted to $TW$) is the inverse to $\Pi^{\gCoul}_*$.
\end{proof}

\section{The Seiberg-Witten equations} \label{sec:SWe}
Let $(Y,g,\spinc,\Spin)$ be as above. We let $\rho:TY \to \text{End}(\Spin)$ be the Clifford multiplication.  
Further, for $a \in \Omega^1(Y; i\R)$, we will use $D_a : \Gamma(\Spin) \to \Gamma(\Spin)$ to denote the Dirac operator corresponding to the connection $A_0 + a$, and $D$ for the case of $a = 0$.  

Consider the Chern-Simons-Dirac (CSD) functional, $\L$, on $\C(Y)$:
\[
\L(a,\phi) = \frac{1}{2} \Bigl(\int_Y \langle \phi, D_a \phi \rangle  - \int_Y a \wedge da \Bigr).  
\]
We let $\X$  denote the $L^2$-gradient of CSD:
$$ \X(a, \phi) =  (*da + \tau(\phi,\phi), D_a \phi),$$
where $\tau(\phi,\phi)=\rho^{-1}(\phi \phi^*)_0$ is a quadratic function coming from the Clifford multiplication. 

It is not difficult to check that the CSD functional is gauge-invariant (since $b_1(Y) = 0$). 
The critical points of $\L$ are the solutions to the {\em Seiberg-Witten equations}, 
$$\X(a, \phi)=0.$$ If the spinor $\phi$ is identically 0, then the solution $(a,\phi)$ is said to be {\em reducible}.  

 By measuring the length of the enlarged local Coulomb projections of tangent vectors to $W$, we obtain a Riemannian metric $\tilde{g}$ on $W$. Explicitly, for any tangent vector $(b, \psi)$ to $(a, \phi) \in W$, we set\footnote{The $L^2$ inner product is Hermitian, and thus the real part gives a real inner product.  This is what we need to define a Riemannian metric.}  
\begin{equation}
\label{eq:gtilde}
\langle (b, \psi), (b', \psi') \rangle_{\tilde{g}} = \Re \; \langle \Pi^{\elCoul}_{(a, \phi)} (b, \psi), \Pi^{\elCoul}_{(a, \phi)} (b', \psi') \rangle_{L^2}.
\end{equation}

For future reference, let us mention that, since $\Pi^{\elCoul}_{(a, \phi)}$ is an $L^2$-orthogonal projection, we can also write
\begin{equation}
\label{eq:gtilde2}
\langle (b, \psi), (b', \psi') \rangle_{\tilde{g}} = \Re \; \langle  (b, \psi), \Pi^{\elCoul}_{(a, \phi)} (b', \psi') \rangle_{L^2}.
\end{equation}

The metric $\tilde g$ has the property that the trajectories of the gradient flow of $\L$ restricted to $W$ are precisely the global Coulomb projections of the original gradient flow trajectories in $\C(Y)$.  Therefore, in the global Coulomb slice with the metric $\tilde{g}$, the (downward) gradient flow trajectories are given by
\begin{equation}
\label{eq:gradL}
\frac{d}{dt} \gamma(t) = - (\Pi^{\gCoul}_*)_{\gamma(t)} \X (\gamma(t)),
\end{equation}
where $\gamma(t)=(a(t), \phi(t)).$ Note that $\X = \grad \L$ is perpendicular to the level sets of $\L$ with respect to the $L^2$ metric, which contain the gauge orbits, so $\X$ is automatically contained in the local Coulomb slices. We can split the right hand side of \eqref{eq:gradL} into a linear part $l$ and a nonlinear part $c$, and re-write the flow equation as
\[
\frac{d}{d t} \gamma(t) = - (l + c)(\gamma(t)), 
\]
where
\begin{eqnarray}
l(a,\phi) &=& (*da, D \phi) \label{eq:lmap} \\
c(a,\phi) &=& (\pi \circ \tau(\phi,\phi),\rho(a)\phi +  \xi(\phi)\phi) \label{eq:cmap},
\end{eqnarray}
with $\xi(\phi): Y \to i\rr$ being characterized by $d\xi(\phi) = (1-\pi)\circ \tau(\phi, \phi)$ and $\int_Y \xi(\phi) =0.$ Recall also that $\pi$ denotes the orthogonal projection to $\ker d^*$.

The $\tilde g$-gradient of the restriction $\L|_W$ extends to a map
$$\X^{\gCoul} = l + c: W_{k} \to W_{k-1},$$
such that $l$ is a linear Fredholm operator, and $c$ is quadratic.\footnote{We chose our conventions to be in agreement with \cite{KMbook}.  In \cite{Spectrum}, the map $l+c$ went from $W_{k+1}$ to $W_k$.} The linear operator $l$ is self-adjoint with respect to the $L^2$ inner product, but not necessarily $\tilde{g}$.  Further, the map $c$ is continuous as a map from $W_k$ to $W_k$, and thus compact as a map from $W_k$ to $W_{k-1}$.  The corresponding flow lines are called {\em Seiberg-Witten trajectories} (in Coulomb gauge). Such a trajectory $\gamma=(a(t), \phi(t)): \rr \to W$ is said to be {\em of finite type} if $\L(\gamma(t))$ and $\|\phi(t)\|_{C^0}$ are bounded in $t$.

\section{Finite-dimensional approximation} \label{sec:fdax}
For $\lambda > 1$, let us denote by $\vml$ the finite-dimensional subspace of $W$ spanned by the eigenvectors of $l$ with eigenvalues in the interval $(-\lambda,\lambda)$.\footnote{In \cite{Spectrum}, the role of $\vml$ was played by $W^{\mu}_{\lambda}$, the subspace spanned by eigenvectors with eigenvalues between $\lambda$ and $\mu$.  In this book we restrict to $\mu = - \lambda$; this produces the same spectrum.  The reason for the change is that in Section~\ref{sec:StabilityPoints} we will need to turn the parameter space for $\lambda$ into a manifold with boundary, and it is easier to do so with only one degree of freedom. Also, notice that in \cite[Section 4]{Spectrum}, the approximation $W^{\mu}_{\lambda}$ was initially defined as the span of eigenvectors with eigenvalues in $(\lambda, \mu]$, but later changed so that it is the image of the smoothed projection $p^{\mu}_{\lambda}$. This means using the open interval $(\lambda, \mu)$ instead of $(\lambda, \mu]$, and in fact it is easier to define the approximation as such from the beginning. In our setting we use $(-\lambda, \lambda)$.}  The $L^2$ orthogonal projection from $W$ to $\vml$ will be denoted $\tilde{p}^\lambda$. We modify this to make it smooth in $\lambda$, using the following preliminary definition:
\begin{equation}
\label{eq:plprel}
\pmlprel = \int^1_0 \beta(\theta) \tilde{p}^{\lambda - \theta}_{-\lambda + \theta} d \theta, 
\end{equation}
where $\beta$ is a smooth, non-negative function that is non-zero exactly on $(0,1)$, and such that $\int_\rr \beta(\theta) d\theta =1$.  

Our $\pmlprel$ was the one used in \cite{Spectrum}, where it was denoted $\pml$. In this book, it will be convenient to arrange for the smoothed projection to be the actual projection $\tilde{p}^\lambda$ at an infinite sequence of $\lambda$'s. Let us fix such a sequence:
$$ \llambda_1 < \llambda_2 <  \dots$$
such that $\llambda_i \to \infty$ and none of the $\llambda_i$ are eigenvalues of $l$. Fix also disjoint intervals $[\llambda_i - \epsilon_i, \llambda_i + \epsilon_i]$ that do not contain eigenvalues of $l$. Choose smooth bump functions $\beta_i: (0,\infty) \to [0,1]$ supported in $[\llambda_i - \epsilon_i, \llambda_i + \epsilon_i]$ and with $\beta_i(\llambda_i)=1.$ Then set
 \begin{equation}
\label{eq:pl}
\pml = \sum_i \beta_i(\lambda) \tilde{p}^\lambda + \bigl(1-\sum_i \beta_i(\lambda) \bigr) \pmlprel. 
\end{equation}

We now have that $\pml$ is smooth in $\lambda$, and $\pml=\tilde{p}^{\lambda}$ for $\lambda \in \{\llambda_1, \llambda_2, \dots \}.$ Moreover, for any $\lambda$, the image of $\pml$ is the subspace $\vml$.

As an aside, let us remark that in \cite{Spectrum} there is a different definition of the $L^2_k$ Sobolev norm on $W$, using $l$ instead of the covariant derivative $\nabla$. This produces a norm equivalent to the usual one. With the definition in \cite{Spectrum}, the $L^2$ projection $\tilde{p}^\lambda$ would have been the orthogonal projection to $\vml$ with respect to the $L^2_k$ metric for any $k$, and we would have had $\| \pml (x) \|_{L^2_k} \leq \|x \|_{L^2_k}.$ Given our choice of Sobolev norms,  the same inequality holds up to a constant: 
$$\| \pml (x) \|_{L^2_k} \leq \Theta_k \|x \|_{L^2_k},$$
where $\Theta_k$ depends only on $k$ and the Riemannian manifold $Y$.  We choose $\Theta_0 = 1$.  

On $\vml$, we consider the flow equation
\begin{equation}
\label{eq:approxgrad}
\frac{d}{d t} \gamma(t)=-(l + \pml c)(\gamma(t)). 
\end{equation}

We refer to solutions of \eqref{eq:approxgrad} as  {\em approximate Seiberg-Witten trajectories}. 

\begin{remark} \label{rem:notgradient}
If we consider the restriction of the CSD functional to $\vml$, its gradient with respect to $\tilde g$ is
$$ \tilde{p}^{\lambda}_{\tilde g} (l + c) \tilde{p}^{\lambda}_{\tilde g}= l + \tilde{p}^{\lambda}_{\tilde g} c \tilde{p}^{\lambda}_{\tilde g},$$
where $\tilde{p}^{\lambda}_{\tilde g}$ denotes the $\tilde g$-orthogonal projection onto $W^{\lambda}$.
It would be rather cumbersome to work with these projections, so we replaced $\tilde{p}^{\lambda}_{\tilde g}$ with the $L^2$ orthogonal projection $\tilde{p}^{\lambda}$. When $\lambda$ is one of the cut-offs $\llambda_i$, on $\vml$ we have
$$ \tilde{p}^{\lambda} (l + c) \tilde{p}^{\lambda} = l + \tilde{p}^{\lambda} c = l +\pml c.$$ 
However, even in this case (when $\lambda=\llambda_i$), we expect that $ l +\pml c$ is neither the $L^2$ nor the $\tilde g$ gradient of a function. We will show in Chapter~\ref{sec:quasigradient} that $l +\pml c$ is a quasi-gradient, so it can still be used to do Morse theory.
\end{remark}

Fix a natural number $k \geq 5$.  There exists a constant $R > 0$, such that all Seiberg-Witten trajectories $\gamma: \rr \to W$ of finite type are contained in $B(R)$, the ball of radius $R$ in $W_k$.  The following is a corresponding  compactness result for approximate Seiberg-Witten trajectories:

\begin{proposition}[Proposition 3 in \cite{Spectrum}] \label{prop:proposition3}
For any $\lambda$ sufficiently large (compared to $R$), if $\gamma: \mathbb{R} \to \vml$ is a trajectory of $(l + \pml c)$, and $\gamma(t)$ is in $\overline{B(2R)}$ for all $t$, then in fact $\gamma(t)$ is contained in $B(R)$.
\end{proposition}

This was proved in \cite{Spectrum} with $\pmlprel$ instead of $\pml$, but the same arguments work in our setting.

\section{The Conley index and the Seiberg-Witten Floer spectrum}\label{sec:conleyswf}
Recall the definition of the Conley index from Section~\ref{subsec:ConleyMorse}, and of the $S^1$-equivariant refinement $I_{S^1}$ mentioned at the end of Section~\ref{sec:combinedMorse}. 

With this in mind, we are ready to define the Seiberg-Witten Floer spectrum.  We fix $k$, $R$, and sufficiently large $\lambda$ such that Proposition~\ref{prop:proposition3} applies.  We consider the vector field $u^\lambda(l + \pml c)$ on $\vml$, where $u^\lambda$ is a smooth, $S^1$-invariant, cut-off function on $W^\lambda$ that vanishes outside of $B(3R)$.  This generates the flow $\phi^\lambda$ that we will work with.  Denote by $S^\lambda$ the union of all trajectories of $\phi^\lambda$ inside $B(R)$. Recall from Proposition~\ref{prop:proposition3} that these are the same as the trajectories that stay in $\overline{B(2R)}$. This implies that $S^\lambda$ is an isolated invariant set. 

Since everything is $S^1$-invariant, we can construct the equivariant Conley index $I^\lambda = I_{S^1}(\phi^\lambda,S^\lambda)$.  We must de-suspend appropriately to make the stable homotopy type independent of $\lambda$:
\[
\SWF(Y,\spinc,g) =  \Sigma^{-W^{(-\lambda,0)}} I^\lambda,
\]
where $W^{(-\lambda, 0)}$ denotes the direct sum of the eigenspaces of $l$ with eigenvalues in the interval $(-\lambda, 0)$. As we vary the metric $g$, the spectrum $\SWF(Y, \spinc, g)$ varies by suspending (or de-suspending) with copies of the vector space $\cc$. In \cite{Spectrum}, this indeterminacy is fixed by introducing a quantity $n(Y,\spinc, g) \in \qq$ (a linear combination of eta invariants), and setting  
$$\SWF(Y, \spinc) = \Sigma^{-n(Y, \spinc, g) \cc} \SWF(Y, \spinc, g),$$
where the de-suspension by rational numbers is defined formally.  For the definition of $n(Y,\spinc,g)$, see \eqref{eq:ng} below.  We have:

\begin{theorem}[Theorem 1 in \cite{Spectrum}]
The $S^1$-equivariant stable homotopy type of $\SWF(Y,\spinc)$ is an invariant of the pair $(Y, \spinc)$.  
\end{theorem}

\begin{remark}
The construction of $\SWF(Y, \spinc)$ in \cite{Spectrum} used the smoothed projections $\pmlprel$. By interpolating linearly between $\pmlprel$ and $\pml$, and using the homotopy invariance properties of the Conley index, we see that the definitions from $\pml$ and $\pmlprel$ yield the same $\SWF(Y, \spinc)$.
\end{remark}

\chapter{Monopole Floer homology}
\label{sec:HM}

In this chapter we review the definition of monopole Floer homology given by Kronheimer and Mrowka in their book \cite{KMbook}.  Recall that we are considering the case of a rational homology sphere $Y$ and Spin$^c$ structure $\spinc$, which is necessarily torsion.  In this case, all of the reducible solutions to the Seiberg-Witten equations are gauge equivalent.  While we worked in the Coulomb gauge for the Floer spectrum, for now, we will return to the entire configuration space $$\sC(Y) =  \Omega^1(Y;i \R) \oplus \Gamma(\mathbb{S})$$ (where we have made this identification via a fixed flat Spin$^c$ connection).

Here are the main ideas in the construction of monopole Floer homology: One would like to proceed by analogy with Morse homology; that is, to build a chain complex whose generators are given by gauge-equivalence classes of critical points of the CSD functional and whose differential counts gradient trajectories between them.  However, there are several technical issues that need to be addressed. One issue is that the gauge group does not act freely near the reducible solutions to the Seiberg-Witten equations, and these points cause a serious problem.  This problem was solved by Kronheimer and Mrowka by blowing up the singular set, in a way similar to the construction of $S^1$-equivariant Morse homology in Section~\ref{subsec:CircleMorse}. The second problem is that even after performing a blow-up, transversality may still not be satisfied for the moduli spaces.  For this reason, special perturbations to the CSD functional must be introduced.

\section{Seiberg-Witten equations on the blow-up} \label{sec:SWblowup}
We will often work on cylinders of the form $Z  = I \times Y$, where $I$ is an interval (possibly $\mathbb{R}$). The Spin$^c$ structure on $Y$ induces a Spin$^c$ structure on $Z$ with unitary rank 4 bundle $\Spin^+ \oplus \Spin^-$ on $Z$; here, $\Spin^{\pm}$ are the $\mp 1$ eigenspaces of $\rho(d\vol_Z)$, where $\rho$ is the induced Clifford multiplication for differential forms on $Z$.  Both $\Spin^+$ and $\Spin^-$ can be identified with the pull-backs of $\Spin$ under the projection $\pp: Z \to Y$. Again identifying 1-forms with Spin$^c$ connections, we have 
\[
\sC(Z) = \{(a,\phi) \mid  a \in \Omega^1(Z; i\R), \phi \in \Gamma(\Spin^+)\}. 
\] 
This is acted on by the gauge group $\G(Z) = C^{\infty}(Z, S^1).$ An element of $\sC(Z)$ in temporal gauge (i.e., such that the $dt$ component of $a$ is zero) corresponds to a path $\gamma(t) = (a(t),\phi(t))$ in $\sC(Y)$, and we will often not distinguish between these.  Note that every element of $\sC(Z)$ is gauge-equivalent to a configuration in temporal gauge.  

We write down the four-dimensional Seiberg-Witten equations on $Z$ as $\F(a,\phi) = 0$, where 
\[
\F(a,\phi) = (d^+a - \rho^{-1}((\phi\phi^*)_0),D^+_a \phi) \in \Omega^2_+(Z;i\R) \oplus \Gamma(\Spin^-),  
\]
and $D^+_a : \Gamma(\Spin^+) \to \Gamma(\Spin^-)$ is the Dirac operator and $d^+ a$ is the self-dual part of $da$.  If $\F(\gamma) = 0$, then $\gamma$ represents a {\em Seiberg-Witten trajectory}.  If $\gamma$ is the constant path, then this is giving a solution to the Seiberg-Witten equations on $Y$.  In general, $\F(\gamma) = 0$ is equivalent to $\gamma$ being a downward gradient trajectory of $\L$.  We will often go back and forth between this notion of trajectories of $\grad \L = \X$ on $Y$ and solutions to the Seiberg-Witten equations on $Z$.  We will also think of $\F$ as a section of $\sC(Z)$ into $\V(Z)$, the trivial $\Gamma(Z; i \Lambda^2_+ T^*Z  \oplus \Spin^-)$ bundle over $\sC(Z)$.

As discussed above, in order to deal with the reducible solution, we must blow-up our configuration spaces.  We first consider  
\[
\sC^\sigma(Y) = \{(a,s,\phi) \mid s \geq 0, \| \phi \|_{L^2} = 1 \} \subset  \Omega^1(Y; i\R) \times \mathbb{R}_{\geq 0 } \times \Gamma(\Spin).  
\]
We can define $\sC^\sigma(Z)$ similarly, as the space of triples $(a, s, \phi) \in \Omega^1(Z; i\R) \times \mathbb{R}_{\geq 0 } \times \Gamma(\Spin^+)$ with $\| \phi \|_{L^2(Z)} = 1$.  The {\em blown-up Seiberg-Witten equations} on $\sC^\sigma(Z)$ are given by $\Fsigma(a, s, \phi) = 0$, where $\Fsigma$ is a section of a bundle $\V^\sigma(Z)$ over $\C^\sigma(Z)$, defined to be the pullback of $\V(Z)$ under the blow-down from $\C^\sigma(Z)$ to $\C(Z)$; see \cite[p.115]{KMbook}. Explicitly, we have
\[
\Fsigma(a,s,\phi) = (d^+a - s^2 \rho^{-1}((\phi\phi^*)_0),D^+_a \phi).
\]  

However, there is a variant of the four-dimensional blow-up that is more directly related to paths in $\C^\sigma(Y)$. This is the so-called {\em $\tau$ model} defined in \cite[Section 6.3]{KMbook}. Precisely, we let
$$ \C^\tau(Z) \subset \Omega^1(Z; i\R) \times C^\infty(I) \times C^\infty(Z; \Spin^+)$$
be the space of triples $(a, s, \phi)$ with $s(t) \geq 0$ and $\|\phi(t)\|_{L^2(Y)}=1$ for all $t \in I$. After moving it into temporal gauge, an element of $\C^\tau(Z)$ determines a path $\gamma^\sigma(t) = (a(t),s(t),\phi(t))$ in $\C^\sigma(Y)$.  

In temporal gauge on $\sC^\tau(Z)$, the blown-up Seiberg-Witten equations can be written as 
\begin{align*}
& \frac{d}{dt} a(t) = - *da(t) - s(t)^2 \tau(\phi(t), \phi(t)) \\
& \frac{d}{dt} s(t) = - \Lambda(a(t),s(t),\phi(t)) s(t) \\
& \frac{d}{dt} \phi(t) = - D_{a(t)} \phi(t) + \Lambda(a(t),s(t),\phi(t)) \phi(t),
\end{align*}
where 
\begin{equation}
\label{eq:LambdaNoq}
\Lambda(a,s,\phi) =  \langle \phi, D_a \phi \rangle_{L^2}.
\end{equation}
  Furthermore, the right hand sides of these equations induce a vector field, denoted $\X^\sigma$, on $\sC^\sigma(Y)$ whose flow  trajectories are precisely the solutions to the three equations above.  The zeros of $\X^\sigma$ on $\sC^\sigma(Y)$ are easily rephrased in terms of the zeros of $\X$ on $\sC(Y)$:  

\begin{proposition}[Proposition 6.2.3 in \cite{KMbook}]\label{prop:swblownup} If $s > 0$, then $(a,s,\phi) \in \sC^\sigma(Y)$ is a zero of $\X^\sigma$ in $\sC^\sigma(Y)$ if and only if $(a,s\phi)$ is a zero of $\X$ in $\sC(Y)$.  If $s = 0$, then $(a,s,\phi)$ is a zero if and only if $(a,0)$ is a zero of $\X$ and $\phi$ is an eigenvector of $D_a$.  
\end{proposition}

\begin{remark}
\label{rem:zeroA}
Since $b_1(Y)=0$, we have that $(a,0)$ is a zero of $\X$ if and only if $a=0$.
\end{remark}

Throughout this section, we fix an integer $k \geq 2$.  We may also blow up the Sobolev completions to obtain spaces $\sC^\sigma_{k}(Y)$ such that $\X^\sigma$ extends.  Similarly, we can complete the gauge group $\G$ to $\G_{k}$.  This leads to the various blown-up configuration spaces mod gauge:
\[
\B(Y) = \sC(Y)/\G, \; \; \B^\sigma(Y) = \sC^\sigma(Y)/\G, \; \; \B_{k}(Y) = \sC_{k}(Y)/\G_{k+1}, \; \; \B^\sigma_{k}(Y) = \sC^\sigma_{k}(Y)/\G_{k+1}. 
\]   

Recall that for $j \leq k$, $\T_j$ denotes the $L^2_j$ completion of $T \C_k(Y)$.  We have the analogous construction in the blow-up, $\T^\sigma_j$, which decomposes as
\begin{equation}
\label{eq:BlowUpDecompose}
 \T^{\sigma}_j = \J^{\sigma}_{j} \oplus \K^{\sigma}_{j},
\end{equation}
where $\J^{\sigma}_k$ consists of the tangents to the gauge orbits.  More explicitly, the $L^2_j$ completion of the tangent space at $x = (a,s,\phi) \in \C^\sigma_k(Y)$ is 
$$
\T^\sigma_{j,x} = \{(b,r,\psi) \mid \Re \langle \phi, \psi \rangle_{L^2} = 0 \} \subset L^2_j(Y;iT^*Y) \oplus \rr \oplus L^2_j(Y; \Spin).
$$
At $\phi = 0$, we also want to consider $\T^{\red}_{j}$, the completion of the tangent bundle to $L^2_k(\Omega^1(Y; i\R))$.  We have an analogous splitting of $\T^{\red}_{k,a}$ into $\J^{\red}_{k,a} = L^2_k(\im d)$ and $\K^{\red}_{k,a} = L^2_k(\ker d^*)$.   
  
In four dimensions, we can similarly divide the blown up configuration space $\C^\sigma(Z)$ by gauge to obtain a space $\B^{\sigma}(Z)$. In the $\tau$ model, we define the quotient configuration space $\B^\tau(Z) = \C^{\tau}(Z)/\G(Z)$. Also, starting from the bundle $\V(Z)$ over $\C(Z)$, we obtain a bundle $\V^{\tau}(Z)$ over $\C^{\tau}(Z)$. Explicitly, the fiber of $\V^\tau(Z)$ over $(a, s, \phi)$ is
\begin{align*}
\V^{\tau}(Z)_{(a, s, \phi)} &= \{(b, r, \psi) \mid \Re \langle \phi(t), \psi(t) \rangle_{L^2(Y)}=0, \ \forall t \} \\
& \subset C^\infty(Z; i\Lambda^2_+ T^*Z) \oplus C^\infty(\rr) \oplus C^{\infty}(Z; \Spin^-).
\end{align*}

In the $\tau$ model, the blown-up Seiberg-Witten equations on $Z$ can be written as the zeros of a section $\F^\tau$ of the bundle $\V^\tau(Z)$; see \cite[Equations (6.11)]{KMbook} for the exact formula. In temporal gauge, we simply have
$$ \F^\tau = \frac{d}{dt} + \X^\sigma.$$

When $Z$ is a compact cylinder, that is, $Z = [t_1, t_2] \times Y$, we define the Sobolev completions $\sC_{k}(Z)$, $\B_k(Z)$, $\V_{k}(Z)$, and similarly with $\sigma$ or $\tau$ superscripts. We also have completed tangent bundles $\T_j(Z)$, $\T^{\sigma}_j(Z)$ and $\T^{\tau}_j(Z)$. We should note that $\sC^\sigma_{k}(Y)$ and $\C^\sigma_k(Z)$ are Hilbert manifolds with boundary, so the tangent bundles  make sense.  
The $\tau$ model $\C^\tau_k(Z)$ is not a Hilbert manifold, even with boundary. Nevertheless, it is a closed subset of the Hilbert manifold $\tC^\tau_k(Z)$, which is the $L^2_k$ completion of
\begin{equation}
\label{eq:tildeC}
\tC^\tau(Z) = \{(a, s, \phi) \mid \|\phi(t)\|_{L^2(Y)}=1, \ \forall  t\}  \subset \Omega^1(Z; i\R) \times C^\infty([t_1,t_2]) \times C^\infty(Z; \Spin^+);
 \end{equation}
see \cite[Section 9.2]{KMbook}. By the completed tangent bundle to $\C^\tau_k(Z)$ we simply mean the restriction of the completed tangent bundle to $\tC^\tau_k(Z)$, i.e. the $L^2_j$ completion of the bundle $\T^\tau(Z)$ with fibers  
\begin{align*}
\T^{\tau}(Z)_{(a, s, \phi)} &= \{(b, r, \psi) \mid \Re \langle \phi(t), \psi(t) \rangle_{L^2(Y)}=0, \ \forall t \} \\
& \subset C^\infty(Z; i T^*Z) \oplus C^\infty([t_1,t_2]) \oplus C^{\infty}(Z; \Spin^-).
\end{align*}
The bundle $\V^\tau(Z)$ extends naturally over $\tC^\tau(Z)$ as well.  When discussing $\V^\tau(Z)$, this will refer to the extension and will explicitly mention the restriction to $\C^\tau(Z)$ explicitly when it is used.  Finally, we will need to quotient the extended spaces $\tC^\tau_k(Z)$ by gauge, which we write as $\tB_k^\tau$. 
When $Z$ is non-compact, for example $Z=\rr \times Y$, it is often more useful to consider the local Sobolev completions $L^2_{k, loc}$; see \cite[Section 13.1]{KMbook} for more details.

As mentioned in Chapter~\ref{sec:spectrum}, gauge transformations preserve the property of being a Seiberg-Witten trajectory; this extends through blow-ups and Sobolev completions.  Monopole Floer homology will be defined as the Morse homology (for manifolds with boundary) of a perturbation of the Seiberg-Witten equations on $\B^\sigma_{k}(Y)$.

\section{Perturbed Seiberg-Witten equations}
\label{sec:perturbedSW}

Consider a function $f: \sC(Y) \to \mathbb{R}$ which is gauge-invariant.  A {\em perturbation} is a section $\q: \sC(Y) \to \T_{0}$. We will call $\q$ the {\em formal gradient of $f$}, or $\grad f$ (even though it's not actually a gradient), if for all $\gamma \in C^\infty([0,1],\sC(Y))$ 
\[
\int^1_0 \Bigl \langle \frac{d}{dt} \gamma(t),\q(\gamma(t)) \Bigr  \rangle_{L^2} dt = f \circ \gamma(1) - f \circ \gamma(0).  
\]
We will use the notation 
$$\X_\q:=\X + \q = (\grad \L) + \q.$$  
We will also write $\L_\q$ for $\L + f$, so $\Xq = \grad \L_\q$.  Note that $\q$ is described by its components $\q^0$ to $\Omega^1(Y; i\R)$ and $\q^1$ to $\Gamma(\mathbb{S})$. More generally, from now on we will use the superscripts $0$ and $1$ to describe the $1$-form and spinorial parts of vectors in $\C(Y)$.

The perturbation $\q$ also induces a section $\qhat: \sC(Z) \to \V_{0}(Z)$.  When $Z=[t_1, t_2] \times Y$ is a compact cylinder, we will require that this actually extends to a smooth section $\sC_{k}(Z) \to \V_{k}(Z)$.      

The flow trajectories $(a(t),\phi(t))$ of $\X_\q$ are solutions to the equations
\begin{align*}
& \frac{d}{dt} a = - *da - \tau(\phi, \phi) - \q^0(a,\phi) \\
& \frac{d}{dt} \phi = - D_a \phi - \q^1(a,\phi).  
\end{align*}
Recast on $Z$, these are the {\em perturbed Seiberg-Witten equations}, $\F_{\q} = \F + \qhat = 0$.

We now move to the blow-up.  The vector field $\X_q$ induces a vector field $\Xqsigma$ on $\sC^\sigma(Y)$. Precisely, on $\sC^\sigma(Y)$, we define 
\[
\q^\sigma(a,s,\phi) = (\q^0(a,s\phi), \Re \langle \qtilde^1(a,s,\phi),\phi \rangle_{L^2} \cdot s, \qtilde^1(a,s,\phi)^\perp),  
\]    
where $\qtilde^1(a,s,\phi) = \int^1_0 \D_{(a,st\phi)} \q^1(0,\phi) dt$ and $\perp$ means projection onto the real orthogonal complement of $\phi$.  We can obtain $\qhatsigma$ similarly, with components $(\qhatzerosigma,\qhatonesigma)$.  This leads to the flow equations for $\Xqsigma$:
\begin{align*}
&  \frac{d}{dt} a = - *da - s^2 \tau(\phi, \phi) - \q^0(a,s\phi), \\
& \frac{d}{dt} s = - \Lambda_\q(a,s,\phi)s, \\
& \frac{d}{dt} \phi = - D_a \phi - \qtilde^1(a,s,\phi) + \Lambda_\q(a,s,\phi)\phi,
\end{align*}
where 
\begin{equation}
\label{eq:Lambdaq}
\Lambda_\q(a,s,\phi) = \text{Re}\langle \phi,D_a \phi + \qtilde^1(a,s,\phi) \rangle_{L^2}.
\end{equation}
 These are the solutions to the {\em perturbed Seiberg-Witten equations on the blow-up} in temporal gauge, or $\Fsigma_{\q} = 0$ for short.  We call $\Lambda_\q(a,s,\phi)$ the {\em spinorial energy} of $(a,s,\phi)$.  Note that for an irreducible stationary point $(a,s,\phi)$, we must have $\Lambda_\q(a,s,\phi) = 0$.  

We can now state the perturbed analogue of Proposition~\ref{prop:swblownup}.  First, we define the operator
\[
D_{\q,a}: L^2_k(Y; \mathbb{S}) \to L^2_{k-1}(Y; \mathbb{S}), \ \ \ 
D_{\q, a} (\phi) =D_a \phi + \D_{(a,0)} \q^1(0,\phi).  
\] 

\begin{proposition}[Proposition 10.3.1 of \cite{KMbook}]
The element $(a,s,\phi)$ in $\sC^\sigma_{k}(Y)$ is a zero of $\Xqsigma$ if and only if: 
\begin{enumerate}[(a)]
\item $s \neq 0$ and $(a,s\phi)$ is a stationary point of $\X_\q$, or 
\item $s = 0$ and $(a,0)$ is a stationary point of $\X_\q$ and $\phi$ is an eigenvector of $D_{\q,a}$.  
\end{enumerate}
\end{proposition}

\begin{remark}
A stationary point of the form $(a,0)$ is called {\em reducible}. Recall from Remark~\ref{rem:zeroA} that the unperturbed Seiberg-Witten vector field $\X$ has a unique reducible stationary point mod gauge, namely $a=0$. We claim that, when $\q$ is small, the perturbed field $\Xq$ also has a unique reducible stationary point. Indeed, notice that the linear map $*d: L^2_k(Y;iT^*Y) \to L^2_{k-1}(Y; iT^*Y)$ induces an invertible map from $(\ker d^*)_k$ to $(\ker d^*)_{k-1}$.  Note also that $(\ker d^*)_j$ is the $L^2$ orthogonal complement in $L^2_j(Y; iT^*Y)$ of the tangents to the gauge orbits of connections.   Therefore, by the inverse function theorem, for a small perturbation $\q$, there is a unique $a \in L^2_k(Y; iT^*Y)$ mod gauge, in a neighborhood of $0$, satisfying $*da + \q^0(a,0) = 0$.  (Here we are using that $\q$ is $L^2$ orthogonal to the tangents of the gauge-orbits, since it is the formal gradient of a gauge-equivariant function.)  By elliptic bootstrapping, $a$ is smooth.  Further, by gauge-equivariance, we have $\q^1(a,0) = 0$ and we can conclude that $(a,0)$ provides the unique reducible solution to $\Xq$ near zero. Using the compactness properties of the perturbed Seiberg-Witten equations, it also follows that (for sufficiently small $\q$) there are no reducible solutions outside the fixed neighborhood of $0$.

 We will not, however, use the uniqueness of reducible solutions for small perturbations in this book.    
\end{remark}

The (perturbed, blown-up) Seiberg-Witten map $\Fsigma_\q$ can be viewed as a section of the bundle $\V^{\sigma}(Z)$ over $\C^\sigma(Z)$. In the $\tau$ model, there is a similar section $\F^{\tau}_\q$ of $\V^\tau(Z)$ over $\tC^\tau(Z)$. In temporal gauge, we can simply write
$$ \F^{\tau}_{\q} = \frac{d}{dt} + \Xqsigma.$$

\section{Tame perturbations}
\label{sec:tame}

We are interested in studying the moduli spaces of flows of $\Xqsigma$ connecting stationary points.  In order to obtain the desired compactification results for the moduli spaces, we require some conditions on $\q$.  

Let $Z = [t_1, t_2] \times Y$ be a compact cylinder. For every one-form $a \in \Omega^1(Z; i\R)$, we can consider the Sobolev norm $L^2_{k, a}$ defined using as covariant derivative the connection corresponding to $a$, namely $\nabla_{A_0 + a}$. The $L^2_{k, a}$ norm is equivalent to the usual Sobolev norm, and gives rise to the same Sobolev completion. However, when we want to state global bounds it becomes important to specify the precise Sobolev norm. We will take the usual norm on $\C_{k}(Y), \C_{k}(Z)$. However, on the bundles $\V_{k}(Z)$ and $T\C_{k}(Z)$, in the fiber over $(a, \phi)$ we will take the $L^2_{k, a}$ norm; this turns them into gauge-invariant normed vector bundles.

\begin{definition}[Definition 10.5.1 in \cite{KMbook}]
\label{def:tame}
Fix an integer $k \geq 2$. Suppose that a section $\q : \sC(Y) \to \T_{0}$ is the formal gradient of a continuous, $\G$-invariant function.  We say $\q$ is a {\em $k$-tame perturbation} if, for any compact cylinder $Z = [t_1, t_2] \times Y$, the following hold:  
\begin{enumerate}[(i)]
\item $\qhat$ defines an element of $C^{\infty}(\C_{k}(Z), \V_{k}(Z))$;
\item $\qhat$ also defines an element of $C^0(\C_{j}(Z), \V_{j}(Z))$, for all integers $j \in [1,k]$;
\item  The derivative $\D \qhat \in C^\infty(\sC_{k}(Z),\Hom(\T_{k}(Z),\V_{k}(Z)))$ extends to a smooth map into $\Hom(\T_{j}(Z),\V_{j}(Z))$, for all integers $j \in [-k,k]$;
\item There exists a constant $m$ such that 
\[
\| \q (a,\phi) \|_{L^2} \leq m (\| \phi \|_{L^2} + 1),
\]
for all $(a,\phi) \in \sC_{k}(Y)$;
\item For any $a_0 \in i \Omega^1(Z)$, there is a function $h:\mathbb{R} \to \mathbb{R}$ such that 
\[
\| \qhat(a,\phi) \|_{L^2_{1,a}} \leq  h(\| (a, \phi)\|_{L^2_{1,a_0}}), 
\]
for all $(a,\phi) \in \sC_{k}(Z)$;
\item $\q$ extends to a $C^1$ section of $\sC_{1}(Y)$ into $\T_{0}$.
\end{enumerate}

We say that $\q$ is tame if it is $k$-tame for all $k \geq 2$.
\end{definition}

\begin{remark}
\label{rem:qhatq}
Conditions (i), (ii), (iii) and (v) are phrased in terms of the four-dimensional perturbation $\hat{\q}$, but they readily imply the analogous statements for $\q$. For example, from part (i) in Definition~\ref{def:tame}, we know that for a tame perturbation $\q$, $\hat{\q}$ gives a smooth map from $\C_{k}(Z)$ to $\V_{k}(Z)$, for all $k \geq 1$. For $x \in \C_{k}(Y)$, we can consider the element of $\C_{k}(Z)$ that is constantly $x$ in every slice $\{t\} \times Y$. Its image under $\hat{\q}$ is constantly $\q(x)$ in every slice, so we conclude that $\q(x)$ is in $\V_{k}(Y)$. 
\end{remark}

For a tame perturbation $\q$, the set of solutions to $\X_\q = 0$, topologized as a subspace of $\B_{k}(Y)$, is compact for tame $\q$.  In this case, any trajectory of $\Xqsigma$ between two stationary points in $\sC^\sigma_{k}(Y)$  is gauge-equivalent to a smooth trajectory, which thus lives in $\sC^\sigma(Y)$; no Sobolev completion is necessary.  

Kronheimer and Mrowka construct explicitly a collection of tame perturbations which are the gradient of cylinder functions.  We will not define these here.  However, we will always need to work with a space of perturbations that does contain them.    

\begin{definition}
A separable Banach space $\P$ is a {\em large Banach space of tame perturbations} if there exists a map $\P \to C^0(\sC(Y),\T_{0})$ which takes $x$ to $\q_x$ such that: 
\begin{enumerate}[(i)]
\item the image of this map contains a countably infinite family of perturbations which are gradients of cylinder functions, 
\item $\q_x$ is tame for all $x \in \P$, 
\item for $k \geq 2$, the map $\P \times \sC_{k}(I \times Y) \to \V_{k}(I \times Y)$ is smooth for any compact interval $I$ in $\mathbb{R}$, 
\item the map $\P \times \sC_{1}(Y) \to \T_{1}(Y)$ is continuous, and 
\item there exist a constant $m$ and a map $\mu:\mathbb{R} \to \mathbb{R}$ such that $\P \times \sC_{1}(Y) \to \T_{1}(Y)$ satisfies 
\begin{align*}
\|\q_x(a,\phi)\|_{L^2} & \leq m \| x \| ( \|\phi\|_{L^2} + 1) \\
\|\q_x(a,\phi)\|_{L^2_{1,a_0}} & \leq \| x \| \mu(\|(a,\phi)\|_{L^2_{1,a_0}}).  
\end{align*}
\end{enumerate}
\end{definition}

\section{Very tame perturbations}
\label{sec:verytame}
We now make a digression. For the purposes of doing finite dimensional approximation (as we shall do in this book), we need slightly stronger assumptions on the perturbation $\q$ than the ones in Definition~\ref{def:tame}. We give the relevant definitions here. These concepts are new material; they did not appear in \cite{KMbook}.

Let us recall that a linear operator between Banach spaces, $f: E \to F$, is continuous if and only if it is bounded, i.e., there is a constant $K$ such that $\|f(x) \| \leq K \|x\|$ for all $x$.  This condition is also equivalent to requiring $f$ to take bounded sets to bounded sets. For nonlinear operators, we introduce the following terminology.

\begin{definition}
\label{def:fb}
Let $E$ and $F$ be Banach spaces. A map $f: E \to F$ is called {\em functionally bounded} if $f(B) \subset F$ is bounded whenever $B \subset E$ is bounded. In other words, there exists a function $h: \R \to \R$ such that
$$ \| f(x) \| \leq h(\|x\|),$$
for all $x \in E$.
\end{definition}

We use the term {\em functionally bounded}, rather than {\em bounded}, to prevent confusion with 
the usual notion of bounded functions in analysis, which requires that $\| f(x) \| \leq K$ for some $K$.

For non-linear operators, continuity does not imply functional boundedness; see \cite{StackExchange}. We will need both of these conditions, so let us denote by
$$ C^0_{\fb}(E, F)$$
the space of continuous and functionally bounded operators from $E$ to $F$. Further, for $m \geq 1$ (and also for $m=\infty$), we let
$$ C^m_{\fb}(E, F) = \{ f \in C^m(E, F) \mid \D^\ell f \in C^0_{\fb}(E, \Hom(E^{\times \ell}, F)), \ \ell=0, \dots, m \}.$$

We can similarly define the space of $C^k_{\fb}$ sections of a bundle. 

With this in mind, we present the following strengthening of Definition~\ref{def:tame}.

\begin{definition}
\label{def:verytame}
Fix $k \geq 2$. A $k$-tame perturbation $q$ is called {\em very $k$-tame} if the following additional conditions are satisfied:
\begin{enumerate}[(i)]
\item $\qhat$ defines an element of $C^{\infty}_{\fb}(\C_{k}(Z), \V_{k}(Z))$;
\item $\qhat$ also defines an element of $C^0_{\fb}(\C_{j}(Z), \V_{j}(Z))$, for all integers $j \in [1,k]$;
\item  $\D \qhat$ defines an element of $C^\infty_{\fb}(\sC_{k}(Z),\Hom(T\sC_{j}(Z),\V_{j}(Z)))$, for all integers $j \in [-k,k]$.
\end{enumerate}
We say that $\q$ is {\em very tame} if it is very $k$-tame for all $k \geq 2$.
\end{definition}

Luckily, the kind of perturbations considered by Kronheimer and Mrowka in \cite{KMbook} satisfy these additional conditions.

\begin{lemma}
\label{lem:verytame}
Let $\q$ be the gradient of a cylinder function, as in \cite[Section 11]{KMbook}. Then, $\q$ is very tame.
\end{lemma}
\begin{proof}
This is an immediate consequence of Proposition 11.4.1 and Lemma 11.4.4 in \cite{KMbook}, along the lines of the proof of Theorem 11.1.2 on p.190 of \cite{KMbook}.
\end{proof}

We will need to ensure that both the perturbations we use are very tame and the space of perturbations is sufficiently large to guarantee transversality.  This motivates the following definition.
\begin{definition}\label{def:large-verytame}
A {\em large Banach space of very tame perturbations} is a large Banach space of tame perturbations where each perturbation is additionally very tame.  
\end{definition}

Using Lemma~\ref{lem:verytame}, the proof of the existence of such a space follows as in the proof of Theorem 11.6.1 in \cite{KMbook}.

\section{Admissible perturbations} \label{sec:AdmPer}
For Morse homology, we need non-degeneracy of the Hessian at the critical points and the Morse-Smale condition.  For monopole Floer homology, we must establish the analogues of these: non-degeneracy of stationary points and regularity of the moduli spaces.  We first focus on non-degeneracy.  

Fix a tame perturbation $\q$ and Sobolev number $k$.  Note that $\Xqsigma$ takes $\C^\sigma_k(Y)$ to $\T^\sigma_{k-1}$.  Note also that the stationary points of $\Xqsigma$ cannot be isolated like in Morse theory.  This is because the gauge group preserves stationary points.  Therefore, this will be a Morse-Bott condition, in the sense that the stationary points will be non-degenerate in the directions transverse to the gauge orbits.    

\begin{definition}
A stationary point $\x$ of $\Xqsigma$ is {\em non-degenerate} in $\sC^\sigma_{k}(Y)$ if $\Xqsigma$ is transverse to the sub-bundle $\J^\sigma_{k-1}$ at $\x$.  
\end{definition}

Much like the existence of stationary points, the non-degeneracy of the stationary points of $\Xqsigma$ can be described in terms of $\X_\q$ on $\sC_{k}(Y)$.

\begin{proposition}[Proposition 12.2.5 in \cite{KMbook}]\label{prop:nondegeneracycharacterized}
A stationary point $(a,s,\phi)$ is a {\em non-degenerate} zero of $\Xqsigma$ on $\sC^\sigma_{k}(Y)$ for $s>0$ if $\X_\q$ is transverse to $\J_{k-1}$ at $(a,s\phi)$. If $s = 0$, and $\phi$ is an eigenvector of $D_{\q,a}$ with eigenvalue $\lambda$, then $(a, 0, \phi)$ is non-degenerate if and only if the following three conditions are satisfied: 
\begin{enumerate}
\item $\lambda \neq 0$, 
\item $\lambda$ is a simple eigenvalue, 
\item $\X_\q$ is transverse to $\J^{\red}_{k-1}$ at $a$.  
\end{enumerate}
\end{proposition}

It turns out that one can always find many perturbations such that the stationary points are non-degenerate.  

\begin{theorem}[Theorem 12.1.2 in \cite{KMbook}]
For any large Banach space of tame perturbations $\P$, there is a residual subset of $\P$ such that for each perturbation $\q$, all of the stationary points of $\Xqsigma$ are non-degenerate.  
\end{theorem}

Given such a perturbation $\q$, the gauge-equivalence classes of solutions to the perturbed Seiberg-Witten equations on $(Y,\spinc)$ are isolated (this comes from being transverse to the gauge orbits).  Since these equivalence classes form a compact subset of $\sC_{k}(Y)/\G_{k+1}$ because of tameness, there are at most finitely many irreducible solutions in $\B^\sigma_{k}(Y)$.  

For the rest of the subsection, we will assume that $\q$ has been chosen so as to be tame and that the stationary points are non-degenerate.  In analogy with the finite-dimensional case we have 
\begin{definition}
Let $[\x] = [(a,s,\phi)]$ be a gauge-equivalence class of zeros of $\Xqsigma$ in $\B^\sigma_{k}(Y)$.  If $s \neq 0$, then we say that $[\x]$ is {\em irreducible}.  If $[\x]$ is {\em reducible}, then $[\x]$ is {\em boundary-stable} (respectively boundary-unstable) if $\Lambda_\q(\x) >0$ (respectively $<0$).           
\end{definition}

Define $\Crit^o$, $\Crit^s$, and $\Crit^u$ to be the sets of gauge-equivalence classes of irreducible, boundary-stable, and boundary-unstable stationary points of $\Xqsigma$ on $\B^\sigma_{k}(Y)$. Set $\Crit = \Crit^o \cup \Crit^s \cup \Crit^u.$

Next, we discuss regularity for trajectories. Let $Z = \rr \times Y$. For $\x, \y \in \C^\sigma(Y)$, pick a smooth path $\gamma_0: \rr \to \C^\sigma(Y)$ such that $\gamma_0(t)=\x$ for $t \ll 0$ and $\gamma_0(t)=\y$ for $t \gg 0$. Then, define
\begin{equation}
\label{eq:Cktau}
\sC^{\tau}_k(\x, \y)= \{\gamma \in \C_{k, loc}^{\tau}(Z) \mid \gamma - \gamma_0 \in L^2_k(Z; iT^*Z) \times L^2_k(\rr; \rr) \oplus L^2_k(Z; \Spin^+)\},  
\end{equation}
Note that here we are imposing $L^2_k$ and not $L^2_{k,loc}$ conditions on $\gamma - \gamma_0$. Thus, $\sC^{\tau}_k(\x, \y)$ is equipped with a natural metric, $d(\gamma,\gamma') = \| \gamma - \gamma' \|_{L^2_k}$.  We have an analogously defined space $\tC^\tau_k(\x,\y)$ where we remove the condition $s(t) \geq 0$, as for the definition of $\tC^\tau_k(Z)$.  

We let $\B^{\tau}_k(\x, \y)$ (respectively $\tB^\tau_k(\x,\y)$) be the quotient of $\sC^{\tau}_k(x,y)$ (respectively $\tC^\tau_k(\x,\y)$) by the action of gauge transformations $u:Z\to S^1$ such that $1-u \in L^2_{k+1}(Z, \cc)$. The space $\B^{\tau}_k(\x, \y)$ depends only on the classes $[\x]$ and $[\y]$, up to canonical diffeomorphism. Consequently, we can use the notation $\B^{\tau}_k([\x], [\y])$.  Similarly for $\tB^\tau_k([\x],[\y])$.  

We are now interested in studying gauge-equivalence classes of trajectories between two stationary points.  Define $\tau_t : \rr \times Y \to \rr \times Y$ to be translation by $t$ and let $\gamma_x$ (respectively $\gamma_y$) denote the elements of $\C^\tau_{k,loc}(\rr \times Y)$ which are $x$ (respectively $y$) in each slice.  
\begin{definition}\label{def:sw-moduli-space}
For $[\x]$ and $[\y]$ in $\Crit$, we define the {\em moduli space of trajectories from $[\x]$ to $[\y]$} as 
\[
M([\x],[\y]) = \{[\gamma] \in \B^\tau_{k,loc}(\rr \times Y) \mid \F^\tau_\q(\gamma) = 0, \lim_{t \to -\infty} [\tau^*_t \gamma] = [\gamma_\x], \lim_{t \to +\infty} [\tau^*_t \gamma] = [\gamma_\y] \},
\]
where the limits are taken in $\B^\tau_{k,loc}(\rr \times Y)$.  
If $[\x]$ is boundary-stable and $[\y]$ is boundary-unstable, then we say we are in the {\em boundary-obstructed} case.  Finally, we will decorate a moduli space by $\breve{M}$ if we want the result after quotienting by the usual $\mathbb{R}$-action.  
\end{definition}

It is proved in \cite[Theorem 13.3.5]{KMbook} that every class in $M([\x], [\y])$ has a gauge representative in $\B^{\tau}_k([x], [y])$. The proof uses exponential decay estimates for Seiberg-Witten trajectories.

\begin{remark}
As mentioned above, each (blown-up, perturbed) Seiberg-Witten trajectory on $\sC^\sigma_{k}(Y)$ is gauge-equivalent to a smooth trajectory on $\sC^\sigma(Y)$, for $k \geq 2$.  In particular, the moduli spaces are homeomorphic for any $k \geq 2$.   
\end{remark}

In order to show that the moduli spaces are smooth manifolds, we must show that they are (locally) the preimage of a regular value of a Fredholm map. Just as in the case of compact cylinders, one can define a bundle $\V_{k-1}^\tau(Z)$ over $\tC_{k}^\tau(\x,\y)$, such that $\F_\q^\tau$ provides a section.  The zero set of this section, restricted to $\C_k^\tau(\x,\y)$, thus describes the Seiberg-Witten trajectories asymptotic to $\x$ and $\y$.  The advantage of working with $\tC_{k}^\tau(\x,\y)$ is that we can differentiate $\F_\q^\tau$ in this setting.  

Recall from \eqref{eq:muy} that $\mu_Y(f)$ denotes the average value of a function $f$ over the three-manifold $Y$. For $\gamma = (a,s,\phi) \in \tC^\tau_{k}(\x,\y)$  and $1 \leq j \leq k$, define the linear operator\footnote{In Equation~\eqref{eq:Qgamma}, $\D^{\tau}$ is a covariant derivative on the space of paths, which was simply denoted $\D$ in \cite{KMbook}. See Section~\ref{sec:linearized} below for more details.}
\begin{equation}
\label{eq:Qgamma}
Q_{\gamma} = \D^{\tau}_\gamma \F^\tau_\q \oplus \dd^{\tau,\dagger}_\gamma : \T^\tau_{j,\gamma}(Z) \to \V^\tau_{j-1,\gamma}(Z) \oplus L^2_{j-1}(Z;i\mathbb{R}),  
\end{equation}
where 
\begin{equation}
\label{eq:dtaudagger}
\dd^{\tau,\dagger}_\gamma(b,r,\psi) = -d^*b + is^2 \text{Re}\langle i\phi,\psi \rangle + i |\phi|^2 \text{Re} \ \mu_Y \langle i\phi,\psi \rangle
\end{equation}
is a variant of the formal adjoint to the infinitesimal (four-dimensional) gauge action. The condition $\dd^{\tau, \dagger}_\gamma = 0$ describes a four-dimensional local Coulomb slice $\K^\tau_{j, \gamma} \subset \T^\tau_{j, \gamma}$. When $j = k$, the slice $\K^\tau_{k,\gamma}$ can be viewed as the tangent space to $\tB^{\tau}_{k}([x], [y])$ at $[\gamma]$.  In general, the bundle $\K^\tau_j$ defines a gauge-equivariant bundle over $\tC^\tau_k(x,y)$ which descends to a bundle over $\tB^\tau_k([x],[y])$.  If $Q_\gamma$ is surjective (or, equivalently, the restriction $\D^\tau_\gamma \F^\tau_\q|_{\K^\tau_{j, \gamma}}$ surjects onto $\V^\tau_{j-1,\gamma}$) for all $[\gamma]$ in $M([\x],[\y])$, then $M([\x],[\y])$ is a smooth manifold.    

\begin{definition}
\label{def:regM}
The moduli space $M([\x],[\y])$ is {\em regular} if $Q_\gamma$ is surjective for all $[\gamma] \in M([\x],[\y])$, unless we are in the boundary-obstructed case.  If $M([\x],[\y])$ is boundary-obstructed, regularity means that the cokernel of $Q_\gamma$ is dimension one for each $[\gamma] \in M([\x],[\y])$.      
\end{definition}

Theorem 15.1.1 in \cite{KMbook} guarantees the existence of a tame perturbation $\q$ such that $M([\x],[\y])$ are regular for all $[\x]$ and $[\y]$.  In particular, this implies that the moduli spaces $M([\x],[\y])$ are all smooth manifolds, even in the boundary-obstructed case, by \cite[Proposition 14.5.7]{KMbook}.  In particular, we have that $\ind Q_\gamma = \dim M([\x],[\y])$, except in the boundary-obstructed case, where $\ind Q_\gamma = \dim M([\x],[\y]) -1$.  

If $[\x]$ and $[\y]$ are reducible and $M([\x],[\y])$ is regular, then define $M^{\red}([\x],[\y])$ to be the subset of $M([\x],[\y])$ consisting of reducible trajectories.  Note that $M^{\red}([\x],[\y])$ is either empty (this is the case when one of $[\x]$ or $[\y]$ is irreducible) or all of $M([\x],[\y])$, except when $[\x]$ is boundary-unstable and $[\y]$ is boundary-stable, in which case it is $\partial M([\x],[\y])$.   

\begin{definition}
\label{def:admi}
An {\em admissible} perturbation is a tame perturbation such that all the stationary points are non-degenerate and all the moduli spaces are regular. 
\end{definition}

\begin{theorem}[Theorem 15.1.1 in \cite{KMbook}]\label{thm:admissibleperturbationsexist} For any large Banach space of tame perturbations $\mathcal{P}$, there exists an admissible perturbation $\q \in \mathcal{P}$.  
\end{theorem}
As discussed after Definition~\ref{def:large-verytame}, there exist large Banach spaces of very tame perturbations.  Therefore, there exist perturbations which are both very tame and admissible.

\section{Orientations}
\label{sec:or2}
From now on we fix an admissible perturbation $\q$. 

\begin{theorem}[Corollary 20.4.1 of \cite{KMbook}]  
The moduli spaces $M([\x],[\y])$ are orientable manifolds.
\end{theorem}

Let us sketch the construction of orientations on $M([x], [y])$, following \cite[Section 20]{KMbook}. This is similar to the discussion of specialized coherent orientations in Morse theory (see Section~\ref{sec:or1}). 

To orient $M([x], [y])$, we need to orient the determinant lines $\det(Q_{\gamma})$, where $Q_{\gamma}$ is as in \eqref{eq:Qgamma}. For arbitrary $x, y \in \C_k^{\sigma}(Y)$ (not necessarily stationary points), we consider instead a compact interval $I=[t_1, t_2]$, and a configuration $\gamma \in \C_k^{\tau}(I \times Y)$ whose restrictions to $\{t_1\} \times Y$ and $\{t_2\} \times Y$ are gauge equivalent to $x$, resp. $y$. To any such $\gamma$ we can associate an operator 
$Q_{\gamma}$ by the same formula as for \eqref{eq:Qgamma}. To make it Fredholm, we need to add suitable boundary conditions. At the boundary component $\{t_1\} \times Y$, consider the subspaces
$$ H_1^{\pm} = \K^{\pm}_{1/2, x} \oplus L^2_{1/2}(Y; i\R) \subset \T^{\sigma}_{1/2, x}(Y) \oplus L^2_{1/2}(Y; i\R).$$
Here, we have decomposed the local Coulomb slice $\K^{\sigma}_{1/2, x} \subset \T^{\sigma}_{1/2, x}(Y) $ as
$$\K^{\sigma}_{1/2, x} =  \K^{-}_{1/2, x} \oplus \K^{+}_{1/2, x},$$
using the spectral subspaces (the direct sum of nonpositive, resp. positive eigenspaces) of the Hessian
$$ \Hess^{\sigma}_{\q, x} := \Pi_{\K^{\sigma}_{1/2, x}} \circ \D_{\gamma}\Xq^{\sigma}|_{\K^{\sigma}_{1/2, x}}.$$
(See Section~\ref{sec:HessBlowUp} below for more details about the Hessian.)

We consider similar subspaces $\K_2^{\pm}$ at the other boundary component $\{t_2\} \times Y$. 
Then, we define a Fredholm operator
\begin{equation}
\label{eq:Pgamma}
 P_{\gamma} = \bigl( Q_{\gamma}, -\Pi_1^+, \Pi_2^- \bigr) : \T^{\tau}_{1, \gamma}(I \times Y) \to \V^{\tau}_{0, \gamma} \oplus L^2(I \times Y; i\R) \oplus H_1^+ \oplus H_2^-,
 \end{equation}
where $\Pi_i^{\pm}$ denotes the composition of restriction to $\{t_i\} \times Y$ with the projection to $H_i^{\pm}$.

Let $\Lambda(\gamma)$ be the set of orientations of $\det(P_{\gamma})$. It is proved in \cite[Proposition 20.3.4]{KMbook} that there are canonical identifications between the different $\Lambda(\gamma)$ when we fix the (gauge equivalence classes of the) endpoints $x$ and $y$ of $\gamma$. Thus, we can write $\Lambda([x], [y])$ for $\Lambda(\gamma)$. Further, there are natural composition maps
$$ \Lambda([x], [y]) \times \Lambda([y], [z]) \to \Lambda([x], [z]).$$

Departing slightly from the terminology and conventions in \cite{KMbook}, we define an {\em orientation data set} $o$ for the admissible perturbation $\q$ to consists of elements $o_{[x], [y]} \in \Lambda([x], [y])$, one for each pair $([x], [y])$, such that we have the relations
$$ o_{[x], [y]} \cdot o_{[y], [z]} = o_{[x], [z]}.$$

An orientation data set can be constructed as follows.  In \cite[p. 385-390]{KMbook} it is shown that when $[x]$ and $[y]$ are reducible, the set $\Lambda([x], [y])$ can be canonically identified with $\zz/2 = \{\pm 1\}$, in a way compatible with concatenation; we then choose $o_{[x], [y]}$ to be the element $+1$ in this case. Then, we fix a reducible $[x_0]$ and pick arbitrary elements $o_{[x_0], [x]}$ for all irreducibles $[x]$. This uniquely determines the data set $o$, using the concatenation property.

Next, let $[x]$ and $[y]$ be stationary points, and consider a trajectory $[\gamma] \in M([x], [y])$. Let $\gamma_0$ be its restriction to a large compact interval $I=[t_1, t_2] \subset \R$. It is proved in \cite[Section 20.4]{KMbook} that an orientation for $\det(P_{[\gamma_0]})$ determines one for $\det(Q_{\gamma})$. Therefore, an orientation data set for $\q$ produces orientations for all determinant lines $\det(Q_{\gamma})$, and hence for the moduli spaces $M([x], [y])$. From here we also get orientations on the quotients $\breve M([x], [y])$.

We can orient the moduli spaces $M^{\red}([\x],[\y])$ and $\breve{M}^{\red}([\x],[\y])$ in a similar manner.

\begin{remark}
In \cite[Section 20]{KMbook}, the discussion was more general, allowing the perturbation $\q$ to vary for $t \in [t_1, t_2]$. For the purposes of this book, it suffices to work with a single perturbation $\q$; we adjusted our discussion accordingly, for simplicity. 

Furthermore, the notion of orientation data set did not appear in \cite{KMbook}. Whereas we trivialize the sets $\Lambda([x], [y])$ and then let the monopole Floer complex be generated by stationary points $[x]$ (cf. Section~\ref{sec:mfh} below), Kronheimer and Mrowka define the complex more canonically, as generated by orientations of $\Lambda([x_0], [x])$ (where $[x_0]$ is a fixed reducible basepoint), with two opposite orientations being set to be the negative of each other. The two definitions are readily seen to be equivalent.
\end{remark}

\section{Monopole Floer homology} \label{sec:mfh}
We will now define the monopole Floer chain complex $\cmto(Y,\spinc,\q)$ analogous to the construction of Morse homology for manifolds with boundary in Section~\ref{subsec:morseboundary}.  

The set $\Crit^o$ of gauge-equivalence classes of irreducible stationary points of $\Xqsigma$ is finite. There are also countably many reducible stationary points, corresponding to the eigenvectors of $D_{\q, a}$. Indeed, the operator $D_{\q, a}$ is ASAFOE in the sense of \cite[Definition 12.2.1]{KMbook}, so its spectrum is discrete by \cite[Lemma 12.2.4]{KMbook}.

In \cite[Chapter V]{KMbook}, Kronheimer and Mrowka give an analysis of the compactifications of the moduli spaces via broken flow lines.  Rather than state the general result, we simply point out that given an admissible perturbation, if $\breve{M}([\x],[\y])$ or $\breve{M}^{\red}([\x],[\y])$ is  0-dimensional, then this moduli space is compact.       

Let the groups $C^\theta$ be freely generated over $\mathbb{Z}$ by $\Crit^\theta$ for $\theta \in \{o,s,u\}$.  The monopole Floer chain groups are given by   
\[
\cmto(Y,\spinc,\q) = C^o \oplus C^s.
\]   
Here, the relative $\mathbb{Z}$-grading on $\cmto(Y,\spinc,\q)$ (since $\spinc$ is torsion) is given by the expected dimension of $M([\x],[\y])$ (the index of $Q_\gamma$), since in this case, we cannot have that $M([\x],[\y])$ is boundary-obstructed.  Since $Y$ is a rational homology sphere, this number is well-defined (there is only one homotopy class of paths from $[\x]$ to $[\y]$).  Recall that we need a shift of grading by 1 in the boundary-obstructed case.

For $\theta, \varpi \in \{o, s, u\}$, define $\partial^\theta_\varpi$ and $\bar{\partial}^\theta_\varpi$ by 
\begin{align}
\label{eqn:boundaryirred} & \partial^\theta_\varpi([\x]) = \sum_{[\y] \in \Crit^\varpi} \# \breve{M}([\x],[\y]) [\y], \\
\label{eqn:boundaryred} & \bar{\partial}^\theta_\varpi([\x]) = \sum_{[\y] \in \Crit^\varpi} \# \breve{M}^{red}([\x],[\y]) [\y], 
\end{align}
for $[\x] \in \Crit^\theta$, where we only sum over $[\y]$ such that the relevant moduli spaces are 0-dimensional.  Here, $\#$ means the signed count of points in this oriented, compact $0$-dimensional manifold.  Also, we only allow $\theta, \varpi$ such that these counts make sense: for $\partial$, we want $\theta \in \{o,u\}$ and $\varpi \in \{ o,s\}$ while for $\bar{\partial}$, we ask that $\theta, \varpi \neq o$.  The boundary operator on $\cmto(Y,\spinc,\q)$ is given by 
\begin{equation}\label{eqn:cmboundary}
\check{\partial} = \begin{bmatrix} \partial^o_o & - \partial^u_o \bar{\partial}^s_u \\ \partial^o_s & \bar{\partial}^s_s - \partial^u_s \bar{\partial}^s_u \end{bmatrix}.
\end{equation}
The admissibility of $\q$ guarantees that $\check{\partial}$ squares to zero and we can take homology, $\hmto(Y,\spinc,\q)$.  As in Remark~\ref{rmk:noMred}, the moduli spaces $M^{\red}([x],[y])$ are either empty or equal to $M([x],[y])$, except in the case that $[x]$ is boundary-unstable and $[y]$ is boundary-stable; in this exceptional case, the counts of $\breve{M}^{\red}([x],[y])$  do not arise in $\check{\partial}$.

Similar to the Morse homology for circle actions, monopole Floer homology can be given the structure of a $\mathbb{Z}[U]$-module.  This is defined in \cite{KMbook} in terms of evaluations of suitable \v{C}ech cochains on the $2$-dimensional moduli spaces of parameterized trajectories. For our purposes, it is more convenient to use an alternate, equivalent definition of the $U$ map, which is taken from \cite[Section 4.11]{KMOS}. The latter definition involves counting points in zero-dimensional spaces, and this will make it easier to prove a stability result for trajectories in Section~\ref{subsec:U-maps}.

Although the $U$ map on $\hmto$ comes from a more general cobordism construction, we will restrict our attention to the specific case we need. Let $p \in \rr \times Y$ be a basepoint, and $B_p$ a standard ball neighborhood of $p$ in $\rr \times Y$. Let $\B^{\sigma}_{k}(B_p)$ be the the blown-up configuration space of $L^2_k$ connections and spinors on $B_p$, modulo the gauge group $$\G_{k+1}(B_p) = \{u: B_p \to S^1 \mid u \in L^2_{k+1}\}.$$ Note that $\B^{\sigma}_k(B_p)$ is a Hilbert manifold with boundary, and is a free quotient by $\G_{k+1}(B_p)$. There is a natural complex line bundle $E_p^{\sigma}$ over $\B^{\sigma}_k(B_p)$, induced from the map $\G_{k+1}(B_p) \to S^1, \ u \mapsto u(p)$.

For any $[\x] \in \Crit^{\theta}, [\y] \in \Crit^{\varpi}$, because of unique continuation, there is a well-defined restriction map 
$$ r_p :  M([\x],[\y]) \to \B^\sigma_k(B_p),$$
which is an embedding.
Pick a smooth section $\sect$ of $E_p^{\sigma}$ such that $\sect$ is transverse to the zero section, and the zero set $\Zs$ of $\sect$ intersects all the moduli spaces $M([\x], [\y])$ and $M^{\red}([\x], [\y])$ transversely. By analogy with the finite-dimensional case from Section~\ref{subsec:UMorse}, for $(\theta, \varpi) \in \{(o,o), (o,s), (u,o),(u,s)\}$, we define $m^\theta_\varpi: \mathfrak{C}^\theta \to \mathfrak{C}^\varpi$ by 
\begin{equation}\label{eqn:mU}
m^\theta_\varpi([\x]) = \sum_{[\y] \in \mathfrak{C}^\varpi} \# (M([\x],[\y]) \cap \Zs) \cdot [\y], \text{ for } [\x] \in \mathfrak{C}^\theta,   
\end{equation}
Similarly, for the reducibles, we set 
\begin{equation}\label{eqn:mUred}
\bar{m}^\theta_\varpi([\x]) = \sum_{[\y] \in \mathfrak{C}^\varpi} \# (M^{\red}([\x],[\y]) \cap \del \Zs) \cdot [\y], \text{ for } [\x] \in \mathfrak{C}^\theta.
\end{equation}
Note that we are using parameterized trajectories.  We only consider terms in the sums where the dimension of the moduli spaces being considered is two.   

Finally, we can define $\check{m}: \cmto \to \cmto$ by 
\begin{equation}\label{eqn:floercap}
\check{m} = \begin{bmatrix}  m^o_o & - m^u_o \bar{\partial}^s_u - \partial^u_o \bar{m}^s_u \\  m^o_s & \bar{m}^s_s - m^u_s \bar{\partial}^s_u - \partial^u_s \bar{m}^s_u  \end{bmatrix}.
\end{equation}

This induces a chain map.  Therefore, we can define the $U$-action on $\hmto$ by the map induced by $\check{m}$.  By construction, this map lowers the relative grading by 2.  

\begin{theorem}[Kronheimer-Mrowka, Theorem 23.1.5 of \cite{KMbook}]
The relatively-graded $\mathbb{Z}[U]$-module $\hmto(Y,\spinc,\q)$ is an invariant of the pair $(Y,\spinc)$.   
\end{theorem}

Therefore, we just use the notation $\hmto(Y,\spinc)$.  From here we can easily define Bloom's variant of monopole Floer homology, $\hmtilde$.  

\begin{definition}[\cite{Bloom}]
The {\em tilde-flavor of monopole Floer homology}, $\hmtilde(Y,\spinc)$, is the homology of the mapping cone of the map $\check{m}$ on $\cmto(Y,\spinc,\q)$.  
\end{definition}

\section{Gradings}\label{sec:modifications}
Recall that for two stationary points $\x$ and $\y$, the relative homological grading between $[\x]$ and $[\y]$ is defined by
\[
\gr([\x], [\y]) = \ind Q_{\gamma},
\]
where $\gamma \in \C^{\tau}_k(\x, \y)$ is any path, and $Q_{\gamma}$ is the operator from \eqref{eq:Qgamma}. (See \cite[Definition 14.4.4]{KMbook}.) 

Suppose that $[\x]$ and $[\y]$ are reducibles with the same blow-down projection $(a, 0)$, and correspond to eigenvalues $\mu$ and $\nu$ of $D_{\q, a}$ with $\mu \geq \nu$. Then, an alternate formula for the relative grading is given in \cite[Corollary 14.6.2]{KMbook}:
\begin{equation}
\label{eq:GradingRed}
 \gr([\x], [\y]) = \begin{cases} 
2i(\mu, \nu) & \text{if $\mu$ and $\nu$ have the same sign,} \\
2i(\mu, \nu)-1 &\text{otherwise.}
 \end{cases}
 \end{equation}
where $i(\mu, \nu)$ denotes the number of eigenvalues in the interval $(\nu, \mu].$ 

There is also an absolute rational grading $\gr^{\Q}$ defined in \cite[Section 28.3]{KMbook}. Given a cobordism $W$ from the round sphere $S^3$ to $Y$, we have
$$ \gr^{\Q}([x]) := - \gr_z([x_0], W, [x]) + \frac{c_1(\t)^2 - \sigma(W)}{4} - \iota(W),$$
where: $[x_0]$ is the reducible on $S^3$ with the lowest positive eigenvalue; $\t$ is a Spin$^c$ structure on $W$ that restricts to $\s$ on $Y$; $\gr_z([x_0], W, [x])$ is the relative grading between $x_0$ and $x$ for the monopole map associated to $(W, \t)$, as in \cite[Section 25]{KMbook}; the subscript $z$ refers to a connected component of the configuration space of $W$, and in our case $z$ is uniquely determined by the Spin$^c$ structure $\t$; and $\iota(W) = (\chi(W) + \sigma(W))/2=b^+(W) - b_1(W)$. 

For future reference, let us rephrase this definition in terms of a four-manifold $X$ with boundary $Y$. We can then take $W$ to be the complement of a four-ball in $X$. Using additivity of $\gr$ and a standard calculation on $B^4$ we can replace $W$ with $X$ and write
\begin{equation}
\label{eq:grQx}
 \gr^{\Q}([x]) = -\gr_z(X, [x]) + \frac{c_1(\t)^2 - \sigma(X)}{4} - b^+(X) + b_1(X) -1,
 \end{equation}
where $\gr_z(X,[x])$ is the expected dimension of the moduli space $M_z(X^*; [x])$ defined in \cite[Section 24]{KMbook}. Here, $X^*$ refers to $X$ after attaching to it a cylindrical end of the form $[0, \infty) \times Y$.

\chapter{Reduction to the Coulomb gauge}\label{sec:coulombgauge}

In this chapter, we recast the constructions from Chapter~\ref{sec:HM} entirely in the global Coulomb slice
\[
W =  \ker d^* \oplus \Gamma(\Spin) \subset \sC(Y).
\]
 This is needed in order to make contact with the construction of the Seiberg-Witten Floer spectrum from Chapter~\ref{sec:spectrum}, for which we used finite dimensional  approximation on $W$.  Throughout this section $k$ will denote a fixed integer at least 2.  

\section{Bundle decompositions and projections}
\label{sec:decompositions}
Recall that in Section~\ref{sec:coulombs} we introduced the global Coulomb slice $W$, the global Coulomb projection $\Pi^{\gCoul}$, the infinitesimal global Coulomb projection $\Pi^{\gCoul}_*$, the local Coulomb slice $\K$, the (infinitesimal) local Coulomb projection $\Pi^{\lCoul}$, the enlarged local Coluomb slice $\Ke$, the enlarged local Coulomb projection $\Pi^{\elCoul}$, and the metric $\tilde g$ on $W$. Further, in Section~\ref{sec:SWblowup} we introduced the bundle decomposition of $\T_k$ into the tangents to the gauge orbits $\J_k$ and the local Coulomb slice $\K_k$; we also mentioned a similar decomposition in the blow-up, $\T_k^{\sigma} = \J^{\sigma}_{k} \oplus \K^{\sigma}_{k}.$

In this section we explore these constructions further. In particular, we extend the gauge projections to the blow-up, and describe a few bundle decompositions that are related to global Coulomb gauge. 

For $j \leq k$, recall that $\T^{\gCoul}_{j}$ is the trivial vector bundle with fiber $W_{j}$ over $W_{k}$. Unlike what we did for $\T_{j}$ in Section~\ref{sec:tame}, in the case of $\T^{\gCoul}_j$ there is no need to define the Sobolev norms using the varying covariant derivatives $\nabla_{A_0+a}$. The Sobolev norm on $W_j$ defined by $\nabla = \nabla_{A_0}$ is invariant under the residual gauge action by $S^1$, and this norm is what we shall use on each tangent space.  Thus, $\T^{\gCoul}_j$ is exactly the trivial normed bundle $W_k \times W_j$. Again, we keep the notation $\T^{\gCoul}_j$ (rather than just $W_k \times W_j$) to emphasize the bundle structure.  Note that these two Sobolev norms are equivalent in the following strong sense.  For $x = (a,\phi) \in W_k$ with $\| a \|_{L^2_k} < R$, there exists a constant $C(R)$ such that 
$$
\frac{1}{C(R)} \| v \|_{L^2_j} \leq \| v \|_{L^2_{j,a}} \leq C(R) \| v \|_{L^2_j}
$$
for all $v \in \T_{j,x}$ and $j \leq k$.  Therefore, these Sobolev norms will also be equivalent when working over paths of three-dimensional configurations, and so even in the four-dimensional setting, we are content to use $\nabla$. 

For $x = (a, \phi) \in W_k$, we let $\J^{\gCoul}_x \subset \T^{\gCoul}_k$ be the (real) span of $(0, i\phi)$. This is the tangent space to the $S^1$-orbit at $x$. Since $\J^{\gCoul}_x$ is canonically isomorphic to $\mathbb{R}$, there is no need for a subscript $j$ for Sobolev regularity.

The intersection of a gauge orbit with global Coulomb gauge is depicted schematically in Figure~\ref{fig:orbits}.

\begin{figure}
\begin{center}
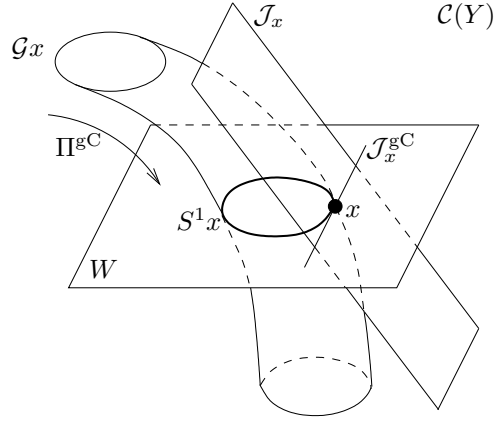
\end{center}
\caption {We show the gauge orbit $\G x$ intersecting $W$ in a circle. Also shown are the tangent spaces to the orbits, and the global Coulomb projection $\Pi^{\gCoul}$.} 
\label{fig:orbits}
\end{figure}

\begin{lemma}\label{lem:gaugeprojections}
Fix $j \geq 0$.  For each $x =(a,\phi) \in W_k$ with $\phi \neq 0$, we have $\J_{j,x} \cap \T^{\gCoul}_{j,x} = (\Pi^{\gCoul}_*)_x(\J_{j,x}) = \J_x^{\gCoul}$.
\end{lemma}
\begin{proof}
Let $x = (a,\phi) \in W_{k-1}$.  We recall from \cite[Page 140]{KMbook} that 
\[
\J_{j,x} = \{(-d\zeta, \zeta \phi) \mid \zeta \in L^2_{j+1}(Y;i\mathbb{R})\}.   
\]
First, we compute $\J_{j,x} \cap \T^{\gCoul}_{j,x}$.  Suppose $(b,\psi) \in \J_{j,x} \cap \T^{\gCoul}_{j,x}$.  If we write $(b,\psi) = (-d\zeta, \zeta\phi)$, we have $d^*d\zeta = 0$.  Since $Y$ is a rational homology sphere, this implies that $\zeta$ is constant.  Therefore, 
\[
\J_{j,x} \cap \T^{\gCoul}_{j,x} = \{(0,it\phi) \mid t \in \mathbb{R} \} = \J_x^{\gCoul}.  
\]
Now, we study $(\Pi^{\gCoul}_*)_x(\J_{j,x})$.  
Let $(-d\zeta,\zeta\phi) \in \J_{j,x}$. Recall that 
\[
(\Pi^{\gCoul}_*)_x(-d\zeta,\zeta\phi) = (-d\zeta - d \xi, \zeta \phi +  \xi \phi),
\]
where $\xi:Y \to i\mathbb{R}$ satisfies $d^*(-d\zeta - d\xi) = 0$ and $\int_Y \xi = 0$.  Again, since $Y$ is a rational homology sphere, this implies that $\zeta + \xi$ is a constant $it$ for some $t \in \mathbb{R}$.  Therefore, we obtain
\[
(\Pi^{\gCoul}_*)_x(-d\zeta,\zeta\phi) = (0,it\phi).   
\]  
This implies that $
(\Pi^{\gCoul}_*)_x(\J_{j,x}) = \J_x^{\gCoul}$, as desired.
\end{proof}

Analogous to the splitting of $\mathcal{T}_{j}$ as $\J_j \oplus \K_j$, there exists a decomposition 
\begin{equation}
\label{eq:JKcoulomb}
\mathcal{T}^{\gCoul}_{j} =\J^{\gCoul} \oplus \K^{\agCoul}_{j}.
\end{equation}
Here, $\K^{\agCoul}_{j}$ is defined to be the orthogonal complement of $\J^{\gCoul}$ with respect to the $\tilde{g}$-metric on $W$.  For each $x = (a,\phi) \in W_k$ and $j \leq k$, this space can be written explicitly as 
\begin{align*}
\K^{\agCoul}_{j,x} &= \{(b,\psi) \in \T_{j,x}^{\gCoul}  \mid  \langle (0, i\phi), (b, \psi) \rangle_{\tilde g} = 0 \} \\
&= \{(b,\psi) \in \T_{j, x}  \mid d^*b = 0,  \langle (0, i\phi), (b, \psi) \rangle_{\tilde g}= 0\}.
\end{align*}
We call $\K^{\agCoul}_{j,x}$ the {\em anticircular global Coulomb slice} at a point $x=(a,\phi) \in W_k$.  Observe that, if $d^*b=0$, then the vector $(b, 0)$ is $L^2$-perpendicular to all $(-d\zeta, \zeta \phi) \in \J_{j,x}$, and hence lies in the local Coulomb slice $\K_{j,x} \subset \Ke_{j,x}$. Using the formula \eqref{eq:gtilde2} for the $\tilde g$-inner product, we get that
\begin{equation}
\label{eq:gtildezero}
  \langle (0, i\phi), (b, 0) \rangle_{\tilde g} =   \Re \langle (0, i\phi), \Pi^{\elCoul}(b, 0) \rangle_{L^2} = \Re \langle (0, i\phi), (b, 0) \rangle_{L^2} = 0. 
  \end{equation}

We deduce that the condition $\langle (0, i\phi), (b, \psi) \rangle_{\tilde g}= 0$ is equivalent to
$$ \langle (0, i\phi), (0, \psi) \rangle_{\tilde g}= 0.$$
For simplicity, we will write $ \langle  \psi_1, \psi_2 \rangle_{\tilde g}$ and $\|\psi\|^2_{\tilde g}$ for expressions of the form $ \langle (0, \psi_1), (0, \psi_2) \rangle_{\tilde g}$ and $\|(0, \psi) \|_{\tilde g}^2$. With this in mind, we can write the anticircular global Coulomb slice as
 \begin{align*}
\K^{\agCoul}_{j,x} &= \{(b,\psi) \in \T_{j,x}^{\gCoul}  \mid  \langle i\phi, \psi \rangle_{\tilde g} = 0 \} \\
&= \{(b,\psi) \in \T_{j, x}  \mid d^*b = 0,  \langle  i\phi, \psi \rangle_{\tilde g}= 0\}.
\end{align*}

We define 
$$\Pi^{\agCoul} = \Pi_{\K^{\agCoul}_j} \circ \Pi^{\gCoul}_* :\T_j \to \K_j^{\agCoul}$$ to be the composition of infinitesimal global Coulomb projection $\Pi^{\gCoul}_*: \T_j \to \T^{\gCoul}_j$ with  $\tilde g$ orthogonal projection onto $\K^{\agCoul}_j$. 
See Figure~\ref{fig:coulomb}.

\begin{figure}
\begin{center}
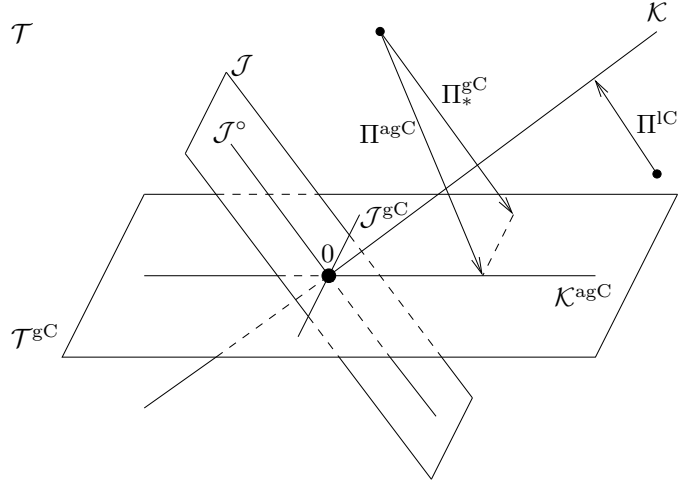
\end{center}
\caption {Different Coulomb slices and projections inside the tangent space $\T_{j, x}$. We drop the subscripts $j$ and $x$ from notation for convenience.} 
\label{fig:coulomb}
\end{figure}

\begin{remark}
The anticircular global Coulomb slice $\K^{\agCoul}_{j, x}$ contains, but does not equal, the intersection $\K_{j,x} \cap \T_j^{\gCoul}$. Indeed, the latter consists of vectors $(b,\psi)$ satisfying $d^*b = 0$ and $\Re \langle i \phi, \psi \rangle = 0$ (pointwise).  These conditions imply 
$$\langle (0, i\phi), (b,\psi) \rangle_{\tilde{g}} = \Re \langle (0,i\phi), \Pi^{\elCoul}_x (b,\psi) \rangle_{L^2} = \Re \langle (0, i\phi), (b,\psi) \rangle_{L^2} = 0,$$ and thus $(b,\psi) \in \K^{\agCoul}_{j, x}$.  Here we are using that $\Pi^{\elCoul}$ is an $L^2$ projection, which can easily be verified.
\end{remark}

\begin{remark}
The subspaces $\J^{\gCoul}$ and $\K_j^{\agCoul}$ do not form Hilbert bundles over the whole of $W_k$, because $\J^{\gCoul}$ is smaller (and $\K_j^{\agCoul}$ larger) at reducibles, compared to irreducibles.
\end{remark}

Recall that on $W$ we have a natural metric $\tilde  g$, defined by measuring the $L^2$ norm of the enlarged local Coulomb projections; cf. Equation \eqref{eq:gtilde}. At this point it is helpful to extend the $\tilde g$-inner product on $\T_j^{\gCoul}$ to the bigger bundle $\T_j$ (restricted to $W_k$). Consider the bundle decomposition over $W_k$,
\begin{equation}
\label{eq:decomposetilde1}
\T_j=\Jo_j \oplus \T^{\gCoul}_j, 
\end{equation}
where $\Jo_j$ is the tangent to the orbit of the normalized gauge group $\Go$; that is, $\Jo_j \subset \J_j$ consists of vectors $(-d\zeta, \zeta \phi)$ with $\int_Y \zeta=0$. Given $x \in W_k$ and vectors $v, w \in \T_{j,x}$, decompose them according to \eqref{eq:decomposetilde1} as
\begin{equation}
\label{eq:sums}
 v = v^{\circ} + v^{\gCoul}, \ \ \  w = w^{\circ} + w^{\gCoul},
 \end{equation}
where $v^{\gCoul} = (\Pi^{\gCoul}_*)_x(v)$ and $w^{\gCoul} = (\Pi^{\gCoul}_*)_x(w)$. Then, set
\begin{equation}
\label{eq:gtildefull}
\langle v, w\rangle_{\tilde g} = \Re \langle v^{\circ}, w^{\circ} \rangle_{L^2} + \langle v^{{\gCoul}},  w^{{\gCoul}} \rangle_{\tilde g},
\end{equation}
where in the last inner product we use the formula \eqref{eq:gtilde}. With this definition, the direct sum decompositions \eqref{eq:decomposetilde1} and
\begin{equation}
\label{eq:decomposetilde2}
\T_j=\J_j \oplus \K^{\agCoul}_j
\end{equation}
are orthogonal for $\tilde g$. Moreover, we can think of $\Pi^{\gCoul}_*$ and $\Pi^{\agCoul}$ as the $\tilde g$-orthogonal projections from $\T_j$ to $\T^{\gCoul}_j$ and $\K^{\agCoul}_j$, respectively.

Next, we consider the blow-up $\C^{\sigma}_k(Y)$ from Section~\ref{sec:SWblowup}. This has tangent bundle $\T^\sigma$ with Sobolev completions $\T^{\sigma}_j$. The tangents to the gauge orbits at $x=(a, s, \phi)$ form the subspace $ \J^{\sigma}_{x, j}$, consisting of vectors of the form $(-d\zeta, 0, \zeta \phi)$. We also have a local Coulomb slice in the blow-up, given by the following formula from \cite[Definition 9.3.6]{KMbook}:
\[
\K^\sigma_{j,x} = \{(b,t,\psi) \in \mathcal{T}^\sigma_{j,x} \mid  -d^*b + i s^2 \Re \langle i \phi, \psi \rangle = 0, \Re \langle i \phi, \psi \rangle_{L^2} = 0 \}.   
\]
The local Coulomb projection on the blow-up, $\Pi^{\lCoul,\sigma}_{(a,s,\phi)}: \T_{j,x}^\sigma\to \K_{j,x}^\sigma$, is defined to be
\begin{equation}
\label{eq:piLCsigma}
\Pi^{\lCoul,\sigma}_{(a,s,\phi)}(b,r,\psi) = (b-d\zeta,r, \psi + \zeta \phi),
\end{equation}
where $\zeta$ is such that $-d^*(b-d\zeta) + is^2 \Re \langle i \phi, \psi +  \zeta \phi \rangle = 0$ and $\Re \langle i \phi, \psi+\zeta \phi \rangle_{L^2} = 0.$

On the blow-up $\C^{\sigma}(Y)$ there is a natural $L^2$ metric, obtained from the inclusion $\T^{\sigma}_j \subset L^2_j(Y; iT^*Y) \oplus \R \oplus L^2_j(Y; \Spin)$. Note that the direct sum decomposition
$$ \T_j^{\sigma} = \J_j^\sigma \oplus \K_j^\sigma$$
is not orthogonal with respect to this $L^2$ metric. Rather, on the irreducible locus, the decomposition is  orthogonal with respect to the pull-back of the $L^2$ metric on $\C(Y)$. On the other hand, this pull-back does not produce a non-degenerate metric on the whole $\C^{\sigma}(Y)$.

Now consider the residual gauge action of $S^1$ on $W$. It is convenient to blow up $W$ at its fixed locus, just as we did with the configuration space $\C(Y)$. The blow-up of $W$ is the space
$$ W^{\sigma}=\{(a, s, \phi) \mid d^*a =0, s \geq 0, \| \phi\|_{L^2}=1\} \subset \C^{\sigma}(Y).$$
There is a natural identification of $W^\sigma/S^1$ with $\B^\sigma(Y)$, defined in Section~\ref{sec:SWblowup}. We have Sobolev completions $W^{\sigma}_k$ and tangent bundles $\T^{\gCoul, \sigma}_j$ for $j \leq k$.

The global Coulomb projection $\Pi^{\gCoul}: \C(Y) \to W$ induces a global Coulomb projection between the blow-ups:
\begin{equation}
\label{eq:piCS}
 \Pi^{\gCoul, \sigma}: \C^{\sigma}(Y) \to W^{\sigma}, \ \ (a, s, \phi) \mapsto (a-df, s, e^{f}\phi),
 \end{equation}
where $f = Gd^*a$.  Furthermore, if $x = (a,s,\phi) \in W^{\sigma}$, then the differential of $\Pi^{\gCoul,\sigma}$ is given by 
\begin{equation}
\label{eq:piCoulSigma}
(\Pi^{\gCoul, \sigma}_*)_x : \T_x^\sigma \to \T_x^{\gCoul,\sigma}, \ \ ((a,s, \phi),(b,r,\psi)) \mapsto (\pi(b),r,\psi + (Gd^*b)\phi). 
\end{equation}

At a point $x=(a, s, \phi) \in W^{\sigma}$, we let $\J_x^{\gCoul, \sigma} = \R\langle (0, 0, i\phi) \rangle$ be the tangent to the residual $S^1$ gauge orbit. Lemma~\ref{lem:gaugeprojections} can be easily adapted to show that:
\begin{equation}\label{eq:PigCJ}
\J^{\sigma}_{j,x} \cap \T^{\gCoul, \sigma}_{j,x} = (\Pi^{\gCoul, \sigma}_*)_x(\J^{\sigma}_{j,x}) = \J_x^{\gCoul, \sigma}.
\end{equation}

Next, we define the anticircular global Coulomb slice in the blow-up.  For $x = (a,s,\phi) \in W^\sigma_k$, let 
\begin{align}
\label{eq:Kagcsigma}
\K^{\agCoul,\sigma}_{j,x} &= \{(b,r,\psi)  \in \mathcal{T}_{j,x}^{\gCoul,\sigma} \mid  \langle i\phi, \psi \rangle_{\tilde{g}} = 0 \} \\  
&= \{(b,r,\psi) \in \mathcal{T}_{j,x}^\sigma \mid d^*b = 0, \langle i\phi, \psi \rangle_{\tilde{g}} = 0\}, \notag
\end{align}
Note that the condition $(b,r,\psi) \in \mathcal{T}_{j,x}^\sigma$ already implies that $\Re \langle \phi, \psi \rangle_{L^2} = 0$. Furthermore, here and later, by $\langle \psi, i\phi \rangle_{\tilde{g}}$ we implicitly mean that the inner product is taken in the blow-down projection; that is, we consider 
$\langle i\phi, \psi \rangle_{\tilde{g}(a,s\phi)}.$ It is useful to compare $\langle i\phi, \psi \rangle_{\tilde{g}}$ with the result of taking the $\tilde{g}$-inner product of $(0,i\phi)$ with the image of $(b,r,\psi)$ in the blow-down: $(b,r\phi + s\psi)$.  This yields
\begin{equation}\label{eq:tgs-blowdown-inner}
\langle (b,r\phi + s\psi), (0, i\phi) \rangle_{\tilde{g}} = s \langle (0,\psi), (0,i\phi) \rangle_{\tilde{g}} + r \langle (0,\phi), (0, i\phi) \rangle_{\tilde{g}} = s \langle \psi, i\phi \rangle_{\tilde{g}}, 
\end{equation}
since $(0,\phi) \in \K_{(a,s\phi)}$ and $\Re \langle \phi, i \phi \rangle_{L^2} = 0$.  Note that if we tried to define $\K^{\agCoul,\sigma}_{(a,s,\phi)}$ using the $\tilde{g}$-inner product of $(0,i\phi)$ with $(b,r\phi + s\psi)$, we would not obtain a bundle, since this would be larger at reducibles.  On the other hand, for irreducibles, we have the following.  

\begin{lemma}
At every irreducible $x=(a,s,\phi) \in W^{\sigma}_k$, the infinitesimal blow-down projection 
$$ (b, r, \psi) \mapsto (b, r\phi + s\psi)$$
induces a linear isomorphism from $\K^{\agCoul,\sigma}_{j, (a, s, \phi)}$ to $\K^{\agCoul}_{j,(a, s\phi)}$. 
\end{lemma}
\begin{proof}
This follows from \eqref{eq:tgs-blowdown-inner}, since at an irreducible we have $s \neq 0$.  
\end{proof}

We have direct sum decompositions
\begin{equation}\label{eq:TsigmaJsigma}
\T_{j,x}^{\sigma} = \J_{j,x}^{\sigma} \oplus \K_{j,x}^{\agCoul, \sigma}
\end{equation}
and
\begin{equation}
 \T_{j,x}^{\gCoul, \sigma} = \J_x^{\gCoul, \sigma} \oplus \K_{j,x}^{\agCoul,\sigma}.
 \end{equation}

We define the anticircular global Coulomb projection on the blow-up, 
$$\Pi^{\agCoul,\sigma}_x: \T_{j,x}^{\sigma} \to \K_{j,x}^{\agCoul, \sigma},$$ 
to be the projection with kernel $\J_{j,x}^{\sigma}$. It can be viewed as the composition of the map $(\Pi^{\gCoul, \sigma}_*)_x$ from \eqref{eq:piCoulSigma} with the projection with kernel $\J_x^{\gCoul, \sigma}$. This last projection, which is the restriction of $\Pi^{\agCoul,\sigma}_x$ to $ \T_{j,x}^{\gCoul, \sigma}$, can be written explicitly as
\begin{equation}
\label{eq:antiproj}
(b,r, \psi) \mapsto (b,r,\psi) - \frac{\langle  i\phi, \psi \rangle_{\tilde{g}}}{\| i\phi\|_{\tilde g}^2} \cdot (0,0,i\phi).
\end{equation}



From \eqref{eq:TsigmaJsigma} we see that the anticircular global Coulomb slice is a true ``infinitesimal slice'' to the whole gauge group in $\C^\sigma(Y)$; i.e., a complement to the tangent space to the gauge orbits. Another such complement is the local Coulomb slice in the blow-up, $\K^{\sigma}_j$. The two slices are related as follows:

\begin{lemma}
\label{lem:bijection}
Let $x = (a, s, \phi) \in W^{\sigma}_k$. Then:

$(a)$ The local Coulomb projection $\Pi^{\lCoul, \sigma}_x$ induces a linear isomorphism between the slices $ \K_{j,x}^{\agCoul, \sigma}$ and $\K^{\sigma}_{j,x}$. Its inverse is the anticircular global Coulomb projection $\Pi^{\agCoul,\sigma}|_{\K^{\sigma}_{j,x}}$.  

$(b)$ If $s=0$, then $\K^{\agCoul,\sigma}_{j,x} = \K^\sigma_{j,x}$, and $\Pi^{\agCoul, \sigma}_x|_{\K^{\sigma}_{j,x}} = (\Pi^{\gCoul,\sigma}_*)_x|_{\K^{\sigma}_{j, x}} : \K^\sigma_{j,x} \to \K^{\agCoul,\sigma}_{j,x}$ is the identity.

$(c)$ For any $x$, we have that $\Pi^{\agCoul, \sigma}_x|_{\K^{\sigma}_{j,x}} = (\Pi^{\gCoul,\sigma}_*)_x|_{\K^{\sigma}_{j, x}}$.  
\end{lemma}

\begin{proof}
$(a)$ As previously noted, both slices are complements to $\J^{\sigma}_{j,x}$. Further, both local Coulomb projection and anticircular global Coulomb projection are given by adding the suitable elements in $\J^{\sigma}_{j,x}$ that move to the other slice. This implies that the two maps are inverse to each other.  

$(b)$ If $s=0$, we see from \eqref{eq:piLCsigma} and \eqref{eq:piCoulSigma} that $\Pi^{\lCoul,\sigma}_x = (\Pi^{\gCoul, \sigma}_*)_x$. The conclusion follows since $(\Pi^{\gCoul,\sigma}_*)_x$ is both idempotent and invertible.

$(c)$ The case $s=0$ was studied in part (b). Therefore, it now suffices to consider the case when $s \neq 0$, so that $x$ is irreducible. Then, the local slices are isomorphic to the corresponding ones in the blow-down, and the projections commute with the infinitesimal blow-down map, so we can simply work in the blow-down. 

Let $\Phi = s \phi$ and $z = (a, \Phi)$. We need to check that if we have a vector $v \in \K_{j,z}$, then its projection $w = (\Pi^{\gCoul}_*)_z(v)$ lands in the anticircular global Coulomb slice. In other words, we know that $\Re \langle (0, i \Phi), v \rangle_{L^2}=0$, and we want to check that $\langle (0, i\Phi), w \rangle_{\tilde g} = 0$. Recall from \eqref{eq:backandforth} that $\Pi^{\elCoul}_z$ and $(\Pi^{\gCoul}_*)_z$ are inverse to each other; hence, $\Pi^{\elCoul}_z(w)=v$. Using \eqref{eq:gtilde2}, we get
$$\langle (0, i\Phi), w \rangle_{\tilde g} = \Re \langle (0, i\Phi), \Pi^{\elCoul}_z(w) \rangle_{L^2} = \Re \langle (0, i \Phi), v \rangle_{L^2}=0,$$
as desired.
\end{proof}

Recall that in Section~\ref{sec:coulombs} we defined an enlarged local Coulomb slice $\Ke$, complementary to the orbit of the normalized gauge group $\Go$. There is a similar enlarged local Coulomb slice in the blow-up,
 \[
\Kesigma_{j,x} = \{(b,t,\psi) \in \mathcal{T}^\sigma_{j,x} \mid  -d^*b + i s^2 \Re \langle i \phi, \psi \rangle^{\circ} =0 \}.   
\]
Note that the condition $-d^*b + i s^2 \Re \langle i \phi, \psi \rangle^{\circ} =0$ simply means that $-d^*b + i s^2 \Re \langle i \phi, \psi \rangle$ is a constant function.

We define the enlarged local Coulomb projection on the blow-up, $\Pi^{\elCoul,\sigma}_{(a,s,\phi)}: \T_{j,x}^\sigma\to \Kesigma_{j,x}$, by
\begin{equation}
\label{eq:pieLCsigma}
\Pi^{\elCoul,\sigma}_{(a,s,\phi)}(b,r,\psi) = (b-d\zeta,r, \psi + \zeta \phi),
\end{equation}
where $\zeta$ is such that $\int_Y \zeta =0$ and $-d^*(b-d\zeta) + is^2 \Re \langle i \phi, \psi +  \zeta \phi \rangle^\circ =0$. 

Let 
\begin{equation}
\label{eq:Jcircsigma}
\J^{\circ,\sigma}_{j,x} = \{(-d\xi, 0, \xi\phi) \mid \int_Y \xi =0\} \subset \J^{\sigma}_{j,x}
\end{equation}
be the tangent to the orbit of $\Go$ in the blow-up. We have direct sum decompositions
\begin{equation}
\label{eq:xavi}
 \T^{\sigma}_{j,x} = \J^{\circ,\sigma}_{j,x} \oplus \T^{\gCoul, \sigma}_{j,x}
  \end{equation}
and
\begin{equation}
\label{eq:xavi2}
 \T^{\sigma}_{j,x} = \J^{\circ,\sigma}_{j,x} \oplus \Kesigma_{j,x}.
 \end{equation}
Thus, $\T^{\gCoul, \sigma}_{j,x}$ and $\Kesigma_{j,x}$ are both infinitesimal slices for the action of $\Go$.

Here is the analogue of Lemma~\ref{lem:bijection}; see also \eqref{eq:backandforth} for the corresponding result in the blow-down.

\begin{lemma}
\label{lem:bijection2}
Let $x = (a, s, \phi) \in W^{\sigma}_k$. Then:

$(a)$ The enlarged local Coulomb projection $\Pi^{\elCoul, \sigma}_x$ induces a linear isomorphism between the slices $ \T_{j,x}^{\gCoul, \sigma}$ and $\Kesigma_{j,x}$. Its inverse is the infinitesimal global Coulomb projection $\Pi^{\gCoul,\sigma}_*|_{\Kesigma_{j,x}}$.

$(b)$ If $s=0$, then $\T^{\gCoul,\sigma}_{j,x} = \Kesigma_{j,x}$, and $(\Pi^{\gCoul,\sigma}_*)_x|_{\Kesigma_{j, x}} : \Kesigma_{j,x} \to \T^{\gCoul,\sigma}_{j,x}$ is the identity.
\end{lemma}

Finally, at $x \in W^{\sigma}_k$, let us introduce the {\em shear map}
$$ S_x: \T_{j,x}^{\sigma} \to \T_{j,x}^{\sigma}$$
given by
\begin{equation}
\label{eq:shear}
\J^{\circ, \sigma}_{j,x}  \oplus \Kesigma_{j,x} \to \J^{\circ, \sigma}_{j,x} \oplus \T^{\gCoul, \sigma}_{j,x}, \ \ \ \
v \oplus w \mapsto v  \oplus (\Pi^{\gCoul, \sigma}_*)_x(w). 
\end{equation}
Its inverse is
$$
S_x^{-1}:  \J^{\circ, \sigma}_{j,x} \oplus \T^{\gCoul, \sigma}_{j,x} \to \J^{\circ, \sigma}_{j,x} \oplus \Kesigma_{j,x}, \ \ \ \ 
v \oplus w \mapsto v \oplus \Pi^{\elCoul,\sigma}_x(w).
$$

We can write, in a more compressed form,
\begin{align*}
S_x(v) &= v - \Pi^{\elCoul,\sigma}_x(v) + (\Pi^{\gCoul, \sigma}_*)_x(v), \\
S_x^{-1}(v) &=  v + \Pi^{\elCoul,\sigma}_x(v) - (\Pi^{\gCoul, \sigma}_*)_x(v).
\end{align*}

Putting together all the maps $S_x$, we obtain an automorphism $S$ of the bundle $\T_j^{\sigma}$ over $W_k^{\sigma}$. 

\begin{remark}
\label{rem:shear}
With regard to the infinitesimal slices to the whole gauge action, observe that $S_x$ maps $\K^{\sigma}_{j,x}$ to $\K^{\agCoul, \sigma}_{j,x}$; compare with Lemma~\ref{lem:bijection}. Also, $S_x$ preserves (but does not act by the identity on) the infinitesimal orbit space $\J^{\sigma}_{j,x}$.
\end{remark}

\section{Choices of gauge on cylinders}
\label{sec:cylinders}
Let $Z = I \times Y$ be a cylinder. Recall that $\C(Z)$ consists of pairs $(a, \phi)$ with $a \in \Omega^1(Z; i\R)$ and $\phi \in \Gamma(\Spin^+)$. We write such a pair as a path 
$$(a(t) + \alpha(t)dt, \phi(t)),  \ \ t \in I,$$ where $a(t) \in \Omega^1(Y; i\R)$, $\alpha(t) \in C^{\infty}(Y; i\R)$, and $\phi(t) \in \Gamma(\Spin)$. 

Recall that if $\alpha(t)=0$ we say that $(a, \phi)$ is in temporal gauge and that any configuration can be put into temporal gauge using the action of $\G(Z)$.

We seek an analogue of $\C(Z)$ adapted to global Coulomb gauge. The first guess is to consider $W(Z)$,  the subspace of $\C(Z)$ consisting of configurations in temporal gauge, and that are also in (three-dimensional) Coulomb gauge at each slice $\{t\} \times Y$; that is,
$$ W(Z) = \{(a, \phi) \in \Gamma(Z, \pp^*(iT^*Y \oplus \Spin)) \mid a(t) \in  \ker d^*,\ \forall t\},$$
where $\pp: Z \to Y$ denotes the projection. Note that an arbitrary configuration $(a, \phi) \in \C(Z)$ cannot always be moved into $W(Z)$ by a four-dimensional gauge transformation; instead, we can move it into temporal gauge, and then Coulomb project in each three-dimensional slice. Further, on $W(Z)$ we could consider an action by the group of slicewise constant gauge transformations
$$ \G^{\gCoul}(Z) = C^{\infty}(I; S^1),$$
with $u \in   \G^{\gCoul}(Z) $ acting by $(a(t), \phi(t)) \to (a(t), u(t)\phi(t))$.

As we saw in Section~\ref{sec:SWe}, the global Coulomb projections of Seiberg-Witten trajectories are the solutions of 
\begin{equation}
\label{eq:swc}
 \Bigl( \frac{d}{dt} + \Xgc \Bigr) (a(t), \phi(t)) = 0,
\end{equation}
where $(a, \phi) \in W(Z)$.  

Note that the equations \eqref{eq:swc} are invariant under the action of constant $u \in S^1$, but not under all of $\G^{\gCoul}(Z)$. This is similar to what happens in $\C(Z)$: the Seiberg-Witten equations are invariant under the whole group $\G(Z)$, but once we move to temporal gauge and write the equations as a gradient flow, we are only left with the action of $\G(Y)$, constant in $t$. 

In view of this, a better analogue of $\C(Z)$ is defined using the following notion:

\begin{definition}
A configuration $(a(t) + \alpha(t)dt, \phi(t)) \in \Omega^1(Z; i\R) \oplus \Gamma(\Spin^+)$ is said to be in {\em pseudo-temporal gauge} if for each $t$, the component $\alpha(t)$ is constant as a function on $Y$. 
\end{definition}

We let $\C^{\gCoul}(Z)$ consist of pairs $(a(t) + \alpha(t)dt, \phi(t))$ in pseudo-temporal gauge, and such that $(a(t), \phi(t))$ is in slicewise Coulomb gauge, i.e.:
$$d(\alpha(t)) =0, \ \ d^*(a(t)) = 0, \ \ \forall t \in I.$$

The elements of $\G^{\gCoul}(Z)$ act on $\C^{\gCoul}(Z)$ as usual gauge transformations:
$$ u: \bigl( a(t) + \alpha(t)dt, \phi(t)\bigr ) \mapsto \bigl(a(t) + (\alpha(t)-u(t)^{-1}\frac{du(t)}{dt})dt, u(t)\phi(t) \bigr).$$

Consider the process of moving an arbitrary configuration in $\C(Z)$ into pseudo-temporal gauge by an element of $\G(Z)$, and then applying slicewise global Coulomb projection to land in $\C^{\gCoul}(Z)$. Under this process, the Seiberg-Witten equations turn into:
\begin{equation}
\label{eq:swC}
 \Bigl( \frac{d}{dt} + \Xgc \Bigr) (a(t), \phi(t)) + (0, \alpha(t)\phi(t)) = 0,
\end{equation}
for $(a(t) + \alpha(t)dt, \phi(t)) \in \C^{\gCoul}(Z)$. These equations are invariant under the action of $\G^{\gCoul}(Z)$.

Of course, every solution of \eqref{eq:swC} can be transformed into a solution of \eqref{eq:swc} by moving into temporal gauge. Most of the time we will work with solutions of \eqref{eq:swc}, i.e., trajectories of $\Xgc$. However, when considering infinitesimal deformations of such trajectories (as we shall do in Section~\ref{sec:linearized}, for example), it is important to allow the more general pseudo-temporal gauge in order to obtain a good Fredholm problem.

By analogy with the section $\F : \C(Z) \to \V(Z)$ from Section~\ref{sec:SWblowup}, we write $\Fgc(a, \phi)$ for the left hand side of Equation~\eqref{eq:swC}. We view $\Fgc$ as a section
$$\Fgc: \C^{\gCoul}(Z) \to \V^{\gCoul}(Z),$$
where $\V^{\gCoul}(Z)$ is the trivial $W(Z)$ bundle over $\C^{\gCoul}(Z)$.

It is worth comparing $\V^{\gCoul}(Z)$ to the bundle $\V(Z)$ from Section~\ref{sec:SWblowup}, whose fibers were $\Gamma(Z; i\Lambda^2_+ T^*Z \oplus \Spin^-)$. In our setting, we identify self-dual imaginary two-forms on $Z$ with sections $a = (a(t))$ of $\pp^*(iT^*Y)$, via sending a section $a$ to $*a + dt \wedge a$, where $*$ here denotes the Hodge star operator on $Y$. We then impose a Coulomb gauge condition $d^*a(t) = 0$ for all $t$. Further, the bundle $\Spin^-$ can be identified with $\pp^* \Spin$. Thus, the fibers become $W(Z)$, as in our definition of $\V^{\gCoul}(Z)$.

If $Z$ is a compact cylinder, then starting from the space $W(Z)$ we can consider Sobolev completions $W_{k}(Z)$ and blow-ups $W_k^{\sigma}(Z)$ and $W_k^{\tau}(Z)$.  The space $W_k^\tau(Z)$ is a subset of the Banach manifold $\widetilde{W}^\tau_k(Z)$, the latter of which is obtained by removing the condition $s(t) \geq 0$; compare with \eqref{eq:tildeC}. Similarly, we can define $\C^{\gCoul}_k(Z)$, its $\sigma$ and $\tau$ blow-ups, and a Banach manifold $\widetilde{\C}^{\gCoul, \tau}_k(Z)$. 

The tangent space to $W_k(Z)$ can be completed to $T_jW(Z)$ (for $j \leq k$), which is a trivial bundle with fiber $W_j(Z)$. The tangent bundle to $\C^{\gCoul}_k(Z)$ is denoted $\T^{\gCoul}_{k}(Z)$, and has completions $\T^{\gCoul}_{j}(Z)$. There are blown-up analogues $\T^{\gCoul, \sigma}_j(Z)$ and $\T^{\gCoul, \tau}_j(Z)$, the latter of which we think of as a bundle over $\tC^{\gCoul,\tau}_k(Z)$. Similar constructions can be done for the bundle $\V(Z)$. 

If $Z$ is an infinite cylinder, we will also consider $L^2_{j, loc}$ completions. These are denoted $W_{k, loc}(Z), \C^{\gCoul}_{k, loc}(Z)$, and so on.

\section{Controlled Coulomb perturbations} \label{sec:ccp}
In Section~\ref{sec:perturbedSW} we discussed how one can perturb the Seiberg-Witten equations by a formal gradient $\q$. In Coulomb gauge, a perturbation is an $S^1$-equivariant vector field 
$$ \eta : W \to \T^{\gCoul}_0.$$
We write $\eta^0$ and $\eta^1$ for the connection and spinor components of $\eta$.

We are interested in perturbations $\eta$ that are $\tilde{g}$-formal gradients of functions $f: W \to \mathbb{R}$. Given such an $\eta$, by applying it slicewise in $t$ on a cylinder $Z = I \times Y$, we 
obtain a section
\begin{equation}
\label{eq:hateta}
\hat \eta: \C^{\gCoul}(Z)  \to \V^{\gCoul}(Z), \ \ \ 
\hat \eta(a(t) + \alpha(t)dt, \phi(t)) = \eta(a(t), \phi(t)).
\end{equation}
We require that $\hat \eta$ preserves Sobolev regularity. We will also need properties of $\eta$ analogous to some of the very tameness properties of $\q$. (Compare Definitions~\ref{def:tame} and \ref{def:verytame}.)

\begin{definition}\label{def:controlled}
Suppose that $\eta=(\eta^0, \eta^1)$ is the $\tilde{g}$-formal gradient of an $S^1$-equivariant function $f: W \to \mathbb{R}$.  Let $\hat{\eta}$ be the induced four-dimensional perturbation. We say $\eta$ is a {\em controlled Coulomb perturbation} if for all integers $k \geq 2$, and all compact cylinders $Z = [t_1, t_2] \times Y$,
\begin{enumerate}[(i)]
\item $\hat \eta$ defines an element of $C^{\infty}_{\fb}(\C^{\gCoul}_{k}(Z), \V^{\gCoul}_{k}(Z))$;
\item  The first derivative $\D \hat\eta \in C^\infty_{\fb}(\C^{\gCoul}_{k}(Z),\Hom(\T^{\gCoul}_{k}(Z),\V^{\gCoul}_{k}(Z)))$ extends to a $C^\infty_{\fb}$ map into $\Hom(\T^{\gCoul}_{j}(Z),\V^{\gCoul}_{j}(Z))$, for all integers $j \in [-k,k]$;
\end{enumerate}
\end{definition}

\begin{remark}
In \cite{KMbook}, Properties (ii), (iv), (v), and (vi) in Definition~\ref{def:tame} are used to get bounds on the stationary points and trajectories of $\Xq$.  When we do finite-dimensional approximation, we have a priori bounds since we restrict to the ball $B(2R)$.  Therefore, there is no need to include analogues of these properties in Definition~\ref{def:controlled} above.  
\end{remark}

We recall that the infinitesimal global Coulomb projection of the Seiberg-Witten vector field $\X=\grad \L$ is given by $\X^{\gCoul}=l + c$, where $l$ and $c$ are defined by \eqref{eq:lmap} and \eqref{eq:cmap} respectively.  Suppose that $\q$ is an abstract perturbation given by the formal gradient of a function $f$.  Let
\begin{equation}
\label{eq:etaq}
\etaq(a, \phi) = (\Pi^{\gCoul}_*)_{(a,\phi)} \q(a, \phi)
\end{equation}
be the infinitesimal global Coulomb projection of $\q$.  The infinitesimal global Coulomb projection of $\Xq = \X + \q$ is then  $$
\Xqgc=\X^{\gCoul} + \etaq,$$ which is the gradient of $(\L + f)|_W$ with respect to the metric $\tilde{g}$ introduced in \eqref{eq:gtilde}.

\begin{lemma}\label{lem:tamecoulomb}
Let $\q$ be a very tame perturbation.  Then, $\etaq$ is a controlled Coulomb perturbation.  
\end{lemma}

Before proving the claim, we will need to prove some additional properties of the infinitesimal global Coulomb projection.  
\begin{lemma}\label{lem:igc-fb}
Fix $j \in [-k,k]$.   Then, 
\begin{enumerate}[(a)]
\item \label{igc-fb:a} the map $\Pi^{\gCoul}_* : \T_j \to \T^{\gCoul}_j$ is functionally bounded as a map from $\C_k(Y)$ to $\Hom(\C_j(Y), W_j)$;
\item \label{igc-fb:b} if $X  \in C^n_{\fb}(W_k, \C_k(Y))$ satisfies that $\D X: W_k \to \Hom(W_k, \C_k(Y))$ extends to an element of $C^{n-1}_{\fb}(W_k, \Hom(W_j, \C_j(Y)))$, then $\D (\Pi^{\gCoul}_* \circ X) : W_k \to \Hom(W_j, W_j)$ is also in $C^{n-1}_{\fb}$; 
\item \label{igc-fb:c}the analogues of \eqref{igc-fb:a} and \eqref{igc-fb:b} hold for $I \times Y$ as well for  $I \subseteq \mathbb{R}$ a closed interval (with $\Pi^{\gCoul}_*$ applied slicewise). 
\end{enumerate}
\end{lemma}
\begin{proof}
\eqref{igc-fb:a} Let $(a,\phi) \in \C_k(Y)$.  By definition, 
$$
(\Pi^{\gCoul}_*)_{(a,\phi)}(b,\psi) = (\pi(b), e^{Gd^* a}((Gd^* b)\phi + \psi)).  
$$
We are interested in showing that this expression has $L^2_j$ norm bounded in terms of the $L^2_k$ norm of $(a,\phi)$ and the $L^2_j$ norm of $(b,\psi)$.  This is clear for the first term $\pi(b)$, because $\pi$ is a linear, continuous operator taking $L^2_j$ forms to $L^2_j$ forms. For the second term, the continuity and bilinearity of the Sobolev multiplication $L^2_j \times L^2_k \to L^2_j$ gives bounds 
$$
\|(Gd^* b)\phi\|_{L^2_j} \leq C \| Gd^*b \|_{L^2_j} \| \phi \|_{L^2_k}.
$$   
We then use the linearity and continuity of $Gd^*$ to bound $ \| Gd^*b \|_{L^2_j}$ in terms of $\|b \|_{L^2_j}$. The functional boundedness of the exponential map gives a bound on $e^{Gd^* a}$. Using the Sobolev multiplication again, we get the desired $L^2_j$ bounds on $e^{Gd^* a}((Gd^* b)\phi + \psi))$.

\eqref{igc-fb:b} By Lemma~\ref{lem:igc}, we have that $\Pi^{\gCoul}_*$ is smooth and therefore we are only interested in functional boundedness. Precisely, to show that $\D (\Pi^{\gCoul}_* \circ X)$ is in $C^{n-1}_{\fb}$, we need to check that  $\D^m (\Pi^{\gCoul}_* \circ X)$ is functionally bounded for all $m=1, \dots, n$.


We first consider the case $m=1$.  Let $(a, \phi) \in W_k$ and $(b,\psi) \in W_j$. It is straightforward to compute    
\begin{align}\label{eq:igc-fb-d}
\D_{(a,\phi)} (\Pi^{\gCoul}_* \circ X) (b,\psi) =  \Big(\pi(\D_{(a,\phi)} X^0(b,\psi)), \D_{(a,\phi)} X^1(b,\psi) + (G&d^* \D_{(a,\phi)} X^0(b,\psi)) \phi \\
&+ (Gd^* X^0(a,\phi))\psi \Big).\notag
\end{align}
Since $X$ and $\D X$ are functionally bounded (the latter as a map from $W_k$ to $\Hom(W_j, \C_j(Y))$), we get bounds on $X(a,\phi)$ and $\D_{(a,\phi)} X^0(b,\psi)$ in terms of the $L^2_k$ norm of $(a,\phi)$ and the $L^2_j$ norm of $(b,\psi)$.  To obtain the desired $L^2_j$ bounds on \eqref{eq:igc-fb-d}, as above, we apply the continuity and bilinearity of the Sobolev multiplication $L^2_k \times L^2_j \to L^2_j$.     
  
Finally, we give the argument for the second derivative (as this will illustrate the appropriate Sobolev norms) and allow the reader to complete the proof by induction.  We consider the derivative of $\D (\Pi^{\gCoul}_* \circ X)$, which we think of as a map from $W_k$ to $\Hom(W_k \times W_j,W_j)$.  Again, let $(a,\phi) \in W_k$, $(b,\psi) \in W_j$.  We also denote an $L^2_k$ tangent vector to $(a,\phi)$ by $(\alpha, \zeta)$.  Direct computation shows 
\begin{align*}
\D^2_{(a,\phi)} (\Pi^{\gCoul}_* \circ X)((\alpha, \zeta), (b,\psi)) = \Big(\pi(\D^2_{(a,\phi)} X^0)(\alpha,\zeta), (b,\psi))&,  \\
(Gd^*(\D^2_{(a,\phi)} X^0)((\alpha,\zeta), (b,\psi))) \phi &+ Gd^*(\D_{(a,\phi)} X^0(b,\psi)) \zeta \\ &+ (\D^2_{(a,\phi)} X^1)(\alpha, \zeta),(b,\psi)\Big). 
\end{align*}
Again, the functional boundedness of $X$ and $\D X$ together with Sobolev multiplication give the desired result.   

\eqref{igc-fb:c} Similar arguments apply to establish the four-dimensional analogues.  
\end{proof}

\begin{proof}[Proof of Lemma~\ref{lem:tamecoulomb}] First, note that if $\q$ is the $L^2$-formal gradient of $f: \C(Y) \to \rr$, then $\etaq$ is the $\tilde{g}$-formal gradient of the restriction of $f$ to $W$.

Since $\eta_\q (a,\phi) = (\Pi^{\gCoul}_*)_{(a,\phi)} \q(a,\phi)$, the result follows by combining Lemmas~\ref{lem:igc} and \ref{lem:igc-fb} with the corresponding Properties (i) and (iii) of $\q$ in Definition~\ref{def:verytame}, since $\q$ is a very tame perturbation.  
\end{proof}

\section{Trajectories in global Coulomb gauge} \label{sec:trajGC}
Let $\q$ be a very tame perturbation and $\etaq$ the induced controlled perturbation in Coulomb gauge.

Let us write $\Xqgc = ((\Xqgc)^0, (\Xqgc)^1)$ where, as usual, the superscript $0$ denotes the form part and the superscript $1$ denotes the spinorial part. The vector field $\Xqgc$ on $W$ induces a vector field $\Xqgcsigma$ on $W^{\sigma}$, given by
\begin{equation}\label{eq:Xqgcsigmaformula}
\Xqgcsigma (a, s, \phi) = ((\Xqgc)^0(a, s\phi), \Lambda_\q(a, s, \phi)s,  (\widetilde\Xqgc)^1(a, s, \phi) -  \Lambda_\q(a, s, \phi) \phi),
\end{equation}
where
\begin{equation}
\label{eq:widetildeX}
(\widetilde\Xqgc)^1(a, s, \phi) = \int_0^1 \D_{(a, sr\phi)} (\Xqgc)^1(0, \phi) dr,
\end{equation}
and
\begin{equation}
\label{eq:Lambdaqas}
 \Lambda_\q(a, s, \phi) = \Re\langle \phi, (\widetilde \Xqgc)^1(a, s, \phi) \rangle_{L^2}.
 \end{equation}
Since $\Xqgc = \Pi^{\gCoul}_* \circ \Xq$, we have that $\Xqgcsigma = \Pi^{\gCoul,\sigma}_* \circ \Xqsigma$.  
\begin{lemma}\label{lem:continuityxqsigma}
The vector field $\Xqgcsigma : W^\sigma_k \to \mathcal{T}^{\gCoul,\sigma}_{k-1}$ is smooth.  
\end{lemma}
\begin{proof}
We write $\Xqgc = \Pi^{\gCoul}_* \circ (\X + \q)$.  Lemma~\ref{lem:igc} shows that $\Pi^{\gCoul}_*: \T_{k-1} \to \T^{\gCoul}_{k-1}$ is smooth.  Since $\X = \grad \L$ and $\q$ are smooth as maps from $\C_k(Y)$ to $\T_{k-1}$, we get that $\X^{\gCoul}_\q : W_k \to \T^{\gCoul}_{k-1}$ is smooth.  Hence, the induced vector field on the blow-up, $\Xqgcsigma$, is also smooth (compare \cite[Lemma 10.2.1]{KMbook}).      
\end{proof}

We are interested in the dynamics of the vector field $\Xqgcsigma$ on $W^{\sigma}$.

Every stationary point of  $\Xqsigma$ on $\C^\sigma(Y)$ can be moved into $W^\sigma$ by the global Coulomb projection.  Conversely, every stationary point $x$ of $\Xqgcsigma$ on $W^\sigma$ is also a stationary point of $\Xqsigma$, since $\Xqsigma$ lands in $\K_x^{\sigma}$ and infinitesimal global Coulomb projection induces an isomorphism from $\K_x^{\sigma}$ to $\K_x^{\agCoul,\sigma}$.  Thus, $\Pi^{\gCoul, \sigma}$ induces a bijection
\begin{equation}
\label{eq:EquivStat1}
\{ \text{stationary points of } \Xqsigma \bigr \} / \G \ \xrightarrow{\mathmakebox[2em]{\cong}} \ \{ \text{stationary points of } \Xqgcsigma \} / S^1.
\end{equation}

Note that the condition of being irreducible, boundary stable, or boundary unstable is preserved by this bijection.  By contrast, a trajectory\footnote{From now on, by trajectory we will always mean a trajectory of finite type as in Section~\ref{sec:SWe}.} of  $\Xqgcsigma$ on $W^{\sigma}$ is {\em not} a trajectory of $\Xqsigma$; still, we have:
 
\begin{proposition}
\label{prop:CPtraj}
The trajectories of $\Xqgcsigma$ on $W^{\sigma}$ are precisely the global Coulomb projections of the trajectories of $\Xqsigma$ on the blow-up $\C^\sigma(Y)$.  In fact, global Coulomb projection induces a bijection
\begin{equation}
\label{eq:EquivTraj1}
\{ \text{trajectories of } \Xqsigma \bigr \} / \G \ \xrightarrow{\mathmakebox[2em]{\cong}}\ \{ \text{trajectories of } \Xqgcsigma \} / S^1,
\end{equation}
where $\G$ acts on trajectories $x(t)$ by three-dimensional gauge transformations, constant in $t$.
\end{proposition}

\begin{remark}
Because $\q$ is a very tame perturbation, every $L^2_k$ trajectory of $\Xqsigma$ is gauge-equivalent to a smooth one.  Furthermore, every trajectory of $\Xqgcsigma$ is smooth.      
\end{remark}

\begin{proof}[Proof of Proposition~\ref{prop:CPtraj}]
If $\gamma$ is a smooth trajectory of $\Xqsigma$, its global Coulomb projection $\gamma^{\flat}$ is a trajectory of $\Pi^{\gCoul, \sigma}_* \circ  \Xqsigma$, which is exactly $\Xqgcsigma$.

Conversely, given a trajectory $\gamma$ of $\Xqgcsigma$, we can lift it to a smooth trajectory $\gamma^{\sharp}$ of $\Xqsigma$ as follows. Consider the submanifold $\mathcal{O}(\gamma)_k \subset \C_k^{\sigma}(Y)$ consisting of all points that are on the gauge orbits of points on $\gamma$. Thus, any $x \in \mathcal{O}(\gamma)_k$ can be written as $x=u \cdot \gamma(t_0)$ where $u$ is an $L^2_{k+1}$ gauge transformation on $Y$ and $t_0 \in \rr$. The vector $v_x = \Xqsigma(x)$ is in the local Coulomb slice $\K^\sigma_{x, k-1}$. Consider the push-forward $(u^{-1})_* v_x$, which is in the local Coulomb slice to $u^{-1} \cdot x = \Pi^{\gCoul}(x)$. Now, $u^{-1} \cdot x$ is part of the trajectory $\gamma$, and hence is smooth. Therefore, $(u^{-1})_* v_x= (\Xqsigma)_{u^{-1} \cdot x}$ is also smooth. Pushing it back by the $L^2_{k+1}$ gauge transformation $u$ yields the vector $v$, which we now see that it must be in $L^2_k$. Thus, the vectors $v_x$ form a true vector field on the Banach manifold $\mathcal{O}(\gamma)_k$. By integrating this vector field starting at a smooth point $x \in \mathcal{O}(\gamma)_k$, we obtain a lift $\gamma^\sharp$ of $\gamma$. We can do this for any $k$, and obtain the same lift; hence, $\gamma^\sharp$ is smooth.

In the above construction, note that the lift $ \gamma^\sharp$ is unique up to transformation by an element in $\G$. Indeed, in four-dimensions, after moving to temporal gauge, there are only gauge transformations which are constant in $t$. (Compare \cite[Proposition 7.2.1]{KMbook}.)
\end{proof}

Using the formula~\eqref{eq:piCS} for $(\Pi^{\gCoul,\sigma}_*)_{(a, s, \phi)}$, we can describe the trajectories of $\Xq^{\gCoul, \sigma}$ more explicitly, as the paths $(a(t),s(t),\phi(t))$ in slicewise Coulomb gauge that satisfy
\begin{align}
\nonumber &  \frac{d}{dt} a = - *da - s^2\pi ( \tau(\phi, \phi)) - \pi(\q^0(a,s\phi)), \\
\label{eqn:xqgcsigma} & \frac{d}{dt} s = - \Lambda_\q(a,s,\phi)s, \\
\nonumber & \frac{d}{dt} \phi = - D_a \phi - \tilde{q}^1(a,s,\phi) + \Lambda_q(a,s,\phi)\phi - s^2 Gd^*(\tau(\phi,\phi))\phi - Gd^*(\q^0(a,s\phi))\phi.
\end{align}
Alternatively, we can write the trajectory equations as 
\begin{align}
\nonumber   \frac{d}{dt} a &= - *da - s^2\pi ( \tau(\phi, \phi)) - \pi(\q^0(a,s\phi)), \\
\label{eq:xqgcsigmaalternate}  \frac{d}{dt} s &= - \Lambda_\q(a,s,\phi)s, \\
\nonumber  \frac{d}{dt} \phi &= - D_a \phi -  s^2 \xi(\phi)\phi - \widetilde{\eta}_{\q}^1(a,s,\phi) + \Lambda_\q(a,s,\phi)\phi \\
\nonumber & = -D \phi - \tilde{c}^1(a,s,\phi) -  \widetilde{\eta}_{\q}^1(a,s,\phi) + \Lambda_\q(a,s,\phi)\phi,
\end{align}
where $\xi(\phi)$ is as in \eqref{eq:cmap}, and we use $\widetilde{\eta}_{\q}^1(a,s,\phi)$ to denote $ \int_0^1   \D_{(a, sr\phi)} {\eta}_{\q}^1(0,\phi) dr$.  

By making use of $\G^{\gCoul}(Z)$-equivariance as in \eqref{eq:hateta}, the equations defining the flow  of $\Xq^{\gCoul, \sigma}$ define a section
\begin{equation}\label{eq:Fqgctau}
 \F^{\gCoul, \tau}_{\q}: \C^{\gCoul, \tau}(Z) \to \V^{\gCoul, \tau}(Z),
\end{equation}
so that in temporal gauge we have
$$ \F^{\gCoul, \tau}_{\q} = \frac{d}{dt} + \Xqgcsigma.$$
It follows that the space of trajectories of $\Xqgcsigma$ modulo the action of $S^1$ is the zero set of $\F^{\gCoul,\tau}_{\q}$ modulo the action of $\G^{\gCoul}(Z)$.  As in Section~\ref{sec:AdmPer}, it will be useful to consider the extension of $\F^{\gCoul,\tau}_{\q}$ to $\tC^{\gCoul,\tau}(Z)$, so that we may differentiate this map.

\section{Hessians}
In Section 12.4 in \cite{KMbook}, Kronheimer and Mrowka study the derivative of the vector field $\Xq^{\sigma}$ on $\C^{\sigma}_k(Y)$. At each $x \in \C^{\sigma}_{k}(Y)$, this is an operator between the local Coulomb slices:
$$ \Hess^{\sigma}_{\q, x} : \K^{\sigma}_{k,x} \to \K^{\sigma}_{k-1, x},$$ 
called the {\em Hessian} (in the blow-up). The Hessian is needed in Floer theory for several reasons.  First, the non-degeneracy of a stationary point $x$ is expressed in terms of the surjectivity of $\Hess^{\sigma}_{\q, x}$. Second, we need to study Hessians at all points (not necessarily stationary) in order to construct the operator
\begin{equation}
\label{eq:dHess}
 \frac{d}{dt} + \Hess^{\sigma}_{q},
 \end{equation}
which can be applied to vector fields along a path $\gamma$ in $\C^{\sigma}(Y)$. If $\gamma$ is a flow trajectory for $\Xq^{\sigma}$, then the surjectivity of $d/dt + \Hess^{\sigma}_{q}$ indicates that the moduli space of trajectories is regular at $\gamma$. Also, for arbitrary $\gamma$ (not necessarily a flow trajectory), we need the operator~\eqref{eq:dHess} to describe the relative grading of stationary points, and to construct orientations on the moduli spaces of trajectories.

We will do the same analysis in global Coulomb gauge. We will first define a $\tilde{g}$-Hessian before the blow-up, at irreducible points $(a, \phi)$ (that is, those with $\phi \neq 0$). We will then construct a $\tilde{g}$-Hessian on the blow-up, well-defined at all points.

\subsection{The Hessian at irreducibles} \label{sec:Hess}
We start by recalling the original Hessian at irreducibles, as defined in Section 12.3 of \cite{KMbook}: 
\begin{equation}
\label{eq:HessIrr}
\Hess_{\q,x} = \Pi^{\lCoul}_x \circ \D_{x} \Xq : \K_{k,x} \to \K_{k-1,x}.
\end{equation}
This formula applies to any $x=(a, \phi) \in \C_k(Y)$ with $\phi \neq 0$. 

\begin{remark}
When $x$ is a stationary point, we can actually drop the projection $\Pi^{\lCoul}_x$ from the formula for $\Hess_{\q,x}$. This is because $\Xq$ is a formal gradient, so it is orthogonal to the gauge orbits; therefore, it can be viewed as a section of $\K_{k-1}$, and its derivative at a zero point must automatically land in the local Coulomb slice.  
\end{remark}

The $\tilde{g}$-Hessian is defined as: 
\begin{equation}
\label{eq:HessGCoul}
\Hess^{\tilde{g}}_{\q,x} = \Pi^{\agCoul}_x \circ \Dg_x \Xqgc : \K^{\agCoul}_{k,x} \to \K^{\agCoul}_{k-1,x}. 
\end{equation}
Here, $\Dg$ denotes the connection on $\T^{\gCoul}$ induced by the $\tilde g$ metric on $W_k$. To put this in context, recall that the derivative of a vector field is well-defined at a stationary point $x$, but when $x$ is not stationary it depends on the choice of a connection in the tangent bundle. In \eqref{eq:HessIrr}, we simply viewed the vector field $\Xq$ as a map into affine space, and took its derivative as such. Doing the same thing in global Coulomb gauge would mean using the connection $\D$ induced by the ordinary $L^2$ metric. However, $\Xqgc$ is the $\tilde{g}$-gradient  of $(\L + f)|_{W}$, and taking its ordinary derivative $\D_x \Xqgc$ would yield an operator that is not symmetric. It is more natural to use the connection $\Dg$, which is explicitly given by the formula:
\begin{equation}
\label{eq:Dg}
\Dg (X) =   \Pi^{\gCoul}_* \circ \D ( \Pi^{\elCoul}(X)) \circ \Pi^{\elCoul}.
\end{equation}
Here, $X$ is a vector field on $W_k$, and $\Pi^{\elCoul}(X)$ is a section of $\T_k$ over $W_k$. We implicitly extend $\Pi^{\elCoul}(X)$ to a $\Go_{k+1}$-invariant vector field on the whole configuration space $\C_k(Y)$; this allows us to take its derivative in the direction $\Pi^{\elCoul}(Z)$, for some $Z$.

Let us check that the formula \eqref{eq:Dg} produces a connection compatible with $\tilde g$. If $X, Y$ and $Z$ are vector fields on $W_k$, we have:
\begin{align*}
Z \langle X, Y \rangle_{\tilde g} &=  Z \Re \langle\Pi^{\elCoul}(X), \Pi^{\elCoul}(Y) \rangle_{L^2}   \\
&= \Pi^{\elCoul}(Z) \Re \langle\Pi^{\elCoul}(X), \Pi^{\elCoul}(Y) \rangle_{L^2}\\
&= \Re \langle (\D ( \Pi^{\elCoul}(X))) (\Pi^{\elCoul}(Z)),  \Pi^{\elCoul}(Y) \rangle_{L^2} + \langle   \Pi^{\elCoul}(X), (\D ( \Pi^{\elCoul}(Y))) (\Pi^{\elCoul}(Z))\rangle_{L^2} \\
&= \Re \langle (\Dg X) (Z),  \Pi^{\elCoul}(Y) \rangle_{L^2} + \langle   \Pi^{\elCoul}(X), (\Dg Y) (Z)\rangle_{L^2} \\
&=  \Re \langle  \Pi^{\elCoul}( (\Dg X) (Z)),  \Pi^{\elCoul}(Y) \rangle_{L^2} + \langle \Pi^{\elCoul}(X),  \Pi^{\elCoul}((\Dg Y) (Z))\rangle_{L^2} \\
&= \Re \langle (\Dg X)(Z), Y \rangle_{\tilde g} + \langle  X, (\Dg Y) (Z)\rangle_{\tilde g}.
\end{align*}
Here, we have implicitly extended $\Pi^{\elCoul}(X)$ and $\Pi^{\elCoul}(Y)$ to $\Go_{k+1}$-invariant vector fields on the configuration space. The second equality above is true because the function $$\Re \langle\Pi^{\elCoul}(X), \Pi^{\elCoul}(Y) \rangle_{L^2} $$ is $\Go_{k+1}$-invariant, and therefore its derivative in the direction $Z-\Pi^{\elCoul}(Z) \in \J^{\circ}_{k}$ vanishes. For some of the other equalities, we used repeatedly the fact that the $L^2$ inner product with something in the enlarged local Coulomb slice is unchanged by applying either $\Pi^{\gCoul}_*$ or $ \Pi^{\elCoul}$ to the other factor. This last fact is true because both of these operations consist in adding a vector in $\J^{\circ}_{k}$, and such vectors are $L^2$ perpendicular to the enlarged local Coulomb slice. 

One can also check that the connection $\Dg$ is torsion-free:
\begin{align*}
(\Dg X)(Y) - (\Dg Y)(X) &= \Pi^{\gCoul}_* \circ \D ( \Pi^{\elCoul}(X)) (\Pi^{\elCoul}(Y))- \Pi^{\gCoul}_* \circ \D ( \Pi^{\elCoul}(Y)) (\Pi^{\elCoul}(X))\\
&= \Pi^{\gCoul}_* [\Pi^{\elCoul}(X), \Pi^{\elCoul}(Y)] \\
&=[X, Y].
\end{align*}
To see the last equality, it suffices to check how the respective vector fields act on $\Go_{k+1}$-invariant functions on the configuration space. On such functions $f$, the action of a vector field does not change if we apply $\Pi^{\elCoul}$ or $\Pi^{\gCoul}_*$ to that vector field (because this means changing it by a vector field in a direction tangent to $\Go_{k+1}$). The equality follows easily from this observation.

\begin{remark}
It follows from the above discussion that $\Dg$ is the exact analogue of the Levi-Civita connection on the infinite-dimensional manifold $W$ with respect to the $\tilde{g}$-metric.  Therefore, when we compute the $\tilde{g}$-Hessian of a real-valued function on $W$, we will obtain a symmetric operator.  
\end{remark}

From Section~\ref{sec:coulombs} we know that the projection $\Pi^{\gCoul}_*$ maps $\Ke$ isomorphically onto $\T^{\gCoul}$, with inverse $\Pi^{\elCoul}$. Since $\Xqgc = \Pi^{\gCoul}_* \Xq$, we also have $\Xq = \Pi^{\elCoul} \Xqgc$. From here, in view of \eqref{eq:Dg}, we can obtain a simpler formula for the Hessian $\Hess^{\tilde{g}}_{\q,x}$:
\begin{equation}
\label{eq:altHess}
\Hess^{\tilde{g}}_{\q,x} = \Pi^{\agCoul}_x \circ \D_x \Xq \circ \Pi^{\elCoul}_x : \K^{\agCoul}_{k,x} \to \K^{\agCoul}_{k-1,x}. 
\end{equation}

\begin{remark}
The vector field $\Xqgc$ points $\tilde{g}$-orthogonally to the $S^1$-orbits on $W_{k-1}$. Hence, when $x$ is a stationary point of $\Xqgc$, the image of $\D^{\tilde g}\Xqgc$ must be contained in $\K_{k-1}^{\agCoul}$. Further, in this case the derivative of $\Xqgc$ is independent of which connection we use.  Therefore, when $x$ is stationary, we can write  
\begin{equation}
\label{eq:statHess}
 \Hess^{\tilde{g}}_{\q,x} =  \D_x \Xqgc = \Dg_x \Xqgc  : \K^{\agCoul}_{k,x} \to \K^{\agCoul}_{k-1,x}.
 \end{equation}
\end{remark}

It is shown in \cite[Proposition 12.3.1]{KMbook} that $\Hess_{\q,x}$ is a Fredholm operator of index zero.  Similarly, using the $\tilde{g}$-Hessian we have:

\begin{lemma}\label{lem:hessianfredholm}
For any $x=(a, \phi) \in W_k$ with $\phi \neq 0$, the operator $\Hess^{\tilde{g}}_{\q,x}: \K^{\agCoul}_{k,x} \to \K^{\agCoul}_{k-1,x}$ is Fredholm of index zero.  Therefore, it is surjective if and only if it is injective.  
\end{lemma}

\begin{proof} We see from \eqref{eq:altHess} that $\Hess^{\tilde{g}}_{\q,x}$ is the conjugate of $\Hess_{\q, x}$ by the isomorphism $\Pi^{\agCoul}_x:  \K_{k,x} \to \K^{\agCoul}_{k,x}$, with inverse $\Pi^{\elCoul}_x$ (see Lemma~\ref{lem:bijection}). Since $\Hess_{\q, x}$ is Fredholm of index $0$, so is $\Hess^{\tilde{g}}_{\q, x}$.
 \end{proof}

At this point it is helpful to recall the proof of \cite[Proposition 12.3.1]{KMbook}, which says that $\Hess_{\q,x}$ is a Fredholm operator of index zero. That proof will serve as a model for some of our later arguments. 

We start by recalling \cite[Definition 12.2.1]{KMbook}:

\begin{definition}
An operator $L$ is called {\em $k$-almost self-adjoint first-order elliptic} ($k$-ASAFOE) if it is of the form
$$ L=L_0+h$$
where
\begin{itemize}
\item $L_0$ is a first-order, self-adjoint, elliptic differential operator (with smooth coefficients) acting on sections of a vector bundle $E \to Y$;\\
\item $h: C^\infty(Y; E) \to L^2(Y; E)$ is a linear operator on sections of $E$ that extends to a bounded map on $L^2_j(Y; E)$ for all $j$ with $|j| \leq k$.
\end{itemize}
\end{definition}

For $|j| \leq k$, note that a $k$-ASAFOE operator $L_0+h$ is Fredholm of index zero when viewed as a map $L^2_j(Y;E) \to L^2_{j-1}(Y; E)$. Indeed, this statement is true for $L_0$ (by ellipticity and self-adjointness), and adding $h$ is just a compact perturbation.

Note that $\Hess_{\q, x}$ is not $k$-ASAFOE, because it does not act on {\em all} sections of a vector bundle. The remedy in \cite{KMbook} was to introduce an extended Hessian
$$ \widehat{\Hess}_{\q, x} : \T_{k, x} \oplus L^2_k(Y; i\R) \to \T_{k-1, x} \oplus L^2_{k-1}(Y; i\R)$$
by the formula
 \begin{equation}
 \label{eq:ExtHess}
  \widehat{\Hess}_{\q, x} = \begin{pmatrix}
\D_x \Xq & \dd_x \\
\dd_x^* & 0 
\end{pmatrix}.
\end{equation}
Here, $\dd_x$ encodes the infinitesimal gauge action at $x=(a, \phi)$:
$$\dd_x (\xi) = (-d\xi, \xi \phi)$$
and its adjoint 
$$\dd_x^*(b, \psi) = -d^* b + i \Re\langle i\phi, \psi\rangle$$
can be used to define the local Coulomb slice $\K_{k, x}$ by the condition $\dd_x^* =0$.

The extended Hessian is a self-adjoint $k$-ASAFOE operator acting on sections of the bundle $iT^*Y \oplus \Spin \oplus i\R$. With respect to the decomposition $\T_{j, x}\oplus L^2_j(Y; i\R)  = \J_{j, x} \oplus \K_{j, x} \oplus L^2_j(Y; i\R)$, with $j=k$ for the domain and $j=k-1$ for the target, we can write\footnote{In \cite[Equation (12.7)]{KMbook}, the term $h_1$ was inadvertently missing.}
$$
  \widehat{\Hess}_{\q, x} = \begin{pmatrix}
h_1 & h_2 & \dd_x \\
h_3 &  \Hess_{\q, x} & 0  \\
\dd_x^* & 0 & 0
\end{pmatrix}. 
$$
Here, the $h_i$ terms preserve Sobolev regularity and hence are compact as maps $L^2_k \to L^2_{k-1}$. (As an aside, note also that the $h_i$ terms are zero at stationary points.) Thus, dropping $h_1, h_2$ and $h_3$ from the expression above yields a Fredholm operator of index $0$. In turn, since the $h_i$ are compact, this shows that $\Hess_{\q, x}$ must be Fredholm of index $0$.

\subsection{The Hessian on the blow-up} \label{sec:HessBlowUp}
We now define the analogous Hessian on the blow-up.  
In Section 12.4 of \cite{KMbook}, the Hessian on the blow-up is defined by:
\[
\Hess^{\sigma}_{\q,x} = \Pi^{\lCoul,\sigma}_x \circ \D^\sigma_x \Xqsigma : \K^{\sigma}_{k,x} \to \K^{\sigma}_{k-1,x}.
\]
Note that since $\C^\sigma_k(Y)$ is only a submanifold of a vector space, the derivative of the section $\Xqsigma$ need not land in its tangent bundle $\T^\sigma_{k-1}$. What we actually mean\footnote{In \cite{KMbook}, the notation $\D$ is used instead of $\D^\sigma$.  We use $\D^\sigma$ to distinguish this derivative from the $L^2$ derivative $\D$ in the ambient vector space $L^2(Y;iT^*Y) \oplus \R \oplus L^2(Y;\Spin)$, which we will also work with.} by  $\D^\sigma_x \Xqsigma$ is the composition of the derivative in the larger vector space with the $L^2$ orthogonal projection onto $\T^\sigma_{k-1}$. To obtain the Hessian, we further project to the corresponding local Coulomb slice by $\Pi^{\lCoul,\sigma}_x$.

To show that $\Hess^{\sigma}_{\q,x} $ is Fredholm of index zero, Kronheimer and Mrowka introduced the extended Hessian in the blow-up:
\begin{equation}
\label{eq:ExtHessBlowUp}
  \widehat{\Hess}^{\sigma}_{\q, x} = \begin{pmatrix}
\D^\sigma_x \Xqsigma & \dd^{\sigma}_x \\
\dd_x^{\sigma, \dagger} & 0 
\end{pmatrix}.
\end{equation}
Here,
\begin{align}
\label{eq:dsigma}
 \dd^{\sigma}_x (\xi) &= (-d\xi, 0,\xi \phi), \\
\label{eq:dsigmadagger}
 \dd^{\sigma, \dagger}_x (b, r, \psi) &= -d^*b + is^2\Re\langle i\phi, \psi \rangle + i |\phi|^2 \Re \mu_Y(\langle i\phi, \psi\rangle),
 \end{align}
so that $ \K_j^{\sigma} = \ker \dd^{\sigma, \dagger}_x$,  $\J_j^{\sigma} = \im \dd^{\sigma}_x$ are the local Coulomb slice and the tangent to the gauge orbit, as in \eqref{eq:BlowUpDecompose}. Note that on the blow-up, $\dd_x^{\sigma, \dagger}$ is not quite the adjoint to $\dd_x^{\sigma}$, just as the decomposition $\T^{\sigma}_j = \J^{\sigma}_j \oplus \K^{\sigma}_j$ is not $L^2$ orthogonal. Therefore, unlike  the extended Hessian $\widehat{\Hess}_{\q, x}$, the operator $\widehat{\Hess}^{\sigma}_{\q, x}$ is not symmetric. Moreover, $\widehat{\Hess}^{\sigma}_{\q,x}$ does not a priori act on sections of a vector bundle, but rather on the subspace of $ \T^{\sigma}_{j, x} \oplus L^2_j(Y; i\R)$, where we have the condition $\Re \langle \phi, \psi\rangle_{L^2} = 0$. Nevertheless, we can combine the $r$ and $\psi$ components of $(b, r, \psi) \in \T^{\sigma}_{j, x}$ into
$$ \psis = \psi + r \phi,$$
and then we can think of $\widehat{\Hess}^{\sigma}_{\q, x}$ as acting on sections of $iT^*Y \oplus \Spin \oplus i\R$. It then becomes a $k$-ASAFOE operator (hence Fredholm of index zero), and starting from here one can show that $\Hess^{\sigma}_{\q,x} $ is Fredholm of index zero---much as in the proof of Lemma~\ref{lem:hessianfredholm}. Further, when $x$ is a non-degenerate stationary point, $\Hess^{\sigma}_{\q,x} $ is invertible and has real spectrum (even though it is not self-adjoint). See \cite[Section 12.4]{KMbook} for details.

Let us now move to global Coulomb gauge. Fix $x = (a, s, \phi) \in W^{\sigma}$. In this setting, we define a $\tilde{g}$-Hessian in the blow-up by:
\begin{equation}
\label{eq:HessianBlowupCG}
\Hess^{\tilde{g},\sigma}_{\q,x} = \Pi^{\agCoul,\sigma}_x \circ   \Dgs_x \Xqgcsigma : \K^{\agCoul,\sigma}_{k,x} \to \K^{\agCoul,\sigma}_{k-1,x},  
\end{equation}
where
\begin{equation}
\label{eq:DGS}
 \Dgs X := \Pi^{\gCoul, \sigma}_* \circ  \D^\sigma(\Pi^{\elCoul, \sigma}(X)) \circ \Pi^{\elCoul, \sigma}.
 \end{equation}
is the connection on $\T^{\gCoul, \sigma}$ defined by analogy with \eqref{eq:Dg}.

Since $\Xqgcsigma = \Pi^{\gCoul, \sigma}_* \circ \Xqsigma$, in view of Lemma~\ref{lem:bijection2}(a) we  have $\Xqsigma = \Pi^{\elCoul, \sigma} \circ \Xqgcsigma$. Therefore, we can re-write the blown-up $\tilde{g}$-Hessian as
\begin{equation}
\label{eq:Hess2}
\Hess^{\tilde{g},\sigma}_{\q,x} =  \Pi^{\agCoul,\sigma}_x \circ  \D^\sigma_x \Xqsigma \circ \Pi^{\elCoul, \sigma}_x.
\end{equation}

\begin{remark}
\label{rem:HessEqual}
When $x$ is on the blow-up locus (i.e., $s=0$), we know from Lemma~\ref{lem:bijection} (b) that the anticircular global Coulomb slice coincides with the local Coulomb slice and $\Pi^{\lCoul, \sigma}_x = \Pi^{\agCoul, \sigma}_x$.  We see from \eqref{eq:Hess2} that $\Hess^{\tilde{g},\sigma}_{\q,x}$ agrees with the ordinary blow-up Hessian $\Hess^{\sigma}_{\q,x}$. 
\end{remark}

\begin{lemma}
\label{lem:HessB}
$(a)$ For any $x \in W^{\sigma}_k$, the operator $\Hess^{\tilde{g},\sigma}_{\q,x}$ is Fredholm of index zero.

$(b)$ When $x$ is a non-degenerate stationary point of $\Xqgcsigma$, the operator  $\Hess^{\tilde{g},\sigma}_{\q,x} $ is invertible and has real spectrum.
\end{lemma}

\begin{proof}
When $x$ is reducible, we have  $\Hess^{\tilde{g},\sigma}_{\q,x} = \Hess^{\sigma}_{\q,x}$, and the corresponding results for $\Hess^{\sigma}_{\q,x}$ were established in \cite[Section 12.4]{KMbook}.  
 
 When $x$ is irreducible, $\Hess^{\tilde{g},\sigma}_{\q,x} $ is conjugate to $\Hess^{\tilde{g}}_{\q,x}$  via the blow-down map. Thus, part (a) is a consequence of Lemma~\ref{lem:hessianfredholm}. For part (b), since $x$ is stationary, by \eqref{eq:altHess}, we have that $\Hess^{\tilde{g}}_{\q,x} = \D^{\tilde g}_{x} \Xqgc$, and the latter operator is invertible and self-adjoint (being a formal Hessian). The conclusion follows.
\end{proof}

Let us also mention:

\begin{lemma}\label{lem:hessiancontinuity}
The map $\Hess^{\tilde{g},\sigma}_{\q} : \K_k^{\agCoul,\sigma} \to \K_{k-1}^{\agCoul,\sigma}$ is a continuous bundle map.  
\end{lemma}
\begin{proof}
This follows from the continuity of the factors in \eqref{eq:Hess2}. In particular, recall that $\Pi^{\agCoul, \sigma}$ is the composition of $\Pi^{\gCoul,\sigma}_*$ with the projection \eqref{eq:antiproj} onto $\K^{\agCoul, \sigma}$. The continuity of $\Pi^{\gCoul,\sigma}_*$ can be proved by an argument similar to that in Lemma~\ref{lem:igc}. As for the projection \eqref{eq:antiproj}, it can be seen from that formula that it is a continuous operation. Finally, $ \Pi^{\elCoul, \sigma}$ is continuous by the same arguments as in Lemma~\ref{lem:elc}.
\end{proof}

\subsection{The split extended Hessian on the blow-up}
\label{sec:ExtSpilt}
For the purposes of this book, it is helpful to work with a different extended Hessian than the one in \eqref{eq:ExtHessBlowUp}. Precisely, we replace \eqref{eq:dsigmadagger} with an operator
$$ \dd^{\sp, \sigma, \dagger}_x : \T^{\sigma}_{j,x} \to L^2_j(Y; i\rr),$$
defined as follows. We decompose the domain $\T^{\sigma}_{j,x}= \K^\sigma_{j,x} \oplus \J^{\sigma}_{j,x} $ into three summands $\K^\sigma_{j,x} \oplus \J^{\gCoul, \sigma}_{j,x} \oplus \J^{\circ, \sigma}_{j,x}$, and also decompose the codomain $L^2_{j-1}(Y; i\R)$ into $(\im d^*)_{j-1}  \oplus i\R$ (that is, into functions that integrate to zero and constant functions). With respect to these decompositions, we let 
\begin{equation}
\label{eq:decomposing}
 \dd^{\sp, \sigma, \dagger}_x = \begin{pmatrix} 0 & 0 & -d^* \\
 0 & i \Re \langle i\phi, \cdot \rangle_{L^2}  & 0 \end{pmatrix},
\end{equation}
where $-d^*$ acts on the first component of $(-d\xi, 0, \xi\phi) \in  \J^{\circ, \sigma}_{j,x}$, and $i \Re \langle i\phi, \cdot \rangle_{L^2}$ acts on the last component of $ (0, 0, it\phi) \in\J^{\gCoul, \sigma}_{j,x}$, $t\in \rr$. Note that $\| \phi \|_{L^2} = 1$, so $i \Re \langle i\phi, it\phi \rangle_{L^2}$  simply equals $it$.

One can also write a more compressed formula for $ \dd^{\sp, \sigma, \dagger}_x$. Recall from 
\eqref{eq:piLCsigma} that $\Pi^{\lCoul,\sigma}_{(a,s,\phi)}(b,r,\psi) = (b-d\zeta,r, \psi + \zeta \phi),$ 
for some $\zeta = \zeta(x, b,r,\psi): Y \to i\rr$. Then:
\begin{align*}
 \dd^{\sp, \sigma, \dagger}_x (b,r,\psi) &=  d^*d \zeta+  \mu_Y(\zeta).
\end{align*}

Note that the kernel of $ \dd^{\sp, \sigma, \dagger}_x$ is the local Coulomb slice $\K^{\sigma}_{j,x}$, just as for $\dd^{\sigma, \dagger}_x$.
 
  We now define the {\em split extended Hessian in the blow-up} to be
  \begin{equation}
\label{eq:SplitExtHessBlowUp}
  \widehat{\Hess}^{\sp, \sigma}_{\q, x} = \begin{pmatrix}
\D^\sigma_x \Xqsigma & \dd^{\sigma}_x \\
\dd_x^{\sp, \sigma, \dagger} & 0 
\end{pmatrix}: \T^{\sigma}_{j,x} \oplus L^2_j(Y; i\rr) \to \T^{\sigma}_{j-1,x} \oplus L^2_{j-1}(Y; i\rr).
\end{equation}

 At any $x$, by combining $\psi$ and $r$ into $\psis=\psi + r\phi$ as in the case of $ \widehat{\Hess}^{\sigma}_{\q, x}$, we see that $ \widehat{\Hess}^{\sp, \sigma}_{\q, x}$ is a $k$-ASAFOE operator. We also have:

\begin{lemma}
\label{lem:eHsplit}
If $x$ is a non-degenerate stationary point of $\Xqgcsigma$, then $\widehat{\Hess}^{\sp, \sigma}_{\q, x}$ is invertible and has real spectrum.
\end{lemma}

\begin{proof}
The argument is similar to that in \cite[proof of Lemma 12.4.3]{KMbook}.  The operator $ \widehat{\Hess}^{\sp, \sigma}_{\q, x}$ has a block form where one block is ${\Hess}^{\sigma}_{\q, x}$ and the other is 
\begin{equation}\label{eq:block-hess-dsigma}
\begin{pmatrix}
0 &\dd^{\sigma}_x \\
\dd^{\sp, \sigma, {\dagger}}_x & 0
\end{pmatrix}.
\end{equation}
It is established in \cite[proof of Lemma 12.4.3]{KMbook} that the operator ${\Hess}^{\sigma}_{\q, x}$  is invertible and has real spectrum.  To justify that \eqref{eq:block-hess-dsigma} is invertible and has real spectrum, it suffices to prove that the operator $\dd^{\sp, \sigma, \dagger}_x \dd^{\sigma}_x$ is self-adjoint and strictly positive. 

To see this last fact, we compute
$$\dd^{\sp, \sigma, \dagger}_x \dd^{\sigma}_x (\xi) = \dd^{\sp, \sigma, \dagger}_x (-d\xi, 0, \xi \phi) = \Delta \xi + \mu_Y(\xi).$$

With respect to the orthogonal decomposition $L^2_{j-1}(Y; i\R)=(\im d^*)_{j-1}  \oplus i\R$, we have
$$ \dd^{\sp, \sigma, \dagger}_x \dd^{\sigma}_x = \begin{pmatrix}
\Delta & 0 \\
0 & 1
\end{pmatrix}.$$
Both diagonal entries, $\Delta$ and $1$, are self-adjoint and positive on the respective summands. This completes the proof. 
\end{proof}

\begin{remark}
We could have defined a split extended Hessian in the blow-down as well (at irreducibles), but we had no need for that construction.
\end{remark}

\subsection{The $\tilde g$--extended Hessian on the blow-up} \label{sec:ExtHessBlowUp}
We now construct another extended Hessian in the blow-up, using the $\tilde{g}$ metric. The definition is somewhat similar to that of the split extended Hessian in Section~\ref{sec:ExtSpilt} above. Precisely, we set
 \begin{equation}
 \label{eq:ExtHessBlowUpGC}
   \widehat{\Hess}^{\tilde{g}, \sigma}_{\q, x} = \begin{pmatrix}
S_x \circ \D^\sigma_x \Xq^{\sigma} \circ S_x^{-1} & \dd^{\sigma}_x \\[.5em]
\dd^{\sigma, \tilde{\dagger}}_x & 0 
\end{pmatrix}.
\end{equation}
Here, $S_x$ is the shear map from \eqref{eq:shear}, and to define $\dd^{\sigma,\tilde{\dagger}}_x: \T^{ \sigma}_{j,x} \to L^2_{j-1}(Y; i\R)$, we use the decompositions
$$\T^{\sigma}_{j,x}=\K^{\agCoul, \sigma}_{j,x} \oplus \J^{\gCoul, \sigma}_{j,x} \oplus \J^{\circ, \sigma}_{j,x}, \ \ \ \ \ L^2_{j-1}(Y; i\R) = (\im d^*)_{j-1}  \oplus i\R,$$ 
and set
\begin{equation}
\label{eq:decomposingtilde}
 \dd^{\sigma, \tilde{\dagger}}_x = \begin{pmatrix} 0 & 0 & -d^* \\
 0 & i\langle i\phi, \cdot \rangle_{\tilde g}  & 0 \end{pmatrix}.
\end{equation}

Alternatively, we have the formula 
\begin{equation}
\label{eq:dcsdagger}
\dd^{\sigma, \tilde{\dagger}}_x(b, r, \psi) = -d^* b +  i \langle i\phi, \psi + (Gd^*b) \phi \rangle_{\tilde{g}}.
\end{equation}
  
Observe the kernel of $\dd^{\sigma, \tilde{\dagger}}_x$ is the anticircular global Coulomb slice $\K^{\agCoul, \sigma}_{j,x}$, given by the conditions $d^*b=0$ and $\langle i\phi, \psi\rangle_{\tilde g}=0$.
  
At any $x$, by combining $\psi$ and $r$ into $\psis=\psi + r\phi$ as before, we would like to claim that $ \widehat{\Hess}^{\tilde{g}, \sigma}_{\q, x}$ is a $k$-ASAFOE operator as a first step to proving that $\widehat{\Hess}^\sigma_{\q,x}$ is invertible and has real spectrum at non-degenerate stationary points.  It turns out that if $x$ is not a stationary point, then $\widehat{\Hess}^{\tilde{g}, \sigma}_{\q, x}$ is only $(k-1)$-ASAFOE, so we will have to be careful.  To establish these properties, we will use a different operator which agrees with $ \widehat{\Hess}^{\tilde{g}, \sigma}_{\q, x}$ at stationary points.  This operator will also be useful in Section~\ref{sec:linearized} and is described in the following lemma.

\begin{lemma}\label{lem:fakehessian}
Let $x \in W^\sigma_k$ and $1 \leq j \leq k$.  Consider the operator $\Hx: \T^\sigma_{j,x} \oplus L^2_j(Y;i\R) \to  \T^\sigma_{j-1,x} \oplus L^2_{j-1}(Y;i\R) $ given in block form by 
$$
\Hx = \begin{pmatrix} (\D^\sigma_x \Xqgcsigma) \circ \Pi^{\gCoul,\sigma}_* & \dd^\sigma_x \\ \dd^{\sigma,\tilde{\dagger}}_x & 0  \end{pmatrix}.
$$ 
\begin{enumerate}[(a)]
\item\label{fakehessian-asafoe} Under the identification of $\T^\sigma_{j,x} \oplus L^2_j(Y;i\R)$ with $L^2_j(Y;iT^*Y \oplus \Spin \oplus \R)$ and likewise for $j-1$, the operator $\Hx$  is $(k-1)$-ASAFOE with linear part $$L_0 = \begin{pmatrix} *d & 0 & -d \\ 0 & D & 0 \\ -d^* & 0 & 0 \end{pmatrix};$$  
\item \label{fakehessian-compact-asafoe} when $j=k$, the operator $\Hx$ differs from $L_0$ by a compact operator from $L^2_k$ to $L^2_{k-1}$;
\item \label{fakehessian-k-asafoestationary} if $x$ is a stationary point of $\Xqgcsigma$, then $\Hx$ is $k$-ASAFOE; 
\item \label{fakehessian-realstationary} if $x$ is a stationary point of $\Xqgcsigma$, then $\Hx = \widehat{\Hess}^{\tilde{g}, \sigma}_{\q, x}$.
\end{enumerate} 
\end{lemma}
\begin{proof}
\eqref{fakehessian-asafoe} For notation, we write $L^\sigma:\T^\sigma_{j,x} \oplus L^2_j(Y;i\R) \to \T^\sigma_{j-1,x} \oplus L^2_{j-1}(Y;i\R)$ for the operator induced by $L_0$.  Our goal is to show that $\Hx$ differs from $L^\sigma$ by bounded operators from $L^2_{j}$ to $L^2_{j}$ for $1 \leq j \leq k-1$.  (Technically, we must show that these are induced by operators from $C^\infty$ to $L^2$, but this will be clear from the explicit description.)  We break up the analysis of $\Hx$ into how it acts on $L^2_j(Y;i\R)$, $\T^{\gCoul,\sigma}_{j,x}$ and $\J^{\circ,\sigma}_{j,x}$.  In fact, we will show the difference is bounded from $L^2_j$ to $L^2_j$ even when $j = k$, except for one term.    

First, consider $\beta \in L^2_j(Y;i\R)$.  We have that $\Hx(\beta) = ( - d\beta , 0, \beta \phi)$.  Then, we have $\Hx(\beta) - L^\sigma(\beta) = \beta \phi$, which is bounded as a linear map from $L^2_j$ to $L^2_{j}$ by Sobolev multiplication whenever $1 \leq j \leq k$ since $\phi \in L^2_k$.  

Next, let $v  = (-d\xi, 0, \xi \phi) \in \J^{\circ,\sigma}_{j,x}$.  Then, since $v$ is in the kernel of the infinitesimal global Coulomb projection  $\Pi^{\gCoul,\sigma}_*$, we have 
$$\Hx(v) = (0, \dd^{\sigma,\tilde{\dagger}}_x(v)) = (0, d^* d \xi) \in \T^\sigma_{j-1,x} \oplus L^2_{j-1}(Y;i\R).$$  In the second component, $\Hx(v)$ and $L^\sigma(v)$ 
agree. Thus, it remains to show that the component of $L^\sigma(v)$ landing in $\T^{\sigma}_{j-1,x}$ is bounded from $L^2_{j}$ to $L^2_{j}$ for $1 \leq j \leq k-1$.  We have 
$$L^\sigma(v) - \Hx(v)  = (0, \Re \langle D(\xi \phi), \phi \rangle_{L^2}, D(\xi \phi) - \Re \langle D(\xi \phi), \phi \rangle_{L^2}\phi) \in T^\sigma_{j-1,x}.$$
This differs from $v \mapsto (0,0,D(\xi\phi))$ by a bounded operator from $L^2_j$ to $L^2_k$, so we focus on $D(\xi\phi) = \rho(d \xi) \phi + \xi D\phi$.  Since $v \in L^2_j$, we see that $d\xi \in L^2_j$ and thus $\rho(d\xi) \phi$ is bounded in $L^2_j$ by Sobolev multiplication.  Therefore, $v \mapsto (0,0,\rho(d\xi) \phi)$ is bounded from $L^2_j$ to $L^2_{j}$ (even if $j = k$).  For the term $\xi D\phi$, we note that $D\phi \in L^2_{k-1}$, so the map $(-d\xi, 0, \xi \phi) \mapsto \xi D\phi$ is bounded as a linear map from $L^2_{j}$ to $L^2_{j}$ as long as $j \leq k-1$.  This establishes the desired form for $v \in\J^{\circ,\sigma}_{j,x}$.              

It thus remains to compare $\Hx$ and $L^\sigma$ on $\T^{\gCoul,\sigma}_{j,x}$.  Let $v = (b,r,\psi) \in \T^{\gCoul,\sigma}_{j,x}$.  First, note that the component of $\Hx(v)$ landing in $L^2_{j-1}(Y;i\R)$ is given by $i \langle i \phi, \psi \rangle_{\tilde{g}}$ which is bounded as a map from $\T^{\gCoul,\sigma}_{j,x}$ to $L^2_j(Y;i\R)$ for $j \leq k$; this is compatible with the fact that $L^\sigma(b,r,\psi)$ has no component landing in $L^2_{j-1}(Y;i\R)$ since $d^*b = 0$.  Thus, it remains to focus on the component of $\Hx(v)$ contained in $\T^\sigma_{j,x}$.  Since $\widehat{\Hess}^\sigma_{\q,x}$ is $k$-ASAFOE with linear term also given by $L_0$, it suffices to show that $\D^\sigma_x \Xqsigma$ differs from $\D^\sigma_x \Xqgcsigma$ by bounded operators from $L^2_j$ to $L^2_j$.  Direct computation shows 
$$
\Xqsigma (a,s,\phi)- \Xqgcsigma(a,s,\phi) = \left( dGd^* \q^0(a,s\phi), 0, -(Gd^*\q^0(a,s\phi)) \phi \right). 
$$ 
Recall that to compute $\D^\sigma$, we compute the $L^2$ derivative as a map into the affine space $L^2(Y;iT^*Y) \oplus \R \oplus L^2(Y;\Spin)$ and then apply $L^2$ projection to $\T^\sigma_{0,x}$. We first compute the affine derivative  
\begin{align*}
\left( \D_{(a,s,\phi)} \Xqsigma- \D_{(a,s,\phi)} \Xqgcsigma \right)(b,r,\psi) = \big(&dGd^*\D_{(a,s\phi)} \q^0(b, r\phi + s\psi) ,0, \\
& - Gd^*\D_{(a,s\phi)} \q^0(b, r\phi + s\psi) \phi - Gd^*\q^0(a,s\phi) \psi \big).
\end{align*}
Projecting, we obtain
\begin{align}\label{eq:Dsigma-XqvsXqgc}
\left(\D^\sigma_{(a,s,\phi)} \Xqsigma - \D^\sigma_{(a,s,\phi)} \Xqgcsigma \right)(b,r,\psi)= \big(&dGd^*\D_{(a,s\phi0} \q^0(b, r\phi + s\psi) ,0, \\ \notag &- Gd^*\D_{(a,s\phi0} \q^0(b, r\phi + s\psi) \phi - Gd^*\q^0(a,s\phi) \psi  + \\ \notag & \hspace{1.5in} \Re \langle Gd^*\q^0(a,s\phi) \psi , \phi \rangle_{L^2} \phi\big).
\end{align}
Here we are using that $Gd^*\D_{(a,s\phi)} \q^0(b, r\phi + s\psi) \phi$ is real $L^2$ orthogonal to $\phi$ since $\q^0$ is purely imaginary.  This operator is seen to be bounded from $L^2_j$ to $L^2_j$ for $j \leq k$ since $\q$ is tame and because $G$ raises Sobolev regularity by 2.  \\

\noindent \eqref{fakehessian-compact-asafoe} From the above argument, we saw that $\Hx$ differs from $L^\sigma$ by a bounded map from $L^2_j$ to $L^2_j$ for all $1 \leq j \leq k$, except possibly for a single term: $(-d\xi, 0, \xi\phi) \mapsto \xi D\phi$ where $(-d\xi, 0, \xi\phi) \in \J^{\circ,\sigma}_{k,x}$.  By postcomposing with the compact inclusion of $L^2_j$ into $L^2_{j-1}$, we have the desired result, except for this exceptional term.  In this case, since $D\phi \in L^2_{k-1}$, we have that this operator is bounded as a linear map from $L^2_{k-1}$ to $L^2_{k-1}$.  Therefore, we precompose with the compact inclusion from $L^2_{k}$ to $L^2_{k-1}$ to obtain the desired compactness.  \\

\noindent \eqref{fakehessian-k-asafoestationary} As in the above case, we previously showed that $\Hx$ differs from $L^\sigma$ by a bounded map from $L^2_j$ to $L^2_j$ for all $1 \leq j \leq k$, except for the term involving $D\phi$.  If $x$ is a stationary point of $\Xqgcsigma$, then we know that $\phi$ is actually in $L^2_{k+1}$ and thus $D\phi$ is contained in $L^2_k$, and the argument proceeds with $j = k$.  \\

\noindent \eqref{fakehessian-realstationary} Let $x$ be a stationary point of $\Xqgcsigma$ (and hence also $\Xqsigma$).  By \eqref{eq:ExtHessBlowUpGC}, it suffices to show that $S_x \circ \D^\sigma_x \Xqsigma \circ S^{-1}_x$ agrees with $\D^\sigma_x \Xqgcsigma \circ \Pi^{\gCoul,\sigma}_*$ on $\T^\sigma_{j,x}$.  First, suppose $v \in \J^{\circ,\sigma}_{j,x}$.  Then, we have that $S^{-1}_x(v) = v$.  It then follows that $S_x \circ \D^\sigma_x \Xqsigma \circ S^{-1}_x(v) = 0$, since $v$ is tangent to the gauge orbit at a stationary point.  This agrees with with the fact that the kernel of $\Pi^{\gCoul,\sigma}_*$ is $\J^{\circ,\sigma}$.  

Next, consider the case $v \in \T^{\gCoul,\sigma}_{j,x}$.  Since $S^{-1}_x(v) - v$ is contained in $\J^{\circ,\sigma}_{j,x}$, we have that $\D^\sigma_x \Xqsigma \circ S^{-1}_x (v) = \D^\sigma_x \Xqsigma(v) $.  Further, since $x$ is a stationary point, we have that $\D^\sigma_x \Xqsigma(v) \in \K^\sigma_{j,x}$, and therefore $S_x(\D^\sigma_x \Xqsigma(v)) = \Pi^{\gCoul,\sigma}_* \D^\sigma_x \Xqsigma(v)$.  Thus, it suffices to prove that 
\begin{equation}\label{eq:almost-Dsigma-XqvsXqgc}
\Pi^{\gCoul,\sigma}_* \D^\sigma_x \Xqsigma(v) = \D^\sigma_x \Xqgcsigma (v). 
\end{equation}
Since $\D^\sigma_x \Xqgcsigma(v) \in \T^{\gCoul,\sigma}_{j,x}$, we have that \eqref{eq:almost-Dsigma-XqvsXqgc} is equivalent to
\begin{equation}\label{eq:Pigc-Dsigma-XqvsXqgc}
\Pi^{\gCoul,\sigma}_* \D^\sigma_x \Xqsigma(v) = \Pi^{\gCoul,\sigma}_* \D^\sigma_x \Xqgcsigma (v). 
\end{equation}
Thus, we can establish \eqref{eq:Pigc-Dsigma-XqvsXqgc} if applying infinitesimal global Coulomb projection to \eqref{eq:Dsigma-XqvsXqgc} vanishes.  We split \eqref{eq:Dsigma-XqvsXqgc} into two parts.  The first term is
$$
(0, 0, Gd^*\q^0(a,s\phi) \psi  + \Re \langle Gd^*\q^0(a,s\phi) \psi , \phi \rangle_{L^2} \phi).
$$ 
Since $x$ is a stationary point of $\Xqsigma$, we have that $ -\q^0(a,s\phi) = *da$.  Therefore, $Gd^*\q^0(a,s\phi) = 0$ and the above expression vanishes.  Thus, it suffices to show that infinitesimal global Coulomb projection vanishes on the remaining term 
$$
(dGd^*\D_{(a,s\phi0} \q^0(b, r\phi + s\psi) ,0, - Gd^*\D_{(a,s\phi0} \q^0(b, r\phi + s\psi) \phi). 
$$  
This term is contained in $\J^\circ_{j,x}$, which is precisely the kernel of the infinitesimal global Coulomb projection.  This completes the proof.
\end{proof}

With the above lemma, we can now show that $\widehat{\Hess}^{\tilde{g},\sigma}_{\q,x}$ shares two important properties with $\widehat{\Hess}^{\sigma}_{\q,x}$ and $\widehat{\Hess}^{\sp, \sigma}_{\q,x}$.  
\begin{lemma}
\label{lem:eH}
If $x$ is a non-degenerate stationary point of $\Xqgcsigma$, then $\widehat{\Hess}^{\tilde{g}, \sigma}_{\q, x}$ is invertible and has real spectrum.
\end{lemma}

\begin{proof}
At stationary points, with respect to the decomposition $\T^{\sigma}_{j,x}=\K^{\sigma}_{j,x} \oplus \J^{\sigma}_{j,x}$, the derivative $\D_x \Xq^{\sigma}$ takes the form
$$ \begin{pmatrix}
\Hess^{\sigma}_{\q,x} & 0 \\ 0 & 0
\end{pmatrix}.$$

Recall from Remark~\ref{rem:shear} that $S_x$ maps $\K^{\sigma}_{j,x}$ to $\K^{\agCoul, \sigma}_{j,x}$ and preserves $\J^{\sigma}_{j,x}$. Therefore, after conjugating by the shear $S_x$, we have
$$ 
S_x \circ \D^\sigma_x \Xq^{\sigma} \circ S_x^{-1} = 
\begin{pmatrix}
\Hess^{\tilde g, \sigma}_{\q,x} & 0 \\ 0 & 0
\end{pmatrix},$$
with respect to the decomposition $\K^{\agCoul, \sigma}_{j,x} \oplus \J^{\sigma}_{j,x}$.

We deduce that the operator $ \widehat{\Hess}^{\tilde{g}, \sigma}_{\q, x}$ has a block form where one block is ${\Hess}^{\tilde{g}, \sigma}_{\q, x}$ and the other is 
\begin{equation}\label{eq:block-hess-dsigma-tilde}
\begin{pmatrix}
0 &\dd^{\sigma}_x \\
\dd^{\sigma, \tilde{\dagger}}_x & 0
\end{pmatrix}.
\end{equation}
By Lemma~\ref{lem:HessB}, ${\Hess}^{\tilde{g}, \sigma}_{\q, x}$ is invertible and has real spectrum.  
It remains to check that \eqref{eq:block-hess-dsigma-tilde} is invertible and has real spectrum. As in the proof of Lemma~\ref{lem:eHsplit}, we do this by checking that $\dd^{\sigma, \tilde{\dagger}}_x \dd^{\sigma}_x$ is self-adjoint and strictly positive. 

Indeed, we have 
$$ \dd^{\sigma, \tilde{\dagger}}_x \dd^{\sigma}_x (\xi) = \Delta \xi + \mu_Y(\xi) \|i \phi \|_{\tilde g}^2.$$
In block form with respect to the decomposition $L^2_{j-1}(Y; i\R)=(\im d^*)_{j-1}  \oplus i\R$, we have
$$ \dd^{\sp, \sigma, \dagger}_x \dd^{\sigma}_x = \begin{pmatrix}
\Delta & 0 \\
0 &  \|i \phi \|_{\tilde g}^2
\end{pmatrix}.$$
Both diagonal entries are self-adjoint and positive. (Note that $\phi$ is nonzero, because it was normalized to have unit $L^2$ norm.)
 \end{proof}

\subsection{Interpolations} \label{sec:interpol}
In the proof of Proposition~\ref{prop:FredholmCoulomb} below, we will need to interpolate between the extended Hessians $\widehat{\Hess}^{\sigma}_{\q, x}$ and $\widehat{\Hess}^{\tilde{g}, \sigma}_{\q, x}$, by going through $k$-ASAFOE operators that are still invertible and have real spectrum.

We can do this in two steps. First, we interpolate linearly between $\widehat{\Hess}^{\sigma}_{\q, x}$ and the split extended Hessian $\widehat{\Hess}^{\sp,\sigma}_{\q, x}$.

\begin{lemma}
\label{lem:eHrho}
If $x$ is a non-degenerate stationary point of $\Xqgcsigma$, then for any $\rho \in [0,1]$, we have that
$$(1-\rho) \cdot \widehat{\Hess}^{\sigma}_{\q, x} + \rho \cdot \widehat{\Hess}^{\sp, \sigma}_{\q, x}$$ 
is invertible and has real spectrum.
\end{lemma}

\begin{proof}
We use a block decomposition as in the proofs of Lemma 12.4.3 in \cite{KMbook} and of Lemma~\ref{lem:eHsplit} above. It suffices to check that
$$(1-\rho)\cdot \dd^{\sigma, \dagger}_x \dd^{\sigma}_x + \rho \cdot \dd^{\sp, \sigma, \dagger}_x \dd^{\sigma}_x$$
is self-adjoint and strictly positive. This is true because both terms are self-adjoint and strictly positive.
\end{proof}

For the second step, we interpolate between $\widehat{\Hess}^{\sp, \sigma}_{\q, x}$ and $\widehat{\Hess}^{\tilde{g}, \sigma}_{\q, x}$. We do this by considering the family of metrics on $\T_{j,x}$ given by
$$ g_{\rho}= (1-\rho) \cdot g_{L^2} + \rho \cdot \tilde g, \ \ \rho \in [0,1],$$
where $g_{L^2}$ denotes the $L^2$ metric. We consider the $g_{\rho}$-orthogonal complements to $\J_{j,x}$ and $\J^{\circ}_{j,x}$, which we denote by $\K^{\rho}_{j,x}$ and $\K^{\rho, \operatorname{e}}_{j,x}$, respectively. After blowing-up, we obtain a complement $\K^{\rho, \sigma}_{j,x}$ to $\J^{\sigma}_{j,x}$ and a complement $\K^{\rho, \operatorname{e}, \sigma}_{j,x}$ to $\J^{\circ, \sigma}_{j,x}$. We construct the shear map $S^{\rho}_x$ that takes $\K^{\operatorname{e},\sigma}_{j,x}$ into $\K^{\rho, \operatorname{e}, \sigma}_{j,x}$ and is the identity on $\J^{\circ, \sigma}_{j,x}$. Further, we let
\begin{equation}
\label{eq:decomposingrho}
 \dd^{\rho, \sigma, \dagger}_x = \begin{pmatrix} 0 & 0 & -d^* \\
 0 & i\langle i\phi, \cdot \rangle_{g_{\rho}}  & 0 \end{pmatrix} : \K^{\rho, \sigma}_{j,x} \oplus \J^{\gCoul, \sigma}_{j,x} \oplus \J^{\circ, \sigma}_{j,x} \to (\im d^*)_{j-1}  \oplus i\R,
\end{equation}
so that $ \dd^{0, \sigma, \dagger}_x= \dd^{\sp, \sigma, \dagger}_x$ and $\dd^{1, \sigma, \dagger}_x
=  \dd^{\sigma, \tilde{\dagger}}_x$. Finally, define
\begin{equation}
 \label{eq:ExtHessBlowUpRho}
   \widehat{\Hess}^{\rho, \sigma}_{\q, x} = \begin{pmatrix}
S^{\rho}_x \circ \D^\sigma_x \Xq^{\sigma} \circ (S^{\rho}_x)^{-1} & \dd^{\sigma}_x \\[.5em]
\dd^{\rho, \sigma, {\dagger}}_x & 0 
\end{pmatrix}.
\end{equation}

The same arguments as in Lemmas~\ref{lem:eH} and ~\ref{lem:eHrho} give the following:
\begin{lemma}
\label{lem:eHrho2}
If $x$ is a non-degenerate stationary point of $\Xqgcsigma$, and $\rho \in [0,1]$, then $\widehat{\Hess}^{\rho, \sigma}_{\q, x}$ is invertible and has real spectrum.
\end{lemma}

\section{Non-degeneracy of stationary points in Coulomb gauge} \label{sec:nondegCoulomb}
Recall that an irreducible stationary point $x = (a,\phi)$ of $\Xq$ is non-degenerate if $\Xq$ is transverse to the gauge orbit at $x$ or, equivalently, transverse to the subbundle $\J_{k-1}$ at $x$.  We would like to rephrase this condition both in terms of Coulomb gauge and in terms of Hessians. 
  
\begin{lemma}\label{lem:nondegeneracycoulomb}
Let $x \in W_k$ be an irreducible stationary point of $\Xqgc$.  The following are equivalent: 
\begin{enumerate}[(i)]
\item  $\x$ is non-degenerate (i.e. $\Xq$ is transverse to $\J_{k-1}$ at $x$), 
\item  $\Hess_{\q,x} : \K_{k,x} \to \K_{k-1,x}$ is surjective, 
\item $\Xqgc$ is transverse to $\J^{\gCoul}_{x}$, 
\item $\Hess^{\tilde{g}}_{\q,x} : \K^{\agCoul}_{k,x} \to \K^{\agCoul}_{k-1,x}$ is surjective. 
\end{enumerate}
\end{lemma}

\begin{proof}
The equivalence (i) $\iff$ (ii) is proved in \cite[Lemma 12.4.1]{KMbook}. It is a consequence of the fact that at a critical point, with respect to the decompositions $\T_j = \J_j \oplus \K_j$ (with $j=k$ for the domain and $j=k-1$ for the image), the derivative $\D_x \Xq$ has the block form
$$\begin{pmatrix}
0 & 0 \\
0 & \Hess_{\q, x}
\end{pmatrix}.$$

For the equivalence (iii) $\iff$ (iv), we apply a similar reasoning: $\Dg_x \Xqgc$ vanishes on $\J^{\gCoul}_{x}$, the tangents to the $S^1$-orbits in $W_k$; therefore,  $\Dg_x \Xqgc$ has a block form with respect to the decomposition $\T^{\gCoul}_j = \J^{\gCoul} \oplus \K^{\gCoul}_j$ from \eqref{eq:JKcoulomb}; and the only nonzero entry in this block form is $\Hess^{\tilde{g}}_{\q,x}$. (Since we are at a stationary point, note that in fact $\Dg_x \Xqgc = \D_x \Xqgc.$)

Finally, for the equivalence (ii) $\iff$ (iv), recall from \eqref{eq:altHess} and the proof of Lemma~\ref{lem:hessianfredholm} that $\Hess^{\tilde{g}}_{\q,x}$ is the conjugate of $\Hess_{\q, x}$ by the isomorphism $\Pi^{\agCoul}_x$, with inverse $\Pi^{\elCoul}_x$. Hence, one Hessian is surjective if and only if the other one is. 
 \end{proof}

In the blow-up, we can also rephrase non-degeneracy of stationary points in terms of the Hessian. Rather than stating the exact analogue of Lemma~\ref{lem:nondegeneracycoulomb}, let us emphasize the following result:
\begin{lemma}\label{lem:nondegeneracycoulombblowup}
Let $x$ be a stationary point of $\Xqgcsigma$.  Then, $x$ is non-degenerate $\iff\Hess^{\tilde{g},\sigma}_{\q,x}$ is injective $\iff  \Hess^{\tilde{g},\sigma}_{\q,x}$ is bijective.  
\end{lemma}
\begin{proof}
If $x$ is irreducible, $\Hess^{\tilde{g},\sigma}_{\q,x}$ is conjugate to $\Hess^{\tilde{g}}_{\q,x}$ via the blow-down (which identifies $\K^{\agCoul,\sigma}$ with $\K^{\agCoul}$) and we may apply Lemmas~\ref{lem:hessianfredholm} and ~\ref{lem:nondegeneracycoulomb}.  

If $x$ is a reducible stationary point in the blow-up, we have $\Hess^{\tilde{g},\sigma}_{\q,x} = \Hess^\sigma_{\q,x}$; see Remark~\ref{rem:HessEqual}. The claim then follows from the analogous statement for $\Hess^\sigma_{\q,x}$ in \cite[Section 12.4]{KMbook}. 
\end{proof}

So far we have only worked with $W^\sigma$.  The above constructions can also be phrased in terms of the quotient $W^\sigma/S^1$.  Given $x \in W^{\sigma}$, we will write $[x]$ for its class in $W^\sigma/S^1$. Note that the bundles $\K^{\agCoul,\sigma}_j$ are $S^1$-invariant and therefore the space $\K^{\agCoul,\sigma}_{j,x}$ is canonically identified with the tangent space at $[x]$ in the $L^2_j$ completion of the tangent bundle to $W^\sigma/S^1$.  The vector field $\Xqgcsigma$ on $W_k^\sigma$ is $S^1$-invariant and takes values in $\K^{\agCoul,\sigma}_{k-1}$. Thus, it descends to a vector field, denoted $\Xqagcsigma$, on $W_k^\sigma/S^1$.

\begin{lemma}
\label{lem:idCoulomb}
We have the following identifications, given by composing global Coulomb projection with projection to the quotient by $S^1$:
\begin{equation}
\label{eq:EquivStat2}
\{ \text{stationary points of } \Xqsigma \bigr \} / \G_{k+1} \ \xrightarrow{\mathmakebox[2em]{\cong}} \ \{ \text{stationary points of } \Xqagcsigma \}
\end{equation}
and
\begin{equation}
\label{eq:EquivTraj2}
\{ \text{trajectories of } \Xqsigma \bigr \} / \G_{k+1} \ \xrightarrow{\mathmakebox[2em]{\cong}} \ \{ \text{trajectories of } \Xqagcsigma \}.
\end{equation}
\end{lemma}

\begin{proof}
This is immediate from \eqref{eq:EquivStat1} and \eqref{eq:EquivTraj1}.
\end{proof}

Let $x$ be a stationary point of $\Xqgcsigma$, so that $[x]$ is a stationary point of $\Xqagcsigma$. Since $\Xqagcsigma$ is a section of $\K^{\agCoul,\sigma}_{k-1}$, under the identification of $\K^{\agCoul,\sigma}_{k-1,x}$ with the $L^2_{k-1}$-completion of $T_{[x]} (W^\sigma_k/S^1)$, the derivative $\D^\sigma_{[x]}\Xqagcsigma := \D^\sigma_x \Xqgcsigma = \D^{\tilde{g},\sigma}_x \Xqgcsigma$ takes values in $\K^{\agCoul,\sigma}_{k-1, x}$. In view of Equation~\eqref{eq:Hess2}, we have
$$ \Hess^{\tilde{g},\sigma}_{\q,x} = \D^\sigma_{[x]}\Xqagcsigma.$$

The following is then a direct consequence of Lemma~\ref{lem:nondegeneracycoulombblowup}:

\begin{lemma}
\label{lem:rephraseStat}
In terms of the identification \eqref{eq:EquivStat2}, non-degeneracy of a stationary point $x$ of $\Xqsigma$ is  equivalent to the injectivity (or bijectivity) of $\D^\sigma\Xqagcsigma$ at the corresponding point $[\Pi^{\gCoul, \sigma}(x)] \in W^\sigma/S^1$.
\end{lemma}

Just as in Lemma~\ref{lem:rephraseStat} we rephrased the non-degeneracy of stationary points, our next goal will be to rephrase the regularity condition on the moduli spaces of trajectories of $\Xqsigma$ in terms of global Coulomb gauge; that is, re-write it as a condition on the moduli spaces of trajectories of $\Xqagcsigma$.  This will be accomplished in Section~\ref{sec:NDTcoulomb}. Before that,  as preliminary steps, we will:
\begin{itemize}
\item Embed the moduli space of trajectories of $\Xqagcsigma$ into a larger space of paths in Coulomb gauge (in Section~\ref{sec:path});
\item Describe the tangent bundle to this path space (in Section~\ref{sec:4Dcoulomb}); 
\item Linearize the equations that define the moduli space in Coulomb gauge (in Section~\ref{sec:linearized}). 
\end{itemize}

\section{Path spaces}
\label{sec:path}
Fix points $x,y \in W^{\sigma}$, and a smooth path $\gamma_0$ in $W^{\sigma}$ from $x$ to $y$, such that $\gamma_0(t)$ agrees with $x$ near $-\infty$ and agrees with $y$ near $+\infty$. Let $Z = \rr \times Y$. By analogy with the definition of $\C_k^\tau(x, y)$ in \eqref{eq:Cktau}, we define the space of four-dimensional configurations
\[
\C_k^{\gCoul, \tau}(x, y) = \{\gamma \in \C^{\gCoul, \tau}_{k, loc}(\rr \times Y) \mid \gamma - \gamma_0 \in L^2_k(Z; iT^*Z) \times L^2_k(\rr; \rr) \oplus L^2_k(Z; \Spin^+)\}.  
\]
We can write $\gamma$ as a path 
$$\gamma(t) = (a(t)+\alpha(t)dt, s(t), \phi(t)).$$ 
Recall from Section~\ref{sec:cylinders} that the condition $\gamma \in \C_{k, loc}^{\gCoul,\tau}(\rr \times Y)$ means that $\gamma$ is in pseudo-temporal gauge ($\alpha(t)$ is constant on each slice), and that the one-form component $a(t)$ is in the kernel of $d^*$, for all $t$. Moreover, since we are in the $\tau$ model, we must have $\|\phi(t)\|_{L^2(Y)}=1$ and $s(t) \geq 0$ for all $t$.  

The space $\C_k^{\gCoul, \tau}(x, y)$ embeds in a Hilbert manifold $\widetilde{\C}_k^{\gCoul,\tau}(x,y)$, defined as above but using $\widetilde{\C}^{\gCoul,\tau}_{k,loc}(\rr \times Y)$; that is, dropping the condition $s(t) \geq 0$.   

Further, in the spirit of Section~\ref{sec:cylinders}, we let $W_k^{\tau}(x, y)$ (respectively $\tW_k^\tau(x,y)$) denote the subset of $\C_k^{\gCoul,\tau}(x, y)$ (respectively $\tC_k^{\gCoul,\tau}(x,y)$) consisting of configurations with $\alpha(t)=0$, i.e., in temporal gauge. 
 
The gauge group $$\G_{k+1}^{\gCoul}(Z) := \{u: \R \to S^1 \mid 1-u \in L^2_{k+1}(\R; \cc) \}$$ 
acts on $\C_k^{\gCoul,\tau}(x, y)$ and $\widetilde{\C}_k^{\gCoul,\tau}(x,y)$.
 
Let $\Btaug_k(x, y)$ denote the quotient of $\C_k^{\gCoul,\tau}(x, y)$ by $\G_{k+1}^{\gCoul}(Z)$. Note that $\Btaug_k(x, y)$ only depends on the classes $[x]$ and $[y]$ in $W^\sigma/S^1$, up to canonical diffeomorphism. Therefore, we will use the notation $\Btaug_k([x], [y])$. The quotient of $\widetilde{\C}_k^{\gCoul,\tau}(x,y)$ by the same gauge action is denoted $\tBtaug_k([x],[y])$. One can check that $\Btaug_k([x], [y])$ and $\tBtaug_k([x],[y])$ are Hausdorff in the quotient topology; compare \cite[Proposition 13.3.4]{KMbook}.

\begin{remark}\label{rmk:no-temporal}
It is important to note that given an element $\gamma \in \C^{\gCoul,\tau}_k(x,y)$, it cannot necessarily be moved to be in temporal gauge by an element of $\G_{k+1}^{\gCoul}(Z)$.  However, we can act by a four-dimensional gauge transformation in $\G_{k+1,loc}^{\gCoul}(Z)$ to move $\gamma$ to Coulomb gauge, but the result will land in $\C^{\gCoul,\tau}_k(x, u y)$, for some $u \in S^1$.  Since $\B^{\gCoul,\tau}_k(x,y)$ is canonically identified with $\B^{\gCoul,\tau}_k(x, uy)$, we can still think of $[\gamma]$ as having a representative in temporal gauge.    
\end{remark}

These constructions are similar to those in Section~\ref{sec:AdmPer}, where we had a space $\B^{\tau}_k([x], [y])$ of configurations in $\C_k^\tau(x, y)$ modulo four-dimensional gauge transformations. One can also consider the corresponding Hilbert  manifolds $\widetilde \B^{\tau}_k([x], [y])$ and $\widetilde \C_k^\tau(x, y)$. 

We would like to relate $\widetilde \B^\tau_k([x],[y])$ to $\tBtaug_k([x],[y])$. We first define a map 
$$  \Pi^{\gCoul, \tau}: \widetilde \C^{\tau}_k(x, y) \to \widetilde \C_k^{\gCoul,\tau}(\Pi^{\gCoul, \sigma}(x), \Pi^{\gCoul, \sigma}(y))$$
by the formula
\begin{equation}
\label{eq:zebra}
 (a(t)+\alpha(t)dt, s(t), \phi(t)) \mapsto \Pi^{\gCoul, \sigma}(a(t), s(t), \phi(t)) + (\mu_Y(\alpha(t))dt, 0, 0).
 \end{equation}

\begin{lemma}
There is a well-defined, continuous map
\begin{equation}
\label{eq:Pigctau}
 \Pi^{[\gCoul], \tau}: \widetilde \B^{\tau}_k([x], [y]) \to \tBtaug_k([x], [y]), \ \ [\gamma] \to [\Pi^{\gCoul, \tau} (\gamma)].
 \end{equation}
 This sends $\B^{\tau}_k([x], [y])$ to $\Btaug_k([x], [y])$.
\end{lemma}

\begin{proof}
Let us first check that $[\Pi^{\gCoul, \tau}(\gamma)]$ does not depend on the choice of representative $\gamma$ for the class $[\gamma]$. Indeed, suppose we change $\gamma$ by a four-dimensional gauge transformation of the form $u: Z \to S^1$. If we ignore the $\alpha(t)dt$ component and write $x(t)=(a(t), s(t), \phi(t))$, we find that $u$ acts on each $x(t)$ as the three-dimensional gauge transformation $u(t) := u|_{\{t\} \times Y}$. Write $u(t) = e^{f(t)}$ with $f(t): Y \to i\R$. Using \eqref{eq:piCS} and the fact that $Gd^*df = f - \mu_Y(f)$, we see that 
$$\Pi^{\gCoul, \sigma} (e^{f(t)} \cdot x(t)) = e^{\mu_Y(f(t))} \Pi^{\gCoul, \sigma}(x(t)).$$
From here we get $$ \Pi^{[\gCoul], \tau}(e^{f} \cdot \gamma)  = e^{\mu_Y(f)} \cdot \Pi^{[\gCoul], \tau}(\gamma),$$
where the gauge transformation $e^f$ on the left is a four-dimensional gauge transformation, and $e^{\mu_Y(f)}$ is the slicewise constant gauge transformation obtained by taking the average of $f$ in each slice. Thus, the class $[\Pi^{\gCoul, \tau}(\gamma)]$ is unchanged.

The fact that slicewise application of $\Pi^{\gCoul, \sigma}$ preserves the four-dimensional $L^2_k$ condition can be seen from the formula \eqref{eq:piCS}, together with the Sobolev multiplication rule on infinite cylinders \cite[Theorem 13.2.2]{KMbook}. Further, averaging $\alpha(t)$ slicewise preserves the $L^2_k$ condition by the Cauchy-Schwarz inequality. Continuity of the map $\Pi^{[\gCoul], \tau}$ follows from similar arguments. 

Finally, since $s(t)$ (and hence the condition $s(t) \geq 0$) is preserved by global Coulomb projection, we have that $\B^{\tau}_k([x], [y])$ is  mapped to $\Btaug_k([x], [y])$.
\end{proof}

Observe that the map \eqref{eq:Pigctau} is surjective but not injective. For example, suppose we have two configurations $\gamma_1$ and $\gamma_2$ that are in temporal gauge and that differ in each slice by a three-dimensional gauge transformation $u: Y \to S^1$ (non-constant in $t$). Then $\Pi^{[\gCoul], \tau}([\gamma_1]) = \Pi^{[\gCoul], \tau}([\gamma_2])$, but $\gamma_1$ and $\gamma_2$ are typically not gauge equivalent as four-dimensional configurations.  

Now suppose that $[x]$ and $[y]$ are stationary points of $\Xqagcsigma$. Define $M^{\agCoul}([x],[y])$ to be the moduli space of trajectories of $\Xqagcsigma$, considered as a subspace of $\Btaug_k([x],[y])$.  We can also define $M^{\agCoul,\red}([x],[y])$ similarly.    

\begin{proposition}\label{prop:gccorrespondence}
Every trajectory of $\Xqagcsigma$ on $W^\sigma/S^1$ (connecting two stationary points $x$ and $y$ as in Definition~\ref{def:sw-moduli-space}) is actually in $\Btaug_k([x], [y])$. Further, the map $\Pi^{[\gCoul], \tau}$ produces a homeomorphism between the moduli space $M([x], [y])$ (consisting of gauge equivalence classes of trajectories of $\Xqsigma$) and $M^{\agCoul}([x],[y])$.
 \end{proposition}

\begin{proof}
The equivalence between trajectories of $\Xqsigma$ and $\Xqagcsigma$ (without the $L^2_k$ conditions) was already established in Lemma~\ref{lem:idCoulomb}.    

Furthermore, as noted in Section~\ref{sec:AdmPer}, it is proved in \cite[Theorem 13.3.5]{KMbook} that every trajectory of $\Xqsigma$ connecting $x$ and $y$ is gauge equivalent to a trajectory in $\C^{\tau}_k(x, y)$. Thus, if we have a trajectory $\gamma$ of $\Xqagcsigma$ connecting two stationary points, we can lift it to one of $\Xqsigma$, and apply a gauge transformation to obtain a trajectory in $\C^{\tau}_k(x, y)$. Since the map $\Pi^{[\gCoul], \tau}$ preserves the $L^2_k$ condition, we deduce that $\gamma$ is in $\Btaug_k([x], [y])$.  
\end{proof}

\begin{remark}\label{rmk:MagcFgc}
It also follows from the proof of Proposition~\ref{prop:gccorrespondence} that $M^{\agCoul}([x],[y])$ is identified with the space of trajectories of $\Xqgcsigma$ from $x$ to $y$ modulo the $S^1$ action, and equivalently, the zero set of $\F^{\gCoul,\tau}_\q$, restricted to $\C^{\gCoul,\tau}_k(x,y)$ modulo the action of $\G^{\gCoul}_{k+1}(Z)$.  
\end{remark}

\section{Four-dimensional Coulomb slices}
\label{sec:4Dcoulomb}
Fix $x, y \in W^{\sigma}$. We aim to prove that $\widetilde{\B}^{\gCoul, \tau}_k([x],[y])$ is a Hilbert manifold, and to identify its tangent space. The discussion here will be modelled on the corresponding one for the space $\widetilde{\B}^\tau_k([x],[y])$, following \cite[Section 14.3]{KMbook}.

Let us first review the analysis for $\widetilde{\B}^\tau_k([x],[y])$. This space is the quotient of $\widetilde{\C}^{\tau}_k(x, y)$ by the gauge action. The $L^2_j$ completion of the tangent space to $ \widetilde{\C}^{\tau}_k(x, y)$ at $\gamma=(a, s, \phi)$ is
\begin{align}
\label{eq:fir}
\T^{\tau}_{j, \gamma} &= \{(b, r, \psi) \mid \Re \langle \phi(t), \psi(t) \rangle_{L^2(Y)}=0, \ \forall t \} \\
& \subset L^2_j(Z; i T^*Z) \oplus L^2_j(\rr; \rr) \oplus L^2_j(Z; \Spin^+). \notag
\end{align}

The derivative of the gauge group action on $\widetilde{\C}^{\tau}_k(x, y)$ is given by
$$
\dd^{\tau}_{\gamma}: L^2_{j+1}(Z; i\R) \to \T^{\tau}_{j, \gamma}, \ \ \ 
\dd^\tau_{\gamma}(\xi) = (-d\xi, 0, \xi \phi).$$

Kronheimer and Mrowka define a local slice for the gauge action, $\Slice^{\tau}_{k, \gamma} \subset  \widetilde{\C}^{\tau}_k(x, y)$, by the equation:
\begin{equation}
\label{eq:sliceC}
 -d^*b + isr \Re \langle i\phi, \psi \rangle + i |\phi|^2 \Re \mu_Y ( \langle i\phi, \psi \rangle ) = 0.
 \end{equation}

By linearizing this equation, they obtain the four-dimensional local Coulomb slice which was previously mentioned in Section~\ref{sec:AdmPer}:
$$\K^\tau_{k, \gamma} = \ker \dd^{\tau, \dagger}_{\gamma} \subset \T^\tau_{k, \gamma},$$
where 
\begin{equation}
\label{eq:dtdagger}
 \dd^{\tau,\dagger}_\gamma(b,r,\psi) = -d^*b + is^2 \text{Re}\langle i\phi,\psi \rangle + i |\phi|^2 \Re\ \mu_Y (\langle i\phi,\psi \rangle).
 \end{equation}
There are also $L^2_j$ completions $\K^\tau_{j, \gamma}$ for $j \leq k$. Let $\J^{\tau}_{j, \gamma}$ denote the image of $\dd^{\tau}_{\gamma}$. In \cite[Proposition 14.3.2]{KMbook}, Kronheimer and Mrowka prove that $\dd^\tau_{\gamma}$ is injective with closed range, $ \dd^{\tau,\dagger}_\gamma$ is surjective, and there is a bundle decomposition
$$ \T^\tau_j = \J^{\tau}_j \oplus \K^{\tau}_j.$$
In turn, this implies that the slices $\Slice^{\tau}_{k, \gamma}$ are Hilbert submanifolds of $\widetilde{\C}^{\tau}_k(x, y)$. Since they provide local models for the quotient $\widetilde{\B}^\tau_k([x],[y])$, they deduce that this quotient is a Hilbert manifold. Furthermore, its tangent space at $[\gamma]$ can be identified with $\K^\tau_{k, \gamma}$.

We now turn to a similar discussion in global Coulomb gauge. 

The $L^2_j$ completion of the tangent space to $ \widetilde{\C}^{\gCoul, \tau}_k(x, y)$ at $\gamma=(a, s, \phi)$ is
\begin{align*}
\T^{\gCoul, \tau}_{j, \gamma} &= \{(b, r, \psi) \mid d\beta(t)=0, \ d^*(b(t))=0, \ \Re \langle \phi(t), \psi(t) \rangle_{L^2(Y)}=0, \ \forall t \} \\
& \subset L^2_j(Z; i \pp^*(T^*Z)) \oplus L^2_j(\rr; \rr) \oplus L^2_j(Z; \Spin^+),
\end{align*}
where we write $b$ as $b(t) + \beta(t)dt$. 

The derivative of the action of $\G_{k+1}^{\gCoul}(Z)$ on $\widetilde{\C}^{\gCoul, \tau}_k(x, y)$ is 
$$
\dd^{\gCoul, \tau}_{\gamma}: L^2_{j+1}(\R; i\R) \to \T^{\tau}_{j, \gamma}, \ \ \ 
\dd^{\gCoul, \tau}_{\gamma}(\xi) = (-\frac{d\xi}{dt} dt, 0, \xi \phi).$$

A suitable local slice for this action is $\Slice^{\gCoul, \tau}_{k, \gamma} \subset \widetilde{\C}^{\gCoul, \tau}_k(x, y)$, defined by:
\begin{equation}
\label{eq:Remu}
  \frac{d\beta}{dt} +  i \langle  i \phi, \psi \rangle_{\tilde{g}} = 0.
\end{equation}
This is already linear, so the same equation defines the corresponding linearized local slice
$$\K^{\gCoul,\tau}_{k, \gamma} = \ker \dd^{\gCoul, \tau, \tilde{\dagger}}_{\gamma}  \subset \T^{\gCoul, \tau}_{k, \gamma},$$
where
\begin{equation}
\label{eq:dgct}
 \dd^{\gCoul, \tau, \tilde{\dagger}}_\gamma(b,r,\psi) =    \frac{d\beta}{dt} + i \langle  i \phi, \psi \rangle_{\tilde{g}}.
 \end{equation}
We have $L^2_j$ completions $\K^{\gCoul, \tau}_{j, \gamma}$ for $j \leq k$, and we denote the image of $\dd^{\gCoul, \tau}_{\gamma}$ by $\J^{\gCoul, \tau}_{j, \gamma}$.

\begin{lemma}
\label{lem:DecomposeGC}
The operator $\dd^{\gCoul, \tau}_{\gamma}$ is injective with closed range, $ \dd^{\gCoul, \tau, \tilde{\dagger}}_\gamma$ is surjective, and there is a bundle decomposition
$$ \T^{\gCoul, \tau}_j = \J^{\gCoul, \tau}_j \oplus \K^{\gCoul, \tau}_j.$$
Further, 
\[
\J^{\gCoul,\tau}_j = \J^{\gCoul,\tau}_0 \cap \T^\tau_j.  
\]
\end{lemma}

The proof of Lemma~\ref{lem:DecomposeGC} is similar to \cite[Proposition 14.3.2]{KMbook}, so we omit it.

We deduce that $\tBtaug_k([x],[y])$ is a Hilbert manifold, locally modelled on the slices $\Slice^{\gCoul, \tau}_{k, \gamma}$. The tangent space to $\tBtaug_k([x],[y])$ at $\gamma$ can be identified with $\K^{\gCoul,\tau}_{k, \gamma}$.

Let us now take a quick look at the map $\Pi^{[\gCoul], \tau}$ from \eqref{eq:Pigctau}. With respect to the smooth structures we just defined, this map is continuously differentiable. Its derivative 
\begin{equation}
\label{eq:Pigctau2}
 (\Pi^{[\gCoul], \tau}_* )_{[\gamma]}: \K^{\tau}_{j, [\gamma]} \longrightarrow \K_{j, \Pi^{[\gCoul], \tau}([\gamma])}^{\gCoul, \tau}
 \end{equation}
is given by the formula
$$(b(t)+\beta(t)dt, r(t), \psi(t)) \mapsto \Pi_{\K^{\gCoul, \tau}_j} \Bigl( (\Pi^{\gCoul, \sigma}_*)_{(a(t), s(t), \phi(t))}(b(t), r(t), \psi(t)) + (\mu_Y(\beta(t))dt, 0, 0) \Bigr),$$
where $\Pi_{\K^{\gCoul, \tau}_j}$ denotes the projection onto $\K^{\gCoul, \tau}_j$ with kernel $ \J^{\gCoul, \tau}_j$.

\section{The linearized equations}
\label{sec:linearized} 
Let $x, y$ be non-degenerate stationary points of $\Xqsigma$.  Recall from Section~\ref{sec:HM} that the moduli space $M(x, y)$ of perturbed Seiberg-Witten trajectories can be described as the zero set of the section
$$ \F^{\tau}_{\q} : \C^{\tau}_k(x, y) \to \V^{\tau}_{k-1}(Z),$$
modulo gauge. In temporal gauge, we have $ \F^{\tau}_{\q} = \frac{d}{dt} + \Xqsigma$.

Recall that $\V^{\tau}_{k-1}(Z)$ is a bundle over the Hilbert manifold $\widetilde{\C}^{\tau}_k(x, y)$, and we have extended the section $\F^{\tau}_{\q}$. In order to understand the local structure of $M(x, y)$, one needs to study the derivative $\D^{\tau}_{\gamma} \F^{\tau}_{\q}$ at paths $\gamma \in M(x, y)$. Further, to be able to define gradings and orientations later, one needs to understand $\D^\tau_{\gamma}  \F^{\tau}_{\q}$ when $\gamma$ is not a trajectory in $M(x, y)$.  Much like using the notation $\D^\sigma$ to clarify derivatives in the blow-up, we write $\D^\tau$ to mean that the derivatives are taken with respect to four-dimensional configurations (as opposed to three-dimensional derivatives slicewise).  The relevant properties of $\D^\tau_{\gamma}  \F^{\tau}_{\q}$ are analyzed by Kronheimer and Mrowka in \cite[Section 14.4]{KMbook}. We will sketch their results, and then do a similar analysis in global Coulomb gauge, with an eye towards the local structure of the moduli spaces $M^{\agCoul}([x], [y])$.

Fix $\gamma \in \widetilde{\C}^{\tau}_k(x, y)$, and assume that $\gamma$ is in temporal gauge. Recall the definition of the tangent space $\T^{\tau}_{j, \gamma}$ from \eqref{eq:fir}. Following  \cite[Section 14.4]{KMbook}, we write an element of $\T^{\tau}_{j, \gamma}$ as $(V, \beta)$, where $V(t) = (b(t), r(t), \psi(t))$ is a path in (a completion of) $\T^{\sigma}(Y)$, and $\beta= (\beta(t))$ is the path in $L^2(Y; i\R)$ that gives the $dt$ component of the connection. Vectors in  $\T^{\sigma}(Y)$ can be differentiated along paths using the covariant derivative
$$ \frac{D^\sigma}{dt} V = \Bigl( \frac{db}{dt}, \frac{dr}{dt}, \Pi^{\perp}_{\phi(t)} \frac{d\psi}{dt} \Bigr),$$
where $\Pi^{\perp}_{\phi(t)}$ denotes the $L^2$ projection to the orthogonal complement of $\phi(t)$. 
We can then write the derivative $\D^\tau_{\gamma} \F^{\tau}_{\q}$ explicitly:
$$ \D^\tau_{\gamma} \F^{\tau}_{\q}(V, \beta) = \frac{D^\sigma}{dt} V  + \D^\sigma\X^{\sigma}_{\q}(V) + \dd^{\sigma}_{\gamma(t)} \beta,$$
where $\dd^{\sigma}$ is as in \eqref{eq:dsigma}.

Let us also recall from Section~\ref{sec:4Dcoulomb} that the $L^2_j$ completion of the tangent space to $\widetilde{\B}^{\tau}_k([x], [y])$ at $\gamma$ is the space $\K^{\tau}_{j,\gamma} = \ker \mathbf{d}^{\tau, \dagger}_\gamma$. With respect to the decomposition $(V, \beta)$,  the map $\dd^{\tau, \dagger}_{\gamma}: \T^{\tau}_{j, \gamma} \to L^2_{j-1}(Z; i\R)$ is given by
$$ \dd^{\tau, \dagger}_{\gamma}(V, \beta) = \frac{d\beta}{dt} + \dd^{\sigma, \dagger}_{\gamma(t)}(V),$$
with $\dd^{\sigma, \dagger}$ as in \eqref{eq:dsigmadagger}.

The local structure of the moduli space $M([x], [y])$ is governed by the operator:
$$ (\D^\tau_{\gamma} \F^{\tau}_{\q})|_{\K^{\tau}_{j,\gamma}} : \K^{\tau}_{j,\gamma} \to 
\V^{\tau}_{j-1,\gamma}(Z).$$

Kronheimer and Mrowka establish the following:

\begin{proposition}[Proposition 14.4.3 in \cite{KMbook}]
\label{prop:Dfredholm} For $1 \leq j \leq k$, the operator  $(\D^\tau_{\gamma} \F^{\tau}_{\q})|_{\K^{\tau}_{j,\gamma}}$ is Fredholm and the index is independent of $j$.
\end{proposition}

\begin{proof}[Sketch of proof] The main tool is Proposition 14.2.1 in \cite{KMbook}, which gives a Fredholmness criterion for differential operators on infinite cylinders. Specifically, it deals with operators of the form
\begin{equation}
\label{eq:Q}
Q=  \frac{d}{dt} + L_0 + h_t : L^2_1(Z; \pp^*E) \to L^2(Z; \pp^*E),
\end{equation}
where $E \to Y$ is a vector bundle, $L_0$ is a first order, self-adjoint elliptic operator acting on sections of $E$, and $h_t$ is a time-dependent bounded operator on $L^2(Y; E)$, varying continuously in the operator norm topology, and assumed to be constant $h_{\pm}$ near the ends of the cylinder $Z$. The proposition says that if $L_0 + h_{\pm}$ are hyperbolic (i.e., their spectrum is disjoint from the imaginary axis), then $Q$ is Fredholm, with index given by the spectral flow of the family $\{L_0 + h_t\}$.   

Proposition 14.2.1 in \cite{KMbook} cannot be applied directly to the operator $( \D^\tau_{\gamma} \F^{\tau}_{\q})|_{\K^{\tau}_{j,\gamma}}$, because its domain is the local Coulomb slice, rather than the space of all sections of a vector bundle. The remedy is to enlarge the operator, much as in the proof of Lemma~\ref{lem:hessianfredholm}. We take the direct sum
\begin{equation}
\label{eq:Qg}
 Q_{\gamma} = \D^\tau_{\gamma} \F^{\tau}_{\q} \oplus  \dd^{\tau, \dagger}_{\gamma} : \T^{\tau}_{j,\gamma} \to \V^{\tau}_{j-1,\gamma} \oplus L^2_{j-1}(Z; i\R),
 \end{equation}
previously considered in \eqref{eq:Qgamma}. We can write
\begin{equation}
\label{eq:Qg2}
 Q_{\gamma} = \frac{D^\sigma}{dt} + \begin{pmatrix} \D^\sigma_{\gamma(t)} \X_{\q}^{\sigma} & \dd^{\sigma}_{\gamma(t)} \\ \dd^{\sigma, \dagger}_{\gamma(t)} & 0 \end{pmatrix}.
 \end{equation}

Further, to deal with the condition $\Re \langle \phi(t), \psi(t)\rangle_{L^2(Y)} = 0$ that defines $\T^{\tau}_{j,\gamma}$, we proceed as in Section~\ref{sec:HessBlowUp}: We combine the $r$ and $\psi$ components into $$ \psis(t) = \psi(t) + r(t) \phi(t).$$

Now $Q_{\gamma}$ is of the desired form \eqref{eq:Q}, with $E$ being the bundle $iT^*Y \oplus \Spin \oplus i\R$ and $L_0$ having the block form
\begin{equation}
\label{eq:L0}
 L_0 = \begin{pmatrix}
*d & 0 & -d \\
0 & D & 0\\
-d^* & 0 & 0
\end{pmatrix}.
\end{equation}

In fact, $L_0 + h_t$ is a slight variant of the extended Hessian from \eqref{eq:ExtHessBlowUp}. On the two ends of the cylinder $Z$ we see exactly the extended Hessians at the non-degenerate stationary points $x$ and $y$. Recall that at a stationary point (say $x=(a, s, \phi)$), the extended Hessian is the operator 
\[
 \widehat{\Hess}^{\sigma}_{\q, x} =\begin{pmatrix}
 0 & 0 & \dd^{\sigma}_x \\
0 & \Hess^{\sigma}_{\q, x} & 0 \\
  \dd^{\sigma, \dagger}_x & 0 & 0
\end{pmatrix}.
\]
This is written in block form with respect to the decomposition $L^2_j(Y;E)= \J_j^{\sigma} \oplus \K_j^{\sigma} \oplus L^2_j(Y; i\R)$; compare Sections~\ref{sec:Hess} and ~\ref{sec:HessBlowUp}.

The extended Hessians at $x$ and $y$ are invertible by the non-degeneracy assumption; moreover, if we had not used the blow-up, they would be self-adjoint and therefore have real spectrum. Given the use of $\psis$ imposed from the blow-up construction, the limiting operators are no longer self-adjoint. Nevertheless, they have real spectrum by \cite[Lemma 12.4.3]{KMbook}. It follows that they are hyperbolic, and the hypotheses of \cite[Proposition 14.2.1]{KMbook} apply. We obtain that $Q_{\gamma}$ is  Fredholm as an operator from $L^2_1$ to $L^2$.  Further, if $L_0 + h_{\pm}$ are $k$-ASAFOE and $h_t$ is compact as a map from $L^2_j$ to $L^2_{j-1}$ for paths constant near the endpoints, then using elliptic estimates, the result is extended to show that $Q_{\gamma}$ as a map 
from $L^2_j$ to $L^2_{j-1}$ is also Fredholm with the same index.   As discussed, strictly speaking, we chose the path $\gamma$ to be constant near the ends.  That case implies the general case by the argument at the end of \cite[proof of Theorem 14.4.2]{KMbook}. Finally, Fredholmness of $Q_{\gamma}$ implies that $( \D^\tau_{[\gamma]} \F^{\tau}_{\q})|_{\K^{\tau}_{j,\gamma}}$ is also Fredholm, of the same index; this follows from the decomposition \eqref{eq:Qg}, together with the surjectivity of $ \mathbf{d}^{\tau, \dagger}_{\gamma}$ (cf. Proposition 14.3.2 in \cite{KMbook}).
\end{proof}

We now move to Coulomb gauge. Consider the moduli space $M^{\agCoul}([x], 
[y]) \subset {\widetilde \B}_k^{\gCoul, \tau}([x], [y])$, as defined before 
Proposition~\ref{prop:gccorrespondence}. Recall from Remark~\ref{rmk:MagcFgc} 
that we can describe $M^{\agCoul}([x], [y])$ as the zero set of 
$$\F^{\gCoul, \tau}_{\q}: \tC_k^{\gCoul, \tau}(x, y) \to \V^{\gCoul, \tau}_{k-1}(Z),$$
restricted to $\C_k^{\gCoul,\tau}(x,y)$ and modulo the action of $\G_{k+1}^{\gCoul}(Z)$. 

Fix $\gamma \in \C^{\gCoul, \tau}(x, y)$ in temporal gauge (that is, in $W^{\tau}_k(x, y)$, cf. Section~\ref{sec:path}). To understand the local structure of $M^{\agCoul}([x], [y])$, we need to study the derivative of the section $\F^{\gCoul, \tau}_{\q}$. There are different ways of doing this, depending on what covariant derivative we choose. The most natural choice is to use a covariant derivative that involves the $\tilde{g}$-metric.  Specifically, for $1 \leq j \leq k$, we write 
$$\D^{\tilde{g},\tau}_{\gamma} \Fqgctau : \T^{\gCoul,\tau}_{j,\gamma} \to \V^{\gCoul,\tau}_{j-1,\gamma},$$ 
\begin{equation}
\label{eq:tangerine}
 \D^{\tilde{g},\tau}_{\gamma} \F^{\gCoul, \tau}_{\q}(V, \beta) =  \frac{\dgs}{dt} V  + (\Dgs_{\gamma(t)} \X^{\gCoul, \sigma}_{\q})(V) + \dd^{\gCoul, \sigma}_{\gamma(t)} \beta,
 \end{equation}
where 
\begin{equation}
\label{eq:dde}
\dd^{\gCoul, \sigma}_x : i\R \to \T^{\gCoul, \sigma}_{j, x}, \ \ \dd^{\gCoul, \sigma}_x (\xi) = (0, 0, \xi \phi)
\end{equation}
is the restriction of the operator $\dd^{\sigma}_x :  L^2_j(Y;i\R) \to \T^{\sigma}_{j, x}$ from \eqref{eq:dsigma} to constant functions. Also, $ \frac{\dgs}{dt}$ represents the covariant derivative (along a path) for the connection $\Dgs$ from \eqref{eq:DGS}, that is,
\begin{equation}
\label{eq:dgs}
\frac{\dgs}{dt} V  = \Pi^{\gCoul, \sigma}_* \circ \D^\sigma_{\Pi^{\elCoul, \sigma}(\tfrac{d\gamma}{dt})} (\Pi^{\elCoul, \sigma}(V)).
\end{equation}

We could also use the covariant derivative coming from the $L^2$ metric. Define
$$\D^\tau_{\gamma} \Fqgctau : \T^{\gCoul,\tau}_{j,\gamma} \to \V^{\gCoul,\tau}_{j-1,\gamma},$$ 
by
\begin{equation}
\label{eq:tangerine2}
 \D^\tau_{\gamma} \F^{\gCoul, \tau}_{\q}(V, \beta) =  \frac{D^\sigma}{dt} V  + (\D^\sigma_{\gamma(t)} \X^{\gCoul, \sigma}_{\q})(V) + \dd^{\gCoul, \sigma}_{\gamma(t)} \beta.
 \end{equation}

While the operator $ \D^{\tilde{g},\tau}_{\gamma} \F^{\gCoul, \tau}_{\q}$ seems more natural, it will be easier to work with $\D^\tau_{\gamma} \Fqgctau$ later on, so we will focus on the latter. When $\gamma$ is a flow trajectory, then the two operators coincide.

As discussed in Section~\ref{sec:4Dcoulomb}, the tangent space to $\widetilde{\B}^{\gCoul, \tau}([x], [y])$ has completions $\K^{\gCoul, \tau}_{j, \gamma} = \ker \dd_{\gamma}^{\gCoul, \tau, \tilde{\dagger}}$. From \eqref{eq:dgct} we can write
$$ \dd_{\gamma}^{\gCoul, \tau, \tilde{\dagger}} = \frac{d\beta}{dt} + \dd^{\gCoul, \sigma, \tilde{\dagger}}_{\gamma(t)}(V),$$
where, at $x=(a, s, \phi) \in W^{\sigma}$,
\begin{equation}
\label{eq:ddedagger}
 \dd^{\gCoul, \sigma, \tilde{\dagger}}_x: \T^{\gCoul, \sigma}_{j, x} \to i\R,  \ \ \ \ 
\dd^{\gCoul, \sigma, \tilde{\dagger}}_x(b, r, \psi) = i \langle  i \phi, \psi \rangle_{\tilde{g}}.
 \end{equation}

Note that $ \dd^{\gCoul, \sigma, \tilde{\dagger}}$ appears as part of the operator  $\dd^{\sigma, \tilde{\dagger}}: \T^{\sigma}_{j, x} \to L^2_{j-1}(Y; i\R)$ from \eqref{eq:decomposingtilde}. Precisely, we decompose the domain of $\dd^{\sigma, \tilde{\dagger}}$ as follows:
$$\T^{\sigma}_{j, x}=\K^{\agCoul, \sigma}_{j, x} \oplus \J^{\gCoul, \sigma}_{j,x} \oplus \J^{\circ, \sigma}_{j,x}=\T^{\gCoul, \sigma}_{j, x} \oplus \J^{\circ, \sigma}_{j,x}.$$
Then, with respect to the last decomposition (where we combine the first two summands into one), we have
\begin{equation}
\label{eq:Decomposed}
 \dd^{\sigma, \tilde{\dagger}} = \begin{pmatrix} 0 & -d^* \\
\dd^{\gCoul, \sigma, \tilde{\dagger}} & 0 \end{pmatrix} : \T^{\gCoul, \sigma}_{j, x} \oplus \J^{\circ, \sigma}_{j,x} \to (\im d^*)_{j-1}  \oplus i\R.
\end{equation}

Returning to our operator $\D^\tau_{\gamma} \Fqgctau$, we consider its restriction to $\K^{\gCoul, \tau}_{j, \gamma}$. The surjectivity of this restriction would imply that $M^{\agCoul}([x], [y])$ is a smooth manifold near $[\gamma]$, and the dimension of its kernel would give the dimension of $M^{\agCoul}([x], [y])$. As a preliminary step towards this discussion, we establish the following:

\begin{proposition}
\label{prop:FredholmCoulomb}
Let $x, y \in W^\sigma_k$ be non-degenerate stationary points of $\Xqsigma$ (and hence of $\Xqagcsigma$). Pick a path $\gamma \in W^{\tau}_k(x, y)$. Then, for $j \leq k$, the operator
$$(\D^\tau_{\gamma} \Fqgctau)|_{\K^{\gCoul,\tau}_{j,\gamma}} : \K^{\gCoul,\tau}_{j,\gamma} \to \V^{\gCoul,\tau}_{j-1,\gamma}$$
is Fredholm. Moreover, the Fredholm index is the same as that of the operator
\[
( \D^\tau_{\gamma} \F^{\tau}_{\q})|_{\K^{\tau}_{j,\gamma}}  : \K^{\tau}_{j,\gamma} \to \V^\tau_{j-1,\gamma}. 
\] 
\end{proposition}

\begin{proof}

To establish the Fredholm property, the arguments  are similar to those in the proof of Proposition~\ref{prop:Dfredholm}, with some modifications. 

We extend our operator $\D^\tau_{\gamma} \Fqgctau$ so that it acts on sections of a vector bundle.  First, define
\begin{equation}
\label{eq:Qggc}
 Q_{\gamma}^{\gCoul}  = \D^\tau_{\gamma} \Fqgctau \oplus \dd^{\gCoul, \tau, \tilde{\dagger}}_{\gamma} : \T^{\gCoul, \tau}_{j, \gamma} \to \V^{\gCoul,\tau}_{j-1,\gamma} \oplus L^2_{j-1}(\R; i\R).
\end{equation}

The codomain of $Q_{\gamma}^{\gCoul, \tilde g, \sigma} $ can be identified with $ \T^{\gCoul, \tau}_{j-1, \gamma}.$
However, even after writing $ \psis(t) = \psi(t) + r(t) \phi(t)$ as before, $\T^{\gCoul, \tau}_{j, \gamma}$ is still not the space of all sections of a vector bundle. Indeed, for $(b(t) + \beta(t)dt, \psis(t)) \in \T^{\gCoul, \tau}_{j, \gamma}$, we still have the conditions $d(\beta(t))=0$ and $d^* (b(t))=0$.

Thus, we need to extend the operator once more.  Consider the linear operator
\begin{equation}
\label{eq:RR}
 R = \frac{d}{dt} + \begin{pmatrix} 0 & -d \\ -d^* & 0 \end{pmatrix}: ( \im d_{(0)} \oplus \im d^*_{(1)} ) \to ( \im d_{(0)} \oplus \im d^*_{(1)}) 
 \end{equation}
where the subscript $(p)$ with $p \in \{0,1\}$ denotes the imaginary $p$-forms on which the respective operator acts, on each slice $\{t\} \times Y$. Precisely, $\im d^*_{(1)}$ is the subset of $L^2_{j}(Z; i\R)$ consisting of functions that integrate to zero slicewise, and $\im d_{(0)} = \ker d_{(1)}$, since $Y$ is a rational homology sphere. 

Next, decompose $\T^{\tau}_{j, \gamma}$ as 
\begin{equation}
\label{eq:Ttau}
\T^{\tau}_{j, \gamma} = \T^{\gCoul, \tau}_{j, \gamma} \oplus ( \J^{\circ, \tau}_{j, \gamma} \oplus \im d^*_{(1)}),
\end{equation}
where $\J^{\circ, \tau}_{j, \gamma}$ consists of time-dependent elements of the spaces $\J^{\circ, \sigma}_{j, \gamma(t)}$ from \eqref{eq:Jcircsigma}. Note that there is a natural identification
$$ \Psi:  \im d_{(0)} \to \J^{\circ, \tau}_{j, \gamma}$$
given at each time $t$ by the formula
$$ \Psi(-d\xi, 0, 0) =(-d\xi, 0, \xi \cdot \phi(t)),$$
where $\phi(t)$ is the spinor component of $\gamma(t)$. If we conjugate the operator $R$ by $\Psi$ on the $\im d_{(0)}$ summand, we obtain an operator
\begin{equation}
\label{eq:RRtilde}
 \Rhat = \Psi \circ \frac{d}{dt}\circ  \Psi^{-1} + \begin{pmatrix} 0 & \dd^{\sigma} \\ -d^* & 0 \end{pmatrix}: (\J^{\circ, \tau}_{j, \gamma} \oplus \im d^*_{(1)} ) \to (\J^{\circ, \tau}_{j, \gamma} \oplus \im d^*_{(1)} ).
 \end{equation}
Observe that $\Psi \circ \frac{d}{dt}\circ  \Psi^{-1} ( -d\xi, 0, \xi \phi) = (-\frac{d\xi}{dt}, 0, \frac{d\xi}{dt} \phi)$.  


Define the new extension of $Q_{\gamma}^{\gCoul} $ to be
\begin{equation}
\label{eq:newextension}
\Qhat_{\gamma}^{\gCoul}  =\begin{pmatrix} Q_{\gamma}^{\gCoul}  & 0 \\ 0 & \Rhat  \end{pmatrix}: \T^{\tau}_{j, \gamma} \to  \T^{\tau}_{j-1, \gamma},
 \end{equation}
with respect to the decomposition \eqref{eq:Ttau}. 

Instead of \eqref{eq:Ttau} we could consider the decomposition $\T^\tau_{j, \gamma} = \V^\tau_{j, \gamma} \oplus L^2_j(Z; i\mathbb{R})$, where we recall that $\V^\tau$ is the space of four-dimensional configurations with trivial $dt$ component that are slicewise in $\T^{\sigma}$. With respect to this new decomposition, we can write
\begin{equation}
\label{eq:Qhgc}
\Qhat_{\gamma}^{\gCoul}  =  \frac{D^\sigma}{dt} +\begin{pmatrix} M & 0 \\ 0 & 0 \end{pmatrix} + \begin{pmatrix} H & \dd^{\sigma}_{\gamma(t)} \\ \dd^{\sigma, \tilde{\dagger}}_{\gamma(t)} & 0 \end{pmatrix}.
 \end{equation}
Here, for a fixed time $t$, the operator $M$ acts by zero on $ \T^{\gCoul, \sigma}_{j, \gamma(t)} \subset \T^\sigma_{j, \gamma(t)} $ and equals the difference $\Psi \circ \frac{d}{dt}\circ  \Psi^{-1} -  \frac{D^\sigma}{dt}$ when applied to elements of  $\J^{\circ, \sigma}_{j, \gamma(t)} \subset \T^\sigma_{j, \gamma(t)}$; that is,
$$ M(-d\xi, 0, \xi \phi) = \Pi^{\perp}_{\phi} (0, 0, \xi \tfrac{d\phi}{dt}).$$ 
Furthermore, with respect to the decomposition $\T^\sigma_{j, \gamma(t)} = \T^{\gCoul, \sigma}_{j, \gamma(t)} \oplus \J^{\circ, \sigma}_{j, \gamma(t)}$, the operator $H$ from \eqref{eq:Qhgc} is given by
\begin{equation}\label{eq:Qgamma-Tgc-Jcirc}
 H = \begin{pmatrix} \D^\sigma_{\gamma(t)} \X^{\gCoul, \sigma}_{\q} & 0 \\ 0 & 0 \end{pmatrix}.
 \end{equation}
  Note that the third term in \eqref{eq:Qhgc} is exactly $\mathcal{H}^\sigma_{\gamma(t)}$ from Lemma~\ref{lem:fakehessian}.  

Using Lemma~\ref{lem:fakehessian} and the arguments in the proof of Proposition~\ref{prop:Dfredholm}, we obtain that $\Qhat_{\gamma}^{\gCoul} $ is of the form \eqref{eq:Q}, with the bundle $E = iT^*Y \oplus \Spin \oplus i\R$ just as for $Q_{\gamma}$. Furthermore, the differential part $L_0$ in $\Qhat_{\gamma}^{\gCoul} $ is the same $L_0$ that appeared in \eqref{eq:L0} for $Q_{\gamma}$. We write
\begin{equation}
\label{eq:tQg}
\Qhat_{\gamma}^{\gCoul}  = \frac{d}{dt} + L_0 + \hat{h}_t^{\gCoul}.
 \end{equation}
It's important to note that $L_0 + \hat{h}^{\gCoul}_t$ is not exactly the operator $\Hx$ arising from Lemma~\ref{lem:fakehessian}, as $\hat{h}^{\gCoul}_t$ also has terms coming from the time derivative of $\phi$. In the limit, as $t \to \pm \infty$ (i.e., as we approach the stationary points $x$ and $y$), we do have from \eqref{eq:Qgamma-Tgc-Jcirc} and Lemma~\ref{lem:fakehessian} that 
\begin{equation}\label{eq:4D-endpoint-Hess}
 L_0 +  \hat{h}_{\pm \infty}^{\gCoul} 
 =  \widehat{\Hess}^{\tilde{g}, \sigma}_{\q, \pm}, 
 \end{equation}
where $ \widehat{\Hess}^{\tilde{g}, \sigma}_{\q}$ is the extended Hessian from   \eqref{eq:ExtHessBlowUpGC}. The extended Hessian is hyperbolic at non-degenerate stationary points; indeed, it has real spectrum by Lemma~\ref{lem:eH}. Using the properties of $\Hx$ established in Lemma~\ref{lem:fakehessian}, we can apply Proposition 14.2.1 in \cite{KMbook} together with the arguments in \cite[p.256]{KMbook} (mentioned in the proof of Proposition~\ref{prop:Dfredholm}) and deduce that $\Qhat_{\gamma}^{\gCoul} $ is Fredholm.  As before, first we establish Fredholmness as a map from $L^2_1$ to $L^2$ for paths $\gamma$ which are constant near the endpoints, and then we extend this to operators from $L^2_j$ to $L^2_{j-1}$ for $1 \leq j \leq k$ and for all paths.  Again the index is independent of $j$.   

We claim that $\Qhat_{\gamma}^{\gCoul} $ has the same Fredholm index as the operator $Q_{\gamma}$ from \eqref{eq:Qg}. Indeed, we can relate them by a continuous family of Fredholm operators, in two steps, along the lines of Section~\ref{sec:interpol}. First, we interpolate linearly from $Q_{\gamma}$ to the operator
$$Q_{\gamma}^{\sp} := \frac{D^\sigma}{dt} + \begin{pmatrix} \D^\sigma_{\gamma(t)} \X_{\q}^{\sigma} & \dd^{\sigma}_{\gamma(t)} \\ \dd^{\sp,\sigma, \dagger}_{\gamma(t)} & 0 \end{pmatrix},$$
which differs from $Q_{\gamma}$ by having $ \dd^{\sp,\sigma, \dagger}_{\gamma(t)}$ instead of  $\dd^{\sigma, \dagger}_{\gamma(t)}$. Second, we use the family of metrics $g_{\rho}$ from Section~\ref{sec:interpol} to define operators $ \Qhat_{\gamma}^{\rho} $ similar to $\Qhat_{\gamma}^{\gCoul}$, such that at $\rho=0$ we have
$$\Qhat_{\gamma}^{0} = Q_{\gamma}^{\sp} $$
and at $\rho=1$ we have
$$ \Qhat_{\gamma}^{1} =\Qhat_{\gamma}^{\gCoul}.$$

To see that the operators considered during these two interpolations are Fredholm and have the same index, observe that they are all of the form \eqref{eq:Q}, with the same differential part $L_0$. Further, at the endpoints they limit to the interpolations between extended Hessians that appeared in Lemmas~\ref{lem:eHrho} and \ref{lem:eHrho2}, and which were shown there to be invertible with real spectrum. 

We have now shown that $Q_{\gamma}$ and  $\Qhat_{\gamma}^{\gCoul}$ have the same index. To go back from $\Qhat_{\gamma}^{\gCoul} $ to $Q_{\gamma}^{\gCoul}$, note that the operator $\Rhat$ defined in \eqref{eq:RRtilde} is bijective. Indeed, it is conjugate to the operator $R$ from \eqref{eq:RR}, so it suffices to check that $R$ is bijective. We have  $R= \frac{d}{dt} + A,$ where 
$$A = \begin{pmatrix} 0 & -d \\ -d^* & 0 \end{pmatrix}$$ 
is a self-adjoint operator acting slicewise. Observe that $A$ is invertible because $b_1(Y)=0$, and we can find a complete orthonormal system of eigenvectors $\{a_n\}$, with $A(a_n) = \lambda_n a_n$, $\lambda_n \neq 0$. An element in the kernel of $R$ would be of the form 
$$\sum_n c_n e^{-\lambda_n t} a_n, \ \ c_n \in \R.$$
However, any such nonzero element increases exponentially at one of the ends of the cylinder $Z$, and hence is not in $L^2_j$. Thus, $\ker (R) = 0$, and by self-adjointness, $\coker(R)=0$.
  
Since $\Rhat$ is bijective, from \eqref{eq:newextension} we deduce that $Q_{\gamma}^{\gCoul} $ is Fredholm, of the same index as $\Qhat_{\gamma}^{\gCoul} $ (and hence as $Q_{\gamma}$). In the proof of Proposition~\ref{prop:Dfredholm} we saw that $Q_{\gamma}$ has the same index as  $(\D^\tau_{\gamma} \F^{\tau}_{\q})|_{\K^{\tau}_{j,\gamma}}$. In a similar fashion, since $\dd^{\gCoul, \tau, \tilde{\dagger}}_{\gamma}$ is surjective by Lemma~\ref{lem:DecomposeGC}, we see that $(\D^\tau_{\gamma} \Fqgctau)|_{\K^{\gCoul,\tau}_{j,\gamma}}$ is Fredholm and of the same index as $Q_{\gamma}^{\gCoul, \tilde g, \sigma} $ by the same arguments as in \cite[Proposition 14.4.3]{KMbook}. The conclusion follows.
\end{proof}

We summarize the results obtained from the proof of Proposition~\ref{prop:FredholmCoulomb}.
\begin{lemma}
\label{lem:allsurjective}
Under the hypotheses of Proposition~\ref{prop:FredholmCoulomb},
\begin{enumerate}[(a)]
\item the operators
$$ (\D^\tau_{\gamma} \Fqgctau)|_{\K^{\gCoul,\tau}_{j,\gamma}}, \ Q_{\gamma}^{\gCoul}, \ \Qhat_{\gamma}^{\gCoul} $$
are all Fredholm of the same index;
\item
 one of the operators
$$ (\D^\tau_{\gamma} \Fqgctau)|_{\K^{\gCoul,\tau}_{j,\gamma}}, \ Q_{\gamma}^{\gCoul}, \ \Qhat_{\gamma}^{\gCoul} $$
  is surjective if and only if any of the two others is surjective. 
  \end{enumerate}
\end{lemma}

\begin{proof}
To establish the relation between $ (\D^\tau_{\gamma} \Fqgctau)|_{\K^{\gCoul,\tau}_{j,\gamma}}$ and $Q_{\gamma}^{\gCoul, \tilde g, \sigma} $ we use Lemma~\ref{lem:DecomposeGC}, which gives the surjectivity of $\dd^{\gCoul, \tau, \tilde{\dagger}}_{\gamma}$. To establish the relation between $Q_{\gamma}^{\gCoul} $ and $\Qhat_{\gamma}^{\gCoul} $, we use the block form \eqref{eq:newextension} and the bijectivity of $\Rhat$. 
\end{proof}

\section{Non-degeneracy of trajectories in Coulomb gauge}
\label{sec:NDTcoulomb}
Let us now suppose that $\gamma$ is a trajectory of $\Xqsigma$ on $\C_k^\sigma(Y)$ between two non-degenerate stationary points $x, y \in \C^\sigma(Y)$.  Recall from Definition~\ref{def:regM} that the moduli space $M(x, y)$ is regular at $\gamma$ if $Q_\gamma$ is surjective. By \cite[Propositions 14.3.2 and 14.4.3]{KMbook}, this is equivalent to the operator 
\[
 (\D^\tau_{\gamma} \F^{\tau}_{\q})|_{\K^{\tau}_{k,\gamma}}  :\K^{\tau}_{k,\gamma} \to \V^\tau_{k-1,\gamma} 
\] 
being surjective. (Compare the proof of Proposition~\ref{prop:Dfredholm}.) We rephrase this condition in Coulomb gauge:

\begin{proposition}\label{prop:Qgammasurjective}
Consider a path ${\gamma} \in \C^{\tau}_k(x, y)$ in temporal gauge. Write $x^\flat = \Pi^{\gCoul, \sigma}(x)$, $y^{\flat} =\Pi^{\gCoul, \sigma}(y)$ and $\gamma^\flat= \Pi^{\gCoul, \tau} (\gamma)$. Then:

$(a)$ The operators $(\D^\tau_{\gamma} \F^{\tau}_{\q})|_{\K^{\tau}_{k,\gamma}} : \K^{\tau}_{k, \gamma} \to \V^\tau_{k-1,\gamma}$ and $(\D^\tau_{\gamma^{\flat}} \Fqgctau)|_{\K^{\gCoul,\tau}_{k,\gamma^{\flat}}} : \K^{\gCoul,\tau}_{k,\gamma^{\flat}} \to \V^{\gCoul,\tau}_{k-1,\gamma^{\flat}}$ have the same Fredholm index.

$(b)$ Suppose that ${\gamma}$ is a trajectory of $\Xqsigma$, so that $[\gamma^{\flat}] \in \Btaug_k([x^\flat], [y^\flat])$ is a trajectory of $\Xqagcsigma$. If $(\D^\tau_{\gamma} \F^{\tau}_{\q})|_{\K^{\tau}_{k,\gamma}}$ is surjective, then so is $(\D^\tau_{\gamma^{\flat}} \Fqgctau)|_{\K^{\gCoul,\tau}_{k,\gamma^{\flat}}}$. 
\end{proposition}

\begin{proof}
$(a)$ We can interpolate between the paths $\gamma$ and $\gamma^{\flat}$ by a continuous family of paths $\{\gamma_s\}_{s \in [0,1]}$, with the endpoints of $\gamma_s$ varying on the gauge orbits of $x$ and $y$. Thus, the operators $Q_{\gamma}$ and $Q_{\gamma^{\flat}}$ are part of a family of Fredholm operators $Q_{\gamma_s}, \ s\in [0,1]$. All Fredholm operators in a continuous family must have the same index. Combining this with Proposition~\ref{prop:FredholmCoulomb}, we obtain $$ \ind((\D^\tau_{\gamma^\flat} \Fqgctau)|_{\K^{\gCoul,\tau}_{k,\gamma^{\flat}}}) = \ind ((\D^\tau_{\gamma^{\flat}} \Fqtau)|_{\K^{\tau}_{k,\gamma^\flat}}) = \ind ((\D^\tau_{\gamma} \F^{\tau}_{\q})|_{\K^{\tau}_{k,\gamma}}),$$
as desired. 

$(b)$ Recall that in Section~\ref{sec:path} we defined the map 
$$  \Pi^{\gCoul, \tau}: \widetilde \C^{\tau}_k(x, y) \to \widetilde \C_k^{\gCoul,\tau}(\x^\flat, \y^\flat)$$
and that after dividing by the gauge groups, this induces a map between the respective quotients:
$$ \Pi^{[\gCoul], \tau}: \widetilde \B^{\tau}_k([x], [y]) \to \tBtaug_k([x], [y]).$$
Note that when we work with gauge equivalence classes in $\C^\sigma(Y)$, we have $[x]=[x^\flat]$ and $[y]=[y^\flat].$ 

The derivative of $ \Pi^{[\gCoul], \tau}$, the map
$$ (\Pi^{[\gCoul], \tau}_* )_{[\gamma]}: \K^{\tau}_{k, \gamma} \longrightarrow \K_{k, \gamma^\flat}^{\gCoul, \tau},$$
was mentioned in \eqref{eq:Pigctau2}.

For $j \leq k$, observe that inside the completed tangent bundle $\T^{\tau}_j$ we have the subbundle $\V^{\tau}_j$ consisting of paths $(b(t), r(t), \psi(t))$ in temporal gauge. The derivative $ (\Pi^{\gCoul, \tau}_*)_{\gamma}$ maps $\V^{\tau}_j$ to $\V^{\gCoul, \tau}_j$ by applying infinitesimal global Coulomb projection slicewise:
$$(\Pi^{\gCoul, \tau}_*)_{\gamma} (b, r, \psi) (t) = (\Pi^{\gCoul, \sigma}_*)_{\gamma(t)} (b(t), r(t), \psi(t)).$$

Consider the diagram
\begin{equation}
\label{eq:diagram}
\xymatrixcolsep{4pc}
\xymatrix{
\widetilde \B^\tau_k([x],[y]) \ar[r]^{\Fqtau} \ar[d]_{\Pi^{[\gCoul],\tau}}  & \V^\tau_{k-1} \ar[d]^{\Pi^{\gCoul,\tau}_*} \\
\widetilde \B^{\gCoul,\tau}_k([x^\flat],[y^\flat]) \ar[r]^{\ \ \ \Fqagctau} & \V_{k-1}^{\gCoul,\tau}.
}
\end{equation}
We claim that this diagram commutes.  This can be seen as follows.  Choose a representative $\gamma$ of $[\gamma] \in \widetilde \B^\tau_k([x],[y])$ in temporal gauge.

Since the elements of $\V^\tau_j$ are in temporal gauge, we can write $\Fqtau = \frac{d}{dt} + \Xqsigma$ and $\Fqgctau = \frac{d}{dt} + \Xqgcsigma$. 

Let us also write $\gamma^\flat(t)=\Pi^{\gCoul,\sigma}(\gamma(t))$ as $g(t)\gamma(t)$, where $g(t)$ is a path of gauge transformations on $Y$.  Since the (perturbed) CSD functional is gauge-invariant, we have that $\Xqsigma$ is gauge-equivariant and thus:
\begin{equation}\label{eq:g(t)}
\Xqsigma (\gamma^{\flat}(t)) = \Xqsigma(g(t)\gamma(t)) = g(t)_* \Xqsigma(\gamma(t)).
\end{equation}
Moreover, direct computation shows that
\begin{equation}\label{eq:g_*}
(\Pi^{\gCoul,\sigma}_*)_{\gamma^{\flat}(t)} \circ g(t)_* = (\Pi^{\gCoul,\sigma}_*)_{\gamma(t)}.
\end{equation}
Since $\Xqgcsigma = \Pi^{\gCoul,\sigma}_* \circ \Xqsigma$, applying $\Pi^{\gCoul,\sigma}$ to \eqref{eq:g(t)} and using \eqref{eq:g_*} shows that
\begin{equation}\label{eq:equivariance}
\Xqgcsigma(\gamma^{\flat}(t)) = (\Pi^{\gCoul,\sigma}_*)_{\gamma^{\flat}(t)} \Xqsigma (\gamma^{\flat}(t)) = (\Pi^{\gCoul,\sigma}_*)_{\gamma(t)}  \Xqsigma(\gamma(t)). 
\end{equation}

By the chain rule, we have 
$$\frac{d\gamma^{\flat}(t)}{dt} = \frac{d}{dt}(\Pi^{\gCoul,\sigma} \circ \gamma(t)) = (\Pi^{\gCoul,\sigma}_*)_{\gamma(t)} \frac{d\gamma(t)}{dt},$$
or, for short,
\begin{equation}\label{eq:dtequivariance}
\frac{d}{dt} (\gamma^{\flat}) = \Pi^{\gCoul,\tau}_*  \Bigl(\frac{d}{dt} \gamma \Bigr).
\end{equation}

Using \eqref{eq:equivariance} and \eqref{eq:dtequivariance} we compute 
\begin{align*}
\Fqagctau  \circ \Pi^{[\gCoul],\tau}(\gamma) &= \frac{d}{dt} (\gamma^{\flat}) + \Xqgcsigma (\gamma^{\flat}) \\
&= \Pi^{\gCoul,\tau}_*  \Bigl(\frac{d}{dt} \gamma \Bigr) + \Pi^{\gCoul,\tau}_*  (\Xqsigma \gamma) \\
&= \Pi^{\gCoul,\tau}_* (\Fqtau (\gamma)).
\end{align*}
Thus, the diagram \eqref{eq:diagram} commutes.

Taking derivatives in \eqref{eq:diagram} and using the fact that $\gamma$ and $\gamma^\flat$ are trajectories of the respective vector fields, we obtain the commutative diagram
\begin{equation}\label{eq:commutingderivatives}
\xymatrixcolsep{5pc}
\xymatrix{
\K^\tau_{k,\gamma} \ar[r]^{\D^\tau_{\gamma}\Fqtau} \ar[d]_{\Pi^{\gCoul,[\tau]}_* }  & \V^\tau_{k-1,\gamma} \ar[d]^{\Pi^{\gCoul,\tau}_*} \\
\K^{\gCoul,\tau}_{k,\gamma^{\flat}} \ar[r]^{\D^\tau_{\gamma^{\flat}} \Fqagctau} & \V_{k-1,\gamma^{\flat}}^{\gCoul,\tau}.
}
\end{equation}

Since the infinitesimal global Coulomb projection $\Pi^{\gCoul, \sigma}_*$ is surjective, we see that the right vertical arrow $\Pi^{\gCoul,\tau}_*$ in \eqref{eq:commutingderivatives} is also surjective.  Using the commutativity of \eqref{eq:commutingderivatives}, we get that if $(\D^\tau_{\gamma} \F^{\tau}_{\q})|_{\K^{\tau}_{k,\gamma}}$ is surjective, then so is $(\D^\tau_{\gamma^{\flat}} \Fqgctau)|_{\K^{\gCoul,\tau}_{k,\gamma^{\flat}}}$.  
 \end{proof}

Recall that Proposition~\ref{prop:gccorrespondence} guaranteed that the moduli spaces $M([x],[y])$ and $M^{\agCoul}([x],[y])$ are homeomorphic.  Proposition~\ref{prop:Qgammasurjective} says more: If $M([x],[y])$ is cut out smoothly and transversely by $\Fqtau$, then also $M^{\agCoul}([x],[y])$ is cut out smoothly and transversely by $\Fqagctau$.

So far we have only discussed the usual moduli spaces $M([x], [y])$. Similar arguments apply to the moduli spaces $M^{\red}([x], [y])$ between reducibles. The end result is that we can identify them with the corresponding moduli spaces $M^{\agCoul,\red}([x], [y])$ in Coulomb gauge, and that regularity of the former implies regularity of the latter.

\section{Gradings}
\label{sec:GradingsCoulomb}
From now on we will assume that we have chosen an admissible perturbation $\q$, so that all the stationary points are non-degenerate, and the moduli spaces $M([x], [y])$ and $M^{\red}([x], [y])$ are regular.  

Let $x, y$ be stationary points in $W^\sigma$.  In view of Lemma~\ref{lem:allsurjective}, Proposition~\ref{prop:Qgammasurjective} and the definition of relative gradings in Section~\ref{sec:modifications}, we have
\begin{align}\label{eq:gradingscoulomb}
\gr(x, y) & =  \ind \ (\D^\tau_{\gamma} \Fqtau)|_{\K^{\tau}_{k,\gamma}}  = \ind Q_{\gamma} \\
&= \ind \ (\D^\tau_{\gamma^{\flat}} \Fqagctau)|_{\K^{\gCoul,\tau}_{k,\gamma^{\flat}}} = \ind Q^{\gCoul}_{\gamma^\flat}, \notag
\end{align}
where $\gamma$ is any path from $x$ to $y$ (not necessarily a trajectory of $\Xqsigma$). 

Thus, the relative gradings can be calculated directly in Coulomb gauge. 

Recall from \eqref{eq:GradingRed} that the relative gradings of reducible stationary points of $\Xq$ with the same connection component can be computed in terms of the spectrum of the operator $D_{\q,a}$; here, an eigenvector $\phi$ of $D_{\q,a}$ has eigenvalue given by the spinorial energy of $(a,0,\phi)$.  It is thus natural to ask how this relates to the analogous setup in Coulomb gauge.  The following tells us that spinorial energy and ``Coulomb gauge spinorial energy'' are equal.   

\begin{lemma} \label{lem:coulomb-gauge-spinorial-same}
Let $x = (a,s,\phi) \in W_k^\sigma$.  Then, 
\begin{equation}\label{eq:spinorials-agree}
\Re \langle \widetilde \Xqgc^1(a,s,\phi), \phi \rangle_{L^2} = \Re \langle \widetilde \Xq^1(a,s,\phi), \phi \rangle_{L^2}.
\end{equation}
\end{lemma}
\begin{proof}
By continuity, it suffices to establish \eqref{eq:spinorials-agree} for irreducibles.  We have 
\begin{align*}
\Re \langle \widetilde \Xqgc^1(a,s,\phi), \phi \rangle_{L^2} &= \frac{1}{s} \Re \langle (\Xqgc)^1(a,s\phi), \phi \rangle_{L^2} \\
&= \frac{1}{s} \Re \langle \Xq^1(a,s\phi) + Gd^*(\Xq^0(a,s\phi))\phi, \phi \rangle_{L^2} \\
&= \frac{1}{s} \Re \langle \Xq^1(a,s\phi), \phi \rangle_{L^2} \\
&= \Re \langle \widetilde \Xq^1(a,s,\phi), \phi \rangle_{L^2},
\end{align*}
where in the penultimate equality, we use that $\Re \langle Gd^*(a,s\phi) \phi, \phi \rangle_{L^2} = 0$ since $Gd^*(a,s\phi)$ is purely imaginary.  
\end{proof}

In light of Lemma~\ref{lem:coulomb-gauge-spinorial-same}, we do not define a separate spinorial energy in Coulomb gauge.   

Consider $x=(a, 0, \phi)$ a reducible stationary point of $\Xqgcsigma$. Then $\phi$ is an eigenvector of  the operator $D^{\gCoul}_{\q, a}$, defined by
$$ D^{\gCoul}_{\q, a} (\psi) := \D_{(a, 0)}  (\Xqgc)^1(0, \psi) = \widetilde \Xqgc^1(a, 0, \psi).$$
Let $\mu$ be the corresponding eigenvalue. Since $x$ is necessarily a stationary point of $\Xqsigma$ as well,  $\phi$ is an eigenvector of $D_{\q,a}$ and $(a,0)$ is a stationary point of $\Xq$.  In particular, $\D_{(a,0)} \Xq(0,\phi)$ is in Coulomb gauge, and thus
$$
\mu \phi = D_{\q,a}(\phi) = D^{\gCoul}_{\q,a} (\phi).
$$    
It follows that 
\begin{equation}\label{eq:LambdaRed}
\Lambda_\q(x) = \langle \widetilde \Xqgc^1(0,\phi), \phi \rangle_{L^2}= \langle \mu\phi , \phi \rangle_{L^2} = \mu,
\end{equation}
which agrees with Lemma~\ref{lem:coulomb-gauge-spinorial-same}.  Note that here we do not need to take real parts, since the relevant operators are self-adjoint at stationary points.

\begin{lemma}
\label{lem:grsp}
Fix a reducible stationary point $(a,0)$ of $\Xqgc$.  For each $N \in \mathbb{N}$, there exist $\omega_1 , \omega_2 > 0$ such that the (finitely many) reducible stationary points of $\Xqagcsigma$ which agree with $(a,0)$ in the blow-down and have grading in the interval $[-N,N]$ are precisely the reducible stationary points with spinorial energy in the interval $[-\omega_1,\omega_2]$.   
\end{lemma}
\begin{proof}
This follows from \eqref{eq:GradingRed} and \eqref{eq:LambdaRed}.
\end{proof}

\section{The cut-down moduli spaces in Coulomb gauge}
\label{sec:UCoulomb}
Recall that in Section~\ref{sec:mfh}, the $U$-map on monopole Floer homology was defined by intersecting the moduli spaces $M([x], [y])$ and $M^{\red}([x], [y])$ with the zero set $\Zs$ of a transverse section $\sect$ of a  complex line bundle $E^{\sigma}$ over $\B^{\sigma}_k(B_p)$. The bundle $E^{\sigma}$ was associated to the map $\G_{k+1}(B_p) \to S^1, \ u \mapsto u(p)$, where $p=(t, q)$ is a point in $\rr \times Y$. Without loss of generality, let us assume that $t=0$.

We will need to modify this definition so that we can relate it to Coulomb gauge. First, whereas in Section~\ref{sec:mfh} we followed \cite{KMOS} and used the restriction of configurations to a standard ball $B_p$ around $p$, we could just as well restrict to any closed neighborhood $N_p$ of $p$ that is a manifold with boundary. It is convenient to take $N_p = [-1, 1] \times Y$. Consider the restriction $r: \B_k(N_p) \to \B_k(B_p)$. On the blow-up, this induces a map
$$ r^{\sigma} : \Ne \to \B^{\sigma}_k(B_p),$$
$$ r^{\sigma}(a, s, \phi) = \bigl (a, s \cdot \|\phi\|_{L^2(B_p)}, {\phi}/{ \|\phi\|_{L^2(B_p)}} \bigr),$$
well-defined on the open subset $\Ne \subset \B^{\sigma}_k(N_p)$ consisting of configurations $(a, s, \phi)$ such that $\phi$ does not vanish identically on $B_p$. Note that, because of the unique continuation principle, the moduli spaces $M([x], [y])$ and $M^{\red}([x], [y])$ are contained in $\Ne$. (Here, we identify trajectories with their restrictions to $N_p$, again using unique continuation.)

Let us pull back $\sect$ under $r^\sigma$ and obtain a section $(r^\sigma)^*\sect$ of the bundle 
$$(E^{\sigma})' := (r^\sigma)^*E^{\sigma}$$ over $\Ne$. Intersecting the moduli spaces with the zero set of $(r^\sigma)^*\sect$ is the same as intersecting them with $\Zs$. Furthermore, we can consider any other transverse section of $(E^{\sigma})'$, and let $\Zs'$ be its zero set. If we define the $U$-maps using intersections with $\Zs'$, standard continuation arguments in Floer theory show that they are chain homotopic to the original ones. 

Note that $(E^{\sigma})'$ is associated to the map $\G_{k+1}(N_p) \to S^1, \ u \mapsto u(0,q)$. For our next modification, let $(E^{\sigma})''$ be the complex line bundle over $\Ne \subset \B^{\sigma}_k(N_p)$ associated to the map 
\begin{equation}
\label{eq:mueq}
\G_{k+1}(N_p) \to S^1, \ e^{f} \mapsto e^{\mu_Y(f(0, \cdot))},
\end{equation}
where $f : N_p \to i\rr$ and $\mu_Y(f(0, \cdot))$ is the average value of $f$ on the slice $\{0\} \times Y$.

We can construct a family of bundles interpolating between $(E^{\sigma})'$ and $(E^{\sigma})''$, by considering the maps
$$ \G_{k+1}(N_p) \to S^1, \ e^{f} \mapsto e^{\lambda \mu_Y(f(0, \cdot)) + (1-\lambda)f(0,q)}.$$
Thus, instead of considering sections of $(E^{\sigma})'$, we could define the $U$-maps using sections of $(E^{\sigma})''$, and the results will be chain homotopic to the originals.

The third modification consists in moving from the $\sigma$ model to the $\tau$ model. As explained in \cite[Section 6.3]{KMbook}, the two models are equivalent in the following sense. Consider the open subset $\Ue \subset \B^{\sigma}_k(N_p)$ consisting of configurations $(a, s, \phi)$ with $\phi|_{\{t\} \times Y} \not \equiv  0$ for all $t \in [0,1]$. As noted in \cite[p.463]{KMbook}, there is a natural map 
$$\varrho: \Ue \to \B^{\tau}_k(N_p).$$
By unique continuation, the moduli spaces $M([x], [y])$ and $M^{\red}([x], [y])$ restricted to $N_p$ yield configurations in $\Ue$, which map homeomorphically onto their image under $\varrho$. 

There is a complex line bundle $E^{\tau}$ over $\B^{\tau}_k(N_p)$ associated to the same map \eqref{eq:mueq} as before. When we pull it back under $\varrho$ and then further restrict to $\Ue \cap \Ne$, we obtain the restriction of the bundle $(E^{\sigma})''$ to $\Ue \cap \Ne$. Thus, we can equivalently define the $U$-maps using the $\tau$ model and intersecting with the zero set $\Zs^{\tau}$ of a transverse section of $E^{\tau}$.

We are now ready to make the connection with configurations in Coulomb gauge. This is done via the global Coulomb projection $\Pi^{[\gCoul],\tau}$. In \eqref{eq:Pigctau} the map $\Pi^{[\gCoul],\tau}$ was defined for configurations on $\rr \times Y$ with fixed asymptotics, but the same formula \eqref{eq:zebra} can be applied to a configuration on $N_p = [-1,1] \times Y$. By a slight abuse of notation, we still write $\Pi^{[\gCoul],\tau}$ for the resulting map $\B^{\tau}_k(N_p) \to \B^{\gCoul, \tau}_k(N_p)$. Here, $\B^{\gCoul, \tau}_k(N_p)$ stands for the quotient of the space $\C^{\gCoul, \tau}_k(N_p)$ by the gauge group $\G^{\gCoul}_{k+1}(N_p)$; cf. Section~\ref{sec:cylinders}.

Note that there is a natural map
\begin{equation}
\label{eq:BW}
\B^{\gCoul, \tau}_k(N_p) \to W^\sigma_{k-1/2}/S^1, \ \ [\gamma] \to [\gamma(0)],
\end{equation}
where the representative $\gamma$ is chosen to be in temporal gauge.

Consider the complex line bundle $E^{\agCoul, \sigma}$ over $W^\sigma_{k-1/2}/S^1$, associated to the $S^1$-bundle $W^\sigma_{k-1/2}$. Pick a section $\sect^{\agCoul}$ transverse to the zero section, and such that the zero set $\Zs^{\agCoul}$ of $\sect^{\agCoul}$ intersects all the moduli spaces $M^{\agCoul}([x], [y])$ and $M^{\agCoul, \red}([x], [y])$ transversely. We obtain cut-down moduli spaces in Coulomb gauge:
$$M^{\agCoul}([x], [y]) \cap \Zs^{\agCoul} \ \ \text{and} \ \ M^{\agCoul, \red}([x], [y]) \cap \Zs^{\agCoul}.$$

Here, we identified $M^{\agCoul}([x], [y])$ and $M^{\agCoul, \red}([x], [y])$ with their images in $W^\sigma_k/S^1 \subset W^\sigma_{k-1/2}/S^1$ at time $t=0$, for simplicity.  (See \cite[Proposition 7.2.1]{KMbook} for the model unique continuation result for this case.)  Alternatively, we could identify them with their images in $\B^{\gCoul, \tau}_k(N_p)$ under restriction (using unique continuation). Consider the line bundle $E^{\agCoul, \tau}$ over $\B^{\gCoul, \tau}_k(N_p)$, pulled back from $E^{\agCoul, \sigma}$ under the map \eqref{eq:BW}. From $\sect^{\agCoul}$  we obtain a section of $E^{\agCoul, \tau}$, and intersections of its zero set with the moduli spaces correspond to intersections of $\Zs^{\agCoul}$ with those moduli spaces.

Observe that the pull-back of $E^{\agCoul, \tau}$ under $\Pi^{[\gCoul],\tau}$ is exactly the bundle $E^{\tau}$ over $\B^{\tau}_k(N_p)$. Thus, we can pull back the section of $E^{\agCoul, \tau}$ and obtain a section of $E^{\tau}$, which we can then take to be the one defining the $U$-maps in the $\tau$ model. We have a commutative diagram:
\[
\xymatrixcolsep{4pc}
\xymatrix{
\ \ M([x],[y]) \ \ \ar[r] \ar[d]_{\Pi^{[\gCoul],\tau}}  & \B^\tau_k(N_p) \ \ar[d]^{\Pi^{[\gCoul],\tau}} \\
M^{\agCoul}([x],[y]) \ar[r] & \B^{\gCoul, \tau}_k(N_p),
}
\]
where the horizontal maps are given by restriction, and are one-to-one (by unique continuation). Since we have established that $\Pi^{[\gCoul],\tau} : M([x],[y]) \to M^{\agCoul}([x],[y])$ is a homeomorphism, we obtain an identification of the cut-down moduli spaces:
\begin{equation}
\label{eq:M1Z}
 M([x], [y]) \cap \Zs^{\tau} \ \cong \ M^{\agCoul}([x], [y]) \cap \Zs^{\agCoul}
 \end{equation}
and
\begin{equation}
\label{eq:M2Z}
 M^{\red}([x], [y]) \cap \Zs^{\tau} \ \cong \ M^{\agCoul, \red}([x], [y]) \cap \Zs^{\agCoul}.
\end{equation}

\section{Orientations}
\label{sec:OrientCoulomb}
In Section~\ref{sec:or2} we explained how the moduli spaces $M([x], [y])$ can be oriented using an orientation data set. That discussion can be adapted to global Coulomb gauge. To orient the spaces $M^{\agCoul}([x], [y])$, we need to trivialize the determinant lines $\det(Q_{\gamma}^{\gCoul} )$. For arbitrary $x, y \in W^{\sigma}$, consider compact intervals $I=[t_1, t_2]$, paths $\gamma \in W^{\tau}(I \times Y)$ restricting to $x$ and $y$ on the two boundary components, and Fredholm operators of the form
\begin{equation}
\label{eq:Pgammagc}
 P_{\gamma}^{\gCoul}  = \bigl( Q_{\gamma}^{\gCoul} , -\Pi_1^{\gCoul,+}, \Pi_2^{\gCoul,-} \bigr) : \T^{\gCoul,\tau}_{1, \gamma}(I \times Y) \to \V^{\gCoul,\tau}_{0, \gamma} \oplus L^2(I; i\R) \oplus H_1^{\gCoul, +} \oplus H_2^{\gCoul, -}.
 \end{equation}
Here, $Q_{\gamma}^{\gCoul} $ has the same expression as in \eqref{eq:Qggc}, and $\Pi_1^{\gCoul,+}$ is the composition of restriction to the boundary with spectral projection onto 
$$ H_1^{\gCoul, +} = \K^{\agCoul, +}_{1/2, x} \oplus i\R \subset  \T^{\gCoul, \sigma}_{1/2, x}(Y) \oplus i\R,$$
with $\K^{\agCoul, +}_{1/2, x} \subset \K^{\agCoul, \sigma}_{1/2, x}$ being the direct sum of all positive eigenspaces of the Hessian $\Hess^{\tilde{g}, \sigma}_{\q, x} - \epsilon$ for small, positive $\epsilon$. The space $H_2^{\gCoul, -}$ and the projection $\Pi_2^{\gCoul,-}$ are defined similarly, using the nonpositive eigenspaces of $\Hess^{\tilde{g}, \sigma}_{\q, y} - \epsilon$.  Note that since $\Hess^{\tilde{g},\sigma}_{\q,x}$ is invertible with real spectrum and is a compact perturbation of a self-adjoint and invertible operator, we see that $\K^{\agCoul,+}_{1/2,x} \oplus \K^{\agCoul,-}_{1/2,x} = \K^{\agCoul,\sigma}_{1/2,x}$ (cf. \cite[p.313]{KMbook}).    

We define an {\em orientation data set in Coulomb gauge} $o^{\gCoul}$ to consist of orientations $o^{\gCoul}_{[x], [y]}$ for $\det(P_{\gamma}^{\gCoul} )$ (for any $\gamma$), satisfying the compatibility condition
$$ o^{\gCoul}_{[x], [y]} \cdot o^{\gCoul}_{[y], [z]} = o^{\gCoul}_{[x], [z]}.$$
An orientation data set in Coulomb gauge produces orientations on the moduli spaces $M^{\agCoul}([x], [y])$ and $M^{\agCoul, \red}([x], [y]).$

\begin{proposition}
\label{prop:orientationsgc}
An orientation data set $o$ (as in Section~\ref{sec:or2}) naturally induces an orientation data set $o^{\gCoul}$ in Coulomb gauge, such that the homeomorphisms constructed in Proposition~\ref{prop:gccorrespondence},
$$ M([x], [y]) \xrightarrow{\cong} M^{\agCoul}([x], [y]),$$
are orientation-preserving, and so are the homeomorphisms
$$ M^{\red}([x], [y]) \xrightarrow{\cong} M^{\agCoul, \red}([x], [y]).$$
\end{proposition}

\begin{proof}
Fix an orientation data set $o$. This trivializes the determinant lines $\det(P_{\gamma})$, for the operators $P_{\gamma}$ from \eqref{eq:Pgamma}. Let us focus on trajectories $\gamma \in W^{\tau}(I \times Y)$. We seek to trivialize the corresponding operators $P_{\gamma}^{\gCoul}$. To go between $P_{\gamma}$ and $P_{\gamma}^{\gCoul}$, we follow the steps in the proof of Proposition~\ref{prop:FredholmCoulomb}, but considering operators defined on compact cylinders, and with spectral projections added at the boundary. Specifically, we can deform $P_{\gamma}$ into an operator of the form
$$\widehat{P}_{\gamma}^{\gCoul}  =\begin{pmatrix} P_{\gamma}^{\gCoul}  & 0 \\ 0 & \widehat{J}  \end{pmatrix},$$
where $\widehat{J}$ is the analogue of $\Rhat$ from \eqref{eq:RRtilde}, with spectral projections added at the boundary. One can check that $\widehat{J}$ is bijective. Hence, a trivialization of $\det(P_{\gamma})$ gives one of $\det(\widehat{P}_{\gamma}^{\gCoul} )$ and then one of $\det(P_{\gamma}^{\gCoul})$. The resulting trivializations are compatible with concatenation, and hence combine into an orientation data set in Coulomb gauge. 

The fact that the homeomorphisms are orientation-preserving is immediate from the construction.
\end{proof}

Observe that Proposition~\ref{prop:orientationsgc} also implies that the identifications \eqref{eq:M1Z} and \eqref{eq:M2Z}, between the cut-down moduli spaces, are orientation-preserving.

\section{Monopole Floer homology in Coulomb gauge}\label{sec:CoulombSummary}

The work done in this chapter allows us to rephrase the definition of monopole Floer homology in terms of configurations in Coulomb gauge.

We fix an admissible perturbation $\q$. The generators of $\cmto$ can be taken to be some of the  stationary points $[x]$ of the vector field $\Xq^{\agCoul, \sigma}$ on $W^\sigma/S^1$; precisely, those that are either in the interior of $W^\sigma/S^1$, or on the boundary and stable. Indeed, by \eqref{eq:EquivStat2}, these are in one-to-one correspondence with the generators $\Crit^o \cup \Crit^s$ considered in Section~\ref{sec:mfh}. Further, since the original generators are non-degenerate, so are the ones in Coulomb gauge, in the sense that $\ker (\D^\sigma_x\Xq^{\agCoul, \sigma}) = 0$ at each stationary $x$; see Lemma~\ref{lem:rephraseStat}.

To define the differential on $\cmto$, we can use the moduli spaces $M^{\agCoul}([x], [y])$ and $M^{\agCoul, \red}([x], [y])$, consisting of trajectories of $\Xq^{\agCoul, \sigma}$. By \eqref{eq:EquivTraj2} and Proposition~\ref{prop:gccorrespondence}, these are in one-to-one correspondence with the moduli spaces of monopoles considered in Section~\ref{sec:mfh}. Moreover, by Proposition~\ref{prop:Qgammasurjective}, since the original moduli spaces are regular, so are the ones in Coulomb gauge. Specifically, this means that the operators $(\D^\tau_{\gamma} \Fqgctau)|_{\K^{\gCoul,\tau}_{k,\gamma}}$ (or, equivalently, $Q^{\gCoul}_{\gamma}$) are surjective for all $[\gamma] \in M^{\agCoul}([x], [y])$, except in the boundary-obstructed case, where the cokernel has dimension 1. 

As shown in \eqref{eq:gradingscoulomb}, we can define the relative gradings between generators $[x]$ and $[y]$ in Coulomb gauge by the index of the operator $Q^{\gCoul}_{\gamma}$, where $\gamma$ is a path from $x$ to $y$ in $W^\sigma$ and that these are the same relative gradings as in $\cmto$.  

Moreover, we can orient the moduli spaces $M^{\agCoul}([x], [y])$ and $M^{\agCoul, \red}([x], [y])$ using an orientation data set in Coulomb gauge, as in Section~\ref{sec:OrientCoulomb}; see Proposition~\ref{prop:orientationsgc}.

With this in mind, we define the differential $\check\del$ on $\cmto$ by the same formulas as \eqref{eqn:boundaryirred}, \eqref{eqn:boundaryred}, and \eqref{eqn:cmboundary}, but using the moduli spaces $M^{\agCoul}([x], [y])$, $M^{\agCoul, \red}([x], [y])$, instead of $M([x], [y]), M^{\red}([x], [y])$. 

Finally, the $\zz[U]$-module structure on $\cmto$ can also be described in Coulomb gauge. We apply the equivalences \eqref{eq:M1Z} and \eqref{eq:M2Z} established in Section~\ref{sec:UCoulomb}. Thus, the $U$-map is given by the formulas similar to \eqref{eqn:mU}, \eqref{eqn:mUred}, and \eqref{eqn:floercap}. For the new formulas, we instead intersect the moduli spaces $M^{\agCoul}([x], [y])$, $M^{\agCoul, \red}([x], [y])$ with the zero set $\Zs^{\agCoul}$ of a generic section of the  complex line bundle $E^{\agCoul,\sigma}$ over $W^{\sigma}/S^1$.

\chapter[Finite-dimensional approximations]
{Finite-dimensional approximations with tame perturbations} \label{sec:finiteapproximations}

\section{Very compactness}\label{sec:verycompact}

Since monopole Floer homology is defined using a perturbation $\q$, we must define an analogue of the Floer spectrum using this perturbation as well.  In order to do this, we must recall a more general setting in which the spectrum can be defined.  

\begin{definition}
\label{def:vc}
Let $Y$ be a closed oriented Riemannian three-manifold with $b_1(Y)=0$, with a spinor bundle $\Spin$. Let $W = \ker d^* \oplus \Gamma(\Spin) \subset \C(Y)$ be the global Coulomb slice from Section~\ref{sec:coulombs}.

A smooth map $\eta:W \to W$ is called {\em very compact} if for all integers $k \geq 5$ and compact cylinders $Z = I \times Y$, the following two conditions are satisfied: 

\noindent $(a)$ The map $\eta$ induces a continuous, functionally bounded map 
\[
\hat{\eta}: W_{k-1}(Z) \to W_{k-1}(Z), \ \ \text{and}
\] 
$(b)$ We can extend the differentials 
\[
 W \times \prod^m_{i = 1} W \to W, \ (x; v_1, \dots, v_m) \to (\D^m\eta)_x(v_1, \dots, v_m)
\]
to continuous maps
\[
W_{k}(Z) \times W_{k-1-i_1}(Z) \times  \ldots \times W_{k-1-i_m}(Z) \to W_{k-1-\sum i_s}(Z).    
\]
\end{definition}

By the argument mentioned in Remark~\ref{rem:qhatq} (that is, working with configurations which are constant in the $\R$ direction), we obtain the analogous regularity statements on $Y$ as well.  We call such a map $\eta$ very compact because condition (a), applied with $k$ instead of $k-1$, guarantees that the induced map $\eta : W_k \to W_{k}$ is functionally bounded, and therefore $\eta : W_k \to W_{k-1}$ is compact (in the sense that it takes bounded sets to precompact sets).  
 
One example of a very compact map is the map $c: W \to W$ from \eqref{eq:cmap}.

\begin{proposition}\label{prop:proposition3perturbed}
Let $\eta: W \to W$ be a very compact map.  Fix $k \geq 5$.  Suppose that there exists a closed, bounded subset $N$ of $W_{k}$ such that the finite type trajectories of $l + \eta$ on $W_{k}$ that are contained in $N$ are actually contained in $W \cap U$ for an open set $U \subset N$.  Then: \\
$(i)$ For $\lambda \gg 0$, trajectories of $l+\pml \eta$ contained in $N$ must be contained in $U$.  \\
$(ii)$ We can define the Floer spectrum $\Sigma^{-W^{(-\lambda,0)}}I^\lambda$ as in \eqref{sec:spectrum}.  This is independent of $\lambda$ up to stable equivalence. \\
$(iii)$ Furthermore, if $\eta$ is $S^1$-equivariant, then the Floer spectrum can be constructed equivariantly, and is an invariant up to $S^1$-equivariant stable equivalence.
\end{proposition} 

We omit the proof of Proposition~\ref{prop:proposition3perturbed}, as it is analogous to that of \cite[Theorem 1]{Spectrum}.  In particular, part (i) corresponds to \cite[Proposition 3]{Spectrum}.  The results in \cite{Spectrum} were for the special case where $\eta = c$, $N = \overline{B(2R)}$, and $U = B(R)$, as in Chapter~\ref{sec:spectrum}.  The properties of $c$ used in those proofs were exactly those listed in Definition~\ref{def:vc}.  

\begin{remark}
Very compactness was defined in \cite[Definition 4]{GluingBF}.  There, condition (a) only required that $\eta:W_k \to W_{k-1}$ be a compact map.  Proposition 5 in \cite{GluingBF} claimed that Proposition~\ref{prop:proposition3perturbed} is true under this weaker hypothesis.  However, this is in fact not sufficient.  In \cite{Spectrum}, Step 1 in the proof of Proposition 3 requires functional boundedness of $\eta$ to bound the derivatives of trajectories, and Step 3 requires continuity of $\hat{\eta}$ in order to do elliptic bootstrapping on $I \times Y$.     
\end{remark}

For future reference, we state a key lemma that is needed for the proof of Proposition~\ref{prop:proposition3perturbed}(i).  Its analogue is contained in the proof of Proposition 3 in \cite{Spectrum}.

\begin{lemma}\label{lem:convergencenoblowup}
Let $I \subseteq \R$ be a closed interval (possibly $\R$).   Under the hypotheses of Proposition~\ref{prop:proposition3perturbed}, suppose we have a sequence of eigenvalues $\lambda_n \to \infty$ and a sequence of trajectories $\gamma_n: I \to W$ of $l + \pmln \eta$, such that $\gamma_n(t) \in N$ for all $t \in I$.  Then there exists a subsequence of $\gamma_n$ for which the restrictions to any compact subinterval\footnote{For closed intervals $I, I' \subseteq \R$, we write $I' \Subset I$ if $I'$ is compact and contained in the interior of $I$.} $I' \Subset I$ converge in the $C^{\infty}$ topology of $W(I' \times Y)$ to a trajectory of $l + \eta$.  
\end{lemma}

The following will allow us to define a Floer spectrum which incorporates the perturbations used in defining monopole Floer homology. 

\begin{proposition}\label{prop:verycompact}
Let $\q$ be a very tame perturbation.  Then the map $\etaq = \Pi^{\gCoul}_* \q:W \to W$ from \eqref{eq:etaq} is very compact.  
\end{proposition}
\begin{proof}
Recall from Lemma~\ref{lem:tamecoulomb} that $\etaq$ is a controlled perturbation, in the sense of  Definition~\ref{def:controlled}.

Thus, to prove Property (a) in the definition of very compactness, we simply apply part (i) of Definition~\ref{def:controlled}, with $k-1$ instead of $k$. This says that $\etaq$ extends to a continuous, functionally bounded map from $W_{k-1}(Z)$ to $W_{k-1}(Z)$.

For Property (b), without loss of generality, assume that $i_m$ is the largest of the integers $i_s$.  Let us study the regularity properties of the derivatives of $\q$.  We apply \cite[Proposition 11.4.1(ii)]{KMbook}, which states that for $p \geq 2$, the map  
\[
\D^m \hat{\q} : (x,v_1,\ldots,v_m) \mapsto \D^m_x \hat{\q}(v_1,\ldots,v_m),
\]
extends to a continuous map from $\C_p(Z) \times \ldots \times \C_p(Z) \times \C_j(Z)$ to $\C_j(Z)$, for $0 \leq j \leq p$.  For clarity, this product is written to be compatible with the formula above, so the copy of $\C_j(Z)$ in the domain consists of tangent vectors and the left-most $\C_p(Z)$ consists of the point where we are computing the higher derivative.  We let $j = k - 1 - i_m$ and 
$$p = \begin{cases}
k-1 & \text{ if } m=1,\\
k-1-\max \{i_s \mid s \neq m \} & \text{ if } m \geq 2.
\end{cases}
$$   Note that since $k - 1 \geq 4$, $k - 1 - \sum i_s \geq 0$, and $i_m = \max i_s$, we have that $k - 1 - i_s \geq 2$ for any $s \neq m$, so $p \geq 2$. Thus, 
\[
\D^m \hat{\q}: \C_p(Z) \times \ldots \times \C_p(Z) \times \C_{k-1-i_m}(Z) \to \C_{k-1-i_m}(Z)
\] 
is continuous. By pre-composing with the continuous inclusions of $\C_{k}(Z)$ and $\C_{k-1-i_s}(Z)$ into $\C_{p}(Z)$ for $s \neq m$ and post-composing with the continuous inclusion of $\C_{k-1-i_m}(Z)$ into $\C_{k-1-\sum i_s}(Z)$, we get that $\D^m \hat{\q}$ extends to a map
\[
\C_{k}(Z) \times \C_{k-1-i_1}(Z) \times \ldots  \times \C_{k-1-i_m}(Z) \to \C_{k-1- \sum i_s}(Z).
\]
The requirement (b) in the definition of very compactness now follows from Lemma~\ref{lem:igc}.
\end{proof} 

We want to apply Proposition~\ref{prop:proposition3perturbed} to $l$ and $$c_{\q} := c+\eta_{\q}.$$ Since both of the terms $c$ and $\etaq$ are very compact, so is $c_{\q}$. However, observe that very compactness is only one of the requirements needed to do finite dimensional approximation. The others are smoothness and boundedness for flow trajectories, as in the statement of Proposition~\ref{prop:proposition3perturbed}.  (The containment in $W \cap U \subset W$ in the statement guarantees smoothness.) In our setting, we know from \cite[Proposition 13.1.2 (i)]{KMbook} that for a tame perturbation, the perturbed Seiberg-Witten trajectories, and thus their global Coulomb projections, are smooth. With regard to boundedness, in Proposition~\ref{prop:proposition3perturbed}, we want to take $U$ to be the open ball of some radius $R \gg 0$ in $W_{k}$, and $N$ to be the closed ball of radius $2R$. This is exactly what was done for the unperturbed Seiberg-Witten trajectories in \cite{Spectrum}. It follows from Proposition 1 of that paper that Seiberg-Witten trajectories inside $\overline{B(2R)}$ live inside $B(R)$, provided $R$ was chosen large enough. The same proof works for perturbed Seiberg-Witten trajectories, since they satisfy similar compactness properties, as detailed in \cite[Section 10.7]{KMbook}.

Therefore, we are now able to define an analogue of the Floer spectrum with the finite-dimensional approximations of $l + c_\q$ instead of $l + c$.  We will denote this new spectrum by $\SWF_\q(Y,\spinc)$.

\begin{proposition}\label{prop:perturbedspectrum}
Let $\q$ be a very tame perturbation.  Then $\SWF_\q(Y,\spinc)$ and $\SWF(Y,\spinc)$ are $S^1$-equivariantly stably homotopy equivalent.  
\end{proposition}   

\begin{proof}
This is similar to the proof that $\SWF(Y, \spinc)$ is independent of the Riemannian metric on $Y$; see \cite[Section 7]{Spectrum}. The key ingredients are the fact that we can interpolate linearly between $\q$ and $0$ such that the hypotheses of Proposition~\ref{prop:proposition3perturbed} are satisfied, and the fact that the homotopy type of the Conley index is invariant under perturbations.
\end{proof}

\begin{remark}
It is worth comparing very compactness with the condition of being controlled, as in Definition~\ref{def:controlled}. If $\q$ is very tame, both of these conditions are satisfied by $\etaq$; see Lemma~\ref{lem:tamecoulomb} and Proposition~\ref{prop:verycompact}. There is some overlap between the two conditions: for example, part (a) in the definition of very compactness is implied by part (i) in the definition of a controlled perturbation. However, neither condition is stronger than the other: part (b) in the definition of very compactness has no analog in the controlled condition; and part (ii) in the controlled condition has no analogue in very compactness due to the constraints on functional boundedness. 

When working with the perturbation $\etaq$, we will need to use both very compactness and the controlled condition. Very compactness is needed to do finite dimensional approximation, and the controlled compactness is necessary to study the properties of stationary points and trajectories in these approximations. 
\end{remark}

From now on, we always assume that our perturbation $\q$ is both very tame (Definition~\ref{def:verytame}) and admissible (Definition~\ref{def:admi}).  For the existence of such $\q$, see Sections~\ref{sec:verytame} and \ref{sec:AdmPer}.

\section[Strategy for the proof]{Strategy for the proof of Theorem~\ref{thm:Main}} \label{sec:strategy}
In order to relate monopole Floer homology to the spectrum, it suffices to instead work with $\SWF_\q(Y,\spinc)$ by Proposition~\ref{prop:perturbedspectrum}.  This latter invariant is clearly closer to monopole Floer homology due to the presence of the perturbation.  However, we still need to relate the vector field $l + c_\q$ used for monopole Floer homology to $l + \pml c_\q$ on $\vml$ .  We will consider a vector field on $W_k$ defined by taking finite-dimensional approximations of the non-linear part of $\Xq$:
\begin{equation}
\label{eq:Xqml}
\Xqmlgc := l + \pml c_\q = l + c + \etaqml,  
\end{equation}
where $\etaqml := \pml c_\q - c.$ We will see that $\etaqml$ is very compact. 
Further, $\Xqmlgc$ induces a vector field $\Xqmlgcsigma$ on the blow-up $W^\sigma$ and thus a vector field $\Xqmlagcsigma$ on the quotient $W^\sigma/S^1$.  

At this point we give an outline of how the proof of Theorem~\ref{thm:Main} is going to go.  

For $N > 0$, let $\Crit_{[-N, N]}$ be the set of stationary points of $\Xqagcsigma$ of grading in $[-N, N]$.  Their $S^1$-orbits form sets of stationary points of $\Xqgcsigma$ on $W^{\sigma}$. Let $\Orb_{[-N, N]}$ denote the union of these orbits. We fix $N$ sufficiently large so that the projection of $\Orb_{[-N, N]}$ to the blow-down $W$ contains all the stationary points of $\Xqgc$; in particular, $\Crit_{[-N, N]}$ should contain all the irreducibles. Further, we assume that no reducible stationary point that is boundary-stable has grading less than $-N$. This ensures that the truncated chain complex $\cmto_{\leq N}(Y, \spinc, \q)$ is generated by $\Crit_{[-N, N]}$. Note that the homology of $\cmto_{\leq N}(Y, \spinc, \q)$ agrees with $\hmto(Y, \spinc, \q)$ in degrees $\leq N-1$.

Let 
\begin{equation}
\label{eq:Nclosed}
\N = \{ x\in W^{\sigma}_k \mid d_{L^2_k}(x, \Orb_{[-N, N]}) \leq 2\delta \},
\end{equation}
where $d_{L^2_k}$ denotes $L^2_k$ distance, and $\delta > 0$ is chosen sufficiently small such that the only stationary points of $\Xqgcsigma$ that are contained in $\N$ are those in $\Orb_{[-N, N]}$.  Similarly, let 
\begin{equation}
\label{eq:Uopen}
\U = \{ x\in W^{\sigma}_k \mid d_{L^2_k}(x, \Orb_{[-N, N]}) < \delta \} \subset \N.
\end{equation}
Thus, $\N/S^1$ and $\U/S^1$ are closed, resp. open, neighborhoods of $\Crit_{[-N,N]}$ in $W^\sigma_k/S^1$.  

Granted this, we will construct a chain complex $\check{C}^{\lambda}$ determined by $\Xqmlagcsigma$, and which will be identified with $\cmto_{\leq N}(Y,\spinc,\q)$.  The chain groups of $\check{C}^{\lambda}$ will be generated by the stationary points of $\Xqmlagcsigma$ that live in $\N/S^1$ (and hence in $\U/S^1$); in particular, this includes all irreducibles $[(a, s, \phi)]$ such that $(a, s\phi) \in \overline{B(2R)}.$  We will see that these stationary points will necessarily be contained in the finite dimensional approximation $(\vml)^\sigma/S^1$.  The differential on $\check{C}^{\lambda}$ will be defined analogously to monopole Floer homology (only counting trajectories that are contained entirely in $(\vml)^\sigma/S^1$).  That this will actually be a chain complex will come from a Morse-Smale stability condition---since we have non-degeneracy of the stationary points and regularity of the moduli spaces for $\Xqagcsigma$, we will show that this holds for $\Xqmlagcsigma$ as well, provided that $\lambda$ is sufficiently large.  Using the inverse function theorem, we will find a correspondence between the stationary points and isolated trajectories of $\Xqmlagcsigma$, on the one hand, and those of $\Xqagcsigma$ on the other.  This will give an explicit identification between $\check{C}^\lambda$ and $\cmto_{\leq N}(Y,\spinc,\q)$.  Here we are using the work of Chapter~\ref{sec:coulombgauge}, where we rephrased $\cmto$ in Coulomb gauge.  By the setup, we will be able to relate the respective orientations of the moduli spaces, gradings, and $U$-actions as well.  

However, $\check{C}^\lambda$ can also be identified with a truncation of the Morse complex (for manifolds with $S^1$ actions) for $B(2R) \cap \vml$ as in Section~\ref{sec:combinedMorse}.  It follows from Equation~\eqref{eq:EquivConleyMorse} that the homology of this Morse complex is isomorphic to $\tH^{S^1}_{\leq M}(\SWF_{\q}(Y,\spinc))$, for some $M > 0$. We can assume that $M > N$, and we get that 
$$\hmto_{\leq N-1}(Y,\spinc,\q) \cong \tH^{S^1}_{\leq N-1}(\SWF_{\q}(Y,\spinc)).$$  
By letting $N$ tend to infinity and applying Proposition~\ref{prop:perturbedspectrum}, we obtain the desired isomorphism in Theorem~\ref{thm:Main}.  

With the above strategy in mind, it will suffice to do most of our analysis (i.e. compactness and non-degeneracy of stationary points and trajectories) in $W^\sigma$ rather than in the quotient $W^\sigma/S^1$.  We are able to do so due to the compactness of the residual gauge group $S^1$.    

\section{Finite-dimensional approximations of perturbations}            
In order to carry out the strategy mentioned above,  we need to understand the properties of the perturbation 
\[
\etaqml= \pml c_\q  - c = (\pml c  - c) + \pml \eta_{\q} , 
\]
and in particular of the term $\pml \eta_q $.  For notation, given a sequence $\lambda_n \to \infty$, we write $\pi_n$ for $p^{\lambda_n}$.  Recall that $\Theta_j$ is a constant such that $\| \pml x \|_{L^2_j} \leq \Theta_j \| x \|_{L^2_j}$ for all $x \in W_j$ and we take $\Theta_0 = 1$.

\begin{fact}\label{fact:pml}

$(a)$ For all $\lambda$, the smoothed projection $\pml$ extends to a continuous, linear map $\pml: W_{j} \to W_{j}$ such that $\| \pml \| \leq \Theta_j$ for all $j$.  (Here, $\| \pml\|$ denotes the norm of $\pml$ as an operator on $W_j$.)  

$(b)$ If $\lambda_n \to \infty$, then $\pi_n = p^{\lambda_n} \to 1$ in the strong operator topology on $W_{j}$ for all $j$.  
\end{fact}

\begin{lemma}\label{lem:finitecoulombtame}
If $\q$ is a very tame perturbation, then $\pml \etaq $ is a controlled Coulomb perturbation for all $ \lambda$.  
\end{lemma}
\begin{proof}
Lemma~\ref{lem:tamecoulomb} says that $\etaq$ is controlled. The claim now follows from Fact~\ref{fact:pml}(a).  
\end{proof}

\begin{lemma}
\label{lem:finitecoulombvc}
The maps $c_\q$, $\pml \etaq $, $\pml c_\q $ and $\etaqml$ are all very compact.  
\end{lemma}

\begin{proof}
This follows from the very compactness of $c$, together with the very compactness of $\etaq$ (cf. Proposition~\ref{prop:verycompact}), and Fact~\ref{fact:pml}(a).
\end{proof}

We also need the following result about $c_{\q}$, which is not subsumed in very compactness.
\begin{lemma}
\label{lem:Dcq}
Let $Z = I \times Y$ for compact $I$.  For $k \geq 3$ and $0 \leq j \leq k$, the map
$$ \D c_{\q}: W_k(Z) \to \Hom(W_j(Z), W_j(Z))$$
is in $C^\infty_{\fb}$.  Therefore, extending by four-dimensional gauge, we have that 
$$ \D c_{\q}: \C^{\gCoul}_k(Z) \to \Hom(\T^{\gCoul}_j(Z), \V^{\gCoul}_j(Z))$$
is in $C^\infty_{\fb}$.  
\end{lemma}

\begin{proof}
We write $c_\q = c + \eta_\q$.  Since $\eta_\q$ is a controlled Coulomb perturbation by Lemma~\ref{lem:tamecoulomb}, we have that $\D \eta_{\q} \in C^\infty_{\fb}(W_k(Z), \Hom(W_j(Z), W_j(Z)))$.  Therefore, it suffices to show that $\D c$ is in $C^\infty_{\fb}$ as well.  Write $x = (a,\phi)$.  

By \eqref{eq:cmap}, we see that $c$ is the composition of $\Pi^{\gCoul}_*$ with the vector field 
\begin{equation}\label{eq:quadraticvf}
X: W_k (Z)\to \C_k(Z), \ (a,\phi) \mapsto (\tau(\phi, \phi), \rho(a) \cdot \phi).  
\end{equation}
We can explicitly compute 
\begin{align*}
(\D_{(a,\phi)} X)(b, \psi) &= (\tau(\phi, \psi) + \tau(\psi, \phi), \rho(b) \cdot \phi + \rho(a) \cdot \psi) \\
(\D^2_{(a,\phi)} X)((\alpha, \zeta),(b,\psi)) &= (\tau(\zeta, \psi) + \tau(\psi, \zeta), \rho(b) \cdot \zeta + \rho(\alpha) \cdot \psi) \\
(\D^3_{(a,\phi)} X) &\equiv 0.  
\end{align*}
From this, it is straightforward to apply the Sobolev multiplication $L^2_k \times L^2_j \to L^2_j$ to see that $X$ satisfies the conditions of Lemma~\ref{lem:igc-fb} with $n = \infty$.  Therefore, we see that $\D c$ is in $C^\infty_{\fb}$.
\end{proof}

The vector field $\Xqmlgc = \Xgc + \etaqml$ from \eqref{eq:Xqml} induces a vector field $\Xqmlgcsigma$ on the blow-up. Similar to \eqref{eq:xqgcsigmaalternate}, this is given by the formula 
\begin{align}
\nonumber \Xqmlgcsigma(a,s,\phi) & =  ({\mathcal{X}}_{\qml}^0(a,s\phi), \Lambda_{\qml}(a,s,\phi)s, D\phi + \widetilde{\pml c_\q }^1(a,s,\phi) - \Lambda_{\qml}(a,s,\phi)\phi) \\
\label{eq:Xqmlgcsigmaformula} & = \Bigl (*da + (\pml c_\q )^0(a,s\phi), \Lambda_{\q^\lambda}(a,s,\phi)s,  \\
\nonumber & \qquad \qquad \qquad D\phi + \int^1_0 \D_{(a,st\phi)} (\pml c_\q )^1 (0,\phi)dt - \Lambda_{\q^\lambda}(a,s,\phi)\phi \Bigr),  
\end{align}
where 
\begin{equation}
\label{eq:lambda-spinorial}
\Lambda_{\qml} = \Re \langle \phi, (\widetilde{\Xqmlgc})^1(a,s,\phi) \rangle_{L^2}.
\end{equation}
For $(a,s,\phi) \in W^\sigma$, we call $\Lambda_{\qml}(a,s,\phi)$ the {\em $\lambda$-spinorial energy}.  As in the case of the spinorial energy, the $\lambda$-spinorial energy of an irreducible stationary point of $\Xqmlgcsigma$ is zero.  Analogous to the definition of $\Fqgctau$ in \eqref{eq:Fqgctau}, we can use $\G^{\gCoul}(Z)$-equivariance to extend the equations defining the flow in \eqref{eq:Xqmlgcsigmaformula} to define a section 
\begin{equation}
\label{eq:Fqmlgctau}
 \Fqmlgctau: \tC^{\gCoul, \tau}(Z) \to \V^{\gCoul, \tau}(Z).
\end{equation}

In the rest of this section we will describe several results about the dynamics of $\Xqmlgcsigma$. These will be put to use, for instance, in Chapters~\ref{sec:criticalpoints}, \ref{sec:gradings}, \ref{sec:trajectories1} and \ref{sec:trajectories2}, where we will relate stationary points and trajectories of $\Xqmlgcsigma$ to those of $\Xqgcsigma$.

Note that being a reducible stationary point $(a,0)$ of $\Xqmlgc$ is equivalent to solving   
\[
*da + (\pml \eta_\q )^0(a,0) = 0, \quad (\pml \eta_\q )^1(a,0) = 0, \]
since $\pml c  (a,0) = 0$.  

A reducible stationary point $(a,0,\phi)$ of $\Xqmlgcsigma$ is determined by a reducible stationary point $(a,0)$ of $\Xqmlgc$ and an $L^2$-unit length eigenvector $\phi$ of the linear operator $D_{\qml,a}$, where
\begin{equation}\label{eq:Dqmlaphi}
D_{\qml,a}(\phi) = \D_{(a,0)} (\Xqmlgc)^1(0,\phi) = D \phi + (\pml)^1(\D_{(a,0)} c_\q(0,\phi)).   
\end{equation}
If $(a,0,\phi)$ is a reducible stationary point of $\Xqmlgcsigma$, then the $\lambda$-spinorial energy is simply the eigenvalue of $D_{\qml,a}$ corresponding to $\phi$.  

We now establish some important analytic properties of the perturbed equations.  
\begin{lemma}
\label{lem:fam}
Fix $k \geq 3$.  
\begin{enumerate}[(a)]
\item \label{fam:a} For a bounded subset $K \subset W_k$ and a bounded subset $J \subset W_j$ (with $0 \leq j  \leq k$), the set 
$$\{\pml (\D_{(a,\phi)} c_\q(b,\psi)) \mid (a,\phi) \in K, (b,\psi) \in J, \lambda \geq 0\}$$ is bounded in $L^2_j$.  These uniform bounds also hold if we include $\lambda = \infty$.

\item \label{fam:b} For a bounded subset $K^\sigma \subset W_k^{\sigma} \subset (\ker d^*)_k \oplus \R \oplus L^2_k(Y; \Spin)$ and a bounded subset $J^\sigma \subset \T_j^{\gCoul, \sigma}|_{K^\sigma}$ (with $0 \leq j  \leq k$), the set 
$$\{ \D^\sigma_{x} (\pml c_\q)^{\sigma}(v) \mid x \in K^\sigma,  \ (x,v) \in J^\sigma, \lambda \geq 0\}$$ is bounded in $L^2_j$.
These uniform bounds also hold if we include $\lambda = \infty$.  

\item\label{fam:c} If $x_n \to x$ in $W^\sigma_k$ and $\lambda_n \to \lambda$ (possibly $\infty$), then $\widetilde{p^{\lambda_n} c_\q }^1(x_n) \to \widetilde{p^{\lambda} c_\q }^1(x)$ in $L^2_k(Y;\Spin)$.  Furthermore, $\widetilde{\Xqmlngc}^1(x_n) \to \widetilde{\Xqmlgc}^1(x)$ in $L^2_{k-1}(Y;\Spin)$.  In this case, we also have $\Lambda_{\q^{\lambda_n}}(x_n) \to \Lambda_{\q^{\lambda}}(x)$ and thus $\Xqmlngcsigma(x_n) \to \Xqmlgcsigma(x)$ in $\T^{\gCoul,\sigma}_{k-1}$.  

\item\label{fam:d} Same as \eqref{fam:c}, but for four-dimensional configurations on compact cylinders.
\end{enumerate}
\end{lemma}

\begin{proof}
\eqref{fam:a} We have $c_\q = c + \eta_\q$.  In the proof of Lemma~\ref{lem:Dcq}, it was shown that $\D c_\q \in C^0_{\fb}(W_k,\Hom(W_j, W_j))$.  
We thus have $L^2_j$ bounds on the set 
$$ \{\D_{(a,\phi)} c_\q(b,\psi) \mid (a,\phi) \in K, (b,\psi) \in J\}.$$
Fact~\ref{fact:pml}(a) now gives uniform $L^2_j$ bounds on 
$$ \{\pml (\D_{(a,\phi)} c_\q(b,\psi)) \mid (a,\phi) \in K, (b,\psi) \in J, \lambda \geq 0\}.$$  

\noindent \eqref{fam:b} Recall that $\D^\sigma (\pml c_\q)^\sigma$ is computed in two steps.  First, we differentiate $(\pml c_\q)^\sigma$, thought of as a map to $L^2(Y; i \rr) \times \rr \times L^2(Y;\Spin)$, as opposed to a section of $\T^{\gCoul}_{k}$.  Then, we apply $L^2$ orthogonal projection to $\T^{\gCoul,\sigma}_k$.  We also have the analogous construction when extending to $L^2_j$ completions. Since $\Pi_{\T^{\gCoul,\sigma}_{j, (a,s,\phi)}}(b,r,\psi) = (b,r,\psi - \Re \langle \psi, \phi \rangle_{L^2} \phi)$, we see that $\Pi_{\T^{\gCoul,\sigma}_j}$ is $L^2_j$ bounded in terms of the $L^2_k$ norm of $(a,s,\phi)$ and $L^2_j$ norm of $(b,r,\psi)$.  Therefore, we focus on the boundedness of the derivative of $(\pml c_\q)^\sigma$ as a map to the extended space $L^2_k(Y; i \rr) \times \rr \times L^2_k(Y;\Spin)$.  From the definitions we compute   
\begin{align}\label{eq:Dhat-finite-compute}
\D_{(a,s,\phi)} (\pml c_\q)^\sigma (b,r,\psi) &= \Big( \D_{(a,s\phi)} (\pml c_\q)^0(b,r\phi + s\psi), \\
\nonumber \Re \langle \widetilde{\pml c_\q}^1(a,s,\phi), & \phi \rangle_{L^2} r + \Re \langle \widetilde{\pml c_\q}^1(a,s,\phi), \psi \rangle_{L^2} s + \Re \langle \D_{(a,s,\phi)} \widetilde{\pml c_\q}^1(b,r,\psi), \phi \rangle_{L^2} s, \\
\nonumber \Re \langle \D_{(a,s,\phi)} \widetilde{\pml c_\q}^1&(b,r,\psi), \phi \rangle_{L^2} \phi + \Re \langle \widetilde{\pml c_\q}^1(a,s,\phi), \psi \rangle_{L^2} \phi + \Re \langle \widetilde{\pml c_\q}^1(a,s,\phi), \phi \rangle_{L^2} \psi\Big),
\end{align} 
where $\widetilde{\pml c_\q}^1(a,s,\phi) = \int^1_0 \D_{(a,sr\phi)} (\pml c_\q)^1(0,\phi) dr$.   It follows from Lemma~\ref{lem:Dcq} and Fact~\ref{fact:pml}(a) that 
$$\widetilde{\pml c_\q}^1(a,s,\phi), \ \D_{(a,s,\phi)} \widetilde{\pml c_\q}^1(b,r,\psi)$$
are $L^2_j$ bounded in terms of the $L^2_k$ bounds on $(a,s,\phi)$ and the $L^2_j$ bounds on $(b,r,\psi)$.  Sobolev multiplication again provides the desired bounds on \eqref{eq:Dhat-finite-compute}. \\

\noindent \eqref{fam:c} A smooth section of $W_k$  to the $L^2_j$ completion of $TW_k$ with $j \leq k$  induces a smooth vector field on the blow-up with the same regularity.  By Proposition~\ref{prop:verycompact} (which applies for any Sobolev coefficient at least three), we see that if $x_n \to x$ in $W^\sigma_k$, we have ${\widetilde{c_\q}}^1(x_n) \to {\widetilde{c_\q}}^1(x)$ in $L^2_k$. Applying Fact~\ref{fact:pml} gives $\widetilde{p^{\lambda_n} c_\q }(x_n) \to \widetilde{p^{\lambda} c_\q }(x)$ in $L^2_k$ as well.  Since $\Xqmlgc = l + \pml c_\q $ and $l : W_k \to W_{k-1}$ are smooth for $k \geq 2$, we obtain the desired convergence results by similar arguments.  \\

\noindent \eqref{fam:d} The argument is similar to the one for \eqref{fam:c}.
\end{proof}

It turns out that we can give more explicit bounds on $(\pml c_\q )^1$:  
\begin{lemma}\label{lem:linearbounds}
There exists a constant $C_{k, j,M}$, depending only on $0 \leq j \leq k$ and $M$, such that if $x = (a,\phi) \in W_k$ satisfies $\| x \|_{L^2_k} \leq M$ then $\| (\pml c_\q )^1(x) \|_{L^2_j} \leq C_{k, j,M} \| \phi \|_{L^2_j}$ for any $\lambda$.   
\end{lemma}

Note that $ \| \phi \|_{L^2_j} \leq \| x \|_{L^2_k} \leq M$, so Lemma~\ref{lem:linearbounds} implies that $\| (\pml c_\q )^1(x) \|_{L^2_j}$ is bounded above uniformly by a constant $C'_{k, j, M}$. We already knew this, owing to the very compactness of $\pml c_\q $. The new content of Lemma~\ref{lem:linearbounds} is that it gives a stronger (linear) bound when $ \| \phi \|_{L^2_j}$ is small.

\begin{proof}[Proof of Lemma~\ref{lem:linearbounds}]
Since $\| \pml \| \leq \Theta_j$, it suffices to find a constant $C$ such that $\| c^1_\q(x) \|_{L^2_j} \leq C \| \phi \|_{L^2_j}$.  First, we bound $c^1$.  
Recall that 
\[
c^1 (a,\phi) = \rho(a) \cdot \phi - Gd^* \tau(\phi,\phi) \phi.   
\]
Since $a$ is $L^2_k$-bounded, Sobolev multiplication implies 
\[
\|\rho(a) \cdot \phi\|_{L^2_j} \leq C_1 \| \phi \|_{L^2_j}.  
\]
Because $\phi$ is $L^2_k$-bounded, so is $Gd^* \tau(\phi,\phi)$.  Again, by Sobolev multiplication, we obtain 
\[
\| Gd^* \tau(\phi,\phi) \phi \|_{L^2_j} \leq C_2 \| \phi \|_{L^2_j}.  
\]
Therefore, it remains to find bounds on $\etaq^1(a,\phi)$.  We have 
\[
\etaq^1(a,\phi) = \q^1(a,\phi) - Gd^* \q^0(a,\phi) \phi.  
\]
Since $\q$ is a very tame perturbation, we obtain $L^2_k$ bounds on $Gd^* \q^0(a,\phi)$ and thus linear bounds on the second term.  Therefore, it remains to give linear bounds on $\|\q^1(a,\phi)\|_{L^2_j}$ in terms of $\| \phi \|_{L^2_j}$.  We may write $\q^1(a,\phi) = \int^1_0 \D_{(a,r\phi)} \q^1(0,\phi) dr$, because the $S^1$-equivariance of $\q^1$ implies $\q^1(a,0) = 0$.  Since the $(a,r\phi)$ are $L^2_k$ bounded and $\D_{(a,r\phi)}\q^1$ is linear in $\phi$, the very tameness of $\q$ implies linear $L^2_j$-bounds on $\D_{(a,r\phi)} \q^1(0,\phi)$, depending only on $\|\phi\|_{L^2_j}$.  Integrating gives linear bounds on $\| \q^1(a,\phi)\|_{L^2_j}$ in terms of $\| \phi \|_{L^2_j}$.  This now completes the proof.     
\end{proof}

In order to make further connection with the finite-dimensional setting of the spectrum, we must ensure that the trajectories of $\Xqmlgcsigma$ are actually contained in the finite-dimensional blow-up of $\vml$, assuming control over the $\lambda$-spinorial energy.  We begin with the non-blown-up case.  

\begin{lemma}\label{lem:trajectoriesinvml}
If $\gamma(t)$ is a trajectory of $\Xqmlgc$ contained in $B(2R)$, then $\gamma(t)$ is contained in $\vml$.  In particular, $\gamma(t)$ is smooth for each $t$.  
\end{lemma}
\begin{proof}
Let $\gamma(t) = (a(t),\phi(t))$.  Fix $\kappa$ with $|\kappa| \geq \lambda$ and let $\gamma_\kappa(t)$ be the $L^2$ projection of $\gamma(t)$ to the $\kappa$-eigenspace of $l$.  It suffices to show that $\gamma_\kappa(t) = 0$.  Since $\gamma(t)$ is a trajectory of $\Xqmlgctau$, we have 
\[ -\frac{d}{dt}\gamma(t) = l(\gamma(t)) + (\pml c_\q )(\gamma(t)). \]
Since $|\kappa| \geq \lambda$, we see $-\frac{d}{dt}{\gamma_\kappa}(t) = l(\gamma_\kappa(t)) = \kappa \gamma_\kappa(t)$.  Thus, $\gamma_\kappa(t) = e^{-\kappa t} \gamma_\kappa(0)$.  Suppose that  $\gamma_\kappa(t) \not\equiv 0$.  Therefore, $\gamma_\kappa(t)$, and thus $\gamma(t)$, are unbounded in the $L^2_k$-norm on $Y$ as $|t|$ increases.  This contradicts $\gamma(t) \in B(2R)$ for all $t$.  Therefore, $\gamma_\kappa(t) \equiv 0$ and $\gamma(t)$ is contained in $\vml$.
\end{proof}

We seek an analogous result to Lemma~\ref{lem:trajectoriesinvml}, but in the blow-up. Apart from the $2R$ bound on the projections of trajectories in the blow-down, we will also need to assume a bound $\omega$ on the absolute values of spinorial energies $\Lambda_{\q}$. With $N$, $\N$ and $\U$ as in Section~\ref{sec:strategy}, note that $\Lambda_{\q}$ is bounded on the compact set $\Orb_{[-N, N]}$, and hence it is bounded on its neighborhood $\N$.  The boundedness can be explained as follows.  Since $\q$ is controlled and our choice of $k$ guarantees that $k - 1 \geq 2$, we have that $\Lambda_\q$ is continuous on $W^\sigma_{k-1}$.  Therefore, $\Lambda_\q$ is bounded on $\N$ since it is a compact subset of $W^\sigma_{k-1}$.  We shall choose $\omega > 0$ so that any point $ x \in \N$ has  
\begin{equation}
\label{eq:omega}
 |\Lambda_{\q}(x) | < \omega.
 \end{equation}

\begin{lemma}\label{lem:blowuptrajectoriesinvml}
For sufficiently large $\lambda$, the following is true.  If $\gamma(t) = (a(t),s(t),\phi(t))$ is a trajectory of $\Xqmlgcsigma$ contained in $B(2R)^\sigma$ and the $\lambda$-spinorial energy of $\gamma(t)$ is in $[-\omega,\omega]$ for all $t$, then $(a(t),\phi(t)) \in \vml$ for all $t \in \mathbb{R}$.  In particular, $(a(t),\phi(t))$ is smooth for all $t \in \mathbb{R}$.  
\end{lemma} 
\begin{proof}
Fix $\lambda \gg \omega$.  Fix $\kappa$ with $|\kappa| \geq \lambda$ and let $(a_\kappa(t), \phi_\kappa(t))$ be the projection of $(a(t),\phi(t))$ to the $\kappa$-eigenspace of $l$.  Since $(a(t),s(t)\phi(t))$ is contained in $\vml$ by Lemma~\ref{lem:trajectoriesinvml}, we have that $a_\kappa(t) \equiv 0$.  It suffices to show that $\phi_\kappa(t) \equiv 0$.  Let $\nu(t) = \Lambda_{\q^\lambda}(a(t),s(t),\phi(t))$.  Since $\gamma(t)$ is a trajectory of $\Xqmlgcsigma$, we have by \eqref{eq:Xqmlgcsigmaformula} that 
\begin{align*}
-\frac{d}{dt}{\phi}(t) &= D\phi(t) + \widetilde{\pml c_\q }^1(a(t),s(t),\phi(t)) - \nu(t) \phi(t) \\
&= D\phi(t) + \int^1_0 \D_{(a(t),rs(t)\phi(t))} (\pml c_\q )^1(0,\phi(t)) dr - \nu(t) \phi(t).
\end{align*}
Since $\pml c_\q (x) \in W^{\lambda}$ for all $x \in W_k$, we have that the projection of $\widetilde{\pml c_\q }^1(a(t),s(t),\phi(t))$ to the $\kappa$-eigenspace of $D$ is trivial.  Therefore, we have 
\[ -\frac{d}{dt}{\phi_\kappa}(t) = D\phi_\kappa(t) -  \nu(t) \phi_\kappa(t) = (\kappa - \nu(t)) \phi_\kappa(t). \] 
By assumption, $\kappa - \nu(t) \gg 0$ for all $t$.  If $\phi_\kappa(t) \not \equiv 0$, then $\|\phi_\kappa(t)\|_{L^2(Y)}$ is unbounded.  However, $\| \phi_\kappa(t) \|_{L^2(Y)} \leq \| \phi(t) \|_{L^2(Y)} \equiv 1$, which is thus a contradiction.  
\end{proof}


For future reference, we will write $L^2_k(Y)$ (resp. $L^2_k(Z)$) to refer to the $L^2_k$-norm of any relevant object, e.g. a connection, spinor, etc.,  on $Y$ (resp. $Z$).  If $Z = I \times Y$, we note that $L^2_{j+1}(Z)$-convergence implies pointwise $L^2_j(Y)$-convergence, uniform on compact sets of $I$.  We will also write $x$, $a$, $s$, $\phi$ (without $t$) if we want to treat this object as a section of a bundle over $I \times Y$ instead of as a path of sections over $Y$.

\chapter{Stationary points}\label{sec:criticalpoints}
Throughout Chapters~\ref{sec:criticalpoints}--\ref{sec:trajectories2} we fix the following:
\begin{itemize}
\item a very tame, admissible perturbation $\q$; this means that $\q$ is very tame in the sense of Definition~\ref{def:verytame}, the stationary points of $\Xqgcsigma$ are non-degenerate, and the associated moduli spaces of trajectories are regular; we will impose additional conditions on $\q$ in Proposition~\ref{prop:ND2} and in Proposition~\ref{prop:MorseSmales};

\item a Sobolev index $k \geq 5$;

\item a bound $R > 0$ such that all the stationary points and finite type trajectories of $\Xqgc$ are contained in $B(2R) \subset W_k$;

\item a value $N>0$ specifying a grading range $[-N, N]$, a closed neighborhood $\N$ and an open neighborhood $\U \subset \N$ of the set of stationary points of $\Xqgcsigma$ in that grading range, as in Section~\ref{sec:strategy}; we also assume that the projection of $\N$ to the blow-down is contained in $B(2R)$ and that $N$ is chosen large enough to contain each reducible stationary point $(a,0,\phi)$ of $\Xqgcsigma$ where $\phi$ is an eigenvector of $D_{\q,a}$ with smallest positive eigenvalue;

\item a strict bound $\omega$ on the absolute values of the spinorial energies of points in $\N$, as in \eqref{eq:omega}.  
\end{itemize}

In analogy with the finite-dimensional setting, for $M \subset W$, we will use $M^\sigma$ to denote the closure in $W^\sigma$ of the preimage of $W - (\ker d^* \oplus 0)$ under the projection from $W^\sigma$ to $W$; and similarly for Sobolev completions. In particular, we can talk about $B(2R)^{\sigma}$. Note that $\N$ is a subset of $B(2R)^{\sigma}$.

Recall that it is our goal to identify some of the stationary points and trajectories of $\Xqagcsigma$ with those of $\Xqmlagcsigma$, for $\lambda$ sufficiently large.  In this section, we deal with stationary points.  First, in Section~\ref{sec:convergence} we show that the stationary points of $\Xqmlagcsigma$ contained in $\N/S^1$ are close to stationary points of $\Xqagcsigma$ in $\N/S^1$, for $\lambda$ sufficiently large.  Then, in Section~\ref{sec:StabilityPoints} an inverse function theorem argument shows that inside $\N/S^1$, the stationary points of $\Xqmlagcsigma$ are in one-to-one correspondence with those of $\Xqagcsigma$.  In Section~\ref{sec:nondegeneratestationarypoints} we prove that the nearby approximate stationary points are non-degenerate.  (This will be needed later, when we define a Morse complex for $\Xqmlagcsigma$.)  In Section~\ref{sec:stationarypointsoutsideN}, we study stationary points outside of $\N/S^1$.  We rephrase these results with gradings of stationary points in Section~\ref{sec:gradingsstationarypoints}.  

As before, if we have a sequence $\lambda_n \to \infty$, we write $\pi_n$ to denote $p^{\lambda_n}$.

\medskip

\section{Convergence}~\label{sec:convergence}  
We first point out a convergence result for stationary points in $W$ (not in the blow-up).  

\begin{lemma}\label{lem:compactnessnoblowup}
Suppose that $x_n$ is a sequence of stationary points for $\Xqmlngc$ in $B(2R)$ where $\lambda_n \to \infty$.  Then, there is a subsequence that converges in $W_{k}$ to $x$, a stationary point for $\Xqgc$.  Further, if $x_n$ are reducible, so is $x$.
\end{lemma} 
\begin{proof}
This follows from applying Lemma~\ref{lem:convergencenoblowup} to constant trajectories, noting that $c_\q$ is a very compact map by Proposition~\ref{prop:verycompact}.  If $x_n = (a_n,0) \in B(2R)$ is a sequence of reducibles, then clearly the limit must be of the form $x = (a,0)$.  
\end{proof}

We now need an analogous compactness result on the blow-up.  It turns out that very compactness of $c_\q$ is not sufficient; we will use some of the other properties of controlled Coulomb perturbations to do this.  

\begin{lemma}\label{lem:compactness}
Fix $\epsilon >0$.  There exists $b \gg 0$ such that for all $\lambda >b$ the following is true.  If $x \in \N \subset W^{\sigma}$ is a zero of $\Xqmlgcsigma$, then there exists $x' \in \N$ such that $\Xqgcsigma(x') = 0$ and $x, x'$ have $L^2_{k}$-distance at most $\epsilon$ in $L^2_k(Y; i T^*Y) \oplus \mathbb{R} \oplus L^2_k(Y;\Spin)$.
\end{lemma}
\begin{proof}
Suppose this is not true.  Then, we can find a sequence $\lambda_n \to \infty$ and corresponding zeros $x_n= (a_n,s_n,\phi_n)$ of $\Xqmlngcsigma$ in $\N$, none of which are within $L^2_k$-distance $\epsilon$ of a stationary point of $\Xqgcsigma$ in $\N$.  We will contradict this by finding a subsequence converging to such a stationary point.

Since the $x_n$ are in $\N$, which is $L^2_k$-bounded, we can extract a subsequence that converges to some $x=(a, s, \phi)$ in $L^2_{k-1}$. Further, after passing to another subsequence, we can assume the $x_n$ are either all reducible or all irreducible.  

First, suppose the $x_n$ are reducible, that is, $s_n = 0$.  By Lemma~\ref{lem:compactnessnoblowup}, the convergence $a_n \to a$ is in the stronger $L^2_k$ norm. Moreover, $a \in (\ker d^*)_k$ is such that $(a,0)$ is a stationary point of $\Xqgc$.  Recall that $\phi_n$ are eigenvectors for $D_{\q^{\lambda_n}, a_n}$ with $L^2$-norm equal to $1$. We claim that $\phi$ is an eigenvector of $D_{\q,a} = D + \D_{(a,0)} (c_\q)^1(0,\cdot)$, and that the convergence $\phi_n \to \phi$ is also in $L^2_k$.  

Let $\kappa_n$ be the associated eigenvalues for $\phi_n$, i.e., the $\lambda_n$-spinorial energies of $x_n$. Since $x_n \to x$ in $L^2_{k-1}$, by Lemma~\ref{lem:fam} (c) we have $|\Lambda_{\q}(x)| \leq \omega.$ By applying Lemma~\ref{lem:fam} (c) again, this time to the $\lambda_n$-spinorial energies of $x_n$, we see that $\Lambda_{\qmln}(x_n) \to \Lambda_{\q}(x)$. Therefore, we also have a bound $\omega' \geq \omega$ on $|\Lambda_{\qmln}(x_n)|=|\kappa_n|$. After passing to a subsequence, we can assume that the $\kappa_n$ converge to some value $\kappa$. Then, we have  
\begin{align}
\nonumber \kappa_n \phi_n &= D_{\q^{\lambda_n},a_n}(\phi_n) \\
\label{eqn:approx-ev}				&= D \phi_n + \D_{(a_n,0)} (\pi_n c_\q)^1(0,\phi_n) \\
\nonumber				&= D \phi_n + \pi_n(\D_{(a_n,0)} c_\q(0,\phi_n))^1.  
\end{align}
Now, since $(a_n,0)$ converges to $(a,0)$ in $L^2_k$ and $(0,\phi_n)$ converge to $(0,\phi)$ in $L^2_{k-1}$, the very compactness of $c_\q$ guarantees that $\D_{(a_n,0)} c_\q(0,\phi_n)$ converges to $\D_{(a,0)} c_\q(0,\phi)$ 
in $L^2_{k-1}$.  As $\pi_n \to 1$ in the strong operator topology on $W_{k-1}$, we must also have
\[ 
\pi_n \bigl( \D_{(a_n,0)} c_\q(0,\phi_n) \bigr)^1 \to \bigl( \D_{(a,0)} c_\q (0,\phi)\bigr)^1 \text{ in $L^2_{k-1}$}.
\]
On the other hand, we have 
\[
\kappa_n \phi_n \to \kappa \phi \text{ in $L^2_{k-1}$}.
\]
Finally, we have that $D \phi_n \to D \phi$ in $L^2_{k-2}$, because $\phi_n \to \phi$ in $L^2_{k-1}$.  Since the convergence of $\pi_n \D_{(a_n,0)} (c_\q)^1(0,\phi_n)$ and $\kappa_n \phi_n$ is in $L^2_{k-1}$, we must in fact have that $D \phi_n \to D \phi$ in $L^2_{k-1}$ by \eqref{eqn:approx-ev}.  Thus, $\phi_n$ converges to $\phi$ in $L^2_k$ and $\kappa_n \phi_n = D_{\q^{\lambda_n}, a_n}(\phi_n) \to D_{\q, a}(\phi) = \kappa \phi$ in $L^2_k$.  Thus, $\phi$ is an eigenvalue of $D_{\q, a}$, so $x=(a, 0, \phi)$ is a reducible stationary point of $\Xqgcsigma$. Since $\N$ is closed in the $L^2_k$ norm and the convergence $x_n \to x$ is in $L^2_k$, we get that $x \in \N$, providing the contradiction.    

We now assume that our sequence $(a_n,s_n,\phi_n)$ consists of irreducibles.  If $s_n \geq \delta > 0$ for all $n$, Lemma~\ref{lem:compactnessnoblowup} guarantees that $(a_n,s_n,\phi_n)$ will converge in $W^\sigma_k$ to an irreducible stationary point $(a,s,\phi)$ of $\Xqgcsigma$, since $\Xqmlngcsigma$ is conjugate to $\Xqmlngc$ via the blow-down.   This is again a contradiction.  

The final case is when $(a_n,s_n,\phi_n)$ is a sequence of irreducible stationary points in $B(2R)^\sigma$ with $s_n \to 0$.  By Lemma~\ref{lem:compactnessnoblowup}, in the blow-down we have $(a_n, s_n \phi_n) \to (a, 0)$ in $L^2_k$. Since $x_n \in \N$, we can find an upper bound on $\|\phi_n\|_{L^2_k}$. By  Lemma~\ref{lem:linearbounds}, we have a constant $C$ such that
\[
|s_n| \cdot \| D\phi_n \|_{L^2_k} = \|D (s_n \phi_n) \|_{L^2_k}  = \| (\pi_n c_\q)^1(a_n,s_n\phi_n)\|_{L^2_k}  \leq C |s_n|.
\]
Thus, we obtain $L^2_{k+1}$-bounds on $\phi_n$.  After passing to a subsequence, we get that $(a_n,s_n,\phi_n)$ converges in $L^2_k$ to $(a,0,\phi)$.  By Lemma~\ref{lem:fam} (c), we have that $\Xqmlngcsigma(a_n,s_n,\phi_n) \to \Xqgcsigma(a,0,\phi)$, and thus $(a,0,\phi)$ is a reducible stationary point.  Moreover, since $\Lambda_{\q^{\lambda_n}}(a_n,s_n,\phi_n)=0$, by taking the limit $n \to \infty$ we obtain $ \Lambda_\q(a,0,\phi)=0$, i.e. $\phi$ is in the kernel of $D_{\q, a}$. This is impossible, because the non-degeneracy condition for reducibles requires that $0$ is not in the spectrum of $D_{\q, a}$.
\end{proof}

\begin{remark}
\label{rem:type}
The proof of Lemma~\ref{lem:compactness} shows that if the zero $x$ of $\Xqmlgcsigma$ is reducible, then we can choose the nearby zero $x'$ of $\Xqgcsigma$ to be reducible as well.  Furthermore, since $\Lambda_{\q^\lambda}(x)$ is close to $\Lambda_{\q}(x')$, we see that $x'$ can be chosen to be stable (resp. unstable) when $x$ is stable (resp. unstable).  
\end{remark}

\begin{corollary}
\label{cor:NisU}
For $\lambda \gg 0$, if $x$ is a zero of $\Xqmlgcsigma$ in $\N$, then $x$ is actually in $\U$ and $|\Lambda_{\q^\lambda}(x)| < \omega.$
\end{corollary}

\begin{proof}
We get that $x \in \U$ by choosing $\epsilon$ sufficiently small in Lemma~\ref{lem:compactness}. If we had a sequence $\lambda_n \to \infty$ and corresponding $x_n \in \N$ with $\Xqmlngcsigma(x_n)=0$ and $|\Lambda_{\qmln}(x_n)| \geq \omega$, then after extracting a subsequence, and making use of Lemma~\ref{lem:compactness}, in the limit we would get a zero of $\Xqgcsigma$ inside $\N$ with $|\Lambda_{\q}|\geq \omega$. This is a contradiction.
\end{proof}

\begin{corollary}
\label{cor:IsFinite}
For $\lambda \gg 0$, all the stationary points of $\Xqmlgcsigma$ in $\N$ live inside the finite-dimensional blow-up $(W^{\lambda})^{\sigma}$.
\end{corollary}

\begin{proof}
This follows from Corollary~\ref{cor:NisU} and Lemma~\ref{lem:blowuptrajectoriesinvml}.
\end{proof}

Observe that the results of this subsection also apply to the stationary points of $\Xqmlagcsigma$ and $\Xqagcsigma$, in the quotient $W^\sigma/S^1$.

\section{Stability} \label{sec:StabilityPoints}
In the previous section we showed that, for $\lambda$ large, the zeros of $\Xqmlagcsigma$ in $\N$ are close to the zeros of $\Xqagcsigma$. Our next goal is to show that near each stationary point of $\Xqagcsigma$ there is exactly one stationary point of $\Xqmlagcsigma$, again for $\lambda$ large. The proof will be an application of the implicit function theorem, using the fact that we have chosen the perturbation $\q$ so that the zeros of $\Xqagcsigma$ are nondegenerate. 

Recall that $p^{\lambda}$ is (roughly) the orthogonal projection to $W^{\lambda}$, modified so that it becomes a smooth function of $\lambda$ for $\lambda \in (0, \infty).$ In order to be able to apply the implicit function theorem at $\lambda = \infty$, we need to find a suitable identification of the interval $(0, \infty]$ with $[0,1)$, so that after this identification we get differentiability at zero.

For the present subsection, let us index (with multiplicity) the eigenvalues of $l$  by $(\lambda_n)_{n \geq 0}$ by the condition that $|\lambda_n| \leq |\lambda_{n+1}|$ for all $n$; note that $\lim_{n \to \infty} |\lambda_n |= \infty.$  Recall that we write 
$$\Xqgc = l + c_\q.$$ 

Pick a homeomorphism $f: (0, \infty] \to [0,1) $ with the following properties:
\begin{itemize}
\item The restriction of $f$ to $(0,\infty)$ is a strictly decreasing diffeomorphism onto $(0, 1)$;
\item $\lim_{n \to \infty} |\lambda_n|^2 f(|\lambda_{n+1}|) = \infty.$ 
\end{itemize}

The second property means that $f$ does not decrease too fast near infinity. We can achieve it, for example, by requiring that $f(|\lambda_{n+1}|) = \frac{1}{|\lambda_n|} + \frac{1}{n}$ for large $n$.

\begin{lemma}
\label{lem:homeo}
The map 
$$h: W_{k} \times (-1,1) \to W_{k-1}, \ \ \ h(x, r) = x- p^{f^{-1}(|r|)}(x)$$
is continuously differentiable, with $\D h_{(x, 0)}(0,1) =0$ for all $x$.
\end{lemma}

\begin{proof}
Continuous differentiability away from $r = 0$ is standard, taking into account that $p^{\lambda}$ are smoothed projections as in \eqref{eq:pl}. 

Since $h(x,0) = 0$, it suffices to show that
$$ \lim_{r \to 0}\frac{ h(x,r)}{r} =0,$$
or, equivalently
\begin{equation} 
\label{eq:qlim}
\lim_{\lambda \to \infty} \frac{\|x-p^{\lambda}(x)\|_{L^2_{k-1}}}{f(\lambda)} = 0.
\end{equation}

The eigenspaces of $l$ are orthogonal in $L^2$. Pick an $L^2$-orthonormal sequence of eigenvectors $w_n$ for $\lambda_n$.  Let us write
$$ x = \sum_n x_n w_n,$$
for a sequence of real numbers $(x_n)_{n \geq 0}.$ We have $\|x\|^2_{L^2} = \sum_n |x_n|^2.$

Note that 
$$ l^{k-1} (x) = \sum_n \lambda_n^{k-1} x_n w_n.$$
Since the $L^2_{k-1}$ norm is equivalent to the one defined using $l$ as a differential operator instead of $\nabla$, we can write
$$ \| x\|^2_{L^2_{k-1}} \approx \sum_n \sum_{j=0}^{k-1} |\lambda_n|^{2j} |x_n|^2 \approx \sum_n |\lambda_n|^{2k-2} |x_n|^2.$$

In fact, we have $x \in W_{k}$, so 
$$\|x\|_{L^2_k}^2 \approx \sum_n |\lambda_n|^{2k} |x_n|^2 < \infty.$$

We have
$$ \|x- p^{|\lambda_n|}(x) \|_{L^2_{k-1}} \approx \sum_{m > n} |\lambda_m|^{2k-2} |x_m|^2 < \frac{1}{|\lambda_n|^2} \sum_{m > n} |\lambda_m|^{2k} |x_m|^2 \leq C \frac{\|x\|_{L^2_k}^2}{|\lambda_n|^2}$$
for some constant $C$ independent of $x$ and $\lambda_n$.  

Recall that we chose $f$ so that $\lim_{n \to \infty} |\lambda_n|^2 f(|\lambda_{n+1}|) = \infty.$ From here we get:
$$ \lim_{n \to \infty} \frac{\|x- p^{|\lambda_n|}(x)\|_{L^2_{k-1}}}{f(|\lambda_{n+1}|)} = 0.$$
The claim \eqref{eq:qlim} follows: For any $\lambda \gg 0$, we choose $n$ such that $\lambda \in [|\lambda_n|, |\lambda_{n+1}|]$, and then we use the fact that both the numerator and the denominator of the right hand side of \eqref{eq:qlim} are nonincreasing functions of $\lambda$.
\end{proof}

Let $[x_0] \in W^\sigma/S^1$ be a non-degenerate, irreducible zero of $\Xqagcsigma$. By Lemma~\ref{lem:rephraseStat}, non-degeneracy means that the linearization $$\D^\sigma_{[x_0]}(\Xqagcsigma) : \K^{\agCoul,\sigma}_{k,[x_0]} \to \K^{\agCoul,\sigma}_{k-1,[x_0]} $$ is an invertible linear operator.

\begin{proposition} \label{prop:nearby}
Let $[x] \in W_k^\sigma/S^1$ be a non-degenerate stationary point of $\Xqagcsigma$.  Then, for any sufficiently small neighborhood $U_{[x]}$ of $[x]$ in $W_k^\sigma/S^1$, for $\lambda \gg 0$ (depending on $U_{[x]}$) there is a unique $[x_\lambda] \in U_{[x]}$ satisfying $\Xqmlagcsigma([x_\lambda])=0$.  
\end{proposition}

\begin{proof}
Consider the vector field $S: (W^\sigma_{k}/S^1) \times (-1,1) \to \K^{\agCoul,\sigma}_{k-1} \times \rr$ given by
$$ S([x], r)=\bigl( \X_{\q^{f^{-1}(|r|)}}^{\agCoul,\sigma}([x]), r \bigr).$$

If we choose a representative $x \in W_k^\sigma$ for $[x]$, we can write 
$$ S([x],r) = \bigl(\bigl[ \bigl ( \Pi^{\agCoul,\sigma} \circ (l + p^{f^{-1}(|r|)} c_\q)^\sigma(x) \bigr) \bigr], r \bigr).$$
Using Lemma~\ref{lem:hessiancontinuity} and its adaption to $\Xqmlgcsigma$ together with Lemma~\ref{lem:homeo}, we deduce that $S$ is continuously differentiable.  Furthermore, from Lemma~\ref{lem:homeo}, the derivative of $S$ at $([x],0)$ is the same as the one obtained without the factor $p^{f^{-1}(|r|)}$; that is, it equals $\begin{pmatrix} \D^\sigma_{[x]}(\Xqagcsigma) & 0 \\ 0 & 1  \end{pmatrix}$, which is invertible by hypothesis.

Note that $(W_k^\sigma/S^1) \times (-1,1)$ is a Banach manifold with boundary. Let $\mathscr{D}(W_k^\sigma/S^1 \times (-1,1))$ be its double, equipped with the associated involution $\iota$.  Since $S$ is tangent to the boundary of $(W_k^\sigma/S^1) \times (-1,1)$, $S$ extends to the double in an $\iota$-invariant way.  We now apply the inverse function theorem to $S$ (on the double) at $([x],0)$. We get that there is a unique solution of $S([y], r)= (0, r)$ near $([x],0)$, for small $r > 0$.  Define $\lambda = f(r)$, and write $[x_\lambda] = [y]$.   

Note that if $[x]$ was irreducible (i.e. in the interior of $(W_k^\sigma/S^1) \times (-1,1)$), then the nearby solution, $[x_\lambda]$, is an irreducible stationary point of $\Xqmlagcsigma$.  If $[x]$ was reducible (i.e. on the boundary of $(W_k^\sigma/S^1) \times (-1,1)$), by $\iota$-invariance, the nearby $[x_\lambda]$ is also on the boundary (i.e., it is a reducible stationary point of $\Xqmlagcsigma$). 
\end{proof}

Recall that in Section~\ref{sec:strategy} we denoted by $\Crit_{[-N, N]}$ the set of stationary points of $\Xqagcsigma$ with grading in the interval $[-N, N]$ (or, equivalently, those inside $\N/S^1$). For simplicity, we write $\Crit_{\N}$ for $\Crit_{[-N, N]}$. Moreover, we denote by $\Crit^{\lambda}_{\N}$ the set of stationary points of $\Xqmlagcsigma$ that live in $\N/S^1$.
 
Combining the results of this subsection and the previous one, we have the following:
\begin{corollary}
\label{cor:corresp} 
 For $\lambda \gg 0$, there is a one-to-one correspondence 
 $$\Xi_{\lambda}: \Crit^{\lambda}_{\N} \to \Crit_{\N}.$$ 
This correspondence preserves the type of stationary point (irreducible, stable, unstable).
\end{corollary}

\begin{proof}
There are only finitely many stationary points in $\Crit_{\N}$. Thus, we can find $\epsilon > 0$ such that the neighborhoods chosen in Proposition~\ref{prop:nearby} around each $[x] \in \Crit_{\N}$ are the balls of $L^2_k$ radius $\epsilon$ around those points. We let $\Xi_{\lambda}^{-1}([x])$ be the unique zeros of $\Xqmlagcsigma$ inside those balls. Since the points in $\Crit_{\N}$ are actually inside the smaller open set $\U/S^1 \subset \N/S^1$, we can assume that $\Xi_{\lambda}^{-1}([x])$ are in $\N/S^1$ as well. By Lemma~\ref{lem:compactness}, we get that any stationary point of $\Xqmlagcsigma$ from $\N/S^1$ must be of the form $\Xi_{\lambda}^{-1}([x])$ for some $[x] \in \Crit_{\N}$. Thus, $\Xi_{\lambda}$ is a one-to-one correspondence.

The claim about preserving type follows from Remark~\ref{rem:type} and the arguments at the end of the proof of Proposition~\ref{prop:nearby}.
\end{proof}

From now on, we will generally use $[x_\infty]$ to denote a stationary point of $\Xqagcsigma$ (corresponding to $\lambda = \infty$).  If we fix $[x_\infty] \in \Crit_{\N}$, observe that $\Xi_{\lambda}^{-1}([x_\infty])$, which is a stationary point of $\Xqmlagcsigma$, is smooth as a function of $\lambda$; this is guaranteed by the implicit function theorem. We will write $[x_\lambda]$ for $ \Xi_{\lambda}^{-1}([x_\infty])$.  

Finally, note that since the reducible stationary points of $\Xqgcsigma$ all correspond to non-zero eigenvalues, the same holds for the corresponding reducible stationary points of $\Xqmlagcsigma$, assuming $\lambda \gg 0$.  

It turns out we also can obtain analogous results for reducible stationary points in the blow-down.  This will be useful when analyzing the reducible stationary points which are in $(B(2R) \cap \vml)^\sigma$, but not in $\N$; see Proposition~\ref{prop:GradingBounds} below.

\begin{lemma}\label{lem:implicitfunctionreducible}
Fix $\epsilon > 0$.  For $\lambda \gg 0$, there is a one-to-one correspondence in $B(2R)$ between reducible stationary points $x_\infty$ of $\Xqgc$ and reducible stationary points $x_\lambda$ of $\Xqmlgc$; further, $x_\lambda$ is $\epsilon$-close to $x$.    
\end{lemma}
\begin{proof}
It is straightforward to verify that since the stationary points of $\Xqagcsigma$ are non-degenerate, so are the reducible stationary points of $\Xqgc$ when restricting to the reducible locus in $W_k$.  By the implicit function theorem, as used in the proof of Proposition~\ref{prop:nearby}, we obtain that near each reducible stationary point $x_\infty$ of $\Xqgc$, there is a unique nearby reducible stationary point $x_\lambda$ of $\Xqmlgc$.  (This is in fact easier than Proposition~\ref{prop:nearby}, since there is no need to blow-up or quotient by $S^1$.)  On the other hand, by Lemma~\ref{lem:compactnessnoblowup}, reducible stationary points of $\Xqmlgc$ are necessarily nearby to reducible stationary points of $\Xqgc$.  This establishes the desired correspondence.  Finally, Lemma~\ref{lem:trajectoriesinvml} implies that $x_\lambda$ is also in $\vml$.  
\end{proof}

\section{Hyperbolicity}\label{sec:nondegeneratestationarypoints}
Recall that we chose $\q$ such that any stationary point $x$ of $\Xqgcsigma$ is non-degenerate, that is, the Hessian $\Hess^{\tilde{g},\sigma}_{\q,x}$ is invertible. By Lemma~\ref{lem:HessB} (b), $\Hess^{\tilde{g},\sigma}_{\q,x}$ has real spectrum, and hence is a hyperbolic operator (i.e., the spectrum of its complexification is disjoint from the real axis).

For the approximate vector field $\Xqmlgcsigma$, we define the $\tilde{g}$-Hessian in the blow-up by analogy with \eqref{eq:HessianBlowupCG}:
\begin{equation}
\label{eq:HessianBlowupCGlambda}
\Hess^{\tilde{g},\sigma}_{\qml,x} = \Pi^{\agCoul,\sigma}_x \circ   \Dgs_x \Xqmlgcsigma : \K^{\agCoul,\sigma}_{k,x} \to \K^{\agCoul,\sigma}_{k-1,x}.  
\end{equation}

As before, if $x$ is a stationary point of $\Xqmlgcsigma$, then it does not matter which connection we use to differentiate $\Xqmlgcsigma$ at $x$.  Therefore, we can simply write
$$\Hess^{\tilde{g},\sigma}_{\qml,x} = \Pi^{\agCoul,\sigma}_x \circ   \D^\sigma_x \Xqmlgcsigma.$$

We will say that $x$ is a {\em hyperbolic} stationary point if $\Hess^{\tilde{g},\sigma}_{\qml,x}$ is hyperbolic. (Generally, it will no longer be the case that $\Hess^{\tilde{g},\sigma}_{\qml,x}$ has real spectrum.) 

If $x$ is in $\N$, then from Corollary~\ref{cor:IsFinite} we know that $x$ lives in the finite dimensional blow-up $(\vml)^{\sigma}$. Furthermore, since $\Xqmlgc = l+\pml c_{\q}$ maps $\vml$ to $\vml$, then $x$ being hyperbolic (as above, in infinite dimensions) implies that $[x]$ is also hyperbolic as a stationary point of $\Xqmlagcsigma$ restricted to the finite dimensional approximation $(W^\lambda)^\sigma/S^1$. Proving hyperbolicity for these stationary points will be the first step towards defining Morse homology with $\Xqmlagcsigma$ on $(W^\lambda)^\sigma/S^1$ as in Section~\ref{sec:combinedMorse}.

Let us start by analyzing the Hessian on the blow-up more carefully.  

Writing $x = (a,s,\phi) \in W_k^\sigma$, we have 
\begin{equation}\label{eqn:lsigma}
l^\sigma(a,s,\phi) = (*da, \langle D\phi, \phi \rangle_{L^2} s, D\phi - \langle D\phi, \phi \rangle_{L^2} D \phi).  
\end{equation}
Thus, we obtain for $v = (b,r,\psi) \in \T^{\gCoul,\sigma}_{j,x}$
\begin{align}\label{eqn:lsigma-derivative}
\nonumber \D^\sigma_{x} l^\sigma(v) &= (*db, \langle D\phi, \phi \rangle_{L^2} r + 2\Re \langle D \psi, \phi \rangle_{L^2} s, \Pi^\perp_{\phi}(D \psi - \langle D\phi, \phi \rangle_{L^2} \psi - 2 \Re \langle D\psi, \phi \rangle_{L^2} \phi)) \\
& = (*db, \langle D\phi, \phi \rangle_{L^2} r + 2\Re \langle D \psi, \phi \rangle_{L^2} s, \Pi^\perp_{\phi}(D \psi) - \langle D\phi, \phi \rangle_{L^2} \psi) \\
\nonumber &= (*db, \langle D\phi, \phi \rangle_{L^2} r + 2\Re \langle D \psi, \phi \rangle_{L^2} s, D \psi - \langle D \psi, \phi \rangle_{L^2} \phi - \langle D\phi, \phi \rangle_{L^2} \psi) 
\end{align}
where $\Pi^\perp_{\phi}$ denotes $L^2$ orthogonal projection onto $\{\psi' \in L^2_{j-1}(Y; \Spin) \mid  \Re \langle \phi, \psi' \rangle_{L^2} = 0 \}$.  Here we are using that $\Pi^\perp_{\phi}(\psi) = \psi$ since $(b,r,\psi) \in \T^\sigma_{j,x}$.  Again, note that here we are taking derivatives with respect to the $L^2$ metric.  
Observe that $L^2_k$ bounds on $(a,s,\phi)$ and $L^2_j$ bounds on $(b,r,\psi)$ easily give bounds on the $r$-component of \eqref{eqn:lsigma-derivative}.  

Recall from Lemma~\ref{lem:HessB} that $\Hess^{\tilde{g},\sigma}_{\q,x}$ has Fredholm index 0, regardless of whether $x$ is a stationary point.  Further, we have that $\Hess^{\tilde{g},\sigma}_{\q^\lambda,x}$ is a compact perturbation of $\Hess^{\tilde{g},\sigma}_{\q,x}$ (roughly since both $\pml c_\q$ and $c_\q$ are compact as maps from $W_k$ to $W_{k-1}$), and thus has Fredholm index 0.  Further, since the inclusion $\K^{\agCoul,\sigma}_{k,x} \to \K^{\agCoul,\sigma}_{k-1,x}$ is compact, we see that for any $z \in \mathbb{C}$, after complexifying, 
\begin{equation}\label{eqn:Hesskappa}
\ind(\Hess^{\tilde{g},\sigma}_{\q,x} - z I) = 0.
\end{equation}

\begin{proposition}
\label{prop:ND}
For $\lambda$ sufficiently large, the stationary points of $\Xqmlgcsigma$ inside $\N$ are hyperbolic.  As a consequence, among the stationary points of the restriction of $\Xqmlagcsigma$ to the finite dimensional space $(W^\lambda)^\sigma/S^1$, all those in $\N/S^1$ are hyperbolic (and hence non-degenerate).
\end{proposition}
\begin{proof}
Recall from Section~\ref{sec:fdax} that we can use an equivalent $L^2_k$ metric on $W$ (and hence $W^\sigma$), where we use the operator $l$ in place of the covariant derivatives.  The statement is independent of which metric we use, so we opt for $l$ in this proof.  

Suppose the claim is false.  Consider a sequence $x_n = (a_n,s_n,\phi_n) \in \N$ of stationary points of $\mathcal{X}^{\gCoul, \sigma}_{\q^{\lambda_n}}$ which are non-hyperbolic, where $\lambda_n \to \infty$.    Note that we have $L^2_k$-bounds on each component of $x_n$.  Moreover, by Lemma~\ref{lem:compactness}, there exists a subsequence of $x_n$ which converges in $W^\sigma_k$ to $x$, a stationary point of $\Xqgcsigma$. By our assumption on $\q$, the stationary point $x$ is hyperbolic.  Since each $x_n$ is non-hyperbolic, we may find a sequence of real numbers $\kappa_n$ such that, after complexifying, $\Hess^{\tilde{g},\sigma}_{\q^{\lambda_n},x_n} - i\kappa_n I$ is not invertible.  Since $\ind(\Hess^{\tilde{g},\sigma}_{\q^{\lambda_n},x_n} - i\kappa_n I) = 0$ by \eqref{eqn:Hesskappa}, these operators have non-trivial kernel, and thus we may find non-zero $v_n \in \K^{\agCoul,\sigma}_{k,x_n} \otimes \mathbb{C}$ such that $\Hess^{\tilde{g},\sigma}_{\q^{\lambda_n},x_n}(v_n) = i \kappa_n v_n$.  For the rest of the proof, we drop the complexified notation for simplicity.  

After rescaling, we can assume that $\|v_n\|_{L^2_k} =1$.  We will show that there exists a subsequence which converges in $\K^{\agCoul,\sigma}_{k}$ to a non-trivial element $v$ of $\K^{\agCoul,\sigma}_{k,x}$ for which $\Hess^{\tilde{g},\sigma}_{\q,x}(v) = i \kappa v$, for some $\kappa \in \mathbb{R}$.  This will contradict the hyperbolicity of $x$.    

We write $v_n = (b_n,r_n,\psi_n)$.  Our first step is to prove that the $\kappa_n$ are bounded, as are the $L^2_{k+1}$-norms of $v_n$. 

 By Lemma~\ref{lem:fam}(b), we get a uniform bound on the $L^2_k$ norms of $\Pi^{\agCoul,\sigma} \circ \D^\sigma_{x_n} (p^{\lambda_n} c_\q)^\sigma(v_n)$.  Further, since $x_n$ is a stationary point of $\Xqmlngcsigma$, we can write   
\begin{align}
\label{eq:hesszero}
i\kappa_n v_n &= \Hess^{\tilde{g},\sigma}_{\q^{\lambda_n},x_n}(v_n) \\
&= \Pi^{\agCoul,\sigma} \circ \D^\sigma_{x_n} \mathcal{X}^{\gCoul,\sigma}_{\q^{\lambda_n}}(v_n) \notag \\
&= \Pi^{\agCoul,\sigma} \circ \D^\sigma_{x_n} l^\sigma(v_n) + \Pi^{\agCoul,\sigma} \circ \D^\sigma_{x_n} (p^{\lambda_n} c_\q)^\sigma(v_n).  \notag
\end{align}
By definition, $\Pi^{\agCoul,\sigma} \circ \D^\sigma_{x_n} l^\sigma(v_n) = \D^\sigma_{x_n} l^\sigma(v_n) - \alpha_n (0,0,i\phi_n)$ for some sequence of real numbers $\alpha_n$.  The $L^2_k$-bounds on $x_n$ and $v_n$ give $L^2_{k-1}$-bounds on $\D^\sigma_{x_n} l^\sigma(v_n)$ by \eqref{eqn:lsigma-derivative}.  By continuity, $\Pi^{\agCoul,\sigma} \circ \D^\sigma_{x_n} l^\sigma(v_n)$ is $L^2_{k-1}$-bounded.  Because $\| \phi_n \|_{L^2_{k-1}} \geq 1$, the sequence $\alpha_n$ is bounded.  Thus, $\alpha_n (0,0,i\phi_n)$ is $L^2_k$-bounded.  

Now note that $\langle D\psi_n, \phi_n \rangle_{L^2} \phi_n$ is $L^2_k$ bounded.  By \eqref{eqn:lsigma-derivative}, we then have that $\D^\sigma_{x_n} l^\sigma(v_n) - (*db_n, 0, D\psi_n - \langle D\phi_n, \phi_n \rangle_{L^2} \psi_n)$ is $L^2_k$ bounded; here we are using the observation that the $L^2_k$ bounds on $x_n$ and $v_n$ guarantee bounds on the $r$-component of $\D_{x_n} l^\sigma(v_n)$.  Combining this with the above discussion, we have that 
\[
\Pi^{\agCoul,\sigma} \circ \D^\sigma_{x_n} l^\sigma(v_n) + \Pi^{\agCoul,\sigma} \circ \D^\sigma_{x_n} (p^{\lambda_n} c_\q)^\sigma(v_n) - (*db_n, 0 , D \psi_n - \langle D \phi_n, \phi_n \rangle_{L^2}  \psi_n)
\]
is $L^2_k$-bounded.  We now see \eqref{eq:hesszero} shows that 
\begin{equation}\label{eqn:brpsi}
i\kappa_n (b_n,r_n,\psi_n) - (*db_n,r_n, D\psi_n - \langle D \phi_n, \phi_n \rangle_{L^2}  \psi_n)
\end{equation}
is $L^2_k$-bounded.  However, note that the operator $(b,r,\psi) \mapsto (*db,r, D \psi - \langle D \phi_n, \phi_n \rangle_{L^2} \psi)$ is $L^2$ self-adjoint. It follows from here that $i \kappa_n v_n$ and $(*db_n,r_n, D \psi_n -  \langle D \phi_n, \phi_n \rangle_{L^2} \psi_n)$ are (real) orthogonal in the complexification, with respect to the $L^2$ inner product. Given that we are working with the $L^2_k$ inner products that are defined using $l$ as the derivative, we see that the orthogonality also holds with respect to $L^2_k$. Thus, the $L^2_k$ bounds on the quantity in \eqref{eqn:brpsi} imply that $\kappa_n v_n$ is $L^2_k$-bounded, and thus the $\kappa_n$ are bounded.  On the other hand, we also obtain $L^2_k$-bounds on $(*db_n, r_n, D\psi_n)$ and thus the $v_n$ are bounded in $L^2_{k+1}$ by the ellipticity of $*d$ and $D$.    

Thus, we can find a subsequence of the $v_n$ that converges in $\K^{\agCoul,\sigma}_{k}$ to an element $v$, which necessarily has $\| v \|_{L^2_k} = 1$.  Because $\K^{\agCoul,\sigma}_{k}$ is closed in the tangent bundle to $W^\sigma_k$, we have $v \in \K^{\agCoul,\sigma}_k$ as well.  After passing to a further subsequence, $\kappa_n \to \kappa$, for some $\kappa \in \mathbb{R}$.  We have
$$ \Hess^{\tilde{g},\sigma}_{\q^{\lambda_n},x_n}(v_n) - \Hess^{\tilde{g},\sigma}_{\q,x_n}(v_n) =  \Pi^{\agCoul,\sigma} \circ \D^\sigma_{x_n} (p^{\lambda_n} c_{\q}  - c_{\q})^\sigma(v_n) \to 0 \text{ in } L^2_{k-1}.$$

By assumption, the sequence $\Hess^{\tilde{g},\sigma}_{\q^{\lambda_n},x_n}(v_n)$ is the sequence $i \kappa_n v_n$. Moreover, $ \Hess^{\tilde{g},\sigma}_{\q,x_n}(v_n)$ converges to $\Hess^{\tilde{g},\sigma}_{\q,x}(v)$ in $L^2_{k-1}$ by the continuity of the Hessian (cf. Lemma~\ref{lem:hessiancontinuity}).  Thus, there exists a non-zero $v \in \K^{\agCoul,\sigma}_k$ such that $\Hess^{\tilde{g}}_{\q,x}(v) = i \kappa v$.  This contradicts the hyperbolicity of $x$.  
\end{proof}

\section{Other stationary points}\label{sec:stationarypointsoutsideN}
Proposition~\ref{prop:ND} was only about the stationary points in $\N$. Recall that $\N$ is a subset of $B(2R)^{\sigma} \subset W^{\sigma}_k$. We do not have any control over the stationary points of $\Xqmlgcsigma$ outside $B(2R)^{\sigma}$, but we can say a bit more about the ones in $B(2R)^{\sigma}$ (and not necessarily in $\N$).

First, recall from Corollary~\ref{cor:IsFinite} that, for $\lambda \gg 0$, all the stationary points of $\Xqmlgcsigma$ in $\N$ are actually inside the finite-dimensional blow-up $(W^{\lambda})^{\sigma}$.

Second, by applying Proposition~\ref{prop:proposition3perturbed} to $N=\overline{B(2R)}$ and $U$ being the blow-down of $\U$, we see that for $\lambda \gg 0$, all the irreducible stationary points of $\Xqmlgcsigma$ in $B(2R)^{\sigma}$ are actually in $\U \subset \N$.


Some of the reducible solutions to $\Xqmlgcsigma$ are in $\N$, and hence close to reducible zeros of $\Xqgcsigma$ with grading in $[-N, N]$.  However, there will be other reducibles in $B(2R)^\sigma$ which may not be in $\N$.  We now study these other reducibles. A reducible $(a, 0, \phi)$ has to satisfy:
$$-*da = (p^{\lambda}c_{\q})^0(a, 0)$$
and
$$D_{\qml,a}(\phi) = D \phi + (\pml)^1(\D_{(a,0)} c_\q(0,\phi)) = \kappa \phi,$$
for some $\kappa \in \rr$. Note that $(a, 0) \in W^{\lambda}$, but $(a, 0, \phi)$ may or may not be in $(W^{\lambda})^{\sigma}$. 

Observe that $D$, $(\pml)^1$ and $\D_{(a,0)} c_\q(0,\cdot )$ are all $L^2$ self-adjoint maps on spinors. (We are using here that the $\tilde g$ metric agrees with the $L^2$ metric at reducibles.) Nevertheless, the product of $(\pml)^1$ and $\D_{(a,0)} c_\q(0,\cdot )$, and hence the operator $D_{\qml,a}$, may not be self-adjoint. On the other hand, for the real numbers $\llambda_i$ defined in Section~\ref{sec:fdax}, the restriction of $D_{\q^{\llambda_i},a}$ to (the spinorial part of) $W^{\lambda}$ is self-adjoint. This is because for $\lambda = \llambda_i$, the map $\pml$ is the honest $L^2$ projection onto $W^\lambda$; therefore, for $(a, 0, \phi) \in (W^{\lambda})^{\sigma}$ we can write
$$ D + (\pml)^1 \D_{(a,0)} c_\q(0,\cdot ) = D + (\pml)^1 \D_{(a,0)} c_\q(0,\cdot ) (\pml)^1,$$
and the right hand side is self-adjoint. 

We will focus on the reducible stationary points of $\Xqmlgcsigma$ that are in $(B(2R) \cap W^{\lambda})^{\sigma}$ for $\lambda = \llambda_i$.  We then have the following strengthening of Proposition~\ref{prop:ND}:

\begin{proposition}
\label{prop:ND2}
We can choose the admissible perturbation $\q$ such that for any $\lambda \in \{\llambda_1, \llambda_2, \dots\}$ sufficiently large, the restriction of $\Xqmlgcsigma$ to $(B(2R) \cap W^{\lambda})^{\sigma}$ has only hyperbolic (and hence non-degenerate) stationary points.   
\end{proposition}

\begin{proof}
As part of the proof of existence of admissible perturbations, Kronheimer and Mrowka showed in \cite[Section 12.6]{KMbook} that for a residual (and hence nonempty) set of tame perturbations $\q$, the reducible stationary points of $\Xq$ are non-degenerate. The key point is in the proof of Lemma 12.6.2 in \cite{KMbook}: There is a large enough space of tame (in fact, very tame) perturbations $\q^{\perp}$ (given by cylinder functions) that vanish at the reducible locus, such that in the tangent space to any $\q^{\perp}$ we can find a $\delta \q^{\perp} = \grad \delta f$, such that the Hessian of $\delta f|_V$ is any chosen $S^1$-equivariant self-adjoint endomorphism of $V = \ker(D_{\q^{\perp}, a})$.  

We can adapt this proof to $\Xqmlgcsigma$, and pick $\q$ such that all the reducibles are non-degenerate for a given $\lambda$. Indeed, we now need to find a $\delta \q^{\perp} = \grad \delta f$ such that the $\tilde{g}$-Hessian of $\delta f|_V$ is any  $S^1$-equivariant self-adjoint endomorphism of $V = \ker(D_{(\q^{\perp})^{\lambda}, a}) \subset \vml$. Since we are at a reducible, on the spinorial part we have that the $\tilde{g}$-Hessian is the same as the usual Hessian. Further, because $\lambda= \llambda_i$, we have that $\pml$ is the $L^2$ orthogonal projection to $\vml$. Hence, we have
$$ \Hess (\delta f|_{\vml}) = \pml \circ \Hess(\delta f)|_{\vml},$$
and we can arrange so that this equals any chosen $S^1$-equivariant self-adjoint endomorphism of the spinors in $\vml$. From here we get the same freedom in choosing $\Hess (\delta f|_V)$, for any subspace $V$ of spinors in $\vml$. The other arguments from \cite[Section 12.6]{KMbook} can then be easily adapted to our setting.

Since $\lambda$ is part of a countable collection $\{\llambda_n\}$, we can find $\q$ such that the reducibles are non-degenerate for all such $\lambda$. Note that non-degeneracy implies hyperbolicity for reducibles, because the relevant operators are self-adjoint. 

Since the irreducible stationary points of $\Xqmlgcsigma$ in $B(2R)^{\sigma}$ are actually in $\U \subset \N$, by applying Proposition~\ref{prop:ND} we can arrange so that these irreducibles are hyperbolic as well.
\end{proof}

Proposition~\ref{prop:ND2} will be useful in showing that the vector field $\Xqmlgcsigma$ in $(B(2R) \cap W^{\lambda})^{\sigma}$ is a Morse-Smale equivariant quasi-gradient. We can then construct a Morse homology group from it, and show that it is the same as Morse homology in $\N \cap (W^{\lambda})^{\sigma}$, in a certain grading range $[-N, N]$. Indeed, we will show that all the other reducible points in $(B(2R) \cap W^{\lambda})^{\sigma}$ cannot be in this grading range; see Proposition~\ref{prop:GradingBounds} below.  Before we can discuss and define gradings on stationary points as in Section~\ref{sec:finite}, we must first establish that $\Xqmlgc$ is indeed a Morse quasi-gradient.   This is the subject of the following section.  We return to discuss gradings on stationary points of $\Xqmlgcsigma$ in Chapter~\ref{sec:gradings} and the Morse-Smale condition in Chapter~\ref{sec:MorseSmale}.

\chapter[The approximate flow as a quasi-gradient]{The approximate flow as a Morse equivariant quasi-gradient}
\label{sec:quasigradient}
Throughout this chapter we assume that the eigenvalue cut-off $\lambda$ is of the form $\llambda_i$ for $i \gg 0$.

Note that $\Xqgc=l+c_{\q}$ is the gradient of the $\L_{\q}$ functional with respect to the $\tilde g$ metric. However, the maps $\pml$ are defined in terms of projections with respect to the usual $L^2$ metric. As discussed in Remark~\ref{rem:notgradient}, the vector field 
$$\Xqmlgc=l+\pml c_{\q}$$ on $\vml$ is neither the $L^2$ nor the $\tilde g$ gradient of the restriction of $\L_{\q}$ to $\vml$. In fact, there is no reason for the derivative of $\Xqmlgc$ at stationary points to have real spectrum (as it would happen for a gradient vector field, with respect to any metric). 

Nevertheless, in this chapter we will be able to prove the following.

\begin{proposition}
\label{prop:AllMS}
We can choose the admissible perturbation $\q$ such that for all $\lambda = \llambda_i$ with $i \gg 0$, the vector field $\Xqmlgc$ on $\vml \cap B(2R)$ is a Morse equivariant quasi-gradient, in the sense of Definition~\ref{def:eqgv}.
\end{proposition}

As discussed in Section~\ref{sec:finite}, having a Morse-Smale equivariant quasi-gradient suffices in order to construct (equivariant) Morse homology; the first step towards this is establishing the Morse condition (cf. Definition~\ref{def:eqgv}).  The additional Morse-Smale condition on trajectories will be shown in Chapter~\ref{sec:MorseSmale}. 

In view of Lemma~\ref{lem:MEquiv}, to check that $\Xqmlgc$ is a Morse equivariant quasi-gradient we need three things: 
\begin{itemize}
\item that the stationary points of $\Xqmlgcsigma$ are hyperbolic;
\item that the operators $\D_x\Xqmlgc$ at reducible stationary points $x$ are self-adjoint;
\item part (d) of Definition~\ref{def:eqgv}. 
\end{itemize}

Hyperbolicity of the stationary points was already checked in Proposition~\ref{prop:ND2}. Self-adjointness of the operators $\D_x\Xqmlgc = l + \pml \D_x c_{\q}$ at reducibles follows from the fact that the metrics $\tilde g$ and $L^2$ coincide there.

We are left to verify part (d) of Definition~\ref{def:eqgv}. Sections~\ref{sec:controlaway}-\ref{sec:consequences} below are devoted to proving this.  

\begin{proposition}\label{prop:LqQuasi}
For each $\lambda \gg 0$, there exists a smooth function $$F_{\lambda}: \vml \cap B(2R) \to \R$$ such that 
\begin{equation}
\label{eq:dFlambda}
\frac{1}{4} \| \Xqmlgc \|^2 _{\tilde{g}} \leq dF_{\lambda}(\Xqmlgc) \leq 4 \| \Xqmlgc \|^2_{\tilde{g}}.
\end{equation}
In particular, $dF_{\lambda}(\Xqmlgc) \geq 0$, with equality only at the stationary points of $\Xqmlgc$.
\end{proposition} 
 
\begin{remark} 
In the statement of Proposition~\ref{prop:LqQuasi}, the constants $\frac{1}{4}$ and $4$ are quite arbitrary. They could be replaced (at the expense of increasing $\lambda$) by $\frac{1}{C}$ and $C$, for any $C > 1$.
\end{remark}

\section{Control away from the stationary points} \label{sec:controlaway}
Since we know that $\Xqgc$ is the $\tilde g$-gradient of the perturbed CSD functional $\Lq$, the first guess is to take $F_{\lambda}$ to be $\Lq$. Then, the desired condition holds away from neighborhoods of the stationary points:

\begin{lemma}
\label{lem:cond1}
Fix $\epsilon > 0$. Then, for all $\lambda \gg 0$, we have 
\begin{equation}
\label{eq:dlq}
\frac{1}{4} \| \Xqmlgc \|^2_{\tilde{g}} < d\L_{\q} (\Xqmlgc) < 4 \| \Xqmlgc \|^2_{\tilde{g}}
 \end{equation}
at any point in $\vml \cap B(2R)$ which is at $L^2_{k-1}$ distance at least $\epsilon$ from all stationary points of $\Xqmlgc$ in $B(2R)$. 
\end{lemma}
\begin{proof}
We suppose this is not true.  Then there exists sequences $\lambda_n \to \infty$ and $x_n \in \vmln \cap B(2R)$ such that $x_n$ is $L^2_{k-1}$ distance at least $\epsilon$ from each stationary point of $\Xqmlngc$ and $(d\L_{\q})_{x_n}(\Xqmlngc)$ violates \eqref{eq:dlq}.  Without loss of generality, we assume that the first inequality in \eqref{eq:dlq} is violated.  (The case of the second inequality is similar.)  Since the $x_n$ are $L^2_k$ bounded, we can extract a subsequence which converges in $L^2_{k-1}$ to some element $x$.  We see that 
$$ (d\L_{\q})_{x_n}(\Xqmlngc) \to (d\L_{\q})_{x}(\Xqgc).$$
Since $d\L_{\q}(\Xqgc) = \| \grad \L_{\q} \|_{\tilde{g}}^2 \geq 0$, we see that $\| \Xqgc(x) \|^2_{\tilde{g}} \leq \frac{1}{4}\| \Xqgc(x) \|^2_{\tilde{g}}$.  Therefore, $x$ is a stationary point of $\Xqgc$.  This implies that $x$ is in $B(2R)$.  For $\lambda \gg 0$, by the work of Section~\ref{sec:StabilityPoints}, we have that $x$ has $L^2_k$ distance (and thus $L^2_{k-1}$ distance) at most $\epsilon/2$ from a stationary point $x_\lambda$ of $\Xqmlgc$.  Since the $x_n$ converge to $x$ in $L^2_{k-1}$, they are eventually within $L^2_{k-1}$ distance $\epsilon$ of $x_{\lambda_n}$ for $n \gg 0$.  Since $x_{\lambda_n}$ is a stationary point of $\Xqmlngc$, this is a contradiction.     
\end{proof}
 
However, $\L_\q$ does not satisfy \eqref{eq:dlq} in the neighborhoods of stationary points. If it did, then by Lemma~\ref{lem:qgCritical2}, any stationary point of $\Xqmlgc$, i.e., zero of $l+ \pml c_{\q}$,  would be a critical point of $\L_{\q}|_{\vml}$. We can write the $\tilde g$-gradient of $\L_{\q}|_{\vml}$ as $l + \tpgml c_{\q}$, where $\tpgml$ is the $\tilde g$-orthogonal projection from $W$ to $\vml$. (Compare Remark~\ref{rem:notgradient}.) However, in general, the condition $(l+\pml c_{\q})(x)=0$ does not imply $(l +\tpgml c_{\q})(x)=0.$

\section{The function $F_\lambda$} \label{sec:Flambda}
To construct the desired function $F_{\lambda}$ as in Proposition~\ref{prop:LqQuasi}, we need to alter $\Lq$ near the stationary points of $\Xqmlgc$. 

Let us first introduce some notation. Given a point $x \in W$, its $S^1$-orbit  can be either a point or a circle. In particular, the stationary points of $\Xqgc$ come in finitely many such orbits, which we denote by $\O^{1}, \dots, \O^{m}.$

Throughout the rest of Chapter~\ref{sec:quasigradient}, we will fix some $\epsilon > 0$ sufficiently small such that it satisfies the following.

\begin{assumption}
\label{as:1} $ $ \\ \vspace{-.16in}
\begin{enumerate}[(a)]
\item The $L^2_{k-1}$ distance between any two orbits $\O^{j}, \O^{j'}$ ($j \neq j'$) is at least $7\epsilon$;
\item If an orbit $\O^{j}$ consists of irreducibles, then the $L^2$ norm of a point in $\O^{j}$ is at least $4\epsilon$.
\end{enumerate}
\end{assumption}

In Section~\ref{sec:controlnearby} we will add another assumption on $\epsilon$. However, Assumption~\ref{as:1} above suffices for the results in the current subsection. 

With $\epsilon$ fixed, we will state our results for $\lambda$ being sufficiently large. Of course, how large $\lambda$ is may depend on $\epsilon$.

It follows from Proposition~\ref{prop:nearby} and Lemma~\ref{lem:implicitfunctionreducible} that, for $\lambda \gg 0$, there is a one-to-one correspondence between the orbits of stationary points of $\Xqgc$, on the one hand, and the orbits of stationary points of $\Xqmlgc$ inside $B(2R)$, on the other hand. Further, this correspondence preserves the type of orbits (reducible or irreducible). Let $\O^{1}_{\lambda}, \dots, \O^{m}_{\lambda}$ be the latter set of orbits, with $\O^{j}_{\lambda}$ corresponding to $\O^{j}$. By choosing $\lambda$ sufficiently large, we can arrange so that, for all $j$, the orbits $\O^{j}$ and $\O^{j}_{\lambda}$ are within $L^2_{k-1}$ distance $\epsilon$ of each other. In view of part (a) in Assumption~\ref{as:1}, this ensures that $\O^{j}_{\lambda}$ and $\O^{j'}_{\lambda}$ are at least $L^2_{k-1}$ distance $5\epsilon$ apart, for $j\neq j'$.

Next, consider the neighborhoods $\nu_{2\epsilon}(\O^{j}_{\lambda})$ of $\O^{j}_{\lambda}$, consisting of points at $L^2_{k-1}$ distance at most $2\epsilon$ from these orbits. Because of our choice of $\epsilon$, all these neighborhoods are disjoint from each other. As an aside, note also that these neighborhoods may well go outside of $B(2R)$, since the latter ball is taken in the $L^2_k$ metric. 

Pick a point $x_{\lambda}^j$ on each orbit $\O^{j}_{\lambda}$. We define functions
$$ \omega_\lambda^j : \nu_{2\epsilon}(\O^{j}_{\lambda}) \to S^1$$
as follows. If $x_{\lambda}^j$ is reducible, so that $\O^{j}_{\lambda} = \{ x_{\lambda}^j \}$, we simply set $\omega_\lambda^j=1$. If $x_{\lambda}^j$ is irreducible, so that $\O^{j}_{\lambda}$ is a circle, we ask that $\omega_\lambda^j(x) \cdot x_{\lambda}^j$ be the point on $\O^{j}_{\lambda}$ that is at minimal $L^2$ distance from $x$. To make sure that $\omega_\lambda^j$ is well-defined, we need to check that this point is unique. The closest point is not unique only for points in the $L^2$-orthogonal complement to the plane $\Span(\O^{j}_{\lambda})$. However, if $x \in \nu_{2\epsilon}(\O^{j}_{\lambda})$, then there is some $x' \in \O^{j}_{\lambda}$ within $L^2_{k-1}$ (and hence $L^2$) distance $2 \epsilon$ from $x$. By part (b) in Assumption~\ref{as:1}, together with the fact that $\O^j_{\lambda}$ and $\O^j$ are within $L^2_{k-1}$ (and hence $L^2$) distance $\epsilon$ from each other, we see that the $L^2$ norm of $x'$ is at least $3\epsilon$. This shows that $x'$ cannot be perpendicular to $x$, and the claim about the uniqueness of the $L^2$-closest point to $x$ follows. 

Explicitly, when $x_{\lambda}^j$ is irreducible, we can write
\begin{equation}\label{eq:omegalambda}
\omega_{\lambda}^j(x)= \frac{\Re \langle x, x^{j}_{\lambda}  \rangle_{L^2}  + i\Re \langle x, ix^{j}_{\lambda}  \rangle_{L^2} }{ \bigl( (\Re \langle x, x^{j}_{\lambda}  \rangle_{L^2})^2  +  (\Re \langle x, ix^{j}_{\lambda}  \rangle_{L^2})^2 \bigr)^{1/2}}.
\end{equation}

Note that the original orbit $\O^{j}$ is at $L^2_{k-1}$ distance at most $\epsilon$ from $\O^{j}_{\lambda}$, and hence is contained in $\nu_{2\epsilon}(\O^{j}_{\lambda})$. Let $x_{\infty}^j$ be the point in $\O^{j}$ that is closest in $L^2$ distance to the chosen basepoint $x_{\lambda}^j \in \O^{j}_{\lambda}$. Then, also $x_{\lambda}^j$ is the $L^2$-closest point to $x_{\infty}^j$ in $\O^{j}_{\lambda}$; in other words, we have 
$$ \omega_{\lambda}^j( x_{\infty}^j)=1.$$

Let $h: [0, \infty) \to \R$ be a smooth, non-increasing function such that $h(x)=1$ for $x \leq 1$ and $h(x)=0$ for $x \geq 2$. Set
$$
 H^j_{\lambda} : \vml \to [0,1], \ \ \ 
 H^j_{\lambda}(x) = h(\epsilon^{-1} d_{L^2_{k-1}} (x, \O^{j}_{\lambda}))$$
where $d_{L^2_{k-1}}$ denotes $L^2_{k-1}$ distance. Note that $ H^j_{\lambda}$ is identically $0$ outside of $\nu_{2\epsilon}(\O^{j}_{\lambda})$, and is identically $1$ in the smaller neighborhood $\nu_{\epsilon}(\O^{j}_{\lambda})$.

We now define
$$ T_{\lambda} : \vml \to W$$
by
\begin{equation}\label{eq:Tlambda}
 T_{\lambda}(x) = x + \sum_{j=1}^m H^j_{\lambda}(x) \cdot \omega^j_{\lambda}(x) \cdot (x_{\infty}^j- x_{\lambda}^j)
\end{equation}
and finally set
\begin{equation}\label{eq:Flambda}
 F_{\lambda} : \vml \to \R, \ \ \ F_{\lambda}= \Lq \circ T_{\lambda}.
\end{equation}
The function $F_\lambda$ is the one we will use to prove Proposition~\ref{prop:LqQuasi}.  Before analyzing this function, we give a more qualitative description for the benefit of the reader.
  
Observe that $T_{\lambda}$ and $F_{\lambda}$ are $S^1$-equivariant, by construction.  In the smaller neighborhood $\nu_{\epsilon}(\O^{j}_{\lambda})$, if we restrict to the affine space perpendicular to $\O^{j}_{\lambda}$ at $x^j_{\lambda}$, we have $ \omega^j_{\lambda} \equiv 1$ there, and hence the map $T_{\lambda}$ is given by translation by $x_{\infty}^j- x_{\lambda}^j$. In particular, $T_{\lambda}(x_{\lambda}^j) = x_{\infty}^j.$ More generally, $T_{\lambda}$ takes the orbit $\O^{j}_{\lambda}$ to $\O^j$. In fact, we can view $\nu_{\epsilon}(\O^{j}_{\lambda})$ as a disk bundle over $\O^{j}_{\lambda}$,  with the projection map given by taking the $L^2$ closest point on the orbit. Then, we can say that inside $\nu_{\epsilon}(\O^{j}_{\lambda})$, the map $T_{\lambda}$ consists of fiberwise translations, arranged so that $\O^{j}_{\lambda}$ is taken to $\O^j$. Further, $T_{\lambda}$ is the identity outside $\nu_{2\epsilon}(\O^{j}_{\lambda})$, and in the intermediate region $\nu_{2\epsilon}(\O^{j}_{\lambda}) \setminus \nu_{\epsilon}(\O^{j}_{\lambda})$, it is given by some interpolation between fiberwise translation and the identity. 

The resulting function $F_{\lambda} : \vml \to \R$ agrees with the perturbed CSD functional $\Lq$ outside $\nu_{2\epsilon}(\O^{j}_{\lambda})$, whereas near $\O^{j}_{\lambda}$ we have arranged so that the points of $\O^{j}_{\lambda}$ became critical points of $F_{\lambda}$. Effectively, this was accomplished by translating $\O^{j}_{\lambda}$ to $\O^j$, and using the fact that $\O^j$ consists of stationary points of $\Xqgc$, i.e., critical points of $\Lq$.

\section{Control in the intermediate region}
Since $F_\lambda$ agrees with $\L_\q$ at any point with $L^2_{k-1}$ distance at least $2\epsilon$ from a stationary point of $\Xqmlgc$, Lemma~\ref{lem:cond1} implies that $dF_\lambda(\Xqmlgc) > 0$ in this region.  In the current subsection, we will be able to use the same arguments to show that $dF_\lambda(\Xqmlgc) > 0$ as long as the distance is at least $\epsilon$, thus gaining control in the intermediate region $\nu_{2\epsilon}(\O^{j}_{\lambda}) \setminus \nu_{\epsilon}(\O^{j}_{\lambda})$.  For notation, we will sometimes write $\tilde g(x)$ for the inner product given by $\tilde g$ on $T_xW$.  We begin with a technical lemma.   

\begin{lemma}\label{lem:approximate-quasi-convergence}
Consider a sequence $x_n$ in $B(2R)$ which converges in $L^2_{k-1}$ to some $x \in W_{k-1}$.  Then if $\lambda_n \to \infty$, $(dF_{\lambda_n})_{x_n}(\Xqmlngc) \to (d\L_\q)_x(\Xqgc)$ in $\mathbb{R}$.  
\end{lemma}
\begin{proof}
We first compute that for any $x'$, 
\begin{align*}
(dF_\lambda)_{x'}(v) &= (d\L_\q)_{T_\lambda(x')}(\D_{x'} T_\lambda)(v) \\
&= \langle \Xqgc (T_\lambda(x')), \D_{x'} T_\lambda (v) \rangle_{\tilde{g}(x')}. 
\end{align*}

We claim it suffices to show that $T_{\lambda_n}(x_n) \to x$ and $(\D_{x_n} T_{\lambda_n})(\Xqmlngc(x_n)) \to \Xqgc(x)$, each in $L^2_{k-2}$.  Indeed, since the $\tilde{g}(x_n)$-metrics converge to $\tilde{g}(x)$, this will imply that $(dF_{\lambda_n})_{x_n}(\Xqmlngc)$ converges to $\langle \Xqgc(x), \Xqgc(x) \rangle_{\tilde{g}(x)}$, which is exactly $(d\L_\q)_x(\Xqgc)$.     

We begin by analyzing the continuity of $T_\lambda$, using \eqref{eq:Tlambda}.  Note that $|H^j_\lambda|$ and $|\omega^j_\lambda|$ are bounded above by 1.  
Since $\lambda_n \to \infty$, then $x^j_\infty - x^j_\lambda \to 0$ in $L^2_{k-1}$ by the discussion after Corollary~\ref{cor:corresp}.  Therefore, we have $T_{\lambda_n}(x_n) \to x$ in $L^2_{k-1}$.      
   
Thus, it remains to analyze $\D_{x} T_{\lambda}$.  We have that 
\[
(\D_x T_{\lambda}) (v) = v + \sum_{j} (dH^j_\lambda)_x(v) \cdot \omega^j_\lambda(x) \cdot (x_\infty - x_\lambda) + \sum_{j} H^j_\lambda(x) \cdot \D_x \omega^j_\lambda(v) \cdot (x_\infty - x_\lambda).
\]
Note that $|(dH^j_\lambda)_x(v)| \leq (C/\epsilon) \| v \|_{L^2_{k-1}}$ for a constant $C$ independent of $x, j,$ and $\lambda$.  (More precisely, $C$ is the $C^0$-norm of $h'$.)  Also, $\omega^j_\lambda$ is a $C^1$-function when restricted to $\nu_{2\epsilon}(\mathcal{O}^j_\lambda)$ whose denominator in \eqref{eq:omegalambda} is bounded below by $3\epsilon^2$, since any point in $\mathcal{O}^j_\lambda$ has $L^2$ norm at least $3\epsilon$.  From this, it is easy to obtain bounds 
\begin{equation}\label{eq:Dxomega-bounds}
|\D_x \omega^j_\lambda(v)| \leq C' \| v \|_{L^2} \leq C' \| v \|_{L^2_{k-1}}
\end{equation}
independent of $x \in B(2R), j,$ and $\lambda$.    Therefore, $(\D_{x_n} T_{\lambda_n})(v_n) \to v$ in $L^2_{k-2}$ for any sequence $v_n$ which converges to $v$ in $L^2_{k-2}$ and is $L^2_{k-1}$ bounded.  Since $\Xqmlngc(x_n) \to \Xqgc(x)$ in $L^2_{k-2}$ and $\Xqmlngc(x_n)$ is $L^2_{k-1}$ bounded (because the $x_n$ are $L^2_k$ bounded), we have that $(\D_{x_n} T_{\lambda_n})(\Xqmlngc(x_n)) \to \Xqmlgc(x)$ in $L^2_{k-2}$.  This suffices to complete the proof.  
\end{proof}

With the above lemma, we now establish the analogue of Lemma~\ref{lem:cond1} for $F_\lambda$.
\begin{proposition}
\label{prop:intermediate}
Fix $\epsilon > 0$ satisfying Assumption~\ref{as:1}.  For $\lambda \gg 0$, we have 
\begin{equation}\label{eq:dFlambda-gradient-inequality}
\frac{1}{4} \| \Xqmlgc \|^2_{\tilde{g}} < dF_\lambda(\Xqmlgc) < 4 \| \Xqmlgc \|^2_{\tilde{g}}
\end{equation}
at any point in $\vml \cap B(2R)$ which is at $L^2_{k-1}$ distance at least $\epsilon$ from any stationary point of $\Xqmlgc$.  
\end{proposition}
\begin{proof}
Suppose that the conclusion is not true.  Then, there exists a sequence $\lambda_n \to \infty$ and a sequence $x_n \in \vmln \cap B(2R)$ which are $L^2_{k-1}$ distance at least $\epsilon$ from any stationary point of $\Xqmlngc$ and $(dF_{\lambda_n})_{x_n}(\Xqmlngc)$ violate \eqref{eq:dFlambda-gradient-inequality}.  Then, there exists a subsequence of the $x_n$ which converges in $L^2_{k-1}$ to an element $x \in W_{k-1}$.  By Lemma~\ref{lem:approximate-quasi-convergence}, we see that $(d\L_\q)_x(\Xqgc)$ must be at most $\frac{1}{4} \| \Xqgc(x) \|^2_{\tilde{g}}$ or at least $4 \| \Xqgc(x) \|^2_{\tilde{g}}$.  Since $(d\L_\q)(\Xqgc) = \| \grad \L_\q \|^2$, we get that $x$ is a stationary point of $\Xqgc$, and thus $x \in B(2R) \subset W_k$.  Since the $x_n$ converge to $x$ in $L^2_{k-1}$, they are eventually within  $L^2_{k-1}$ distance $\epsilon/2$ of a stationary point of $\Xqgc$.  For $\lambda \gg 0$, there is a stationary point of $\Xqmlgc$ within $L^2_{k-1}$ distance $\epsilon/2$ of $x$.  This contradicts the fact that the $x_n$ are not within $L^2_{k-1}$ distance $\epsilon$ of a stationary point of $\Xqmlngc$.     
\end{proof}

\section{The $L^2$ and $\tilde g$ metrics}
A more detailed analysis will be needed to prove the inequality \eqref{eq:dFlambda} in neighborhoods of the stationary points. This will be done in Section~\ref{sec:controlnearby}. As a preliminary step, since $\Xqmlgc$ is an $L^2$ approximation to the $\tilde g$-gradient of $\Lq|_{\vml}$, we will prove a few results relating the $L^2$ and $\tilde g$ metrics.

Let us recall the definition of $\tilde g$ from Section~\ref{sec:SWe}. For $x=(a, \phi) \in W$ and $(b, \psi), (b', \psi') \in T_{(a, \phi)} W$, we have
$$ \langle (b, \psi), (b', \psi') \rangle_{\tilde g} = \langle  \Pi^{\elCoul}_{(a, \phi)} (b, \psi),  \Pi^{\elCoul}_{(a, \phi)} (b', \psi')\rangle_{L^2},$$
where $\Pi^{\elCoul}$ is the enlarged local Coulomb projection. Let us also recall the formula for this projection:
\begin{equation}
\label{eq:pilco}
\Pi^{\elCoul}_{(a, \phi)} (b, \psi) := (b -  d\zeta, \psi + \zeta \phi),
\end{equation}
where $\zeta: Y \to i\R$ is determined (for $\phi \neq 0$) by the conditions $\int_Y \zeta=0$ and 
\begin{equation}
\label{eq:zetaphi}
 \Delta\zeta + |\phi|^2 \zeta - \mu_Y(|\phi|^2\zeta)=- i\Re\langle i\phi , \psi \rangle + i\mu_Y(\Re\langle i\phi , \psi \rangle).
 \end{equation}
The last equality is Equation \eqref{eq:zetafirst}, where we used the fact that $d^*b=0$ for $(b, \psi) \in T_xW$. 

\begin{lemma}
\label{lem:Zeta}
There is a constant $K > 0$ such that, for all $x=(a, \phi) \in B(2R) \subset W_k$ and $(b, \psi) \in \T_{0,x}^{\gCoul} \cong W_0$, if $\zeta$ is the function in \eqref{eq:pilco}, then
\begin{equation}
\label{eq:L21}
 \| \zeta \|_{L^2_1} \leq K \cdot \| \psi \|_{L^2_{-1}}.
 \end{equation}
\end{lemma}

\begin{proof}
This follows from Lemma~\ref{lem:elcUniqueness} by considering the operator $E_{\phi}$ from $L^2_1$ to $L^2_{-1}$. Note that $K$ can be taken to be a constant independent of $\phi$, because the $L^2_{k}$ norm of $\phi$ is bounded.
\end{proof}
Since $\| \psi \|_{L^2_{-1}} \leq \| \psi \|_{L^2}$, Lemma~\ref{lem:Zeta} also gives $L^2_1$ bounds on $\zeta$ in terms of $L^2$ bounds on $\psi$.  

Our next result is about the equivalence between the $L^2$ and $\tilde g$ metrics.

\begin{proposition}
\label{prop:EquivalentMetrics}
There is a constant $C_0 > 0$ such that, for all $x=(a, \phi) \in B(2R) \subset W_k$ and $(b, \psi) \in \T_{0,x}^{\gCoul} \cong W_0$, we have
\begin{equation}
\label{eq:EquivalentMetrics}
{C_0}^{-1} \cdot \| (b, \psi) \|_{L^2} \leq \| (b, \psi) \|_{\tilde g(x)} \leq C_0 \cdot \| (b, \psi) \|_{L^2}.
\end{equation}
\end{proposition}

\begin{proof}
 We begin with the second inequality in \eqref{eq:EquivalentMetrics}. 
Then, with $\zeta$ as in \eqref{eq:pilco}, we have
$$ \| (b, \psi) \|_{\tilde g(x)} \leq  \| (b, \psi) \|_{L^2} + \| (d\zeta, \zeta \phi) \|_{L^2}.$$
We claim that the right hand side is bounded by a constant times $\| (b, \psi) \|_{L^2}$. This follows by the $L^2_1$ control on $\zeta$ from Lemma~\ref{lem:Zeta} and the fact that the condition $x \in B(2R)$ gives $L^2_k$ (and hence $C^0$) bounds on $\phi$.

To prove the first inequality, note that if we view $\Pi^{\elCoul}_x$ as an isomorphism from the global Coulomb slice to the extended local Coulomb slice, then its inverse is the infinitesimal global Coulomb projection, $(\Pi^{\gCoul}_*)_x$, from \eqref{eq:icp}. Thus, if we switch notation and now let $(b, \psi)$ be a vector in the extended local Coulomb slice, the first inequality in \eqref{eq:EquivalentMetrics} can be re-written as
$$  \| (\Pi^{\gCoul}_*)_x(b, \psi) \|_{L^2} \leq C_0 \cdot \| (b, \psi) \|_{L^2}.$$
This holds by Lemma~\ref{lem:igc-fb}.
\end{proof}

Our next goal is to compare the $L^2$- to $\tilde g$-orthogonal projections from $W$ to $\vml$. Since $\lambda$ is of the form $\llambda_i$, we have that $\pml$ is the $L^2$-orthogonal projection to $\vml$, and we write $\pgmlx$ for the $\tilde g(x)$-orthogonal projection to $\vml$. (Compare Remark~\ref{rem:notgradient}.) 

\begin{proposition}
\label{prop:pgpl}
There is a constant $C_1 > 0$ such that, for all $x \in B(2R) \subset W_k$, we have
$$ \| \pgmlx - \pml \| \leq \frac{C_1}{\lambda},$$
where $\pgmlx$ and $\pml$ are viewed as operators from the $L^2$ completion $W_0$ of $W$ to itself.
\end{proposition}

We first need a refinement of Lemma~\ref{lem:Zeta}. Let us denote by $(\vml)^{\perp}_0$ the $L^2$-orthogonal complement to $\vml$ inside $W_0$. In other words, $(\vml)^{\perp}_0$ is the $L^2$ span of the eigenvectors of $l$ with eigenvalues at least $\lambda$ in absolute value.

\begin{lemma}
\label{lem:EstimateZeta}
There is a constant $C_2 > 0$ such that, for all $x=(a, \phi) \in B(2R) \subset W_k$ and $(b, \psi) \in (\vml)^{\perp}_0$, if $\zeta$ is the function in \eqref{eq:pilco}, then
$$ \| \zeta \|_{L^2_1} \leq \frac{C_2}{\lambda} \|\psi \|_{L^2}.$$
\end{lemma}

\begin{proof}
Write $\psi = \sum \psi_\kappa$, where $\psi_\kappa$ are eigenvectors of the Dirac operator $D$, with eigenvalues $\kappa$ such that $|\kappa| > \lambda$. Set
\begin{equation}
\label{eq:D-1}
D^{-1}(\psi) := \sum \kappa^{-1} \psi_\kappa.
\end{equation}
Note that $D(D^{-1}(\psi))=\psi$. Since $D$ is continuous from $L^2$ spinors to $L^2_{-1}$ spinors, we obtain
$$ \| \psi \|_{L^2_{-1}}   \leq K' \| D^{-1}\psi \|_{L^2} \leq \frac{K'}{\lambda} \|\psi\|_{L^2},$$
for some constant $K'$.  Together with Lemma~\ref{lem:Zeta}, this gives the desired inequality.
\end{proof}

\begin{proof}[Proof of Proposition~\ref{prop:pgpl}]
We need to check that, for all $v \in W_0$,
$$  \| \pgmlx(v) - \pml(v) \|_{L^2} \leq \frac{C_1}{\lambda} \| v \|_{L^2},$$
where $C_1$ is independent of $x, v,$ and $\lambda$.  

First, if $v \in \vml$, then note that $\pml_{\tilde{g}(x)}(v) = \pml(v)$ and the claim is trivial.  Therefore, we can assume that $v \in (\vml)^\perp_{0}$.  Thus, we would like to bound 
$$
\| \pml_{\tilde{g}(x)} (v) - \pml(v) \|_{L^2} = \| \pml_{\tilde{g}(x)}(v)\|_{L^2}.
$$
We write $x = (a,\phi)$ and $$w = \pml_{\tilde{g}(x)}(v) \in \vml.$$  Recall from Proposition~\ref{prop:EquivalentMetrics} that there exists a constant $C_0$, independent of $x$, such that $\| v \|_{L^2} \leq C_0 \| v \|_{\tilde{g}(x)}$ and similarly for $w$.  In this case, we have that 
$$ \| \pml_{\tilde{g}(x)}(v) \|^2_{L^2} \leq C^2_0 \| \pml_{\tilde{g}(x)}(v) \|_{\tilde{g}(x)}^2 = C^2_0 \langle v, \pml_{\tilde{g}(x)}(v) \rangle_{\tilde{g}(x)}.$$ 
Therefore, to obtain the desired bounds in the proposition, it suffices to prove that there exists a constant $K_0 > 0$, independent of $x \in B(2R)$ and $v \in  (\vml)^\perp_{0}$, such that 
\begin{equation}\label{eq:g-tilde-projection-bounds}
| \langle v, w \rangle_{\tilde{g}(x)} | \leq \frac{K_0}{\lambda} \| v \|_{L^2} \| w \|_{L^2}.
\end{equation}
We now focus on proving this inequality.

For notation, let $v = (b,\psi)$ and $w = (b', \psi')$. We write $\Pi^{\elCoul}_x(v) = (b - d\zeta, \psi + \zeta \phi)$ and $\Pi^{\elCoul}_x(w) = (b' - d \zeta', \psi' + \zeta' \phi)$.  Since $w = \pml_{\tilde{g}(x)}(v)$, we have that $v$ and $w$ are $L^2$ orthogonal.  From this, we obtain 
\begin{align}
\langle v, w \rangle_{\tilde{g}(x)} &= \langle \Pi^{\elCoul}_x(v), \Pi^{\elCoul}_x(w) \rangle_{L^2} \\
&= -\langle d \zeta, b' \rangle_{L^2} - \langle b, d\zeta' \rangle_{L^2} + \langle d\zeta, d\zeta' \rangle_{L^2} + \langle \zeta \phi , \psi' \rangle_{L^2} \notag \\
& \hskip1cm + \langle \psi, \zeta' \phi \rangle_{L^2} + \langle \zeta \phi , \zeta' \phi \rangle_{L^2} \notag  \\
\label{eq:PilCoul-inner-expand} &= \langle d\zeta, d\zeta' \rangle_{L^2} + \langle \zeta \phi , \psi' \rangle_{L^2} + \langle \psi, \zeta' \phi \rangle_{L^2} + \langle \zeta \phi , \zeta' \phi \rangle_{L^2},
\end{align}
where the third equality comes from the fact that $d^* b = d^* b' = 0$.  Let us now collect the relevant bounds on $\zeta$ and $\zeta'$.  By Lemma~\ref{lem:EstimateZeta}, we have that $\| \zeta\|_{L^2}$ and $\| d\zeta \|_{L^2}$ are bounded above by $\frac{C_2}{\lambda} \| \psi \|_{L^2}$ for some constant $C_2 > 0$.  We do not have $L^2$ bounds on $\zeta'$ or $d\zeta'$ in terms of $\frac{1}{\lambda}$, since $w$ is not an element of $(\vml)^\perp_0$.  However, Lemma~\ref{lem:Zeta} still guarantees $\| \zeta' \|_{L^2}$ and $\| d\zeta'\|_{L^2}$ are bounded above by $K \cdot \|\psi' \|_{L^2}$, independent of $\lambda$ and $x \in B(2R)$.  This gives:
\begin{equation}\label{eq:dzeta-dzeta'}
|\langle d\zeta, d\zeta' \rangle_{L^2}| \leq \frac{C_2 \cdot K}{\lambda} \| \psi \|_{L^2} \| \psi' \|_{L^2} \leq \frac{C_2 \cdot K}{\lambda} \| v \|_{L^2} \| w \|_{L^2}.  
\end{equation}
To establish \eqref{eq:g-tilde-projection-bounds}, it remains to bound the other three terms in \eqref{eq:PilCoul-inner-expand}.  

Since $x \in B(2R) \subset W_k$, we have a uniform bound on $\| \phi \|_{C^0}$ independent of $x$, denoted $K_2$.  Therefore, using a similar argument as for \eqref{eq:dzeta-dzeta'}, we obtain bounds
\begin{align*}
|\langle \zeta \phi, \psi' \rangle_{L^2}| &\leq \frac{C_2 \cdot K_2}{\lambda} \| v \|_{L^2} \| w \|_{L^2},\\
|\langle \zeta \phi, \zeta' \phi \rangle_{L^2}| & \leq \frac{C_2 \cdot K \cdot K_2}{\lambda} \| v \|_{L^2} \| w \|_{L^2}.
\end{align*}   

Thus, the proof will be complete if we obtain similar bounds for $\langle \psi, \zeta' \phi  \rangle_{L^2}$.  We must be careful here, since we do not have bounds on $\zeta'$ in terms of $\frac{1}{\lambda}$.  To handle this, we write 
\begin{align*}
\langle \psi, \zeta' \phi \rangle_{L^2} &= \langle D(D^{-1} \psi), \zeta' \phi \rangle_{L^2} \\
&= \langle D^{-1} \psi, D(\zeta' \phi) \rangle_{L^2}, 
\end{align*}
where $D^{-1}$ is defined as in \eqref{eq:D-1}. Since $(b, \psi) \in (\vml)^\perp_0$, we have that $\| D^{-1} \psi \|_{L^2} \leq \frac{1}{\lambda} \| \psi \|_{L^2}$.  Finally, since $\zeta'$ is $L^2_1$-bounded in terms of $\|\psi\|_{L^2}$, we obtain $L^2_1$-bounds on $\zeta' \phi$ (independent of $x$) by Sobolev multiplication.  In other words, there exists a constant $K_3$ such that $\| D( \zeta' \phi) \|_{L^2} \leq K_3 \| \psi' \|_{L^2}$, since $D$ is continuous from $L^2_1$ to $L^2$.  Thus, we conclude that 
$$ |\langle \psi, \zeta' \phi \rangle_{L^2}| \leq \frac{K_3}{\lambda} \| v \|_{L^2} \| w \|_{L^2}, $$
which completes the proof.
\end{proof}

\section{Control near the stationary points}\label{sec:control-near-stationary}
\label{sec:controlnearby}
We are now ready to prove \eqref{eq:dFlambda} in the neighborhoods $\nu_{\epsilon}(\O^j_\lambda)$ of the orbits of stationary points.  

Consider the stationary point $x^j_{\infty}$ of $\Xqgc$.  Since $x^j_{\infty}$ is non-degenerate, we have that when restricted to the anticircular global Coulomb slice, $l+\D_{x^j_{\infty}} c_{\q}$ is an invertible, self-adjoint operator with respect to the metric $\tilde g$.  We will use this to show that $(l + \pml \D_{x^j_\lambda} c_\q)(y)$ grows at least linearly in $y$ when $y$ is $L^2$ orthogonal to the $S^1$-orbit of $x^j_\lambda = (a^j_\lambda,\phi^j_\lambda)$; this will be key for establishing the inequalities \eqref{eq:dFlambda} near $x^j_\lambda$.  

\begin{lemma}\label{lem:approximate-eigenvalue-bounds}
There exists a constant $\mu_j > 0$ with the following property.  For $\lambda \gg 0$ and any $y \in W_1$ with $\Re \langle y , (0, i\phi^j_\lambda) \rangle_{L^2} = 0$, we have 
\begin{equation}\label{eq:muj}
\| (l + \pml \D_{x^j_\lambda} c_\q)(y) \|_{L^2} \geq \mu_j \| y \|_{L^2_1}.
\end{equation}  
\end{lemma}
Of course, this lemma implies that we have bounds $\| (l + \pml \D_{x^j_\lambda} c_\q)(y) \|_{L^2} \geq \mu_j \| y \|_{L^2}$ as well.  
\begin{proof}
For notational convenience, we omit the index $j$ from the argument.  We suppose that this is not true.  Then, we can find a sequence $y_n \in \vmln$ with $\| y_n\|_{L^2_1} = 1$ and $\Re \langle y_n, (0,i \phi_{\lambda_n}) \rangle_{L^2} = 0$ and a sequence $\lambda_n \to \infty$  such that $\| (l + \pmln \D_{x_{\lambda_n}} c_\q)(y_n) \|_{L^2} \to 0$.  Extract a subsequence for which $y_n$ converges in $L^2$ to some $y \in W_0$.  Since $x_{\lambda_n} \to x_\infty$ in $L^2_k$, the continuity of $\D c_\q$ (Lemma~\ref{lem:Dcq}) implies that 
$$ 
\pmln \D_{x_{\lambda_n}} c_\q(y_n) \to \D_{x_\infty} c_\q (y) \text{ in } L^2.
$$
We see that $l(y_n)$ converges in $L^2$.  Because the $L^2_{-1}$ limit of $l(y_n)$ is $l(y)$, we in fact have that $l(y_n)$ converges to $l(y)$ in $L^2$ and thus $y_n \to y$ in $L^2_1$, and $y \neq 0$.  However, since $(l + \pmln \D_{x_{\lambda_n}} c_\q)(y_n)$ converges to 0, we see that $(l + \D_{x_\infty} c_\q)(y) = 0$.  Bootstrapping further shows that $y$ is actually an element of $W_k$.  

First, suppose that $x_\infty$ is reducible.  In this case, we have contradicted the non-degeneracy of $x_\infty$, as $l + \D_{x_\infty} c_\q :W_k \to W_{k-1}$ is invertible (where we are using non-degeneracy in the blow-down).  Now, suppose that $x_\infty$ is an irreducible stationary point of $\Xqgc$.  In this case, we can only say that $l + \D_{x_\infty} c_\q$ is injective on the $\tilde{g}(x_\infty)$-orthogonal complement of the $S^1$ orbit of $x_\infty$.   Thus, we have that $y$ must have some component tangent to the $S^1$ orbit of $x_\lambda$, i.e., a multiple of $(0,i\phi_\infty)$.  Because $\Re \langle y_n, (0,i\phi_{\lambda_n}) \rangle_{L^2} = 0$ for all $n$, we see that $\Re \langle y, (0,i\phi_\infty) \rangle_{L^2} = 0$ as well.  This is a contradiction.      
\end{proof}

Recall that $\epsilon > 0$ was chosen in Section~\ref{sec:Flambda} to be sufficiently small, depending on two requirements from Assumption~\ref{as:1}. We need an additional requirement. For each $j=1, \dots, m$, let $\mu_j$ be the constant as in Lemma~\ref{lem:approximate-eigenvalue-bounds}.

\begin{assumption}
\label{as:2}
For each $j=1, \dots, m$ and $x \in \nu_{3\epsilon}(\O^j)$, we have
$$ \| \D_x c_{\q} - \D_{x^j_{\infty}} c_{\q} \| \leq \frac{\mu_j}{40 C_0^2}.$$
Here, the operator norm is taken in $\Hom(W_0, W_0)$, and $C_0$ is the constant from Proposition~\ref{prop:EquivalentMetrics}. 
\end{assumption}

Note that the existence of such an $\epsilon$ is guaranteed by the continuity of $\D c_{\q}$, as Lemma~\ref{lem:Dcq}, applied for Sobolev index $k-1$ instead of $k$, implies that 
$$\D c_{\q} : W_{k-1} \to  \Hom(W_0, W_0)$$
is continuous. 

\begin{proposition}
\label{prop:cond2}
Fix $\epsilon > 0$ satisfying Assumptions~\ref{as:1} and \ref{as:2}. Then, for all $\lambda \gg 0$, we have 
\begin{equation}\label{eq:Xqmlgc-gradient-inequality}
\frac{1}{4} \| \Xqmlgc \|^2_{\tilde{g}} \leq dF_\lambda (\Xqmlgc) \leq 4 \| \Xqmlgc \|^2_{\tilde{g}}
\end{equation}
at any point  $x \in \vml \cap B(2R)$ which is at $L^2_{k-1}$ distance at most $\epsilon$ from a stationary point of $\Xqmlgc$.  
\end{proposition}

\begin{proof}
In Section~\ref{sec:Flambda} we observed that, because of Assumption~\ref{as:1} (a) on $\epsilon$, the $S^1$-orbits $\O_{\lambda}^j$ of stationary points $x^j_\lambda$ are at least $L^2_{k-1}$ distance $5\epsilon$ apart from each other. Hence, their neighborhoods $\nu_{2\epsilon}(\O_{\lambda}^j)$ are disjoint.  

Let $x$ be in $ \nu_{\epsilon}(\O_{\lambda}^j)$ for some $j$.  Then, the only contribution to the summation in Equation~\eqref{eq:Tlambda} is from that $j$ and, moreover, we have $H^j_{\lambda}(x)=1$.   For simplicity, we will omit the index $j$ from $x^j_\lambda$, $x^j_\infty$, $\O^j$, $\O_{\lambda}^j$, and $\omega^j_\lambda$ for the rest of this section.  Thus, 
$$T_{\lambda}(x) = x + \omega_{\lambda}(x) \cdot (x_{\infty}- x_{\lambda}).$$

We now expand $dF_\lambda(\Xqmlgc)$:
\begin{align*}
(dF_\lambda)_{x}(\Xqmlgc(x)) &= (d\L_\q)_{T_\lambda(x)} (\D_x T_\lambda)(\Xqmlgc(x)) \\
&= \langle (l + c_\q)(T_\lambda(x)), (\D_x T_\lambda)(l + \pml c_\q)(x) \rangle_{\tilde{g}(T_\lambda(x))} \\
&= \langle (l + c_\q)(T_\lambda(x)), (l + \pml c_\q)(x) + (\D_x \omega_\lambda)(l + \pml c_\q)(x) \cdot (x_\infty - x_\lambda) \rangle_{\tilde{g}(T_\lambda(x))}.  
\end{align*}
By the $S^1$-equivariance of $\Xqmlgc$ and the $S^1$-invariance of $F_\lambda$, it suffices to show that \eqref{eq:Xqmlgc-gradient-inequality} holds when $\omega_\lambda(x) = 1$, that is, when the $L^2$-closest point to $x$ on $\O_{\lambda}$ is exactly $x_\lambda$. (Of course, we automatically have $\omega_\lambda(x) = 1$ if $x_\lambda$ is reducible.)  Further, note that if $x_\lambda$ is reducible, then $\D_x \omega_\lambda = 0$.  

We will analyze  $dF_\lambda (\Xqmlgc)$ by linearizing some of the terms in the inner product above about $x_\lambda$.  We begin by linearizing $(l + \pml c_\q)(x)$.  Let $y = x - x_\lambda$.  Since $\omega_\lambda(x) = 1$, we have that $y$ is orthogonal to the $S^1$ orbit of $x_\lambda$.  Since $x_\lambda$ is a stationary point, we have $(l + \pml c_\q)(x_\lambda) = 0$.  Therefore, the difference between $(l + \pml c_\q)(y)$ and $(l + \pml \D_{x_\lambda}c_\q)(x)$ is given by 
\begin{align}
 (l + \pml c_\q)(x_\lambda + y) - (l + \pml \D_{x_{\lambda}} c_\q)(y) &= l(x_\lambda) + \pml(c_\q(x_\lambda + y)) - \pml (\D_{x_{\lambda}} c_\q(y)) \\
\nonumber &= -\pml \left (c_\q(x_\lambda) + c_\q(x_\lambda + y) - \D_{x_{\lambda}} c_\q(y) \right) \\
\label{eq:Stildelambda} &= \int^1_0 \pml(\D_{x_\lambda + ty} c_\q - \D_{x_{\lambda}} c_\q)(y) dt.  
\end{align}
We denote the term in \eqref{eq:Stildelambda} by $\widetilde{S}_\lambda(y)$.  

Similarly, we would like to linearize $(l + c_\q)(T_\lambda(x))$.  We first recall that $T_\lambda(x_\lambda) = x_\infty$, which is a stationary point of $l + c_\q$.  Second, since $\omega_\lambda(x) = 1$, we have that $x - x_\lambda$ is real $L^2$ orthogonal to $i x_\lambda$, and hence $(\D_{x_\lambda + ty} T_\lambda)(y) = y$ for any $t \in [0,1]$.  Therefore, we can write
\begin{align*}
(l + c_\q)(T_\lambda(x_\lambda + y)) &= (l + c_\q)(T_\lambda(x_\lambda + y))- (l+c_\q)(T_\lambda(x_\lambda)) \\
&= \int^1_0 ( l + \D_{T_\lambda(x_\lambda + ty)} c_\q) \circ (\D_{x_\lambda + ty}T_{\lambda})(y) dt \\
&= \int^1_0 ( l + \D_{T_\lambda(x_\lambda + ty)} c_\q) (y) dt\\
&= \int^1_0 ( l + \D_{x_{\infty} + ty} c_\q) (y) dt.
\end{align*}
We obtain
\begin{equation}
\label{eq:Rtildelambda}
(l + c_\q)(T_\lambda(x_\lambda + y)) - (l + \D_{x_{\infty}} c_\q)(y) = \int^1_0 (\D_{x_\infty + ty} c_\q - \D_{x_\infty} c_\q)(y) dt.
\end{equation}
We let $\widetilde{R}_\lambda(y)$ denote the term on the right hand side of \eqref{eq:Rtildelambda}.

Using these two linearizations and the fact that $\langle (1 - \pml_{\tilde{g}(x')})(u), v \rangle_{\tilde{g}(x')} = 0$ for any $v \in \vml$ and $u, x' \in W_0$, we obtain
\begin{align}\label{eq:dFlambda-linearized-pre}
(dF_\lambda)_{x}(\Xqmlgc(x)) = & \langle (l + \pml_{\tilde{g}(T_\lambda(x))} \D_{x_\infty} c_\q)(y) + \widetilde{R}_\lambda(y), \\
& (l + \pml \D_{x_\lambda} c_\q)(y) + \widetilde{S}_\lambda(y) + (\D_x \omega_\lambda)(l + \pml c_\q)(x) \cdot (x_\infty - x_\lambda)\rangle_{\tilde{g}(T_\lambda(x))}, \notag
\end{align}
where $\widetilde{R}_\lambda(y)$ and $\widetilde{S}_\lambda(y)$ are as defined above.  At this point, it is still too difficult to compare the linear terms to understand this inner product.  Therefore, we will alter our expressions so that the leading terms align.  Write  
\begin{equation}\label{eq:dFlambda-linearized}
(dF_\lambda)_{x}(\Xqmlgc(x)) =  \langle (l + \pml \D_{x_\lambda} c_\q)(y) + R_\lambda(y) , 
(l + \pml \D_{x_\lambda} c_\q)(y) + S_\lambda(y) \rangle_{\tilde{g}(T_\lambda(x))},
\end{equation}
where 
\begin{align}
\label{eq:Rlambda} 
R_\lambda(y) &= \widetilde{R}_\lambda(y) + (\pml_{\tilde{g}(T_\lambda(x))} \D_{x_\infty}c_\q - \pml \D_{x_\lambda} c_\q)(y),  \\
\label{eq:Slambda} 
S_\lambda(y) &= \widetilde{S}_\lambda(y) + (\D_x \omega_\lambda)(l + \pml c_\q)(x) \cdot (x_\infty - x_\lambda).
\end{align}

In Lemmas~\ref{lem:RStildelambda-bounds}, \ref{lem:Rlambda}, and \ref{lem:Slambda} below, we will prove that for our choice of $\epsilon$ as in Assumption~\ref{as:2} and for $\lambda \gg 0$,
\begin{align}
\label{eq:Rtildelambda-bounds}\| \widetilde{R}_\lambda(y) \|_{L^2} &\leq \frac{1}{20 C^2_0} \| (l + \pml \D_{x_\lambda} c_\q)(y) \|_{L^2}  \\
\label{eq:Stildelambda-bounds}\| \widetilde{S}_\lambda(y) \|_{L^2} &\leq \frac{1}{20 C^2_0} \| (l + \pml \D_{x_\lambda} c_\q)(y) \|_{L^2} \\
\label{eq:Rlambda-bounds} 
\| R_\lambda(y) \|_{L^2} &\leq \frac{1}{10C_0^2} \| (l + \pml \D_{x_\lambda} c_\q) (y) \|_{L^2}  \\
\label{eq:Slambda-bounds} 
\| S_\lambda(y) \|_{L^2} &\leq \frac{1}{10  C_0^2} \| (l + \pml \D_{x_\lambda} c_\q) (y) \|_{L^2}  
\end{align}
where $C_0$ is the constant of equivalency between the $\tilde{g}$- and $L^2$-metrics from Proposition~\ref{prop:EquivalentMetrics}.  This will imply the analogous inequalities with $L^2$ replaced by $\tilde{g}(T_\lambda(x))$ and without the $C_0^2$ term.  Let us see why this will complete the proof of the proposition.  
 
Note that in an inner product space with vectors $u, v, w, \tilde{w}$ such that 
\begin{equation}\label{eqn:uvw}
\| v \| \leq \frac{1}{10} \| u\|, \ \| w \| \leq \frac{1}{10} \| u\|, \ \| \tilde{w} \| \leq \frac{1}{20} \| u\|,
\end{equation}
we have the inequalities
\begin{align*}
\frac{19}{20} \| u \| & \leq \| u + \tilde{w} \| \leq \frac{21}{20} \| u \| \\
\left(\frac{9}{10} \right)^2 \| u \|^2 &\leq \langle u + v, u + w \rangle \leq \left(\frac{11}{10}\right)^2 \| u \|^2,
\end{align*}
from which we can deduce 
$$
\frac{1}{2} \| u + \tilde{w} \|^2 \leq \langle u + v, u + w \rangle \leq 2 \| u + \tilde{w} \|^2. 
$$
For our case, we take 
$$ 
u =( l + \pml \D_{x_\lambda} c_\q)(y), \ v = R_\lambda(y), \ w = S_\lambda(y), \ \tilde{w} = \widetilde{S}_\lambda(y),
$$
and $\tilde{g}(T_\lambda(x))$ as the inner product.  Note that $u + \tilde{w} = \Xqmlgc(x)$.  The inequalities \eqref{eq:Stildelambda-bounds}-\eqref{eq:Slambda-bounds} imply that \eqref{eqn:uvw} holds and thus 
$$
\frac{1}{2} \| \Xqmlgc \|^2_{\tilde{g}(T_\lambda(x))} \leq dF_\lambda (\Xqmlgc) \leq 2 \| \Xqmlgc \|^2_{\tilde{g}(T_\lambda(x))}.
$$  
Since $T_\lambda(x) \to x$, the constant of equivalency between the $\tilde{g}(T_\lambda(x))$- and $\tilde{g}(x)$-metrics is at most $\sqrt{2}$ for $\lambda \gg 0$.  From this, we obtain that
$$
\frac{1}{4} \| \Xqmlgc \|^2_{\tilde{g}} \leq dF_\lambda (\Xqmlgc) \leq 4 \| \Xqmlgc \|^2_{\tilde{g}},
$$
completing the proof.
\end{proof}

The rest of the section is now devoted to proving \eqref{eq:Rtildelambda-bounds}- \eqref{eq:Slambda-bounds}. We do this by a series of lemmas.

\begin{lemma}\label{lem:RStildelambda-bounds}
For $\lambda \gg 0$ and any $y \in W_1$ orthogonal to the $S^1$ orbit of $x_\lambda$, we have the following inequalities 
\begin{align*}
\| \widetilde{R}_\lambda(y) \|_{L^2} &\leq \frac{1}{20 C^2_0} \| (l + \pml \D_{x_\lambda} c_\q)(y) \|_{L^2}  \\
\| \widetilde{S}_\lambda(y) \|_{L^2} &\leq \frac{1}{20 C^2_0} \| (l + \pml \D_{x_\lambda} c_\q)(y) \|_{L^2}.
\end{align*}
\end{lemma}
\begin{proof}
We will prove the desired inequality for the case of $\widetilde{S}_\lambda$.  The case of $\widetilde{R}_\lambda$ is similar.  By Lemma~\ref{lem:approximate-eigenvalue-bounds}, $\| (l + \pml \D_{x_\lambda} c_\q)(y) \|_{L^2} \geq \mu\| y \|_{L^2}$.  Therefore, it suffices to establish the upper bound 
\begin{equation}\label{eq:Slambda-eigenvalue-bound}
\| \widetilde{S}_\lambda (y) \|_{L^2} \leq \frac{\mu}{20 C^2_0}\| y \|_{L^2}.
\end{equation}

Since $\| y \|_{L^2_{k-1}} \leq \epsilon$ and $\| x_{\lambda}  -  x_\infty \|_{L^2_{k-1}} \leq \epsilon$, we get that $x_{\lambda} + ty$ is in $\nu_{3\epsilon}(\O)$ for all $t \in [0,1]$. Assumption~\ref{as:2} implies that $\| \D_{x_\lambda + ty} c_\q- \D_{x_{\lambda}} c_\q \| \leq  \frac{\mu}{20 C^2_0}$.  We obtain 
\begin{align*}
\| \int^1_0 \pml( \D_{x_\lambda + ty} c_\q - \D_{x_\lambda} c_\q) (y) dt \|_{L^2} &\leq \left ( \int^1_0 \| \pml \| \cdot \| \D_{x_\lambda + ty} c_\q - \D_{x_\lambda} c_\q\| dt \right ) \cdot \| y \|_{L^2} \\
&\leq \frac{\mu}{20 C^2_0} \|y \|_{L^2},
\end{align*}
which establishes \eqref{eq:Slambda-eigenvalue-bound}.
\end{proof}

\begin{lemma}\label{lem:Rlambda}
For $\lambda \gg 0$, the inequality \eqref{eq:Rlambda-bounds} holds.
\end{lemma}
\begin{proof}
Using Lemma~\ref{lem:RStildelambda-bounds} and the definition of $R_\lambda$ in \eqref{eq:Rlambda}, it suffices to prove that for $\lambda \gg 0$,
$$
\| (\pml_{\tilde{g}(T_\lambda(x))} \D_{x_\infty} c_\q)(y) - (\pml \D_{x_\lambda} c_\q)(y) \|_{L^2} \leq \frac{1}{20 C^2_0} \| (l + \pml \D_{x_\lambda} c_\q) (y) \|_{L^2}.
$$

The proof will be similar to that of Lemma~\ref{lem:RStildelambda-bounds}.  By Lemma~\ref{lem:approximate-eigenvalue-bounds}, $\| (l + \pml \D_{x_\lambda} c_\q)(y) \|_{L^2} \geq \mu \| y \|_{L^2}$.  Therefore, it suffices to establish the upper bound 
\begin{equation}
\label{eq:forty}
 \| \pml_{\tilde{g}(T_\lambda(x))} \D_{x_\infty} c_\q - \pml \D_{x_\lambda} c_\q\|\leq \frac{\mu}{20 C^2_0},
\end{equation}
where the norm is computed as an operator from $W_0$ to $W_0$.  In view of Proposition~\ref{prop:EquivalentMetrics}, we have that as an operator from $L^2$ to $L^2$, $\| \pml_{\tilde{g}(T_\lambda(x))} \| \leq C^2_0$, as $\pml_{\tilde{g}}$ is a $\tilde{g}$-orthogonal projection.  Therefore, 
\begin{align*}
\| \pml_{\tilde{g}(T_\lambda(x))} \D_{x_\infty} c_\q - \pml \D_{x_\lambda} c_\q \| &\leq 
\| \pml_{\tilde{g}(T_\lambda(x))} (\D_{x_\infty} c_\q -  \D_{x_\lambda} c_\q)\|+ \|\pml_{\tilde{g}(T_\lambda(x))} \D_{x_\lambda} c_\q - \pml \D_{x_\lambda} c_\q\| \\
&\leq C^2_0 \| \D_{x_\infty} c_\q - \D_{x_\lambda} c_\q \| + \| \pml_{\tilde{g}(T_\lambda(x))} - \pml \| \cdot \| \D_{x_\lambda} c_\q \|.
\end{align*}  
Since $x_\lambda \to x_\infty$ in $L^2_k$ norm, the continuity of $\D c_\q$ implies that we have uniform bounds on $\| \D_{x_\lambda} c_\q \|$ independent of $\lambda$; further, this implies that $\| \D_{x_\infty} c_\q - \D_{x_\lambda} c_\q \| \to 0$ as $\lambda \to \infty$.  Finally, by Proposition~\ref{prop:pgpl}, $\| \pml_{\tilde{g}(T_\lambda(x))} - \pml \| \to 0$ as $\lambda \to \infty$.  Thus, we conclude that for $\lambda \gg 0$, the operator norm of $\pml_{\tilde{g}(T_\lambda(x))} \D_{x_\infty} c_\q - \pml \D_{x_\lambda} c_\q$ is at most $\frac{\mu}{40  C_0^2}$. This proves \eqref{eq:forty}, which we saw earlier was sufficient to complete the proof. 
\end{proof}

\begin{lemma}\label{lem:Slambda}
For $\lambda \gg 0$, the inequality \eqref{eq:Slambda-bounds} holds. 
\end{lemma}
\begin{proof}
Using Lemma~\ref{lem:RStildelambda-bounds} and the definition of $S_\lambda$ in \eqref{eq:Slambda}, it suffices to prove that for $\lambda \gg 0$,
\begin{equation}\label{eq:Dxomega-Xqmlgc}
\| (\D_x \omega_\lambda)(l + \pml c_\q(x)) \cdot (x_\infty - x_\lambda) \|_{L^2} \leq \frac{1}{20 C^2_0}\| (l + \pml \D_{x_\lambda} c_\q)(y) \|_{L^2}.
\end{equation}
By the definition of $\widetilde{S}_\lambda$, we see that 
$$
(\D_x \omega_\lambda)(l + \pml c_\q(x)) = (\D_x \omega_\lambda)\bigl ( ( l + \pml \D_{x_\lambda} c_\q)(y) + \widetilde{S}_\lambda(y) \bigr ).   
$$
Recall from \eqref{eq:Dxomega-bounds}, $|(\D_x \omega_\lambda)(v)| \leq C' \| v \|_{L^2}$, for a constant $C'$ independent of $x \in B(2R)$. Combining this inequality with Lemma~\ref{lem:RStildelambda-bounds}, we have 
\begin{align*}
|(\D_x \omega_\lambda)(l + \pml c_\q(x))| \leq C'(1 + \frac{1}{20 C^2_0}) \| (l + \pml \D_{x_\lambda} c_\q)(y) \|_{L^2}.
\end{align*} 
Since $\| x_\infty - x_\lambda \|_{L^2} \to 0$, we can choose $\lambda \gg 0$ such that 
$$
|(\D_x \omega_\lambda)(l + \pml c_\q(x))| \cdot \| x_\infty - x_\lambda \|_{L^2} \leq \frac{1}{20 C^2_0} \| (l + \pml \D_{x_\lambda} c_\q)(y) \|_{L^2}.  
$$ 
This establishes \eqref{eq:Dxomega-Xqmlgc}, and the proof is complete.  
\end{proof}

\section[Consequences]{Proposition~\ref{prop:LqQuasi} and its consequences} \label{sec:consequences}
Proposition~\ref{prop:LqQuasi} now follows by combining Propositions~\ref{prop:intermediate} and~\ref{prop:cond2}. Let us also give a few corollaries, which will prove useful in the next sections.

From Proposition~\ref{prop:LqQuasi}, together with the equivalence between the $\tilde{g}$ and $L^2$ metrics (Proposition~\ref{prop:EquivalentMetrics}), we obtain 
\begin{corollary}\label{cor:Flambda-gtilde-bounds}
There exists $C_0 > 0$ such that for $\lambda \gg 0$, 
\begin{equation}\label{eq:Flambda-gtilde-bounds}
\frac{1}{4C_0} \| \Xqmlgc \|^2_{L^2} \leq dF_\lambda(\Xqmlgc) \leq 4C_0 \| \Xqmlgc \|^2_{L^2}
\end{equation}
at any $x \in \vml \cap B(2R)$.
\end{corollary}

We have shown that $\Xqmlgc$ is a Morse equivariant quasi-gradient vector field as in Definition~\ref{def:eqgv} (on a non-compact manifold, as discussed in Section~\ref{sec:combinedMorse}). Lemma~\ref{lem:ao2} implies the following:
\begin{corollary}
\label{cor:endpointsBlowDown}
Let $\gamma: \R \to \vml \cap B(2R)$ be a flow line of $\Xqmlgc$, for $\lambda \gg 0$. Then, $\lim_{t\to -\infty} [\gamma(t)]$ and $\lim_{t\to +\infty} [\gamma(t)]$ exist in $(\vml \cap B(2R))/S^1$, and they are both projections of stationary points of $\Xqmlgc$. 
\end{corollary}

There is also the similar result in the blow-up, which is a consequence of Lemma~\ref{lem:ao3}:
\begin{corollary}
\label{cor:endpointsBlowUp}
Let $[\gamma]: \R \to (\vml \cap B(2R))^{\sigma}/S^1$ be a flow line of $\Xqmlgcsigma$, for $\lambda \gg 0$. Then, $\lim_{t\to -\infty} [\gamma(t)]$ and $\lim_{t\to +\infty} [\gamma(t)]$ exist in $(\vml \cap B(2R))^{\sigma}/S^1$, and they are both projections of stationary points of $\Xqmlgcsigma$. 
\end{corollary}

In fact, one can say more about the limiting behavior of trajectories of $\Xqmlgcsigma$.  Recall that in classical Morse theory, trajectories converge to stationary points with exponential decay (see for example \cite[Section 10.2.b]{AudinDamian}).  A similar argument works for Morse equivariant quasi-gradients as well.  In particular, one obtains that trajectories $[\gamma]: \R \to (\vml \cap B(2R)^\sigma)/S^1$ of $\Xqmlagcsigma$ from $[x_\lambda]$ to $[y_\lambda]$ must be contained in $\B^{\gCoul,\tau}_k([x_\lambda], [y_\lambda])$.  In Chapter~\ref{sec:trajectories1} we will compute this exponential decay in terms of $F_\lambda$ more explicitly and obtain bounds uniform in $\lambda$.     

If $I \subset \R$ is an interval and $\gamma: I \to \vml \cap B(2R)$ is a trajectory of $\Xqmlgc$, its energy is defined\footnote{Our definition of energy differs from that of \cite{KMbook} by a factor of two.} to be
\begin{equation}
\label{eq:energylambda}
 \E(\gamma) = \int_I \left\|\frac{d\gamma}{dt} \right \|_{L^2(Y)}^2 \! \!\!dt \ =\int_I \| \Xqmlgc(\gamma(t)) \|_{L^2(Y)}^2 dt.
 \end{equation}
Note that our Corollary~\ref{cor:Flambda-gtilde-bounds} gives a quantitative version of the quasi-gradient condition $dF_{\lambda}(\Xqmlgc) \geq 0$. A consequence is the following result, which will be useful to us in Chapter~\ref{sec:trajectories1} when we establish exponential decay results for trajectories. It says that the energy of an approximate trajectory is commensurable with the drop in $F_{\lambda}$:
\begin{corollary}
\label{cor:trajectory-Flambda}
There exists $C_0 > 0$ such that, for any $\lambda \gg 0$ and any closed interval $[t_1, t_2] \subset \R$, the following holds. If $\gamma: [t_1, t_2] \to \vml \cap B(2R)$ is a trajectory of $\Xqmlgc$, then
\begin{equation}
\label{eq:trajectory-Flambda} 
\frac{1}{4C_0} \E(\gamma) \leq F_\lambda(\gamma(t_2)) - F_\lambda(\gamma(t_1)) \leq 4C_0 \E(\gamma).
\end{equation}
\end{corollary}
 
\begin{proof} 
Note that
$$F_\lambda(\gamma(t_2)) - F_\lambda(\gamma(t_1))= \int_{t_1}^{t_2} dF_\lambda\Bigl(\frac{d}{dt} \gamma(t)\Bigr ) dt = \int_{t_1}^{t_2} dF_\lambda\bigl(\Xqmlgc( \gamma(t)) \bigr ) dt  .$$
Therefore, the result follows from \eqref{eq:energylambda} and Corollary~\ref{cor:Flambda-gtilde-bounds}.
 \end{proof}

\chapter{Gradings}\label{sec:gradings}

Having established in the previous chapter that $\Xqmlgc$ is a Morse equivariant quasi-gradient for $\lambda = \llambda_i \gg 0$, we are able to define the chain groups of a Morse complex for $\Xqmlagcsigma$ on $(B(2R) \cap \vml)^\sigma/S^1$ as described in Section~\ref{subsec:CircleMorse}.  (After establishing the Morse-Smale condition in Chapter~\ref{sec:MorseSmale}, we will see this indeed gives a complex.)  In this chapter, we relate the gradings of stationary points of $\Xqmlagcsigma$ with the gradings of the stationary points of $\Xqagcsigma$.  An important subtlety here is that stationary points of $\Xqmlagcsigma$ can be thought of as living in the infinite-dimensional manifold $W^\sigma/S^1$ like those of $\Xqagcsigma$ or in the smaller finite-dimensional manifold $(B(2R) \cap \vml)^\sigma/S^1$, and thus there are two ways to define relative gradings.   In Section~\ref{sec:gradingsstationarypoints}, we equate these two relative gradings and relate them to the relative gradings for stationary points of $\Xqagcsigma$ in an appropriate grading range.  In Section~\ref{sec:absgradings}, we define an absolute grading on stationary points of $\Xqmlagcsigma$ (a shift of the Morse index) which we relate to the absolute grading $\gr_{\mathbb{Q}}$ on the stationary points of $\Xqagcsigma$.  

Thus, the work of Chapter~\ref{sec:criticalpoints} together with the claimed grading correspondences will establish an identification between the chain groups of the Morse complex for $\Xqmlagcsigma$ with the monopole Floer complex in an appropriate grading range.  This is summarized for the reader's benefit in Section~\ref{sec:conclusions}.  

\section{Relative gradings of stationary points}\label{sec:gradingsstationarypoints}
We would like to define relative gradings on stationary points of $\Xqmlagcsigma$ in $\N/S^1$ using the analogs of the discussions in Section~\ref{sec:modifications} and Section~\ref{sec:GradingsCoulomb}.  

First, recall from \eqref{eq:Fqmlgctau} that we have a section
$$ \F^{\gCoul, \tau}_{\qml}: \tC^{\gCoul, \tau}(Z) \to \V^{\gCoul, \tau}(Z).$$

We can take the covariant derivatives of this section, $\D^{\tilde{g},\tau}_{\gamma} \Fqmlgctau$ and $\D^\tau_{\gamma} \Fqmlgctau$, just as in \eqref{eq:tangerine} and \eqref{eq:tangerine2}. These agree in the particular case where $\gamma$ is a trajectory of $\Xqmlgcsigma$.  As discussed earlier, it will be easier to work with $\D^\tau_{\gamma} \Fqmlgctau$, so we focus on this operator.

\begin{lemma}
\label{lem:GrPres}
For $1 \leq j \leq k$ and $\lambda \gg 0$, the following is true: for any $[x_\infty],[y_\infty] \in \Crit_{\N}$ and each path $[\gamma_\lambda] \in \B^{\gCoul,\tau}_k([x_\lambda],[y_\lambda])$ with representative $\gamma_\lambda \in \C^{\gCoul,\tau}_k(x_\lambda, y_\lambda)$, the operator 
$$(\D^\tau_{\gamma} \Fqmlgctau)|_{\K^{\gCoul,\tau}_{j,\gamma_\lambda}} : \K^{\gCoul,\tau}_{j,\gamma_{\lambda}} \to \V^{\gCoul,\tau}_{j-1,\gamma_\lambda}$$
is Fredholm, with index equal to $\gr([x_\infty],[y_\infty])$.
\end{lemma}
Note that here we do not need the assumption that $\lambda = \llambda_i$ for the nondegeneracy of $[x_\lambda], [y_\lambda]$, since we have that $[x_\lambda], [y_\lambda]$ are contained in $\N/S^1$ (cf. Proposition~\ref{prop:ND}).  
\begin{proof}
Recall from Proposition~\ref{prop:FredholmCoulomb} that the relative grading between $[x_\infty]$ and $[y_\infty]$ is given by the index of $(\D^\tau_{\gamma} \Fqagctau)|_{\K^{\gCoul,\tau}_{j,\gamma}}$.  Recall further from Lemma~\ref{lem:allsurjective} that the index of $(\D^\tau_{\gamma} \Fqagctau)|_{\K^{\gCoul,\tau}_{j,\gamma}}$ is the same as that of the operator 
 $$\Qhat_{\gamma}^{\gCoul} : \T^{\tau}_{j, \gamma} \to  \T^{\tau}_{j-1, \gamma}$$
from \eqref{eq:newextension}. We can define a similar operator 
$$\Qhat_{\gamma_\lambda, \qml}^{\gCoul}: \T^{\tau}_{j, \gamma_\lambda} \to  \T^{\tau}_{j-1, \gamma_\lambda},$$ by using the perturbation $\qml$ instead of $\q$, and the path $\gamma_\lambda$ instead of $\gamma$. The same arguments as in the proof of Proposition~\ref{prop:FredholmCoulomb} show that $\Qhat_{\gamma_\lambda, \qml}^{\gCoul}$ is Fredholm, so is $(\D^\tau_{\gamma_\lambda} \Fqmlgctau)|_{\K^{\gCoul,\tau}_{j,\gamma_\lambda}}$, and they have the same index. Thus, it remains to show that the operators $\Qhat_{\gamma_\lambda, \qml}^{\gCoul}$ and $\Qhat_{\gamma}^{\gCoul} $ have the same index.

Standard arguments show that the index of the operator $\Qhat^{\gCoul}_\gamma$ is independent of the choice of $\gamma$, for $\gamma$ a smooth path which is asymptotic to $x_\infty$ and $y_\infty$.  Similarly, it suffices to look at the index of $\Qhat^{\gCoul}_{\gamma_\lambda, \qml}$ for one choice of $\gamma_\lambda$.  We choose $\gamma$ and $\gamma_\lambda$ to be constant outside of the interval $[-T,T]$.  As in \eqref{eq:tQg}, we write 
\begin{align*}
\Qhat^{\gCoul}_\gamma &= \frac{d}{dt} + L_0 + \hat{h}^{\gCoul}_t, \\
\Qhat^{\gCoul}_{\gamma_\lambda, \qml} &= \frac{d}{dt} + L_0 + \hat{h}^{\gCoul,\lambda}_t.
\end{align*}
We will define an intermediate operator $\Qhat^{\operatorname{int}, \lambda}$ interpolating between these two operators.  Choose $\beta(t)$ a smooth bump function which is 0 outside of $[-T-2,T+2]$ and is identically 1 in $[-T-1,T+1]$.  Define 
$$
\Qhat^{\operatorname{int},\lambda} = (1-\beta(t)) \cdot \Qhat^{\gCoul}_{\gamma_\lambda, \qml} + \beta(t) \cdot \Qhat^{\gCoul}_{\gamma}.
$$  
We will compare $\Qhat^{\operatorname{int},\lambda}$ to $\Qhat^{\gCoul}_{\gamma_\lambda, \qml}$ and $\Qhat^{\gCoul}_\gamma$.    

We begin by comparing $\Qhat^{\operatorname{int},\lambda}$ to $\Qhat^{\gCoul}_{\gamma_\lambda, \qml}$.  We will show that they differ by a compact operator.  Notice that 
$$
\Qhat^{\gCoul}_{\gamma_\lambda, \qml}  - \Qhat^{\operatorname{int},\lambda} = \beta(t) (\Qhat^{\gCoul}_{\gamma_\lambda, \qml} - \Qhat^{\gCoul}_\gamma) = \beta(t) (\hat{h}^{\gCoul,\lambda}_t - \hat{h}^{\gCoul}_t).
$$  
For each $t$, $\hat{h}^{\gCoul,\lambda}_t - \hat{h}^{\gCoul}_t$ is continuous as a self-map from $L^2_n(Y; iT^*Y \oplus \Spin \oplus i\R)$ for each $1 \leq n \leq j$; therefore, this is compact as a map from $L^2_n(Y)$ to $L^2_{n-1}(Y)$.  Since $\beta(t)$ is smooth and the terms $\hat{h}^{\gCoul,\lambda}_t$ and $\hat{h}^{\gCoul}_t$ vary smoothly in $t$, we have that the four-dimensional map $\beta(t) (\hat{h}^{\gCoul,\lambda}_t - \hat{h}^{\gCoul}_t)$ from $L^2_n(I \times Y)$ to $L^2_{n-1}(I \times Y)$ is compact whenever $I$ is a compact interval.  While the inclusion of $L^2_n(\R \times Y)$ into $L^2_{n-1}(\R \times Y)$ is not compact since $\R \times Y$ is unbounded, we have that $\beta(t) (\hat{h}^{\gCoul,\lambda}_t - \hat{h}^{\gCoul}_t)$ is compact as an operator from $L^2_n(\R \times Y)$ to $L^2_{n-1}(\R \times Y)$ because $\beta$ is compactly supported.  Therefore, $\Qhat^{\operatorname{int},\lambda}$ and $\Qhat^{\gCoul}_{\gamma_\lambda, \qml}$ differ by a compact operator, and hence have the same index.

We now compare $\Qhat^{\operatorname{int},\lambda}$ to $\Qhat^{\gCoul}_{\gamma}$.  We write 
$$\Qhat^{\operatorname{int},\lambda} - \Qhat^{\gCoul}_{\gamma}  = (1 - \beta(t))(\Qhat^{\gCoul}_{\gamma_\lambda, \qml} - \Qhat^{\gCoul}_{\gamma}).  $$ 
Let us analyze this difference more carefully.  When this operator is non-zero (i.e., when $|t| >T + 1$), we have that $\Qhat^{\operatorname{int},\lambda}  - \Qhat^{\gCoul}_{\gamma}$ is given by $(1-\beta(t)) (\hat{h}^{\gCoul,\lambda}_{-\infty} - \hat{h}^{\gCoul}_{-\infty})$ when $t < -T-1$ and $(1-\beta(t)) (\hat{h}^\lambda_{+\infty} - \hat{h}_{+\infty})$ when $t > T + 1$.  We will argue that $(1-\beta(t)) (\hat{h}^{\gCoul,\lambda}_{-\infty} - \hat{h}^{\gCoul}_{-\infty})$, considered as operators from $L^2_j(\R \times Y)$ to $L^2_j(\R \times Y)$, converge to 0 in operator norm as $\lambda \to \infty$.  The same argument will apply for $\hat{h}^{\gCoul,\lambda}_{+\infty} - \hat{h}^{\gCoul}_{+\infty}$, and this will imply that $\Qhat^{\operatorname{int},\lambda} - \Qhat^{\gCoul}_{\gamma}$ converges to $0$ in operator norm as $\lambda \to \infty$.  In turn, this will imply that for $\lambda \gg 0$, the operators $\Qhat^{\gCoul}_{\gamma}$ and $\Qhat^{\operatorname{int},\lambda}$ must have the same index, completing the proof.  

Because $\gamma$ and $\gamma_\lambda$ limit to stationary points of $\Xqgcsigma$ and $\Xqmlgcsigma$ respectively, \eqref{eq:4D-endpoint-Hess} yields that for $t  < -T - 1$,
\begin{equation}\label{eq:hlambdagCoul}
(1-\beta(t)) (\hat{h}^{\gCoul,\lambda}_{-\infty} - \hat{h}^{\gCoul}_{-\infty}) = (1-\beta(t)) (\widehat{\Hess}^{\tilde{g},\sigma}_{\qml,x_\lambda} - \widehat{\Hess}^{\tilde{g},\sigma}_{\q,x_\infty})
\end{equation}
and similarly for $t > T + 1$.  

Recall that $\Hess^\sigma_{\q,x_\infty} = \mathcal{H}^\sigma_{x_\infty}$ by Lemma~\ref{lem:fakehessian}\eqref{fakehessian-realstationary}.  Combining the arguments of Lemma~\ref{lem:fakehessian} with the formula for $\Hess^\sigma_{\q,x_\infty}$ in \cite[p.209]{KMbook}, the operator norm (in three-dimensions) of the difference 
$$
\widehat{\Hess}^{\tilde{g},\sigma}_{\qml,x_\lambda} - \widehat{\Hess}^{\tilde{g},\sigma}_{\q,x_\infty}
$$ 
from $L^2_n$ to $L^2_{n-1}$ can be bounded in terms of the $L^2_n$ norms of $x_\lambda - x_\infty$, $\D_{\pi(x_\infty)} (\Pi^{\elCoul} \circ c_\q) - \D_{\pi(x_\lambda)} (\Pi^{\elCoul} \circ \pml c_\q)$, and $\D_{x_\infty} \widetilde{\Pi^{\elCoul} \circ c_\q}^1 - \D_{x_\lambda} \widetilde{ \Pi^{\elCoul} \circ \pml c_\q}^1$, where $\pi(x_\infty), \pi(x_\lambda)$ denote the images in the blow-down.  Note that since the two extended Hessians have the same first order term, $L_0$, the difference necessarily lands in $L^2_j(Y)$; more generally this difference induces a bounded operator from $L^2_n(Y)$ to $L^2_n(Y)$ for any $1 \leq n \leq j$.  For each $1 \leq n \leq j$, the terms mentioned above converge to 0 in $L^2_n(Y)$ as $\lambda \to \infty$: $x_\lambda - x_\infty$ converges to 0 by the implicit function theorem, while the other two terms converge to zero by combining the arguments of Lemma~\ref{lem:fam} together with the continuity of $\Pi^{\elCoul}$ and its derivatives (Lemma~\ref{lem:elc}).  It follows that the norm of 
$$A_\lambda := \hat{h}^{\gCoul,\lambda}_{-\infty} - \hat{h}^{\gCoul}_{-\infty},$$ thought of as an operator from $L^2_n(Y)$ to $L^2_{n-1}(Y)$, converges to zero as $\lambda \to \infty$.  We must extend this to obtain the analogous convergence on $\R \times Y$.   

Precisely, we seek to show that the norm of $(1-\beta(t))A_\lambda$, thought of as an operator from $L^2_j(\R \times Y)$ to $L^2_{j-1}(\R \times Y)$, converges to zero as $\lambda \to \infty$. Since $\beta(t)$ is smooth and compactly supported, it suffices to consider $A_\lambda$ itself.  Using the analysis on the operator norms from $L^2_n(Y)$ to $L^2_{n-1}(Y)$ above, we have that for $\eta \in L^2_j(\R \times Y)$, 
\begin{align*}
\|A_\lambda (\eta(t)) \|^2_{L^2_{j-1}(\R \times Y)} &= \sum^{j-1}_{n = 0} \int_\R \| A_\lambda (\eta^{(n)}(t)) \|^2_{L^2_{j-n-1}(Y)} dt \\
& \leq \sum^j_{n=0} \int_\R C_\lambda \| \eta^{(n)}(t) \|^2_{L^2_{j-n}(Y)} dt \\
& = C_\lambda \| \eta(t) \|^2_{L^2_{j}(\R \times Y)}, 
\end{align*}
where $C_\lambda \to 0$ as $\lambda \to \infty$.  This gives the desired convergence in operator norm and completes the proof.
\end{proof}

\begin{remark}
We expect that, with more work, one could show that $\Qhat_{\gamma_\lambda, \qml}^{\gCoul}$ converges to $\Qhat_{\gamma}^{\gCoul} $ in operator norm; and hence the two operators have the same index. This would give an alternate proof of the above lemma, without using the intermediate operator $\Qhat^{\operatorname{int},\lambda}$. 
\end{remark}

We define a relative grading on the stationary points of $\Xqmlagcsigma$ in $\N/S^1$ by 
\begin{equation}\label{eq:approx-rel-grading}
\gr([x_\lambda], [y_\lambda]) = \ind(\D^\tau_{\gamma} \Fqmlgctau|_{\K^{\gCoul,\tau}_{k,\gamma}}),
\end{equation}
for any path $\gamma$ from $x_\lambda$ to $y_\lambda$. 

Lemma~\ref{lem:GrPres} immediately implies the following

\begin{corollary}
\label{cor:GrPres}
The correspondence $\Xi_{\lambda}: \Crit_{\N}^{\lambda} \to \Crit_{\N}$ from Corollary~\ref{cor:corresp} preserves relative gradings.
\end{corollary}

Recall from Corollary~\ref{cor:IsFinite} that for $\lambda \gg 0$, the stationary points of $\Xqmlagcsigma$ contained in $\N/S^1$ are contained in $(\vml)^\sigma$, so we are able to treat these stationary points as living in finite dimensions.  We would like to relate the above grading computations to the Morse indices in finite dimensions.   We remind the reader that although $\Xqmlgc$ is a Morse  equivariant quasi-gradient, $\Xqmlagcsigma$ is not a Morse quasi-gradient itself in any obvious way.  However, as discussed in Section~\ref{subsec:CircleMorse}, we can still do the analogous constructions in Morse homology.     

\begin{proposition}
\label{prop:RelationFinite}
Let $[x_\lambda], [y_\lambda] \in \N/S^1$ be stationary points of $\Xqmlagcsigma$.  Then, $\gr([x_\lambda],[y_\lambda])$, where this grading is computed in infinite dimensions, is equal to the difference in gradings of $[x_\lambda],[y_\lambda]$ thought of as stationary points of $\Xqmlagcsigma$ restricted to $(B(2R) \cap \vml)^\sigma/S^1$.   
\end{proposition}
\begin{proof}
We begin by studying the relevant operators.  Let 
$$\gamma: \rr \to (B(2R) \cap \vml)^{\sigma}, \ \gamma(t)=(a(t), s(t), \phi(t))$$ be a path between stationary points  $x_\lambda, y_\lambda \in (B(2R) \cap \vml)^\sigma$. Let also $Z = \R \times Y$.

On the one hand, we have the Fredholm operator
\begin{equation}
\label{eq:Qgammagc}
 Q_{\gamma}^{\gCoul} = \D^\tau_{\gamma} \Fqmlgctau \oplus \dd^{\gCoul, \tau, \tilde{\dagger}}_{\gamma} : \T^{\gCoul, \tau}_{j, \gamma}(x_\lambda,y_\lambda) \to \V^{\gCoul,\tau}_{j-1,\gamma}(Z) \oplus L^2_{j-1}(\R; i\R).
 \end{equation}
 This is similar to the operator $Q_{\gamma}^{\gCoul} $ from \eqref{eq:Qggc}, but uses the perturbation $\qml$ instead of $\q$. (For simplicity, we do not include the perturbation in the notation.)
 
We now restrict $Q_{\gamma}^{\gCoul} $ to
\begin{equation}
\label{eq:Tglambda}
 \T^{\gCoul, \lambda, \tau}_{j, \gamma} = \{ (b, r, \psi) \in \T^{\gCoul,\tau}_{j,\gamma}(x_\lambda, y_\lambda) \mid (b(t), \psi(t))  \in \vml \text{ for all } t \}.
 \end{equation}
After defining $\V^{\gCoul,\lambda, \tau}_{j-1,\gamma}(Z)$ similarly, we get an operator
\begin{equation}
\label{eq:Qgg}
Q_{\gamma}^{\gCoul,  \lambda}= \D^\tau_{\gamma} \Fgctau_{\qml} \oplus \dd^{\gCoul, \tau, \tilde{\dagger}}_{\gamma} : \T^{\gCoul, \lambda, \tau}_{j, \gamma}(x_\lambda, y_\lambda) \to \V^{\gCoul,\lambda, \tau}_{j-1,\gamma}(Z) \oplus L^2_{j-1}(\R; i\R).
\end{equation}

Recall that the index of $Q_{\gamma}^{\gCoul} $ computes the relative gradings between stationary points of $\Xqmlgcsigma$ in infinite dimensions.  We will show below in Lemma~\ref{lem:infinite-finite-same-index} that $Q_{\gamma}^{\gCoul} $ and $Q_{\gamma}^{\gCoul,  \lambda}$ have the same index.  Therefore, it remains to show that the index of $Q_{\gamma}^{\gCoul,  \lambda}$ gives the relative gradings in the Morse complex on $(B(2R) \cap \vml)^\sigma/S^1$ in finite dimensions.  

Note that the domain of $Q_{\gamma}^{\gCoul,  \lambda}$ consists of paths $(V, \beta): \R \to T(\vml)^\sigma \times i\R$, where $\beta$ comes from the $dt$ component of the connection.   Similarly, the codomain also consists of paths $(V, \beta)$ of this form.  We can also view $T_{(a,s,\phi)} (\vml)^\sigma$ as the subspace of $\vml \times \mathbb{R}$ consisting of three-dimensional configurations $((b,\psi),r)$ such that $\Re \langle \phi, \psi \rangle_{L^2(Y)} = 0$.

The first observation is that the $L^2_j$ norm of $V$, as a four-dimensional configuration over $Z$, is equivalent to its $L^2_j$ norm as a map from $\R$ to the finite-dimensional space $\vml \times \mathbb{R}$. Indeed, in principle the former norm takes into account more derivatives, corresponding to the directions along $Y$, whereas the latter just uses derivatives in the $t$ direction. However, since we are in $\vml$,
the $j$th derivatives of $V(t)$ in the $Y$ directions are bounded (in $L^2$) by a constant (of the order of $\lambda^j$) times the $L^2(Y)$ norm of $V(t)$ itself. This shows that the two Sobolev norms are equivalent. From now on we will use the latter norm for $V$, i.e. view $V$ as a map from $\R$ to $\vml \times \mathbb{R}$.

By Proposition~\ref{prop:ND}, the stationary points of $\Xqmlagcsigma$ inside $\N/S^1$ are hyperbolic (where $\Xqmlagcsigma$ is considered as a vector field in either infinite or finite dimensions). Thus, the relative Morse index between these stationary points is well-defined. 
We aim to compare $Q_{\gamma}^{\gCoul,  \lambda}$ with the operator 
\begin{equation}
\label{eq:HessFinite}
 \Pi^{\agCoul,\sigma} \circ ( \frac{D^\sigma}{dt} + \D^\sigma \Xqmlgcsigma) : T_{j, \gamma} \P(x_\lambda, y_\lambda) \to T_{j-1, \gamma} \P(x_\lambda, y_\lambda).
 \end{equation}
 
Here, $T_{j, \gamma} \P(x_\lambda, y_\lambda)$ is the subspace of $L^2_j(\R, T(\vml)^\sigma)$ consisting of the paths $V(t)=(b(t), r(t), \psi(t))$ such that $\langle \psi(t), i\phi(t) \rangle_{\tilde{g}} = 0$ for all $t$, as in \eqref{eq:Kagcsigma}. 

Note that the operator \eqref{eq:HessFinite} can be used to define the Morse index in $(B(2R) \cap \vml)^{\sigma}/S^1$. Indeed, it is an operator of the form~\eqref{eq:lgamma}, which uses the connection $ \Pi^{\agCoul,\sigma} \circ \D^\sigma$ on the tangent bundle $T((\vml)^{\sigma}/S^1) \cong \K^{\agCoul, \sigma}.$

Let us write
\begin{equation}
\label{eq:decomposeTlambda}
\T^{\gCoul, \lambda, \tau}_{j, \gamma} = T_{j, \gamma} \P(x_\lambda, y_\lambda) \oplus  L^2_{j}(\R; i\R) \oplus  L^2_{j}(\R; i\R)
\end{equation}
where the first $L^2_{j}(\R; i\R)$ summand corresponds to the tangent to the $S^1$-orbit (that is, real multiples of $\phi(t)$) and the last summand corresponds to $\beta(t)$. The codomain of $Q_{\gamma}^{\gCoul, {\sigma}, \lambda}$ can also be identified with \eqref{eq:decomposeTlambda}, except we use the Sobolev index $j-1$.

With respect to the decomposition $\T^{\gCoul,\lambda,\tau}_{j,\gamma} = \V^{\gCoul,\lambda,\tau}_{j,\gamma} \oplus L^2_j(\R;i\R)$, we  have
\begin{equation}
\renewcommand*{\arraystretch}{1.3}
\label{eq:QQ}
Q_{\gamma}^{\gCoul, \lambda}= \frac{D^\sigma}{dt} +  \begin{pmatrix}
  \D^\sigma \Xqmlgcsigma  &   \dd^{\gCoul, \sigma}_{\gamma}  \\
 \dd^{\gCoul, \sigma,{\tilde \dagger}}_{\gamma} & 0 
\end{pmatrix},
\end{equation}
where the  operators $\dd^{\gCoul, \sigma}_{\gamma} $ and $\dd^{\gCoul, \sigma,\tilde{ \dagger}}_{\gamma}$  are the ones defined in \eqref{eq:dde} and \eqref{eq:ddedagger}. We can further write the entry  $\D^\sigma \Xqmlgcsigma$ with respect to the decomposition of $T_j(\vml)^{\sigma}$ into the anticircular Coulomb slice and the tangent to the $S^1$-direction, as
\begin{equation}
\label{eq:Kis}
\begin{pmatrix}
 \Pi^{\agCoul,\sigma} \circ \D^\sigma \Xqmlgcsigma & K_2 \\
 K_1 & K_3
\end{pmatrix},
\end{equation}
where the $K_i$ terms are compact and $K_2 = K_3 = 0$ at the endpoints $x_\lambda, y_\lambda$.  (Because we use $L^2$ projections to $\vml$ in $\Xqmlgc$ instead of $\tilde{g}$-projections, $\Xqmlgc$ is not a gradient vector field.  Therefore, $\D \Xqmlgc$ is not symmetric, and we cannot conclude that the term $K_1$ vanishes even at stationary points.)  However, the tangent to the $S^1$-direction depends on time, and hence so does this decomposition of $T_j(\vml)^{\sigma}$. Therefore, adding $\frac{D^\sigma}{dt}$ to $\D^\sigma \Xqmlgcsigma$ does not simply result in adding $\frac{D^\sigma}{dt}$ terms on the diagonal of \eqref{eq:Kis}.  

Let us identify the tangent to the $S^1$-direction with $\rr$, such that the generator $(0, 0, i\phi(t))$ corresponds to $1$. We would like to understand the difference
\begin{equation}
\renewcommand*{\arraystretch}{1.3}
\label{eq:DiffDt}
 \frac{D^\sigma}{dt} - \begin{pmatrix} \Pi^{\agCoul, \sigma} \circ \tfrac{D^\sigma}{dt} & 0 \\ 0 & \tfrac{d}{dt} \end{pmatrix}.
 \end{equation}

If we write an element $V \in \V^{\gCoul,\lambda,\tau}_{j,\gamma}$ as a path
$$ V(t)=V^{\agCoul}(t) + f(t) \cdot (0, 0, i\phi(t)),$$
then the expression \eqref{eq:DiffDt} applied to $V$ becomes
\begin{equation}
\label{eq:DiffDt2}
 \Pi^{\gCoul, \circ} \left(\frac{D^\sigma}{dt} V^{\agCoul}(t)\right) + f(t) \cdot \bigl(0, 0, \Pi^\perp_{\phi(t)} i \frac{d\phi(t)}{dt} \bigr),
 \end{equation}
where $\Pi^{\gCoul, \circ}$ is the projection to the span of $(0, 0, i\phi(t))$, with kernel $\K^{\agCoul, \sigma}$. In other words, if $\psi^{\agCoul}(t)$ is the spinorial component of $V^{\agCoul}(t)$, then
$$ \Pi^{\gCoul, \circ} \left(\frac{D^\sigma}{dt} V^{\agCoul}(t) \right) = \frac{\Big \langle \frac{d}{dt} \psi^{\agCoul}(t), i \phi(t) \Big \rangle_{\tilde g} } {\|i \phi(t) \|^2_{\tilde{g}} } \cdot (0, 0, i\phi(t)).$$
By the Leibniz rule, we have
\begin{align*}
 0 &= \frac{d}{dt} \langle \psi^{\agCoul}(t), i\phi(t) \rangle_{\tilde g} \\
 &=  \frac{d}{dt} \Re \langle (0,\psi^{\agCoul}(t)), \Pi^{\elCoul}  (0,i\phi(t)) \rangle_{L^2(Y)} \\
 &=  \Re \langle  (0,\frac{d}{dt}\psi^{\agCoul}(t)), \Pi^{\elCoul}  (0,i\phi(t)) \rangle_{L^2(Y)} +  \Re \langle  (0, \psi^{\agCoul}(t)), \frac{d}{dt}\Pi^{\elCoul}  (0,i\phi(t)) \rangle_{L^2(Y)} \\
 &=  \langle  \frac{d}{dt}\psi^{\agCoul}(t), i\phi(t) \rangle_{\tilde g} +  \Re \langle  (0,\psi^{\agCoul}(t)), \frac{d}{dt}\Pi^{\elCoul}  (0,i\phi(t)) \rangle_{L^2(Y)}.
 \end{align*}
 Therefore, we can write \eqref{eq:DiffDt2} as
 $$ -  \frac{\Re \langle  (0,\psi^{\agCoul}(t)), \frac{d}{dt}\Pi^{\elCoul} (0,i\phi(t)) \rangle_{L^2(Y)}}{ \| i\phi(t) \|_{\tilde{g}}^2} \cdot (0,0,i\phi(t))+  f(t) \cdot \bigl(0, 0,  \Pi^\perp_{\phi(t)} i \frac{d\phi(t)}{dt} \bigr).$$
 Both of these terms involve the time derivatives of $\phi(t)$, and preserve Sobolev regularity $L^2_j \to L^2_j$ as long as $j < k$. We deduce that the difference \eqref{eq:DiffDt} is compact as a map from $L^2_j$ to $L^2_{j-1}$ for $j < k$. 
 
We would like to compare $Q^{\gCoul,\lambda}_\gamma$ to the  operator 
\begin{equation}
\renewcommand*{\arraystretch}{1.3}
\label{eq:QQ2}
Q_{\gamma}^{\gCoul, \sp, \lambda}= \begin{pmatrix}
 \Pi^{\agCoul,\sigma} \circ ( \tfrac{D^\sigma}{dt} + \D^\sigma \Xqmlgcsigma) & 0 & 0 \\
0 &  \tfrac{d}{dt}   &  \dd^{\gCoul, \sigma}_{\gamma}  \\
0 & \dd^{\gCoul, \sigma,{\tilde \dagger}}_{\gamma} & \tfrac{d}{dt} 
\end{pmatrix},
\end{equation}
 written with respect to the decomposition \eqref{eq:decomposeTlambda}.  We will study the linear interpolations $\alpha Q^{\gCoul,\lambda}_\gamma + (1-\alpha) Q^{\gCoul,\sp, \lambda}_\gamma $.  With respect to the same decomposition, this linear interpolation can be computed at the endpoints to have the form
\begin{equation}
\label{eq:QQGC2}
\renewcommand*{\arraystretch}{1.3}
\begin{pmatrix}
 \Pi^{\agCoul,\sigma} \circ \D^\sigma \Xqmlgcsigma & 0 & 0 \\
\alpha K_1  &  0 &  \dd^{\gCoul, \sigma}_{\gamma}  \\
0 & \dd^{\gCoul, \sigma,{\tilde \dagger}}_{\gamma} &  0
\end{pmatrix},
\end{equation}
where $K_1$ is as in \eqref{eq:Kis}.  First, notice that $ \Pi^{\agCoul,\sigma} \circ \D^\sigma \Xqmlgcsigma$ as a map from $\K^{\agCoul,\sigma}_j$ to $\K^{\agCoul,\sigma}_{j-1}$ is hyperbolic at the endpoints, because the endpoints are hyperbolic stationary points.  (It is easy to see that the hyperbolicity for stationary points extends to $j < k$ as well.)  We also have that $\begin{pmatrix}  0 &  \dd^{\gCoul, \sigma}_{\gamma}   \\ \dd^{\gCoul, \sigma,{\tilde \dagger}}_{\gamma}  & 0 \end{pmatrix}$ is hyperbolic by the proof of Lemma~\ref{lem:eH}.  It is straightforward to deduce from here that \eqref{eq:QQGC2} is hyperbolic at the endpoints, for all $\alpha$.  It follows that $\alpha Q^{\gCoul,\lambda}_\gamma + (1-\alpha) Q^{\gCoul,\sp, \lambda}_\gamma$ are Fredholm as maps from $L^2_j$ to $L^2_{j-1}$ for all $\alpha$ and $j < k$.   
 This implies that they are Fredholm for $j=k$ as well. Indeed, the kernel of the map from $L^2_k$ to $L^2_{k-1}$ is necessarily contained in the kernel from $L^2_{k-1}$ to $L^2_{k-2}$ and thus is finite-dimensional.  A similar argument with the adjoint applies for the cokernel as well.  Therefore, $Q_{\gamma}^{\gCoul, \lambda}$ and $Q_{\gamma}^{\gCoul, \sp, \lambda}$ have the same Fredholm index, being related by a continuous family of Fredholm operators.
 
 For fixed $t$, both $\dd^{\gCoul, \sigma}_{\gamma(t)} $ and $\dd^{\gCoul, \sigma,\tilde{ \dagger}}_{\gamma(t)}$ are just given by multiplication by nonzero constants. Thus, we can write the $2 \times 2$ block at the bottom right of \eqref{eq:QQ2} as 
 \begin{equation}
 \label{eq:dAt}
 \frac{d}{dt} + A(t),
 \end{equation}
  where $A(t)$ is invertible and has real spectrum. Thus, $\{A(t)\}$ has zero spectral flow, and the $2 \times 2$ block has index zero. In fact, given the form of $\{A(t)\}$, the operator \eqref{eq:dAt} itself is invertible.

Hence, $Q_{\gamma}^{\gCoul, \sp, \lambda}$ has the same Fredholm index as the top left entry in \eqref{eq:QQ2}, which is the operator \eqref{eq:HessFinite}.  This completes the proof, modulo the claim that $Q_{\gamma}^{\gCoul} $ and $Q_{\gamma}^{\gCoul, \lambda}$ have the same index, which we prove below.
\end{proof}

\begin{lemma}\label{lem:infinite-finite-same-index}
For $\lambda = \llambda_i \gg 0$, the index of $Q_{\gamma}^{\gCoul} $ is equal to that of $Q_{\gamma}^{\gCoul, \lambda}$.  
\end{lemma}
\begin{proof}
Since $\lambda = \llambda_i$, we have that $\pml$ is a projection defined with respect to the $L^2$ (not $\tilde g$) metric.  

We begin with some further discussion of the relevant spaces.  Define 
\[
\Rcal = \{ (b(t), 0, \psi(t)) \in \T^{\gCoul,\tau}_{\gamma} \mid (b(t), \psi(t)) \in (\vml)^\perp\}
\]
where $(\vml)^\perp$ is the $L^2$-orthogonal complement of $\vml$ in $W$.  In other words, an  element of $(\vml)^\perp$ can be decomposed as a (possibly infinite) sum of the eigenvectors of $l$ in $W$ with associated eigenvalue outside of the interval $(-\lambda, \lambda)$.  Note that $\Rcal$ does not depend on $\gamma$, as $\gamma(t) \in \vml$, so $(b(t), \psi(t)) \in (\vml)^\perp$ automatically implies that $\psi(t)$ is orthogonal to $\gamma(t)$.  Because of this, we have a canonical identification of $\Rcal$ with the space of smooth paths in $(\vml)^\perp$, which comes from simply ignoring the middle component.  It follows that $\Rcal \oplus \T^{\gCoul,\lambda,\tau}_{\gamma} = \T^{\gCoul,\tau}_\gamma$.  Note that elements of $\Rcal$ have no $dt$ component.  We define $\Rcal_j \subset \T^{\gCoul,\tau}_{j,\gamma}$ as the Sobolev completion of $\Rcal$ with respect to the four-dimensional $L^2_j$-norm, so that
$$   \Rcal_j \oplus \T^{\gCoul,\lambda,\tau}_{j,\gamma} = \T^{\gCoul,\tau}_{j,\gamma}.$$

There is an analogous decomposition
\[
\V^{\gCoul,\tau}_{j-1,\gamma} \oplus L^2_{j-1}(\mathbb{R};i\mathbb{R})= \Rcal_{j-1} \oplus \left(\V^{\gCoul,\lambda,\tau}_{j-1,\gamma} \oplus L^2_{j-1}(\mathbb{R};i\mathbb{R})\right).  
\]
With respect to these splittings, we can decompose $Q_{\gamma}^{\gCoul}$ as 
\begin{equation}
\label{eq:Qblock}
\begin{pmatrix}
\Pi_{\Rcal_{j-1}} \circ \D^\tau_\gamma(\Fqmlgctau)|_{\Rcal_j} & 0 \\
* & Q_{\gamma}^{\gCoul,\lambda}
\end{pmatrix}, 
\end{equation}
where the top-right entry of this matrix is zero because $\Xqmlgc$ has image in $\vml$, and $ \Pi_{\Rcal_{j-1}}$ denotes the $L^2$-orthogonal projection to $\Rcal_{j-1}$. Therefore, to show that $Q_{\gamma}^{\gCoul,\lambda}$ has the same index as $Q_{\gamma}^{\gCoul}$, it suffices to show that $\Pi_{\Rcal_{j-1}} \circ \D^\tau_\gamma(\Fqmlgctau)|_{\Rcal_j}$ is invertible.  We now compute this operator explicitly.  

For $(b(t),0,\psi(t)) \in\Rcal_j$, it follows from \eqref{eq:tangerine} that 
\begin{align*}
& (\D^\tau_\gamma \Fqmlgctau)(b(t),0,\psi(t)) 
= \frac{D^\sigma}{dt}(b(t),0,\psi(t)) + \D^\sigma_{\gamma(t)} \Xqmlgcsigma (b(t),0,\psi(t)) \\
&= (\frac{d}{dt}b(t),0,\Pi^\perp_{\phi(t)}\frac{d}{dt}\psi(t)) + \D^\sigma_{\gamma(t)}l^\sigma(b(t),0,\psi(t)) + \D^\sigma_{\gamma(t)}(\pml c_\q)^\sigma(b(t),0,\psi(t)). 
\end{align*}
First, note that $\frac{d}{dt}\psi(t)$ and $\phi(t)$ are $L^2$-orthogonal.  Next, since $\pml c_\q$ has image in $\vml$, and $\Rcal_{j-1}$ consists of paths of configurations orthogonal to $\vml$, it is straightforward to verify that $\Pi_{\Rcal_{j-1}} \circ \D^\sigma_{\gamma(t)}(\pml c_\q)^\sigma(b(t),0,\psi(t)) = 0$. Therefore, we have 
 $$\Pi_{\Rcal_{j-1}} \circ \D^\tau_\gamma(\Fqmlgctau)|_{\Rcal_j}=\frac{d}{dt} + \D^\sigma_{\gamma(t)}l^\sigma.$$  
Using \eqref{eqn:lsigma-derivative} and the fact that $\langle \psi(t), D\phi(t) \rangle_{L^2(Y)} = 0$, one can compute directly that
$$(\D^\sigma_{\gamma(t)}l^\sigma)(b(t),0,\psi(t)) = (*db(t), 0, D\psi(t) -  \langle \phi(t), D\phi(t) \rangle_{L^2(Y)} \psi(t)).  $$

For notational convenience, let us simply ignore the middle component of $(b(t),0,\psi(t))$ in $\Rcal_j$.  Define $$h_t: \Rcal_j \to \Rcal_j, \ \ h_t(b(t), \psi(t)) := (0, - \langle \phi(t), D\phi(t) \rangle_{L^2(Y)} \psi(t)).$$  Showing that $\Pi_{\Rcal_{j-1}} \circ \D_\gamma(\Fqmlgctau)|_{\Rcal_j}$ is invertible is equivalent to showing the invertibility of 
\[
\frac{d}{dt} + l + h_t: \Rcal_j \to \Rcal_{j-1}.   
\]

First, we prove that $\frac{d}{dt} + l + h_t$ is injective.  Since $\gamma(t) \in (\vml)^{\sigma}$ for all $t$ and $\|\phi(t)\|_{L^2(Y)}=1$ for all $t$, it follows that there exists $\epsilon > 0$ such that $|\langle \phi(t), D\phi(t) \rangle_{L^2(Y)}| \leq \lambda - \epsilon$, independent of $t$.  Suppose that $(b(t),\psi(t))$ is in the kernel of $\frac{d}{dt} + l + h_t$ and write $(b(t),\psi(t)) = \sum_{\kappa \geq \lambda} (b_\kappa(t),\psi_\kappa(t))$, where we are summing according to the eigenspace decomposition of $l$.  Note that $(b_\kappa(t), \psi_\kappa(t))$ is in the kernel of $\frac{d}{dt} + l + h_t$ for each $\kappa$.  However, it is straightforward to verify as in the proofs of Lemmas~\ref{lem:trajectoriesinvml} and ~\ref{lem:blowuptrajectoriesinvml}, that since $\kappa \geq \lambda$ we must have that $(b_\kappa(t), \psi_\kappa(t))$ must be unbounded either as $t \to \infty$ or $t \to -\infty$.  This contradicts $(b_\kappa(t), \psi_\kappa(t))$ being an $L^2_j$-path.  

It remains to see that $\frac{d}{dt} + l + h_t$ is surjective.  Note that $\frac{d}{dt} + l + h_t$ naturally extends to an operator on sections from $\R \times Y$ to $p^*(iT^*Y \oplus \mathbb{S})$.  The formal adjoint of this operator is $-\frac{d}{dt} + l + h_t$, which is injective by the same argument as above.  Therefore, the extension of $\frac{d}{dt} + l + h_t$ is surjective.  Since the formal adjoint preserves the condition of paths being in $(\vml)^\perp$, we see that $\frac{d}{dt} + l + h_t$, as defined on $\Rcal_j$, must be surjective.   
\end{proof} 

So far we have only discussed the relative grading between stationary points of $\Xqmlagcsigma$ that live in $\N/S^1$. Let us end with a discussion about the reducible stationary points that are in $(B(2R) \cap \vml)^\sigma/S^1$, but not necessarily in $\N/S^1$. 

For the rest of the subsection we fix $\lambda = \llambda_i$ sufficiently large, and a reducible stationary point $(a, 0)$ of $\Xqmlgc$ in $B(2R)$. Consider the reducible stationary points of $\Xqmlagcsigma$ inside $(B(2R) \cap \vml)^\sigma/S^1$ that are of the form $[(a,0,\phi)]$. We write $\kappa(\phi)$ for the associated eigenvalues. By Proposition~\ref{prop:ND2}, any such $[(a,0,\phi)]$ is hyperbolic, when  thought of as a stationary point on the finite-dimensional manifold $(B(2R) \cap \vml)^\sigma/S^1$. Since $\Xqmlgc$ is a Morse equivariant quasi-gradient, we can compute the relative gradings (in finite dimensions) between these points.  
\begin{lemma}
\label{lem:GradingsReduciblesNotInN}
Let $[(a,0,\phi)]$ and $[(a,0,\phi')]$ be stationary points of $\Xqmlagcsigma$ as above. Assume that $\kappa(\phi) > \kappa(\phi')$.  Then, the relative grading between these points, as computed from $\Xqmlagcsigma$ restricted to the finite-dimensional manifold $(B(2R) \cap \vml)^\sigma/S^1$, is given by 
\begin{equation}\label{eq:approxreduciblegradings}
 \gr([(a,0,\phi)], [(a,0,\phi')]) = \begin{cases}
2i(\kappa(\phi), \kappa(\phi')) & \text{if $\kappa(\phi)$ and $\kappa(\phi')$ have the same sign,} \\
2i(\kappa(\phi),\kappa(\phi'))-1 &\text{otherwise.}
 \end{cases}
\end{equation}
\end{lemma}   
\begin{proof}
This follows from Lemma~\ref{lem:MEquiv}.
\end{proof}

\begin{lemma}
\label{lem:kappas}
Suppose $[(a,0,\phi)]$ is a stationary point of  $\Xqmlagcsigma$ that is contained in $(B(2R) \cap \vml)^\sigma/S^1$ but not in $\N/S^1$. 

(a) If $\kappa(\phi) > 0$, then for all stationary points of $\Xqmlagcsigma$ of the form $[(a,0,\phi')]$ that are contained in $\N/S^1$, we have $\gr([(a,0,\phi)], [(a,0,\phi')])\geq 2.$
 
(b)  If $\kappa(\phi) < 0$, then for all stationary points of $\Xqmlagcsigma$ of the form $[(a,0,\phi')]$ that are contained in $\N/S^1$, we have $\gr([(a,0,\phi)], [(a,0,\phi')])\leq -2.$

\end{lemma}

\begin{proof}
Consider a pair of reducible stationary points $[x]$ and $[y]$ of $\Xqagcsigma$ in $\N/S^1$ with the same connection component.  It follows from Corollary~\ref{cor:GrPres}, Proposition~\ref{prop:RelationFinite}, and Lemma~\ref{lem:GradingsReduciblesNotInN}, that the relative grading between $[x]$ and $[y]$ is the same as the relative grading between $[x_\lambda]$ and $[y_\lambda]$, considered as stationary points of $\Xqmlagcsigma$ on $(B(2R) \cap \vml)^\sigma/S^1$.  Further, the spinorial energies of $[x]$ and $[y]$ are  necessarily close to the $\lambda$-spinorial energies of $[x_\lambda]$ and $[y_\lambda]$ respectively.  Recall that $[x_\lambda]$ and $[y_\lambda]$ are necessarily contained in $\N/S^1$.   Equation~\eqref{eq:GradingRed} and Lemma~\ref{lem:GradingsReduciblesNotInN} give that in each case, the relative gradings are computed in terms of the orderings by eigenvalues, which correspond to ($\lambda$-)spinorial energy.  In particular, this implies that if $[(a,0,\phi)]$ is a reducible stationary point of $\Xqmlagcsigma$ not in $\N/S^1$, then its associated eigenvalue cannot sit between those of $[x_\lambda]$ and $[y_\lambda]$ for any pair $[x_\lambda]$ and $[y_\lambda]$.  The result now follows.      
\end{proof}

The following is an immediate consequence of the proof of Lemma~\ref{lem:kappas}.   

\begin{corollary}\label{cor:lowest-reducible-eigenvalue}
Let $[x]$ denote the reducible stationary point of $\Xqagcsigma$ which has lowest eigenvalue among all reducible stationary points with the same connection component.  Then, $[x_\lambda]$ is the reducible stationary point of $\Xqmlagcsigma$ with the lowest positive eigenvalue among those in $(B(2R) \cap \vml)^\sigma/S^1$ with the same connection component.
\end{corollary}

\section{Absolute gradings}\label{sec:absgradings}
Recall that Theorem~\ref{thm:Main} asserts an isomorphism of $\widetilde{H}^{S^1}_*(\SWF(Y,\spinc))$ with $\hmto(Y,\spinc)$ which respects the absolute gradings.  As our current strategy for the proof of this isomorphism is to identify each of these modules in a certain grading range with the Morse homology of $\Xqmlagcsigma$ on $(B(2R) \cap \vml)^\sigma/S^1$, we need to define an absolute grading on the stationary points of $\Xqmlagcsigma$ which lines up with the gradings coming from $\SWF$ and from $\hmto$.  

In Chapter~\ref{sec:quasigradient}, we showed that $\Xqmlgc$ is a Morse equivariant quasi-gradient.  From this, \eqref{eq:EquivConleyMorse} implies that the Morse complex for $\Xqmlagcsigma$ computes the reduced $S^1$-equivariant homology of the Conley index $I^\lambda$ from Chapter~\ref{sec:spectrum} in a certain grading range.  
Since $\SWF(Y,\spinc) = \Sigma^{-n(Y,\spinc,g) \mathbb{C}} \Sigma^{-W^{(-\lambda, 0)}} I^\lambda$, to have a complex whose homology agrees with that of the reduced $S^1$-equivariant homology of the spectrum, we must shift the gradings accordingly.  Therefore, for $[x_\lambda]$ a stationary point of $\Xqmlagcsigma$, define 
\begin{equation}\label{eq:gr-lambda}
\gr_{\lambda}^{\swf}([x_\lambda]) := \ind\bigl([x_{\lambda}] \text{ in } (\vml)^{\sigma}/S^1 \bigr) - \dim W^{(-\lambda, 0)} - 2n(Y, \s, g),
\end{equation}
where $n(Y, \spinc, g)$ is the quantity mentioned at the end of Chapter~\ref{sec:spectrum}.  Therefore, the Morse complex for $\Xqmlagcsigma$ with absolute grading given instead by $\gr_{\lambda}^{\swf}$ computes $\widetilde{H}^{S^1}_*(\SWF(Y,\spinc))$ in the appropriate grading range by the discussion in Section~\ref{sec:combinedMorse}.

Thus, to connect the gradings on the Floer spectrum with monopole Floer homology, we will need to relate the absolute grading $\gr_{\lambda}^{\swf}$ on $\Crit^\lambda_\N$ with the absolute grading $\gr^{\Q}$ from \eqref{eq:grQx} defined on $\Crit_\N$.  This is the subject of the following proposition.     

\begin{proposition}\label{prop:absolute-gradings-agree}
For any $\lambda = \llambda_i \gg 0$ and $[x] \in \Crit$, we have
\begin{equation}\label{eq:finite-infinite-absolute-gradings}
\gr^{\swf}_{\lambda}([x_\lambda])= \gr^{\Q}([x]).
\end{equation}
\end{proposition}

In order to prove \eqref{eq:finite-infinite-absolute-gradings}, let us now recall the precise definition of $n(Y, \spinc, g)$ from \cite[Equation (6)]{Spectrum}:
\begin{equation}
\label{eq:ng}
 n(Y, \spinc, g) = \ind_{\cc} (D^+) - \frac{c_1(\t)^2 - \sigma(X)}{8}.
 \end{equation}
Here: $X$ is a simply connected, oriented, compact Riemannian four-manifold with boundary $Y$ (such that the metric is a product near the boundary);  $\t$ is a $\spc$ structure on $X$ such that $\t|_Y = \s$;  $D^+$ is the Dirac operator on $(X, \t)$ with Atiyah-Patodi-Singer boundary conditions, and associated to a connection extending $A_0$ on $Y$; and $\ind_{\cc}$ denotes the index of a complex operator, which is twice the real index $\ind_\rr$. The Atiyah-Patodi-Singer boundary conditions mean that the domain of $D^+$ consists of spinors whose restrictions to $Y$ project trivially to the negative eigenspaces of the three-dimensional Dirac operator $D$.

\begin{proof}[Proof of Proposition~\ref{prop:absolute-gradings-agree}]
Let $[x] \in \Crit^s$ be a reducible generator corresponding to the lowest positive eigenvalue of the operator $D_{\q, a}$, where $[x] = [(a,0,\phi)]$. Corollary~\ref{cor:GrPres} and Proposition~\ref{prop:RelationFinite} imply that 
\begin{equation}\label{eq:rel-gr-lambda}
\gr_{\lambda}^{\swf}([x_\lambda]) - \gr_{\lambda}^{\swf}([y_\lambda]) = \gr([x_\infty], [y_\infty]),
\end{equation}
so it suffices to prove the relation for $[x]$. 

Pick $(X, \t)$ as in the definition of $n(Y, \spinc, g)$ above. Then, recall from Section~\ref{sec:modifications} that 
$$ \gr^{\Q}([x]) = -\gr_z(X, [x]) + \frac{c_1(\t)^2 - \sigma(X)}{4} - b^+(X) -1.$$

We seek to show that 
\begin{equation}
\label{eq:grind}
 \gr_z(X, [x]) = \ind_\rr (D^+) - b^+(X) -1 - \ind \bigl ( [x_{\lambda}] \text{ in } (\vml)^{\sigma}/S^1 \bigr ) + \dim W^{(-\lambda, 0)}.
 \end{equation}
 
The quantity $\gr_z(X, [x])$ is the virtual dimension of the Seiberg-Witten moduli space on $X$ (with an added cylindrical end) with asymptotics given by $x$. Following the proof of Lemma 28.3.2 in \cite{KMbook}, we can compute $\gr_z(X, [x])$ by using a reducible configuration on $X$. It then becomes the index of an operator with two parts: one is a perturbed signature operator, and the other is a perturbed Dirac operator. The former would have index $-b^+(X)-1$ if the perturbation $\q$ were zero, but in general it differs from this by the index of a signature operator on the cylinder $[0,1]\times Y$, with boundary data $(0,0)$ and $(\q,a)$. This index can be computed as the spectral flow of the family 
$$ \begin{pmatrix}0 & -d^* \\ -d & *d +  2t\D_{(ta,0)}\q^0 \end{pmatrix} : \Omega^0(Y; i\rr) \oplus \Omega^1(Y; i\rr) \to  \Omega^0(Y; i\rr) \oplus \Omega^1(Y; i\rr) ,\ \ t \in [0,1].$$
 By a compact perturbation that keeps the endpoints fixed, we can change this family of operators into 
\begin{equation}
\label{eq:familynew}
 \begin{pmatrix}0 & -d^* \\ -d & *d +  2t\D_{(a,0)}\q^0 \end{pmatrix}  ,\ \ t \in [0,1].
 \end{equation}
Since $(a,0)$ is a stationary point, we have that 
$$\D_{(a,0)}\q^0 = \begin{pmatrix}0 & 0 \\ 0 & \D_{(a,0)}\etaq^0 \end{pmatrix}$$
with respect to the decomposition of imaginary one-forms into $\ker d \oplus \ker d^*$. Hence, \eqref{eq:familynew} decomposes into a $3 \times 3$ block form, where one block is constantly $\left( \begin{smallmatrix}0 & -d^* \\ -d & 0 \end{smallmatrix} \right)$  and the other is
$$ *d+ 2t \D_{(a,0)}\eta_\q^0 : \ker d^*  \to \ker d^*,\ \ t \in [0,1].$$
Since $\left( \begin{smallmatrix}0 & -d^* \\ -d & 0 \end{smallmatrix} \right)$ has no spectral flow, we have reduced the computation to the spectral flow of this last family.  Let us denote the spectral flow by $\SF(\q)^0.$

The second contribution to $\gr_z(X, [x])$ comes from the index of a perturbed Dirac operator $D^+_{\q,a} - \lambda_0$, where $D^+_{\q,a}$ is an APS operator but with $D_{\q,a}$ on the boundary, unlike $D^+=D^+_{0,0}$. Here $\lambda_0$ is the eigenvalue corresponding to $x$, and the domain of $D^+_{\q,a} - \lambda_0$ consists of spinors whose restrictions to $Y$ project trivially to the eigenspaces of $D_{\q,a}$ for eigenvalues $< \lambda_0$. Since $\lambda_0$ is the lowest positive eigenvalue, the domain is the same as the one we considered for $D^+$ in \eqref{eq:ng}. The two operators $D^+_{\q,a}$ and $D^+_{\q,a} - \lambda_0$ differ by a constant (hence compact) term, and hence have the same index. Note that
$$ \ind_\rr (D^+_{\q,a})- \ind_\rr (D^+) = \SF(\q)^1,$$
where $\SF(\q)^1$ is the spectral flow of the perturbed Dirac operators on $Y$ as we move from $(0,0)$ to $(\q, a)$.  Note that at a reducible stationary point $(a,0)$, we have that $D_a \psi + \D_{(a,0)} \q^1(0,\psi) = D \psi + \D_{(a,0)} (c_\q)^1(0,\psi)$, or in short, $D_{\q,a}^{\gCoul} = D_{\q,a}$.  In particular, we can compute this spectral flow in Coulomb gauge.     

Therefore, we have
$$ \gr_z(X, [x])= \ind_\rr (D^+) - b^+(X) -1 + \SF(\q)^0 + \SF(\q)^1.$$

To obtain \eqref{eq:grind}, it remains to show that
\begin{equation}
\label{eq:grind2}
\dim W^{(-\lambda, 0)} - \ind \bigl([x_\lambda] \text{ in } (\vml)^{\sigma}/S^1 \bigr)  = \SF(\q)^0 + \SF(\q)^1.
 \end{equation}

We now analyze the terms on the left-hand side of \eqref{eq:grind2}.  The first term, $ \dim W^{(-\lambda, 0)}$, is the number of negative eigenvalues of $l$ (counted with multiplicity) between $-\lambda$ and $0$. The second term requires a more careful analysis.  Write $[x_{\lambda}] = (a_\lambda,0,\phi_\lambda)$.  By Corollary~\ref{cor:lowest-reducible-eigenvalue}, $[x_{\lambda}]$ has lowest positive eigenvalue among stationary points in $(B(2R) \cap \vml)^\sigma/S^1$ of $\Xqmlagcsigma$ with connection component $a_\lambda$ (and not just among those in $\N/S^1$).  By Lemma~\ref{lem:MEquiv}, $\ind \bigl([x_{\lambda}] \text{ in } (\vml)^{\sigma}/S^1 \bigr )$ is the sum of two parts.  The first is the number of negative eigenvalues of the linearization of $l + \pml c_{\q}$ restricted to the connection summand of $\vml$, i.e. $*d + \D_{(a_\lambda,0)} (\pml c_\q)^1(\cdot,0)$.  
The second part is the number of negative eigenvalues of the linearization of $l+p^{\lambda} c_{\q}$, restricted to the spinorial summand of $\vml$, at $(a_\lambda, 0)$, i.e. $D + \D_{(a_\lambda,0)} (\pml c_{\q})^1(0, \cdot)$. 

Putting it all together, we find that the left-hand side of \eqref{eq:grind2} is the spectral flow from $l$ to $l + \pml A_\lambda$, where
$$ 
A_\lambda(b,\psi) = (\D_{(a_\lambda,0)} c_\q^0(b,0), \D_{(a_\lambda,0)}c_\q^1(0,\psi)),
$$
when considered as operators on $\vml$.  However, since $\lambda$ is of the form $\llambda_i$, this is the same as the spectral flow from $l$ to $l + \pml A_\lambda \pml$, considered as operators from $W_k$ to $W_{k-1}$.  Thus, to establish \eqref{eq:grind2}, it remains to show that there is no spectral flow from $l + \pml A_\lambda \pml $ to $l + A_\infty$, as operators from $W_k$ to $W_{k-1}$.  For the rest of the discussion, we will only consider operators from $W_k$ to $W_{k-1}$. Further, all  these operators will have index zero (being compact perturbations of $l$), and hence for them injectivity or surjectivity is equivalent to invertibility.  Due to the block form of these operators, this fact remains true if we restrict the operators to either their connection or spinorial components.  

Since $[x_\infty]$ is a non-degenerate stationary point of $\Xqagcsigma$, we have that $l + A_\infty$ is injective.  (This follows from Proposition~\ref{prop:nondegeneracycharacterized} and that this operator is index 0.)  By Proposition~\ref{prop:ND2}, we have that $[x_\lambda]$ is a non-degenerate reducible stationary point of $\Xqmlagcsigma$, and from this it follows that $l + \pml A_\lambda \pml$ is injective.    Since $A_\infty$ and $\pml A_\lambda \pml$ are $L^2$ self-adjoint, it suffices to show that for $\lambda \gg 0$ and $t \in [0,1]$,  $l + h_t^{\lambda}$ is injective, where $h^\lambda_t$ is the compact operator 
\[
h^\lambda_t = t \pml  A_\lambda \pml + (1-t) A_\infty.  
\]
Note that for any sequence $v_n$ which converges to $v$ weakly in $W_k$, any sequence $t_n \in [0,1]$, and $\lambda_n \to \infty$, we have 
\[
h^{\lambda_n}_{t_n}(v_n) \to A_\infty(v) \text{ in } W_{k-1}.
\] 
Now suppose that $l + h_{t_n}^{\lambda_n}$ is not injective for some sequences $t_n \in \mathbb{R}$ and  $\lambda_n \to \infty$. Let $v_n \in W_k$ with $\| v_n \|_{L^2_k} = 1$ be such that $l + h^{\lambda_n}_{t_n}(v_n) = 0$.  The $v_n$ converge weakly in $W_k$ to some $v$ and as discussed, $h^{\lambda_n}_{t_n}(v_n)$ converges in $W_{k-1}$ to $A_\infty(v)$.  Thus, we see that $l(v_n)$ converges in $W_{k-1}$ to $l(v)$.  In particular, $v_n$ converges in $W_k$ to $v$, which consequently has $\| v \|_{L^2_k} = 1$, and $(l + A_\infty)(v) = 0$.  This contradicts the injectivity of $l + A_\infty$. The relation \eqref{eq:grind2} follows.
\end{proof}

Having established Proposition~\ref{prop:absolute-gradings-agree}, we can rephrase Lemma~\ref{lem:kappas} as the following. 

\begin{proposition}\label{prop:GradingBounds}
Any reducible stationary point of $\Xqmlagcsigma$ in $(B(2R) \cap \vml)^\sigma/S^1$ which is in the grading range $[-N,N]$ is contained in $\N/S^1$, assuming $\lambda = \llambda_i$ for some $i \gg 0$.  
\end{proposition}

\section{Conclusions}\label{sec:conclusions}    
Recall, from the discussion at the beginning of Section~\ref{sec:stationarypointsoutsideN}, that approximate irreducible stationary points are necessarily in $\N/S^1$.  Therefore, using the fact that $\Xi_{\lambda}$ is grading preserving and the fact that irreducible stationary points have vanishing ($\lambda$-)spinorial energy, it follows that there exists $N > 0$ such that all irreducible stationary points of $\Xqmlagcsigma$ have grading in $[-N,N]$ for all $\lambda = \llambda_i \gg 0$.  Here, recall that we grade the stationary points of $\Xqmlagcsigma$ using $\gr^{\SWF}_\lambda$, defined in \eqref{eq:gr-lambda}.  We can now summarize the results of this section in the following.  

\begin{proposition}\label{prop:stationarycorrespondence}
Let $\q$ be a very tame, admissible perturbation, and fix $N > 0$ such that all irreducible stationary points of $\Xqmlagcsigma$ have grading in $[-N,N]$ for all $\lambda = \llambda_i \gg 0$.  For all $\lambda = \llambda_i \gg 0$, there is a one-to-one correspondence $\Xi_{\lambda}$ between:
\begin{itemize}
\item the stationary points of $\Xqmlagcsigma$ in $(B(2R) \cap \vml)^\sigma/S^1$ with grading in $[-N, N]$, including all irreducibles, and 
\item
the stationary points of $\Xqagcsigma$ with grading in $[-N, N]$, including all irreducibles. 
\end{itemize}

This correspondence preserves the grading, as well as the type of stationary point (irreducible, stable, unstable). Furthermore, all the stationary points of $\Xqmlagcsigma$ in $(B(2R) \cap \vml)^\sigma/S^1$ with grading in $[-N, N]$ are hyperbolic.
\end{proposition}
\begin{proof}
The conclusion follows from Corollary~\ref{cor:corresp}, Proposition~\ref{prop:ND}, Corollary~\ref{cor:GrPres} and Proposition~\ref{prop:GradingBounds}. \end{proof}

\chapter{The Morse-Smale condition for the approximate flow}\label{sec:MorseSmale}
Recall that in Chapter~\ref{sec:quasigradient}, we established that $\Xqmlgc$ is a Morse equivariant quasi-gradient vector field on $\vml \cap B(2R)$. In this chapter, we show that it is also Morse-Smale, in the sense of Definition~\ref{def:eMSqg}. Recall, from Lemma~\ref{lem:MSlu} and the discussion at the end of Section~\ref{sec:combinedMorse} that the Morse-Smale condition can be rephrased in terms of the surjectivity of a linear operator. In our setting, let 
$$\gamma: \R \to (\vml \cap B(2R))^{\sigma} \subset (\vml)^{\sigma},$$
be a trajectory of $\Xqmlgcsigma$, going between two stationary points $x$ and $y$. Regularity of the moduli space at $[\gamma]$ is equivalent to the surjectivity of the operator
\begin{equation}
\label{eq:Op0}
\Pi^{\agCoul,\sigma} \circ ( \frac{D^{\sigma}}{dt} + \D^\sigma \Xqmlgcsigma) : T_{j, \gamma} \P(x, y) \to T_{j-1, \gamma} \P(x, y)
\end{equation}
which has already made an appearance in \eqref{eq:HessFinite}. Again, while it may seem like a more natural choice to work with an operator where the derivatives are given by a connection coming from the $\tilde{g}$-metric, the choice of connection does not matter at a trajectory.

Alternatively, we can view $\gamma$ as a path in the infinite dimensional space $W^{\sigma}$. The corresponding linearized operator is
\begin{equation}
\label{eq:Op2}
 \D^\tau_{\gamma} \Fqmlgctau : \T_{k,\gamma}^{\gCoul, \tau}(x, y) \to \V^{\gCoul, \tau}_{k-1,\gamma}(Z).
 \end{equation}

 \begin{lemma}
 \label{lem:equivOp}
 The operator \eqref{eq:Op0} is surjective if and only if the operator \eqref{eq:Op2} is surjective.
 \end{lemma}
 \begin{proof}
We use the notation and the results from Lemma~\ref{lem:allsurjective} and the proof of Proposition~\ref{prop:RelationFinite}. 
By the analogue of Lemma~\ref{lem:allsurjective} with $\qml$ instead of $\q$, surjectivity of \eqref{eq:Op2} is equivalent to that of the operator $Q_{\gamma}^{\gCoul}$ from \eqref{eq:Qgammagc}. In the proof of Lemma~\ref{lem:infinite-finite-same-index} we gave a block form \eqref{eq:Qblock} for $Q_{\gamma}^{\gCoul}$, and we also showed that the top left block is invertible. This implies that $Q_{\gamma}^{\gCoul}$ is surjective if and only if the operator $Q_{\gamma}^{\gCoul, \lambda}$ from \eqref{eq:Qgg} is surjective. Finally, $ Q_{\gamma}^{\gCoul, \lambda}$ can be related to \eqref{eq:Op0} through \eqref{eq:QQ2}, using the invertibility of the operator \eqref{eq:dAt}. 
 \end{proof}
 
Thus, we can work with the operators $ \D^\tau_{\gamma} \Fqmlgctau$. Note that we should only ask for their surjectivity when $\gamma$ is boundary-unobstructed. When $\gamma$ is boundary-obstructed, we require surjectivity of its summand $(\D^\tau_{\gamma} \Fqmlgctau)^{\del}$, which acts on the spaces of paths in the boundary of $(\vml)^{\sigma}$.  Similar arguments as above apply for the boundary-obstructed case as well.
 
\begin{proposition}\label{prop:MorseSmales}
We can choose the admissible perturbation $\q$ such that for any $\lambda \in \{\llambda_1, \llambda_2, \dots\}$ sufficiently large, the following holds. Given any flow trajectory $\gamma$ for the restriction of $\Xqmlgcsigma$ to $(B(2R) \cap W^{\lambda})^{\sigma}$, we have that:
\begin{enumerate}[(i)]
\item If $\gamma$ is boundary-unobstructed, then $\D^\tau_{\gamma} \Fqmlgctau$ is surjective;
\item If $\gamma$ is boundary-obstructed, then $(\D^\tau_{\gamma} \Fqmlgctau)^{\del}$ is surjective.
\end{enumerate}
\end{proposition}

\begin{proof}
This is similar to the proof of transversality for moduli spaces of trajectories in \cite[Section 15]{KMbook}. 
So far, we have chosen an admissible perturbation $\q_0$ so that the stationary points of both $\Xqogcsigma$ and $\Xqomlgcsigma$ inside $B(2R)^{\sigma}$ are non-degenerate; cf. Proposition~\ref{prop:ND2}. Consider the blow-down projections of the stationary points of $\Xqogcsigma$; these come in a finite number of $S^1$-orbits. Pick disjoint open neighborhoods of those orbits, and let $U$ be the union of these neighborhoods. By Lemma~\ref{lem:implicitfunctionreducible}, for $\lambda$ large, the blow-down projections of the stationary points of $\Xqomlgcsigma$ from $B(2R)^{\sigma}$ land inside $U$. Consider the set of perturbations
$$ \P_U = \{ \q \in \P \mid \q|_{U} = \q_0|_{U} \}.$$

We can find an open neighborhood $\nu(\q_0)$ of $\q_0$ in $\P_U$ such that for all $\q$ in this neighborhood, we have the same set of stationary points for $\Xqgcsigma$ and $\Xqmlgcsigma$ as for $\q_0$, and therefore we still have nondegeneracy for them. 

We now claim that for a residual set of perturbations $\q$ in $\nu(\q_0) \subset \P_U$, the desired surjectivity conditions hold. Since we work with a countable set of $\lambda$, it suffices to prove this for some fixed $\lambda$, sufficiently large. We define a parametrized map
$$ \mathfrak{M}:  \tC_k^{\gCoul, \tau}(x, y) \times \nu(\q_0) \to \V^{\gCoul, \tau}_{k-1}(Z), \ \ \ \mathfrak{M}(\gamma, \q) = \Fqmlgctau(\gamma).$$
When $\gamma$ is reducible (i.e., contained in the reducible locus), there is a similar map $(\mathfrak{M})^{\del}$ acting on the space of paths in the boundary.

To prove our claim, it is enough to check that the derivative of $\mathfrak{M}$ is surjective at all points $(\gamma, \q)$ in $\mathfrak{M}^{-1}(0)$ when $\gamma$ is irreducible; and that the derivative of $(\mathfrak{M})^{\del}$ is surjective at reducibles.  The proof of these facts is entirely similar to the proof of Proposition 15.1.3 in \cite{KMbook}.
\end{proof}

 Let us now put the results of Chapters~\ref{sec:criticalpoints}, ~\ref{sec:quasigradient}, ~\ref{sec:gradings}, and the current one in context.  Recall that it is our goal to establish an isomorphism between $\hmto(Y,\spinc)$ and the (singular) equivariant homology $\widetilde{H}^{S^1}_*(\SWF(Y,\spinc))$.  To do this, for each $N \gg 0$, we will establish an isomorphism between the truncations $\hmto_{\leq N}(Y,\spinc)$ and $\widetilde{H}^{S^1}_{\leq N}(\SWF(Y,\spinc))$ by passing through an intermediate group, the $S^1$-equivariant Morse homology of $\Xqmlagcsigma$ on $(B(2R) \cap \vml)^\sigma/S^1$ as defined in Section~\ref{sec:finite}.  This group can be defined by Propositions~\ref{prop:AllMS} and ~\ref{prop:MorseSmales} for $\lambda = \llambda_i$ with $i \gg 0$.  We have established in Section~\ref{sec:finite} that the homology of this Morse complex in an appropriate grading range will be isomorphic to the (singular) Borel homology $\widetilde{H}^{S^1}_{\leq N}(\SWF(Y,\spinc))$.  Thus, it remains to identify the $S^1$-equivariant Morse homology of $\Xqmlagcsigma$ with monopole Floer homology in the corresponding grading range.  

At this point, for $\lambda = \llambda_i$ with $i \gg 0$, we have established in Proposition~\ref{prop:stationarycorrespondence} a correspondence between the stationary points of $\Xqagcsigma$ and $\Xqmlagcsigma$ with grading in the interval $[-N,N]$.  Thus, using the results of Chapter~\ref{sec:coulombgauge}, we have an isomorphism on the level of graded chain groups (but not yet on homology) between $\cmto_{\leq N}(Y,\spinc,\q)$ and the $S^1$-equivariant Morse complex for $\Xqmlagcsigma$. 
 
This leaves us with one major step, which is to construct a chain complex isomorphism between the $S^1$-equivariant Morse complex and $\cmto_{\leq N}(Y,\spinc,\q)$ by relating the trajectories of $\Xqagcsigma$ to those in finite dimensions; this will be analogous to the correspondence on the level of stationary points we have established in this section.  This is the focus of Chapter~\ref{sec:trajectories1} and Chapter~\ref{sec:trajectories2}.  Before doing so, in the next section we do some technical work which will allow us to relate paths between stationary points of $\Xqmlgcsigma$ to paths between stationary points of $\Xqgcsigma$.

\chapter{Self-diffeomorphisms of configuration spaces}\label{sec:appendix}
In Section~\ref{sec:StabilityPoints}, we established a correspondence $\Xi_\lambda : \Crit^\lambda_{\N} \to \Crit_{\N}$ between stationary points of $\Xqmlagcsigma$ and $\Xqagcsigma$.  Our goal is to be able to do this for trajectories as well.  In this case, the trajectories live in different path spaces: trajectories of $\Xqagcsigma$ live in $\B_k^{\gCoul,\tau}([x_\infty],[y_\infty])$ while we will see in Chapter~\ref{sec:trajectories1} that trajectories of $\Xqmlagcsigma$ are in $\B_k^{\gCoul,\tau}([x_\lambda],[y_\lambda])$.  Therefore, we need a way to relate these different spaces.  In this section, we will extend the correspondence  from Corollary~\ref{cor:corresp} first to a family of $S^1$-equivariant self-diffeomorphisms of $W^\sigma$, and then to self-diffeomorphisms of $W^\tau(I \times Y)$ for $I \subset \mathbb{R}$ and to other path spaces.  This will be needed in Proposition~\ref{prop:L2k} and Proposition~\ref{prop:stabilitynearby}.  The construction of these maps will use a setup similar to that of the function $T_\lambda$ defined in Section~\ref{sec:Flambda}.  

Before stating the first result, we need some preliminaries.  In this section, in order to make statements about the smoothness of functions with respect to $\lambda$ as a parameter, we do {\em not} restrict to the case that $\lambda = \llambda_i$.  We will not use the results of Chapter~\ref{sec:quasigradient} or Chapter~\ref{sec:gradings}, so this will not be a problem.  For each stationary point $x_\infty$ of $\Xqgcsigma$ in $\N$ let $x_\lambda$ denote the stationary point of $\Xqmlgcsigma$ which is $L^2$ closest to $x_\infty$. (For an explanation of the well-definedness of $x_\lambda$ see Section~\ref{sec:Flambda}.)  In what follows, we will abuse notation and say that a function $G_\lambda$, which depends on $\lambda$, is {\em smooth in $\lambda$ at and near infinity} if there exists $\epsilon > 0$ such that $G_{f^{-1}(r)}$ depends smoothly on $r \in [0, \epsilon)$, where $f: (0, \infty] \to [0,1)$ is the homeomorphism from Section~\ref{sec:StabilityPoints}.

\begin{lemma}\label{lem:Xixylambda}
For $\lambda \gg 0$, there exists an $S^1$-equivariant diffeomorphism $\Xi_{\lambda} : W^\sigma_{0} \to W^\sigma_{0}$ satisfying:
\begin{enumerate}[(i)]
\item\label{xi:xlambda} $\Xi_{\lambda}$ sends $x_\lambda$ to $x_\infty$ for each stationary point $x_\infty \in \N$,  
\item\label{xi:sobolev} $\Xi_{\lambda}$ restricts to a self-diffeomorphism of $W^\sigma_j$ for any $1 \leq j \leq k$, 
\item\label{xi:smooth-f} Let $\Xi_{\infty}$ be the identity. Then, for $0 \leq j \leq k$, $\Xi_{\lambda}: W^\sigma_j \to W^\sigma_j$ and all its derivatives are smooth in $\lambda$ at and near infinity. 
\end{enumerate}
Further, $\Xi_\lambda$ extends to the double $\tW^\sigma_0$ and the analogous properties hold.   
\end{lemma}

Note that since $W^\sigma_j$ is not an affine space, we do not have a natural notion of higher derivatives.  However, we can think of $W^\sigma_j$ as naturally being embedded as a submanifold of the linear space $\What_j$, where we remove the conditions $s \geq 0$ and $\| \phi \|_{L^2} = 1$.  (The same is true for the double $\tW^\sigma_j$.)  Of course, $\What_j \cong W_j \times \mathbb{R}$.  We will then treat $\Xi_\lambda$ as a map from $W^\sigma_j$ to $\What_j$.  After defining $\Xi_\lambda$, it will be clear that this extends naturally to an $L^2_j$ neighborhood $\nu(W^\sigma_j)$ of $W^\sigma_j$ in $\What_j$.  We can make sense of higher derivatives of $\Xi_\lambda$ on $\nu(W^\sigma_j)$, since this is an open submanifold of a linear space, and these are the derivatives we will talk about (and measure in norm).  Further, since we may consider elements of $W^\sigma_j$ as elements of the linear space $\What_j$, it makes sense to take differences of elements there and take their $L^2_j$ norms.

We will give the proof of Lemma~\ref{lem:Xixylambda} in Sections~\ref{sec:define-xi} and \ref{sec:xi-3d-bounds} below.  We will then use this to obtain a number of important technical consequences for diffeomorphisms of path spaces.  Note that given a path $\gamma$ in $W^\sigma_j$, we may apply $\Xi_\lambda$ slicewise to obtain a new path $\Xi_\lambda(\gamma)$ in $W^\sigma_j$.  We will study the regularity of $\Xi_\lambda$ on four-dimensional configurations.  Before restating it, we need some discussion and terminology. 

In parts (ii)-(iv) of the following proposition, we will discuss the smoothness of $\Xi_\lambda : W^\tau_j(x_\lambda, y_\lambda) \to W^\tau_j(x_\infty, y_\infty)$ in $\lambda$. This a priori is not defined, as the domain space is changing.  We postpone the discussion for what we mean by this in Section~\ref{xi:non-compact} (given after the statement of Proposition~\ref{xi:half-cylinder}) for expediency.

\begin{proposition}\label{prop:Xixy-paths}
Let $\Xi_\lambda$ be as above.  Fix stationary points $x_\infty$ and $y_\infty$ of $\Xqgcsigma$ in $\N$ and let $x_\lambda$ and $y_\lambda$ be the stationary points of $\Xqmlgcsigma$ which minimize $L^2$ distance to $x_\infty$ and $y_\infty$ respectively.  Then for $\lambda \gg 0$ and $ 1 \leq j \leq k$, we have the following:  
\begin{enumerate}[(i)]
\item\label{xi:paths} for a compact interval $I \subseteq \mathbb{R}$, the map $\Xi_{\lambda}$ induces an $S^1$-equivariant diffeomorphism of $\tW^\tau_{j}(I \times Y)$ which is smooth in $\lambda$ at and near $\infty$ and preserves $W^\tau_{j}(I \times Y)$, 
\item\label{xi:paths-smooth-f} $\Xi_\lambda$ induces diffeomorphisms from $\tW^\tau_j(x_\lambda,y_\lambda)$ to $\tW^\tau_j(x_\infty,y_\infty)$, which vary smoothly in $\lambda$ at and near $\infty$, 
\item \label{xi:Bgc} $\Xi_\lambda$ induces a diffeomorphism from $\tB^{\gCoul,\tau}_j([x_\lambda],[y_\lambda])$ to $\tB^{\gCoul,\tau}_j([x_\infty],[y_\infty])$, 
\item \label{xi:Vgc} the diffeomorphisms $\Xi_\lambda: \tB^{\gCoul,\tau}_j([x_\lambda],[y_\lambda]) \to \tB^{\gCoul,\tau}_j([x_\infty],[y_\infty])$ from the above item lift to smooth (in domain and also with respect to $\lambda$ at and near $\infty$) bundle maps 
$$
\xymatrix{
\V^{\gCoul,\tau}_j \ar[r]^{(\Xi_\lambda)_*} \ar[d] & \V^{\gCoul,\tau}_j \ar[d] \\
\tB^{\gCoul,\tau}_j([x_\lambda], [y_\lambda]) \ar[r]^{\Xi_\lambda} & \tB^{\gCoul,\tau}_j([x_\infty], [y_\infty]).
}
$$
If $[x_\infty] \neq [y_\infty]$, the analogous statement also holds for $\tB^{\gCoul,\tau}_j([x_\infty],[y_\infty])/\mathbb{R}$.  
\end{enumerate}
\end{proposition}

As for the case of $W^\sigma_j$, there exists a similar way to discuss the derivatives of $\Xi_\lambda$ in four dimensions.  We will think of $\tW^\tau_j(I \times Y)$ as naturally being embedded as a submanifold of the linear space $\What_j(I \times Y)$, where we remove the conditions $s(t) \geq 0$ and $\| \phi(t) \|_{L^2(Y)} = 1$.  Following the discussion above, it turns out that $\Xi_\lambda$ will extend to a neighborhood of $\tW^\tau_j(I \times Y)$ in $\What_j(I \times Y)$.  For simplicity, we will work in the larger affine space to compute derivatives and measure the distance between configurations.  This will not affect any statements about smoothness or regularity.

Before giving the proofs of Lemma~\ref{lem:Xixylambda} and Proposition~\ref{prop:Xixy-paths}, let us state two immediate corollaries of the latter. The first follows from part \eqref{xi:paths}, by continuity:
\begin{corollary}
\label{cor:Xixy-paths}
For $1 \leq j \leq k$, if a sequence $\gamma_n \in W^\tau_{j,loc}(I \times Y)$ converges to some $\gamma_\infty$, then $\Xi_{\lambda_n}(\gamma_n) \to \gamma_\infty$ for any sequence $\lambda_n \to \infty$.
\end{corollary}

The second corollary follows from part \eqref{xi:paths-smooth-f} of Proposition~\ref{prop:Xixy-paths}: 
\begin{corollary}
\label{cor:path-distance} For $1 \leq j \leq k$, $\gamma_0 \in W^{\tau}_j(x_{\infty}, y_{\infty})$, if a sequence $\gamma_n \in W^\tau_j(x_{\lambda_n}, y_{\lambda_n})$ with $\lambda_n \to \infty$ satisfies 
$$\| \gamma_n - \Xi^{-1}_{\lambda_n}(\gamma_0) \|_{L^2_j(\rr \times Y)} \to 0,$$ 
then 
$$ \| \Xi_{\lambda_n}(\gamma_n) - \gamma_0 \|_{L^2_j(\rr \times Y)} \to 0.$$ 
\end{corollary}

We now give the organization of the rest of this section.  In Section~\ref{sec:define-xi}, we give the construction of $\Xi_\lambda$.  In Section~\ref{sec:xi-3d-bounds}, we prove the desired properties of $\Xi_\lambda$ in Lemma~\ref{lem:Xixylambda}.  In Section~\ref{sec:xi-4d-bounds}, we prove parts (i) and (ii) of Proposition~\ref{prop:Xixy-paths}.  Finally, in Section~\ref{sec:xi-ext} we prove parts (iii) and (iv).

\section{The construction of $\Xi_\lambda$}\label{sec:define-xi}
\begin{lemma}\label{lem:xiconstruct}
For $\lambda \gg 0$, there exists an $S^1$-equivariant diffeomorphism $\Xi_{\lambda} : W^\sigma_{0} \to W^\sigma_{0}$ which sends $x_\lambda$ to $x_\infty$ for each stationary point $x_\infty$ of $\Xqgcsigma$ in $\N$.   
\end{lemma}

The construction will be similar to that of the diffeomorphism $T_\lambda$ used in Section~\ref{sec:Flambda} to define $F_\lambda$.  Since the blow-up is not a linear space, there is not a notion of translation like for the definition of $T_\lambda$, so we will work in charts (in the directions orthogonal to the $S^1$ orbits).  The reader may note that here we are working with $W^\sigma_0$, while in the definitions of $T_\lambda$ and $F_\lambda$ from Chapter~\ref{sec:quasigradient}, we only worked with $W_{k-1}$.  The explanation for this is that the function $F_\lambda$ incorporated $c$, which is not well-behaved as a map from $W_0 \to W_0$ (since it contains quadratic terms), so we would not have been able to analyze $F_\lambda$ in lower Sobolev regularity.

A more minor distinction is the following.  To construct $T_\lambda(x)$, we used the function $\omega_\lambda$ defined in \eqref{eq:omegalambda} to find which element of the $S^1$-orbit of an approximate stationary point $x_\lambda$ the point $x$ is closest to, and then translate $x$ by $\omega_\lambda(x)(x_\infty - x_\lambda)$, where $x_\infty$ minimizes $L^2$ distance to $x_\lambda$.  To define $\Xi_\lambda(x)$, we will compare $x$ to $x_\infty$ instead of $x_\lambda$, and then translate.  The reason for this is that we will work with local charts, and it is simplest to work in charts around the same point ($x_\infty$) for every $\lambda$, as opposed to defining $\Xi_\lambda$ in terms of different charts for each $\lambda$ and tracking the changes.      
    
For notation, we choose representatives $x^1_\infty,\ldots, x^m_\infty$ for the $m$ orbits of stationary points of $\Xqagcsigma$ in $\N$.

\begin{proof}[Proof of Lemma~\ref{lem:xiconstruct}]
We will first define $\Xi_\lambda$ on an $S^1$-invariant neighborhood of $x^i_\infty$ in $W^\sigma_0$ for a fixed $i$; for now, we omit the index $i$.  Write $x_\infty = (a_\infty,s_\infty,\phi_\infty)$.  Consider the submanifold 
$$
U_{x_\infty} = \{ (a,s,\phi) \in W^\sigma_0 \mid \Re \langle \phi, i \phi_\infty \rangle_{L^2} = 0, \langle \phi, \phi_\infty \rangle_{L^2} > 0 \}.     
$$
Note that $U_{x_\infty}$ is not open, as it only consists of configurations which are (real) orthogonal to the $S^1$-orbit of $x_{\infty}$.  We will define an $S^1$-equivariant self-diffeomorphism of $S^1 \cdot U_{x_\infty}$ which takes $x_\lambda$ to $x_\infty$ and is the identity outside of a smaller $S^1$-invariant  neighborhood of $x_\infty$.  We will then obtain the desired diffeomorphism $\Xi_\lambda$ by repeating this construction for each orbit of stationary points of $\Xqgcsigma$, and taking $\Xi_\lambda$ to be the identity outside of these neighborhoods.    

First, we observe that there exists a diffeomorphism $G_{x_\infty}$ from $U_{x_\infty}$ to the Hilbert manifold with boundary
$$ 
V_{x_\infty} = \{ (a,s,\phi) \in (\ker d^*)_0 \oplus \mathbb{R} \oplus L^2(Y;\mathbb{S}) \mid s \geq 0, \ \langle \phi, \phi_\infty \rangle_{L^2} = 0 \},
$$
given by 
$$
G_{x_\infty} : (a,s,\phi)  \mapsto \Bigl(a, s, \frac{\phi}{\langle \phi, \phi_\infty \rangle_{L^2}} - \phi_\infty \Bigr).  
$$
We remark that $V_{x_\infty}$ is a submanifold of $\T^{\gCoul,\sigma}_{0, x_\infty}$.  We have that $G_{x_\infty}(x_\infty) = (a_\infty,s_\infty,0)$.  Denote this vector by $z_\infty$.  Note also that in the definition of $G_{x_\infty}$, have that $\Re \langle \phi, \phi_\infty \rangle_{L^2} = \langle \phi, \phi_\infty \rangle_{L^2}$, since $\phi$ is orthogonal to the $S^1$-orbit of $\phi_\infty$. Observe also that the inverse of $G_{x_\infty}$ is given by the formula
$$ G_{x_\infty}^{-1}: (a,s,\phi)  \mapsto  \Bigl(a, s, \frac{\phi + \phi_\infty}{\|  \phi + \phi_\infty \|_{L^2}}\Bigr).$$

Since $x_\lambda= (a_\lambda,s_\lambda,\phi_\lambda)$ and $x_{\infty}$ minimize the $L^2$ distance between their $S^1$-orbits, we deduce that $x_\infty- x_\lambda$ is necessarily orthogonal to the $S^1$-orbit of $x_\infty$.  Further, for $\lambda \gg 0$, $\| x_\infty - x_\lambda \|_{L^2}$ is arbitrarily small, and thus $\langle \phi_\lambda, i \phi_\infty \rangle_{L^2} > 0$, so $x_\lambda$ is contained in $U_{x_\infty}$.  Let $z_\lambda$ denote the image of $x_\lambda$ under $G_{x_\infty}$.  We will construct an interpolation between translation by $z_\infty - z_\lambda$ (which may not be defined near the boundary of $V_{x_\infty}$ if $x_\infty$ is irreducible) and the identity.  

Pick $0 < \delta \ll \frac{1}{2}$ (independent of $\lambda$), such that an $L^2$ ball of size $\delta$ around the origin in $V_{x_\infty}$ has the following properties for $\lambda \gg 0$.  First, this ball must contain $G_{x_\infty}(x_\lambda)$ and be disjoint from a $\delta$-ball around $G_{x_\infty}(x^n_\infty)$ or $G_{x_\infty}(x^n_\lambda)$, should they be defined, for any $n \neq i$ in $\N$.  Further, if $x_\infty$ is irreducible, we choose $\delta$ such that the $\delta$ ball is contained in the interior of $V_{x_\infty}$.  By our choice of $\delta$, if $(a,s,\phi) \in U_{x_\infty}$ satisfies $\langle \phi, \phi_\infty \rangle_{L^2} < \frac{1}{2}$, then $G_{x_\infty}(a,s,\phi)$ is not in this ball.  It is in this ball where we will interpolate between the identity and the translation from $x_\lambda$ to $x_\infty$. 

Let $\alpha: [0,\infty) \to [0,1]$ be a smooth bump function which is 1 on $[0,\delta/2]$ and 0 on $[\delta,\infty)$.   We can now define a map $\Upsilon_{\lambda}$ from $V_{x_\infty}$ to $V_{x_\infty}$ given by 
\begin{equation}\label{xi:association}
\Upsilon_{\lambda}: z \mapsto (1-\alpha(\|\psi\|_{L^2})) \| z + \alpha(\| \psi \|_{L^2}) (z + z_\infty - z_\lambda),
\end{equation}
where we write $z = (b,r,\psi)$.  

By the choice of $\delta$, we have that this induces a well-defined map on $V_{x_\infty}$, since in the case that ${x_\infty}$ is reducible, the $s$-component is unchanged, and when ${x_\infty}$ is irreducible, we have chosen the $\delta$-ball such that translation does not happen near the boundary.  Again, we point out that this map is translation by $z_\infty - z_\lambda$ inside of an $L^2$ ball of size $\delta/2$ in $V_{x_\infty}$ and is the identity outside of a $\delta$ ball.  In particular, we see that $z_\lambda$ is taken to $z_\infty$.  

By the work of Chapter~\ref{sec:criticalpoints}, we have that 
\begin{equation}\label{xi:xlambdaxinfty}
x_\lambda \to x_\infty \text{ in } L^2_j \text{ for all } j \text{ as } \lambda \to \infty 
\end{equation}
and thus 
\begin{equation}\label{xi:vlambdavinfty}
z_\lambda \to z_\infty \text{ in } L^2_j \text{ for all } j \text{ as } \lambda \to \infty.
\end{equation}
We claim that for $\lambda \gg 0$, the map $\Upsilon_{\lambda}$ in \eqref{xi:association} induces a diffeomorphism of $V_{x_\infty}$.  Indeed, since $z_\lambda \to z_\infty$, for $\lambda \gg 0$, the derivative of $\Upsilon_{\lambda}$ is close to the identity, and thus is an isomorphism at each point.  The map is thus a local diffeomorphism, which is the identity outside of a ball.  This is necessarily a self-covering map of a simply-connected space, and so we deduce that this is a global diffeomorphism.  

We can now use $G_{x_\infty}$ to define the self-diffeomorphism $\Xi_\lambda$ on $U_{x_\infty}$.  Since $G_{x_\infty}(x_\lambda) = z_\lambda$ and $G_{x_\infty}(x_\infty) = z_\infty$, we see that 
$$\Xi_{\lambda} := G_{x_\infty}^{-1} \circ \Upsilon_{\lambda} \circ G_{x_\infty} $$
takes $x_\lambda$ to $x_\infty$.  

It follows by the construction that $x_\infty$ is taken to $x_\lambda$.  We then extend $\Xi_\lambda$ to a diffeomorphism on the $S^1$ orbit of $U_{x_\infty}$ as follows.  Define a function $\omega_\infty: S^1 \cdot U_{x_\infty} \to S^1$, similar to $\omega^j_\lambda$ in \eqref{eq:omegalambda}, by 
\begin{equation}
\label{eq:omegainfinity}
\omega_\infty(x)= \frac{\Re \langle \phi, \phi_{\infty}  \rangle_{L^2}  + i\Re \langle \phi, i\phi_{\infty}  \rangle_{L^2} }{ \bigl( (\Re \langle \phi, \phi_\infty  \rangle_{L^2})^2  +  (\Re \langle \phi, i\phi_\infty  \rangle_{L^2})^2 \bigr)^{1/2}}.
\end{equation}
Note that $\omega_\infty$ has the property that if $x \in S^1 \cdot U_{x_\infty}$, then $\overline{\omega_{\infty}(x)} \cdot x \in U_{x_\infty}$.  Therefore, we extend $\Xi_\lambda$ to the $S^1$ orbit of $U_{x_\infty}$ by conjugating by $\omega_{\infty}$:
$$
x \mapsto {\omega}_{\infty}(x) \cdot \Xi_\lambda(\overline{\omega_{\infty}(x)} \cdot x). 
$$
By construction this extension is $S^1$-equivariant.  By repeating the above construction for each stationary point $x^i_\infty$, we extend $\Xi_\lambda$ to neighborhoods of every stationary point of $\N$.  Note that $\Xi_\lambda$ is the identity near the boundary of each such neighborhood.  We finally extend $\Xi_\lambda$ to a diffeomorphism of all of $W^\sigma_0$ by the identity outside of these neighborhoods.  

It is now clear from the construction that $\Xi_\lambda$ is an $S^1$-equivariant diffeomorphism of $W^\sigma_0$.     
\end{proof}

For future reference, we describe an explicit formula for $\Xi_\lambda$.  Let us first introduce some notation to help express $\Xi_\lambda$ more compactly.  For each $i = 1,\ldots, m$, define a function $\beta_i : W^\sigma_0 \to [0,1]$ by 
$$
\beta_i(x) = \begin{cases}
0 & \text{ if } x=(a,s,\phi) \notin S^1 \cdot U^{x^i_\infty} \\
\alpha(\| \frac{\phi}{\langle \phi, \omega^i_\infty(x) \phi^i_\infty \rangle_{L^2}} - {\omega^i_\infty}(x) \phi^i_\infty \|_{L^2}) & \text{ if } x=(a,s,\phi) \in S^1 \cdot U^{x^i_\infty}.
\end{cases} 
$$
Further write $x^i_\lambda = (a^i_\lambda, s^i_\lambda, \phi^i_\lambda)$, where $x^i_\lambda$ is the stationary point of $\Xqmlgcsigma$ corresponding to $x^i_\infty$.  We also write $v^i_\lambda = \frac{\phi^i_\lambda}{\langle \phi^i_\lambda, \phi^i_\infty\rangle_{L^2}} - \phi^i_\infty$.  Note that $G_{x^i_\infty}(x^i_\lambda) = (a^i_\lambda, s^i_\lambda, v^i_\lambda)$.  We get
 
\begin{align}
\label{eq:Xilambda-Ux} \Xi_\lambda& : W^\sigma_0 \to W^\sigma_0 \\
\nonumber (a,s,\phi) &\mapsto \Big(a + \sum_i \beta_i(x) (a^i_\infty - a^i_\lambda),\ s + \sum_i \beta_i(x)(s^i_\infty - s^i_\lambda),  \\ 
\nonumber &\hspace{2.1in}\left.\frac{\phi- \sum_i \beta_i(x)   \langle  \phi,  \phi^i_\infty \rangle_{L^2} \cdot v^i_\lambda }{\| \phi- \sum_i \beta_i(x)   \langle \phi,   \phi^i_\infty \rangle_{L^2} \cdot v^i_\lambda   \|_{L^2}}\right).
\end{align}
Here we take as notational convention that we do not worry about the well-definedness of $\omega^i_{\infty}$ when $\beta_i$ is zero.  Observe that for each $x$, $\beta_i(x)$ is non-zero for at most one $i$.

We now discuss the extension of $\Xi_\lambda$ to the double as claimed in the lemma.  Recall that a stationary point $x_\infty$ is reducible if and only if the approximate stationary point $x_\lambda$ is.  From \eqref{eq:Xilambda-Ux}, we therefore see that for $x \in W^\sigma_0$ with $s \ll 1$ (i.e. $s$ much smaller than the smallest value of $s_\infty$ for $x_\infty$ irreducible), the map $\Xi_\lambda$ preserves the $s$-component.  In particular, $\Xi_\lambda$ preserves and is tangent to the reducible locus.  It follows that $\Xi_\lambda$ extends to the double $\tW^\sigma_0$ as claimed.  It will be clear that the arguments below establishing the desired properties of $\Xi_\lambda$ for $W^\sigma_0$ extend to $\tW^\sigma_0$.  

\begin{remark}
It is worth noting that almost all of of the pieces of $\Xi_\lambda$ are just determined by the spinorial component of $x$.  Further, in the end result, the formula for $\Xi_\lambda$ only uses the function $\omega$ in the bump-like functions $\beta_i$.    
\end{remark}

\section{Three-dimensional properties of $\Xi_{\lambda}$}\label{sec:xi-3d-bounds}
We now show that $\Xi_{\lambda}$ has the stated properties in Lemma~\ref{lem:Xixylambda}.  After that, we will establish some additional bounds that will be useful for Proposition~\ref{prop:Xixy-paths}. 

We begin with the first item in Lemma~\ref{lem:Xixylambda}.  

\begin{proof}[Proof of Lemma~\ref{lem:Xixylambda}\eqref{xi:xlambda}]
It is clear from the construction of $\Xi_\lambda$, that each $x^i_\lambda$ is mapped to $x^i_\infty$.  The $S^1$-equivariance then implies that $x_\lambda$ is mapped to $x_\infty$ for any stationary point $x_\lambda$ of $\Xqmlgcsigma$ in $\N$.
\end{proof}

\subsection{Smoothness properties}

\begin{proof}[Proof of Lemma~\ref{lem:Xixylambda}\eqref{xi:sobolev}]
First, note that the functions $\beta_i$, $\omega^i_{\infty}$, and $\phi \mapsto \langle \phi, \phi^i_\infty \rangle_{L^2}$ are smooth on $L^2_j$ for any $1 \leq j \leq k$, since the $L^2$ inner-product is, and $x^i_\infty \in L^2_k$.  At each step of the construction of $\Xi_\lambda$, where $\Xi_\lambda$ is not the identity (or in other words, some $\beta_i$ is non-zero), we are simply scalar multiplying by $\omega^i_{\infty}$, which is non-zero, or by adding a multiple of $\beta_i$ or $\langle \phi, \phi^i_\infty \rangle_{L^2}$ times $x^i_\infty$ or $\phi^i_\infty$.  Since $x^i_\lambda, x^i_\infty$ are in $L^2_k$, these (invertible) operations necessarily preserve the condition of being an $L^2_j$ configuration for $1 \leq j \leq k$, and we have the desired result.      
\end{proof}

At the beginning of this section, $\Xi_\lambda$ was claimed to extend to a neighborhood of $W^\sigma_j$ in $\What_j$.  This is clear from the explicit description of $\Xi_\lambda$ given in \eqref{eq:Xilambda-Ux} since wherever $\Xi_\lambda$ disagrees with the identity, we have that $\beta_i \neq 0$ for precisely one $i$, and in this case $\langle \phi , \omega^i_\infty \phi^i_\infty \rangle _{L^2}$ is real and positive (in fact, at least 1/2); from this it is easy to see that  $\Xi_\lambda$ extends to a neighborhood of this point in $\What_j$, when we expand the target to $\What_j$ as well.  

The following completes the proof of Lemma~\ref{lem:Xixylambda}.  

\begin{proof}[Proof of Lemma~\ref{lem:Xixylambda}\eqref{xi:smooth-f}]
Since $\Xi_\lambda$ is defined by extending a self-diffeomorphism of $U_{x^i_\infty}$ by $S^1$, and the space $U_{x^i_\infty}$ is defined independently of $\lambda$, it suffices to show that $\Xi_{f^{-1}(r)} |_{U_{x^i_\infty}}$ is smooth at and near $r = 0$.  Recall from Corollary~\ref{cor:corresp} that the correspondence $r \mapsto [x_{f^{-1}(r)}]$ is differentiable at and near $r=0$.  It is easy to check that this lifts to a smoothness statement without quotienting by $S^1$, due to the condition that $\Re \langle x_{f^{-1}(r)}, i x^i_\infty \rangle_{L^2} = 0$.  Using \eqref{eq:Xilambda-Ux}, we see the only terms in $\Xi_{f^{-1}(r)}$ or its derivatives which depend on $r$ are smooth functions of $x_{f^{-1}(r)}$, and thus differentiable at and near $r = 0$.  This gives the desired smoothness statement. 
\end{proof}

\subsection{Some three-dimensional bounds on $\Xi_{\lambda}$}
Before moving on to the four-dimensional properties of $\Xi_\lambda$ in Proposition~\ref{prop:Xixy-paths}, we study a few additional properties of $\Xi_\lambda$.  For notational simplicity, we work in a neighborhood $S^1 \cdot U_{x^i_\infty}$ for some $i$.  For the rest of this subsection, we again omit the index $i$ from the notation.  

It is clear from \eqref{eq:Xilambda-Ux} that the complication in $\Xi_\lambda$ is in the spinorial component.  We begin by simplifying the notation and then establishing a few key bounds.  
We write 
\begin{equation}
\label{eq:Xiflambda} f_\lambda(x) = \phi- \langle \phi, \phi_\infty \rangle_{L^2} \cdot \beta(x) \cdot v_\lambda 
\end{equation}
and thus we can express the spinorial component of $\Xi_\lambda$ as $\frac{f_\lambda}{\| f_\lambda\|_{L^2}}$ for $x \in S^1 \cdot U_{x_\infty}$.  Therefore, $\Xi_\lambda$ can be written a bit more succinctly as 
\begin{equation}\label{xi:simplified}
\Xi_\lambda(a,s,\phi)  = \left(a + \beta(x) (a_\infty - a_\lambda),\ s + \beta(x)(s_\infty - s_\lambda), \frac{f_\lambda(x)}{\| f_\lambda(x) \|_{L^2}} \right).  
\end{equation}  

We point out some bounds that will be useful for us when proving Proposition~\ref{prop:Xixy-paths}.  By the work of Chapter~\ref{sec:criticalpoints}, we have that $\| x_\lambda - x_\infty \|_{L^2_k} \to 0$.  It is then easy to see that for $\lambda \gg 0$, since $\| \phi_\infty \|_{L^2} = 1$, we have  
\begin{equation}\label{xi:phi-lambda-infty}
\| v_\lambda \|_{L^2} \to 0.  
\end{equation}

From \eqref{xi:phi-lambda-infty}, using that the spinorial component of an element of $W^\sigma$ has unit $L^2$-norm,  we can also deduce that for $\lambda \gg 0$: 
\begin{equation}
\label{xi:flambda-L2j} \frac{1}{2} \leq \| f_\lambda(x) \|_{L^2} \leq \frac{3}{2}.
\end{equation}

\section{Four-dimensional properties of $\Xi_{\lambda}$}\label{sec:xi-4d-bounds}
In this section, we prove Proposition~\ref{prop:Xixy-paths}.  For the rest of this section, we let $j$ be an integer such that $1 \leq j \leq k$.  Before proceeding, we remind the reader that the $L^2_j$-norm of a four-dimensional configuration $\gamma$ on $I \times Y$ is expressed as 
$$
\| \gamma \|^2_{L^2_j(I \times Y)} = \sum^j_{n=0} \int^b_a \| \frac{d^n}{dt^n}\gamma(t) \|^2_{L^2_{j - n}(Y)} dt.   
$$

\subsection{Sobolev multiplication and superposition operators}
In order to study how $\Xi_\lambda$ acts on paths in $W^\tau_{j}(I \times Y)$, we will need some elementary variants of Sobolev multiplication and superposition operators.  Many of these are well-known or easily deduced from the definitions and Sobolev multiplication.  We include the proofs so as to reference them for analogous results later on that are less standard.

\begin{lemma}\label{xi:inner-product-path}
Let $I \subseteq \rr$.  Then, taking inner products induces smooth maps
\begin{align}
&\label{inner-product:4d} \What_j(I \times Y) \times \What_j(I \times Y) \to L^2_j(I; \rr), \ (\gamma(t), \eta(t)) \mapsto \langle  \gamma(t), \eta(t) \rangle_{L^2(Y)}, \\
\label{inner-product:3d}&\What_j(I \times Y) \times \What_j \to L^2_j(I; \rr), \ (\gamma,x) \mapsto \langle \gamma, x \rangle_{L^2(Y)}.
\end{align}
\end{lemma}
\begin{proof}
While the smoothness of \eqref{inner-product:4d} is a special case of Sobolev multiplication, we provide a proof, since it makes the second claim easier to justify and we will use similar techniques throughout the rest of the section.  
We compute 
\begin{align*}
\| \langle \gamma , \eta \rangle_{L^2(Y)} \|^2_{L^2_j(I)} &= \sum^j_{n=0} \Bigl\| \frac{d^n}{dt^n} \langle \gamma, \eta \rangle_{L^2(Y)} \Bigr\|^2_{L^2(I)}  \\
& = \sum^j_{n=0} \int_I \Bigl|\sum^n_{i = 0} {n \choose i} \langle \frac{d^i}{dt^i} \gamma , \frac{d^{n-i}}{dt^{n-i}} \eta \rangle_{L^2(Y)} \Bigr |^2 dt \\
& \leq \sum^j_{n=0} \sum^n_{i_1, i_2 = 0} {n \choose i_1}{n \choose i_2} \int_I \Bigl| \langle \frac{d^{i_1}}{dt^{i_1}} \gamma, \frac{d^{n-i_1}}{dt^{n-i_1}} \eta \rangle_{L^2(Y)} \cdot \langle \frac{d^{i_2}}{dt^{i_2}} \gamma, \frac{d^{n-i_2}}{dt^{n-i_2}} \eta \rangle_{L^2(Y)}\Bigr|       dt \\
& \leq \sum^j_{n=0} \sum^n_{i_1, i_2 = 0} {n \choose i_1}{n \choose i_2} \int_I \Bigl\| \frac{d^{i_1}}{dt^{i_1}} \gamma\Bigr\|_{L^2(Y)} \Bigl\| \frac{d^{n-i_1}}{dt^{n-i_1}} \eta\Bigr \|_{L^2(Y)}\Bigl \| \frac{d^{i_2}}{dt^{i_2}} \gamma\Bigr\|_{L^2(Y)} \Bigl\| \frac{d^{n-i_2}}{dt^{n-i_2}} \eta\Bigr \|_{L^2(Y)} dt.      
\end{align*} 

Let us focus on a fixed term in the final term above, i.e., we fix a value of $n, i_1, i_2$.  Note that $n \leq j$, so one of $n - i_1$ or $i_1$ is bounded above by $\frac{j}{2}$.  Without loss of generality, it is $i_1$.  Since $j \geq 1$, we have that 
$$ j - i_1 \geq \frac{j}{2} \geq \frac{1}{2}.$$
(In fact, $j-i_1 \geq 1$, because $j$ and $i_1$ are integers.) We deduce that the composition of $\frac{d^{i_1}}{dt^{i_1}}: \What_{j}(I \times Y) \to \What_{j-i_1}(I \times Y)$ with the restriction map to a slice $\What_{j-i_1}(I \times Y) \to L^2(\{t\} \times Y)$ is continuous.   From here we see that there exists a constant $C_{i_1}$, depending only on $Y$ and $I$, such that 
$$\| \frac{d^{i_1}}{dt^{i_1}} \gamma(t) \|_{L^2(Y)} \leq C_{i_1} \| \gamma \|_{L^2_j(I \times Y)}, \ \text{ for all } t \in I.$$  
We can repeat the same argument for $i_2$ or $n - i_2$.  Again, without loss of generality we assume that $i_2 \leq \frac{j}{2}$, and we obtain bounds      
$$\| \frac{d^{i_2}}{dt^{i_2}} \gamma(t) \|_{L^2(Y)} \leq C_{i_2} \| \gamma \|_{L^2_j(I \times Y)}, \ \text{ for all } t \in I.$$
  
Therefore, we may write 
\begin{align*}
\| \langle \gamma , \eta \rangle_{L^2(Y)} \|^2_{L^2_j(I)} &\leq \sum^n_{i_1, i_2 = 0} {n \choose i_1}{n \choose i_2} \int_I \Big\| \frac{d^{i_1}}{dt^{i_1}} \gamma\Big\|_{L^2(Y)} \Big\| \frac{d^{n-i_1}}{dt^{n-i_1}} \eta\Big \|_{L^2(Y)} \Big\| \frac{d^{i_2}}{dt^{i_2}} \gamma\Big\|_{L^2(Y)}\Big \| \frac{d^{n-i_2}}{dt^{n-i_2}} \eta \Big\|_{L^2(Y)} dt \\
& \leq \sum^n_{i_1, i_2 = 0} {n \choose i_1}{n \choose i_2} C_{i_1} C_{i_2} \| \gamma \|^2_{L^2_j(I \times Y)} \int_I  \Big\| \frac{d^{n-i_1}}{dt^{n-i_1}} \eta \Big\|_{L^2(Y)} \Big \| \frac{d^{n-i_2}}{dt^{n-i_2}} \eta \Big\|_{L^2(Y)}    dt \\
& \leq  \sum^n_{i_1, i_2 = 0} {n \choose i_1} {n \choose i_2} C_{i_1} C_{i_2} \| \gamma \|^2_{L^2_j(I \times Y)} \Big\| \frac{d^{n-i_1}}{dt^{n-i_1}} \eta\Big \|_{L^2(I \times Y)}  \Big\| \frac{d^{n-i_2}}{dt^{n-i_2}} \eta \Big\|_{L^2(I \times Y)} \\
& \leq \sum^n_{i_1, i_2 = 0} {n \choose i_1} {n \choose i_2} C_{i_1} C_{i_2} \| \gamma \|^2_{L^2_j(I \times Y)} \| \eta \|^2_{L^2_j(I \times Y)}, 
\end{align*}
where in the penultimate inequality we have used Cauchy-Schwarz.  An analogous inequality holds if we instead replace $i_1$ or $i_2$ with $n - i_1$ or $n - i_2$ respectively, depending on their values.  

By summing over the relevant constants, we obtain a bound
$$
\| \langle \gamma , \eta \rangle_{L^2(Y)} \|_{L^2_j(I)} \leq C\| \gamma \|_{L^2_j(I \times Y)} \| \eta \|_{L^2_j(I \times Y)}, 
$$
where $C$ is independent of $\gamma, \eta$.  It follows that slicewise inner products define a {\em bounded}, bilinear map.  Such a map is necessarily smooth.  
  
For \eqref{inner-product:3d}, we can repeat the proof above.  Note that an element of $\What_j$ does not determine an element of $\What_j(I \times Y)$ if $I$ is not compact.  However, for each term analyzed in the above argument, we only needed that one of the terms involved (either a derivative of $\gamma$ or of $\eta$) be $L^2(Y)$ bounded uniformly in time, with the other term square-integrable in time.  In this case, the three-dimensional configuration in $\What_j$ is constant in time, and has bounded $L^2(Y)$ norm (since $j \geq 0$), so we have the desired result.  
\end{proof}

\begin{lemma}\label{xi:pointwise-mult}
Let $I \subseteq \rr$.  Then, pointwise multiplication induces smooth maps
\begin{align}
&L^2_j(I; \rr) \times \What_j(I \times Y) \to \What_j(I \times Y), \ (g, \gamma) \mapsto g \gamma, \\
&L^2_j(I; \rr) \times \What_j \to \What_j(I \times Y), \ (g, \gamma) \mapsto g \gamma.
\end{align}
\end{lemma}
\begin{proof}
The argument is similar to that of the above lemma.  We begin with the first equation.  Recall that for any $\gamma \in \What_j(I \times Y)$, we have a norm for $\gamma(t,y)$, defined for any $(t,y) \in I \times Y$; we write this function as $|\gamma|$.  We compute 
\begin{align*}
\| g \gamma \|^2_{L^2_j(I \times Y)} &= \sum^j_{n=0} \int_{I \times Y} \Big| \sum^n_{i = 0}\frac{d^i}{dt^i}(g) \nabla^{n-i} \gamma \Big|^2 \\
& \leq \sum^n_{i_1, i_2 = 0} {n \choose i_1} \cdot {n \choose i_2} \int_{I \times Y} \Big| \frac{d^{i_1}}{dt^{i_1}}(g)  \nabla^{n-i_1} \gamma\Big| \cdot \Big| \frac{d^{i_2}}{dt^{i_2}}(g) \nabla^{n-i_2} \gamma \Big|.      
\end{align*}
We can now apply the same argument as in the first part of Lemma~\ref{xi:inner-product-path} to obtain a bound 
$$
\| g \gamma \|_{L^2_j(I \times Y)} \leq K \| g\|_{L^2_j(I)} \| \gamma \|_{L^2_j(I \times Y)},
$$
where $K$ is a constant independent of $g$ and $\gamma$.  This shows that pointwise multiplication is continuous.  Smoothness again follows by bilinearity.   

The second equation can be obtained by applying the same modifications to the proof as in Lemma~\ref{xi:inner-product-path} to establish \eqref{inner-product:3d}.  
\end{proof}

The following lemma is standard. We include a proof for completeness.
  
\begin{lemma}
\label{lem:composeSobolev}
Let $I \subseteq \rr$ be an interval, and $h: \rr \to \rr$ a smooth function. Consider the transformation of $C^\infty(\mathbb{R})$ given by 
$$ H : g \mapsto h \circ g.$$
\begin{enumerate}[(a)]
\item If $I$ is compact, then $H$ induces a smooth map from $L^2_j(I; \rr)$ to $L^2_j(I; \rr)$.
\item The same holds true for arbitrary $I$ (not necessarily compact), provided that $h(0)=0$. 
\end{enumerate}
\end{lemma}

\begin{proof}
$(a)$ We use Fa{\`a} di Bruno's formula for the higher derivatives of a composition:
\begin{equation}
\label{eq:bruno}
\frac{d^n}{dt^n} h(g(t)) = \sum C(m_1, \dots, m_n) \cdot h^{(m_1 + \dots + m_n)}(g(t)) \cdot \prod_{i=1}^n \bigl( g^{(i)}(t)\bigr)^{m_i}.
\end{equation}
Here, the sum is taken over all $n$-tuples of nonnegative integers $(m_1, \dots, m_n)$ with
$$ 1 \cdot m_1 + 2 \cdot m_2 + \dots + n \cdot m_n = n,$$
and the constants $C(m_1, \dots, m_n)$ are
$$ C(m_1, \dots, m_n) = \frac{n!}{m_1! 1!^{m_1} m_2! 2!^{m_2} \dots m_n! n!^{m_n}}.$$
 
Let $g \in L^2_j(I; \rr)$. We first want to show that $h \circ g$ is also in $L^2_j(I; \rr)$, that is, the $n$th derivative $\tfrac{d^n}{dt^n} h(g(t))$ is in $L^2(I; \rr)$, for all $0 \leq n \leq j$. 

We claim that each summand in the expression \eqref{eq:bruno} is in $L^2$. Indeed, since $g$ is in $L^2_j$ and $j \geq 1$, we have that $g$ is continuous. Further, since $h$ is smooth, we see that  the expression $h^{(m_1 + \dots + m_n)}(g(t))$ is continuous in $t$, and therefore bounded (because $I$ is compact). Moreover, unless $i=n$, the factors $g^{(i)}(t)$ are in $L^2_1$; hence, by Sobolev multiplication, their product is also in $L^2_1$, and thus in $L^2$. It follows that 
\begin{equation}
\label{eq:hmm}
h^{(m_1 + \dots + m_n)}(g(t))\cdot  \prod_{i=1}^n \bigl( g^{(i)}(t)\bigr)^{m_i} \in L^2,
\end{equation}
as long as $m_n =0$, since $h^{(m_1 + \dots + m_n)} \circ g$ is bounded. In the special case $m_n > 0$, we see that $m_n =1$ and all the other $m_i$ must be zero, so $\bigl( g^{(i)}(t)\bigr)^{m_i} = g^{(n)}(t)$ is still in $L^2$, and we get the same conclusion.

We have shown that $H$ maps $L^2_j(I; \rr)$ to $L^2_j(I; \rr)$. To see that $H$ is smooth, we compute its derivatives:
\begin{equation}
\label{eq:drh}
 (\D^r H)_{g}(\xi_1, \dots, \xi_r) = (h^{(r)} \circ g) \cdot \xi_1  \cdot \ldots \cdot \xi_r.
 \end{equation}

When $g, \xi_1, \dots, \xi_r \in L^2_j(I; \rr)$, by applying the fact we just proved with $h^{(r)}$ instead of $h$, and making use of the Sobolev multiplication $L^2_j \times L^2_j \to L^2_j$, we get that $(\D^r H)_{g}(\xi_1, \dots, \xi_r)$ is in $L^2_j$, as desired.

$(b)$ If $I$ is non-compact, we still have that $h^{(m_1 + \dots + m_n)}$ and $g$ are continuous. Further, by the Sobolev embedding theorem, $g \in L^2_j$ implies that $g$ is bounded.  Therefore, it is still the case that $h^{(m_1 + \dots + m_n)}\circ g$ is bounded. We also have the Sobolev multiplication $L^2_j \times L^2_j \to L^2_j$ for $j \geq 1$, so the same arguments as before show that \eqref{eq:hmm} holds, provided that not all $m_i$ are zero. 

When $n=0$ and all $m_i$ are zero, the constant function $1 =  \prod_{i=1}^n \bigl( g^{(i)}(t)\bigr)^{m_i}$ is not in $L^2$. Nevertheless, we claim that $h^{(m_1 + \dots + m_n)}(g(t)) \cdot 1 = h(g(t))$ is in $L^2$. To see this, we make use of the hypothesis $h(0)=0$. The derivative of $h$ at $0$ is
\begin{equation}
\label{eq:hzero}
 h'(0) = \lim_{t \to 0} \frac{h(t)}{t} \in \rr.
 \end{equation}
Recall that $g$ is bounded, so there is some $K > 0$ such that $|g(t)| \leq K$. Let $M$ be the supremum of $h(t)/t$ over the interval $[-K, K]$. From \eqref{eq:hzero} and the continuity of $h$, we see that this supremum is finite. Therefore, we have
$$ |h(g(t)) | \leq M |g(t)|,$$
for all $t \in \rr$. Since $g \in L^2$, this easily implies that $h \circ g$ is in $L^2$. 

We have now shown that $H$ maps $L^2_j(I; \rr)$ to $L^2_j(I; \rr)$. To check that $H$ is smooth, it remains to show that the expression \eqref{eq:drh} is in $L^2_j(I; \rr)$, when $g, \xi_1, \dots, \xi_r \in L^2_j(I; \rr)$. 

By Sobolev multiplication, the product $\xi = \xi_1  \cdot \ldots \cdot \xi_r$ is in $L^2_j$. For $0 \leq n \leq j$, we need to verify that
\begin{equation}
\label{eq:productderivatives}
 \frac{d^n}{dt^n} \bigl( (h^{(r)} \circ g) \cdot \xi \bigr ) = \sum_{s=0}^n  {n \choose s } \frac{d^s}{dt^s} (h^{(r)} \circ g) \cdot  \frac{d^{n-s}}{dt^{n-s}} \xi
 \end{equation}
is in $L^2$. 

For $s \neq 0$, the arguments above (based on Fa{\`a} di Bruno's formula), applied to $h^{(r)}$ instead of $h$, show that the $s^{\operatorname{th}}$ derivative of $h^{(r)} \circ g$ is in $L^2$. Further, since $n-s < n \leq j$, the $(n-s)^{\operatorname{th}}$ derivative of $\xi$ is in $L^2_1$. It follows that \eqref{eq:productderivatives} is in $L^2$. 

In the special case $s=0$, the expression $h^{(r)} \circ g$ is bounded, and $\frac{d^{n}}{dt^{n}} \xi \in L^2$, so we again get that their product is in $L^2$. This concludes the proof.
\end{proof}

Let $I \subseteq \rr$ and $\epsilon > 0$.  Consider the subset $L^2_j(I;\rr_{>\epsilon})$ inside $L^2_j(I;\rr)$ consisting of functions with values in $(\epsilon,\infty)$.  Note that for this subset to be non-empty, the interval $I$ must be compact.  Further, since $j \geq 1$, $L^2_j(I;\rr_{> \epsilon})$ is an open subset of $L^2_j(I;\rr)$.   
\begin{lemma}\label{lem:xi-inversion}
Let $I \subseteq \rr$.  The map $x \mapsto \frac{1}{x}$ induces a smooth map $H: L^2_j(I;\rr_{> \epsilon}) \to L^2_j(I; \rr)$, given by $H(g)(x) = \frac{1}{g(x)}$.  A similar statement applies to $x \mapsto \sqrt{x}$.
\end{lemma}   

\begin{proof}
Let $h: \rr \to \rr$ be any smooth function such that $h(x) = 1/x$ for $x \geq \epsilon$. Then, the result follows from Lemma~\ref{lem:composeSobolev}(a), applied to $h$. The same argument works for $\sqrt{x}$.
\end{proof}

With this, we can study the regularity of the $L^2$ norm as well.  
\begin{lemma}\label{lem:xi-norm}
Let $I \subseteq \rr$.  Let $\gamma \in \What_j(I \times Y)$ have $\| \gamma (t) \|_{L^2(Y)} \geq \epsilon > 0$ for all $t$.  For a neighborhood of $\gamma$ in $\What_j(I \times Y)$, the association $\gamma \mapsto \| \gamma \|_{L^2(Y)}$ is a smooth map to $L^2_j(I;\rr)$.  
\end{lemma}
\begin{proof}
Writing $\| \gamma \|_{L^2(Y)} = \langle \gamma, \gamma \rangle_{L^2(Y)}^{\frac 1 2}$, the result  follows from Lemma~\ref{xi:inner-product-path} and the second statement in Lemma~\ref{lem:xi-inversion}.  
\end{proof}

Finally, we have
\begin{lemma}\label{xi:alpha}
Let $I \subset \rr$ be compact. Postcomposition with the bump function $\alpha$ induces a smooth map $L^2_j(I;\rr_{\geq 0}) \to L^2_j(I;\rr_{\geq 0})$ given by $g \mapsto \alpha \circ g$.  
\end{lemma}
\begin{proof}
This is an immediate consequence of Lemma~\ref{lem:composeSobolev}(a).
\end{proof}

\subsection{On $L^2_{j}$ regularity of $\Xi_{\lambda}$ for compact cylinders}\label{xi:compact}
With the above technical lemmas established, we can now prove the first part of Proposition~\ref{prop:Xixy-paths}, which concerns the $L^2_{j}$ regularity of $\Xi_\lambda$ on compact cylinders. First, recall from Lemma~\ref{lem:Xixylambda} that $\Xi_\lambda$ extends to $\tW^\sigma_j$.  There is an obvious involution $\iota$ on $\tW^\sigma_j$ given by $(a,s,\phi) \mapsto (a,-s,\phi)$, such that $\Xi_\lambda(a,-s,\phi) = \Xi_\lambda(\iota(a,s,\phi)) = \iota (\Xi_\lambda(a,s,\phi))$.  When discussing the terms involved in $\Xi_\lambda$, we will suppress this from the notation; for example, if $(a,s,\phi)$ has $s < 0$, we will write $\beta_i(a,s,\phi)$ instead of $\beta_i(\iota(a,s,\phi))$.  Since $\Xi_\lambda$ is $\iota$-equivariant, it is easy to verify that this will not affect any smoothness or regularity.

\begin{proof}[Proof of Proposition~\ref{prop:Xixy-paths}\eqref{xi:paths}]
Let $I$ be a compact interval.  The idea is to apply the technical lemmas proved above on Sobolev multiplication and superposition operators to the explicit formula for $\Xi_\lambda$ in \eqref{eq:Xilambda-Ux}.  Fix $\gamma \in \tW^\tau_j(I \times Y)$, where we write $\gamma(t) = (a(t),s(t), \phi(t))$.  Recall that $\| \gamma(t) \|_{L^2(Y)} = 1$ for all $t \in I$.  We break $\Xi_\lambda$ into elementary pieces and study their regularity.

First, observe that $\langle \phi(t), \phi^i_\infty \rangle_{L^2(Y)} \in L^2_j(I;\rr)$ by Lemma~\ref{xi:inner-product-path}.  By Lemma~\ref{xi:pointwise-mult}, we have that $\langle \phi(t), \phi^i_\infty \rangle_{L^2(Y)} \cdot v^i_\lambda$ is in $L^2_j$.  

Note that whenever $\gamma(t) \in S^1 \cdot U_{x^i_\infty}$, we have $|\langle \phi(t), \phi^i_\infty \rangle_{L^2(Y)}| \geq \frac{1}{2}$ by our choice of the $\delta$-balls in the definition of $\Xi_\lambda$.  (Here we are using that there exists $u(t) \in S^1$ such that $u(t) \cdot \phi(t) \in U_{x^i_\infty}$.)  Using Lemma~\ref{lem:xi-inversion}, formula \eqref{eq:omegainfinity}, and the Sobolev multiplication $L^2_j(I;\rr) \times L^2_j(I;\rr) \to L^2_j(I;\rr)$, we can deduce that $\omega^i_\infty(\gamma) \in L^2_j(I; \cc)$.  Consequently, it follows from Lemmas~\ref{xi:pointwise-mult}, \ref{lem:xi-inversion}, and \ref{xi:alpha} that $\beta_i(\gamma) \in L^2_j(I;\rr)$.  Therefore, we see by Lemma~\ref{xi:pointwise-mult} that
$$
\left(a + \beta_i(\gamma) (a^i_\infty - a^i_\lambda),\ s + \beta_i(\gamma)(s^i_\infty - s^i_\lambda),f^i_\lambda(\gamma)\right) 
$$
is in $\What_j(I \times Y)$.  The slicewise $L^2$ norm of the spinorial component of this path is bounded below by $1/2$ by \eqref{xi:flambda-L2j}.   Therefore, we see that  
$$
\left(a + \beta_i(\gamma) (a^i_\infty - a^i_\lambda),\ s + \beta_i(\gamma)(s^i_\infty - s^i_\lambda),\frac{f^i_\lambda(\gamma)}{\| f^i_\lambda(\gamma) \|_{L^2(Y)}} \right), 
$$
is contained in $\What_j(I \times Y)$ by Lemmas~\ref{xi:pointwise-mult}, \ref{lem:xi-inversion}, and \ref{lem:xi-norm}.  This expression is $\Xi_\lambda(\gamma)$ by \eqref{xi:simplified}, as $\beta_n(\gamma) = 0$ for any $n \neq i$, when $\beta_i(\gamma) \neq 0$.    

The smoothness of $\Xi_\lambda$ follows from the relevant smoothness established in the sequence of technical lemmas above about Sobolev multiplication and superposition operators.  To see that $\Xi_\lambda$ is a diffeomorphism, we can apply similar arguments to show that $\Xi^{-1}_\lambda$ satisfies the same regularity and smoothness properties.       

We now show that $\Xi_{\lambda}$ depends smoothly on $\lambda$, at and near infinity. 
If we change the definition of $\Xi_\lambda$ by varying $x^i_\lambda$ smoothly,  the arguments above (i.e., the application of Lemmas~\ref{xi:inner-product-path}-\ref{xi:alpha}) show that we will vary the induced diffeomorphism in a smooth way.   This only applies for small variations of $x^i_\lambda$ whose spinorial component stays slicewise in the $L^2$ ball of size $\delta$ around $\phi^i_\infty$, but that is all we will need.  By Corollary~\ref{cor:corresp}, we see that $x_{f^{-1}(r)}$ is smooth at and near $r = 0$, and thus we obtain that $\Xi_{f^{-1}(r)}$ is smooth at and near $r = 0$.

It remains to verify that $\Xi_\lambda$ preserves $W^\tau_{j}(I \times Y)$ (i.e. $\Xi_\lambda$ preserves the condition $s \geq 0$).  This is trivial since $\Xi_\lambda$ takes $W^\sigma_0$ to $W^\sigma_0$.  
\end{proof}

\subsection{On $L^2_j$ regularity of $\Xi_{\lambda}$}\label{xi:non-compact}
Before studying further regularity properties of $\Xi_\lambda$, we need to introduce some further notation, generalizing the construction of $W^\tau_j(x_\lambda, y_\lambda)$. Let $I \subseteq \rr$ be an interval. Given a path $\zeta \in \What_{j, loc}(I \times Y)$ (analogous to the smooth reference path $\gamma_0$ in defining $W^\tau_j(x_\lambda, y_\lambda)$), we define 
$$
 \What_j(I \times Y, \zeta)= \{ \gamma \in \What_{j, loc}(I \times Y) \mid \gamma - \zeta \in \What_{j}(I \times Y) \}.
$$
We equip this space with the $L^2_j$ metric (not $L^2_{j,loc}$).  Note that many different functions can induce the same space.  Further, note that $ \What_j(I \times Y, \zeta)$ is a Banach manifold, since we have an affine identification with $\What_{j}(I \times Y)$.   

Towards studying the map $\Xi_\lambda$ on $\tW^\tau_j(x_\lambda, y_\lambda)$, we will begin with a regularity result for paths in $\tW^\sigma_0$ that are contained entirely in some neighborhood of an orbit of stationary points (or their involutes under $\iota$).  As usual, for an element $x \in \What_j$, we will abusively write $x$ for the induced constant path in $\What_{j,loc}(I \times Y)$.  

\begin{proposition}\label{xi:half-cylinder}
Fix a stationary point $x^i_\infty$ of $\Xqagcsigma$ as in the construction of $\Xi_\lambda$, and let $x^i_{\lambda}$ be the corresponding stationary points of $\Xqmlagcsigma$.  Let $T > 0$.   Let $\gamma \in \tW^\tau_j([T,\infty) \times Y, x^i_\lambda)$ be such that $\beta_i(\gamma) \equiv 1$.  Then, 
\begin{enumerate}[(a)]
\item in a neighborhood of $\gamma$ in $\What_j([T,\infty) \times Y, x^i_\lambda)$, $\Xi_\lambda$ induces a smooth map to $\What_j([T,\infty) \times Y, x^i_\infty)$, and   
\item this family of maps is smooth in $\lambda$ at and near infinity.
\end{enumerate}   
An analogous result applies for the half-cylinder $(-\infty, -T] \times Y$.   
\end{proposition} 

We postpone the proof of Proposition~\ref{xi:half-cylinder}, but do explain what we mean by smoothness of $\Xi_\lambda$ in $\lambda$, since the domain of the function is changing.  Fix a reference path $\gamma_0$ in $W^\tau_{j,loc}([T,\infty) \times Y)$ which agrees with $x_\infty$ for $t \gg 0$.  Note that $\Xi^{-1}_\lambda(\gamma_0)$ provides an $L^2_{j,loc}$ reference path which agrees with $x_\lambda$ for $t \gg 0$ by the work of Section~\ref{xi:compact}.  Then, $\Xi_\lambda$ induces a map, 
$$
\What_j([T,\infty) \times Y) \to \What_j([T,\infty) \times Y), \ \gamma - \Xi^{-1}_\lambda(\gamma_0) \mapsto \Xi_\lambda(\gamma) - \gamma_0,
$$  
defined in a neighborhood of $\gamma - \Xi^{-1}_\lambda(\gamma_0)$, where $\gamma \in \tW^\tau_{j,loc}([T,\infty) \times Y)$ satisfies $\gamma - \Xi^{-1}_\lambda(\gamma_0) \in \What_j([T,\infty) \times Y)$.  Since the domain and target of this map are constant in $\lambda$, we can make sense of smoothness in $\lambda$ at and near $\infty$.  There is an analogous notion of smoothness in $\lambda$ for maps from $\tW^\tau_j(x_\lambda, y_\lambda)$ to $\tW^\tau_j(x_\infty, y_\infty)$ (assuming they extend to neighborhoods in the relevant larger affine space), which is what is meant in Proposition~\ref{prop:Xixy-paths}\eqref{xi:paths-smooth-f} and Proposition~\ref{xi:half-cylinder}.  More generally, given $f : \What_{j,loc}(I\times Y) \to \What_{j,loc}(I \times Y)$, one can analogously define smoothness in $\gamma$ of the family of maps $f: \What_j(I \times Y, \gamma) \to \What_j(I \times Y, f(\gamma))$.  One can check that these notions are independent of the choice of reference path.  

Before proving Proposition~\ref{xi:half-cylinder}, we use it to complete the proof of Proposition~\ref{prop:Xixy-paths}\eqref{xi:paths-smooth-f}.  

\begin{proof}[Proof of Proposition~\ref{prop:Xixy-paths}\eqref{xi:paths-smooth-f}]
Let $\gamma \in \tW^\tau_j(x_\infty, y_\infty)$.  First, we establish that $\Xi_\lambda(\gamma) \in \tW^\tau_j(x_\infty, y_\infty)$.  Fix a reference path $\gamma_0$ from $x_\infty$ to $y_\infty$ which is constant outside of $[-T,T]$.  Write $x^i_\infty$ (respectively $x^n_\infty$) for the indexed stationary point from the construction of $\Xi_\lambda$ that is in the orbit of $x_\infty$ (respectively $y_\infty$).  

We have that $\Xi_\lambda(\gamma) \in \tW^\tau_{j,loc}(\rr \times Y)$ by Proposition~\ref{prop:Xixy-paths}\eqref{xi:paths}.  Since $\gamma \in \tW^\tau_j(x_\infty, y_\infty)$, we have that $\gamma - \gamma_0$ has finite $L^2_j(\rr \times Y)$ norm, with $j \geq 1$. Thus, we see that for $t \gg 0$, we must have that $\| \gamma(t) - y_\infty\|_{L^2(Y)}$ is sufficiently small so that $\gamma(t)$ is contained in $S^1 \cdot U_{x^n_\infty}$ and $\beta_n(\gamma(t)) = 1$.  A similar statement applies for $t \ll 0$, where we see that $\beta_i(\gamma(t)) = 1$.      

By Proposition~\ref{xi:half-cylinder}, we see that 
$$\Xi_\lambda(\gamma|_{[T, \infty)}) \in \What_j([T,\infty) \times Y, \Xi_\lambda(\gamma_0|_{[T,\infty)})).$$  A similar statement applies for $(-\infty, -T] \times Y$ as well.  Since $\gamma|_{[-T,T] } \in \tW^\tau_{j,loc}([-T,T] \times Y)$, we can put these three pieces of $\gamma$ together to see that $\gamma \in \tW^\tau_{j}(x_\infty, y_\infty)$.  

We will use a similar argument for smoothness.  We provide the argument for the existence of the first derivative of $\Xi_\lambda$; the higher derivatives are similar.  Let $\eta \in T_{\gamma} \What_j(\rr \times Y, \gamma_0) \cong \What_j(\rr \times Y)$.  We would like to see that 
\begin{equation}\label{xi:L2j-deriv-exists}
\lim_{h \to 0} \Big\| \frac{\Xi_\lambda(\gamma + h \cdot \eta) - \Xi_\lambda(\gamma)}{h} \Big\|_{L^2_j(\rr \times Y)}
\end{equation}
exists.  This limit needs to be taken with respect to $L^2_j$, and not $L^2_{j,loc}$.  As before, we have that $\| \eta(t) \|_{L^2(Y)}$ is uniformly bounded in $t$, because of the $L^2_j(\rr \times Y)$ bounds on $\eta$.  Therefore, for $h$ sufficiently small, we have that $\gamma + h \cdot \eta(t)$ is in $S^1 \cdot U_{x^n_\infty}$ for $t \geq T$, where we can choose $T$ independent of $h$.  A similar statement applies for $t \leq  -T$.  Therefore, by Proposition~\ref{xi:half-cylinder} the above limit exists in the $L^2_j$ topology when we replace $\rr \times Y$ with $[T,\infty) \times Y$ and with $(-\infty, -T] \times Y$.  By Proposition~\ref{prop:Xixy-paths}\eqref{xi:paths}, we have that the limit in \eqref{xi:L2j-deriv-exists} exists when we replace $\rr$ with $[-T,T]$.  This establishes the existence of the limit on $L^2_j(\rr \times Y)$.        

The smoothness of $\Xi_\lambda$ in $\lambda$ at and near infinity can again be deduced by a similar argument, applying Proposition~\ref{prop:Xixy-paths}\eqref{xi:paths} for a fixed compact interval and Proposition~\ref{xi:half-cylinder} outside of this interval.  
\end{proof}

Before moving on to Propositions~\ref{prop:Xixy-paths}\eqref{xi:Bgc} and \eqref{xi:Vgc} in the final subsection, we will give the promised proof of Proposition~\ref{xi:half-cylinder}.  This will be proved using similar techniques as for Proposition~\ref{prop:Xixy-paths}\eqref{xi:paths}; however, since we work in a region where some $\beta_n \equiv 1$ (and all other $\beta_{n'} \equiv 0$), we do not need to worry about $\omega^n_\infty$ and thus our job will be easier.  

We need one more technical lemma about superposition operators before giving the proof.

\begin{lemma}\label{lem:xi-inversion-global}
Fix $x\in \What_j$ non-zero and $I \subseteq \rr$.  There is a smooth map 
$$
\What_j(I \times Y, x) \to \What_j\Bigl(I \times Y, \frac{x}{\| x \|_{L^2(Y)}}\Bigr), \ \gamma \mapsto \frac{\gamma}{\| \gamma \|_{L^2(Y)}},
$$
defined in a neighborhood of any $\gamma \in \What_j(I \times Y, x)$ with $ \| \gamma(t)\|_{L^2(Y)} \geq \epsilon > 0$ for all $t \in I$.  Further, this family of maps is smooth in $x$.  
\end{lemma}
\begin{proof}
Without loss of generality, we assume that $\|x\|_{L^2(Y)}=1$. (We can reduce to this case by dividing both $\gamma$ and $x$ by $\|x\|_{L^2(Y)}$.)

Let 
$$\zeta := \gamma - x \in \What_j(I \times Y).$$ 
We are interested in showing that the following quantity is in $\What_j(I \times Y)$:
\begin{align*}
\frac{\gamma}{\| \gamma \|_{L^2(Y)}} - x &= \frac{\zeta + x}{\| \zeta + x \|_{L^2(Y)}} - x \\
&= \zeta \cdot (\|\zeta + x \|_{L^2(Y)}^{-1} - 1) + x  \cdot (\|\zeta + x \|_{L^2(Y)}^{-1} - 1) + \zeta.
\end{align*}
Since $\zeta \in \What_j(I \times Y)$ and $x \in \What_j$, in view of Lemma~\ref{xi:pointwise-mult}, it suffices to show that
\begin{equation}
\label{eq:Zeta}
 \|\zeta + x \|_{L^2(Y)}^{-1} - 1 \in L^2_j(I; \rr).
 \end{equation}

Set
$$ g:=  \|\zeta + x \|_{L^2(Y)}^{2} - 1 =  \langle \zeta,\zeta \rangle_{L^2(Y)} + 2\langle \zeta, x \rangle_{L^2(Y)}.$$
By applying Lemma~\ref{xi:inner-product-path}, we see that $g \in L^2_j(I; \rr)$. The expression in \eqref{eq:Zeta} can be written as $h \circ g$, where $h(y)=(y+1)^{-1/2} - 1.$ Note that $h(0)=0$. Thus, we can apply Lemma~\ref{lem:composeSobolev}(b), and deduce that $h \circ g \in L^2_j(I; \rr)$, as desired. 

Smoothness with respect to $\gamma$ follows from the smoothness statements in Lemmas~\ref{xi:inner-product-path}, ~\ref{xi:pointwise-mult}  and ~\ref{lem:composeSobolev}(b). 

Similar arguments can be used to prove smoothness with respect to $x$.
\end{proof}

With the above lemma, we can prove Proposition~\ref{xi:half-cylinder}.  

\begin{proof}[Proof of Proposition~\ref{xi:half-cylinder}]
Recall that by our assumptions on the path $\gamma$, we have that $\beta_i(\gamma) \equiv 1$ for some index $i$, and $\beta_n(\gamma) \equiv 0$ for all $n \neq i$.  Writing $\gamma(t) = (a(t), s(t), \phi(t))$, we have 
$$
\Xi_\lambda(\gamma) = \left (a(t) + a^i_\infty - a^i_\lambda, s(t) + s^i_\infty - s^i_\lambda, \frac{\phi(t) + \langle \phi(t), \phi^i_\infty \rangle_{L^2(Y)} 
\cdot v^i_\lambda}{\| \phi(t) + \langle \phi(t), \phi^i_\infty \rangle_{L^2(Y)} \cdot v^i_\lambda \|_{L^2(Y)}} \right). 
$$
Recall that the spinorial component of $\Xi_\lambda(\gamma)$ in this case is written as $\frac{f_\lambda(\gamma(t))}{\| f_\lambda(\gamma(t))\|_{L^2(Y)}}$.  The result now follows from applying Lemmas~\ref{xi:inner-product-path}, \ref{xi:pointwise-mult} and \ref{lem:xi-inversion-global} using the lower bounds given in \eqref{xi:flambda-L2j}.  
\end{proof}

\section{Extensions of $\Xi_{\lambda}$ to other path spaces}
\label{sec:xi-ext}
With the main technical results established, we are easily able to extend $\Xi_\lambda$ to larger path spaces to complete the proof of Proposition~\ref{prop:Xixy-paths} by proving parts \eqref{xi:Bgc} and \eqref{xi:Vgc}.  

\begin{proof}[Proposition~\ref{prop:Xixy-paths}\eqref{xi:Bgc}]
Recall from Section~\ref{sec:path} the space $\tC^{\gCoul,\tau}_j(x,y)$.  This space is larger than $\tW^\tau_j(x,y)$ due to the condition of pseudo-temporal gauge as opposed to temporal gauge.  Since $\tB^{\gCoul,\tau}_j([x_\lambda],[y_\lambda])$ is the quotient of $\tC^{\gCoul,\tau}_j(x_\lambda, y_\lambda)$ by the gauge group $\G^{\gCoul}_{j+1}(\rr \times Y)$, we will extend $\Xi_\lambda$ to a diffeomorphism 
$$
\Xi_\lambda : \tC^{\gCoul,\tau}_j(x_\lambda,y_\lambda) \to \tC^{\gCoul,\tau}_j(x_\infty, y_\infty),
$$
which commutes with the action by $\G^{\gCoul}_{j+1}(\rr \times Y)$, and this will induce the desired diffeomorphism for $\tB^{\gCoul,\tau}_j$.

Let $(a(t) + \alpha(t)dt, s(t), \phi(t)) \in \tC^{\gCoul,\tau}_j(x_\lambda, y_\lambda)$, where $(a(t), s(t), \phi(t)) \in \widetilde{W}^\tau_j(x_\lambda, y_\lambda)$.  By Proposition~\ref{prop:Xixy-paths}\eqref{xi:paths-smooth-f}, we have that $\Xi_\lambda(a(t), s(t), \phi(t)) \in \widetilde{\C}^{\gCoul,\tau}_j(x_\infty, y_\infty)$.  We now define 
$$
\Xi_\lambda(a(t) + \alpha(t)dt, s(t), \phi(t)) = \Xi_\lambda(a(t), s(t), \phi(t)) + (\alpha(t)dt,0,0).
$$
Since $\Xi_\lambda$ induces a diffeomorphism from $\tW^\tau_j(x_\lambda,y_\lambda)$ to $\tW^\tau_j(x_\infty, y_\infty)$, it is clear that $\Xi_\lambda$ induces a diffeomorphism from $\tC^{\gCoul,\tau}_j(x_\lambda,y_\lambda)$ to $\tC^{\gCoul,\tau}_j(x_\infty, y_\infty)$.  Thus, it remains to see that this induced map respects the (four-dimensional) gauge action.  Let $u \in \G^{\gCoul,\tau}_{j+1}(\rr \times Y)$.  For notation, we will write $u(t) \cdot V(t)$ to mean the path obtained by applying $u(t)$ pointwise.  

Using the $S^1$-equivariance of $\Xi_\lambda :\tW^\sigma_j \to \tW^\sigma_j$, we have 
\begin{align*}
u \cdot  \Xi_\lambda(a(t) + \alpha(t)dt, s(t), \phi(t)) &= u \cdot (\Xi_\lambda(a(t), s(t), \phi(t)) + (\alpha(t)dt,0,0)) \\
&= (-u^{-1}\frac{du}{dt} dt, 0, 0) + u(t) \cdot \Xi_\lambda(a(t),s(t),\phi(t)) + (\alpha(t)dt,0,0) \\
&= (-u^{-1}\frac{du}{dt} dt, 0, 0) + \Xi_\lambda(a(t),s(t), u(t) \cdot \phi(t)) + (\alpha(t)dt,0,0) \\
&= \Xi_\lambda((-u^{-1}\frac{du}{dt} dt + \alpha(t)dt + a(t),s(t), u(t) \cdot \phi(t))) \\
&= \Xi_\lambda( u \cdot (a(t) + \alpha(t)dt,s(t),\phi(t))).  \\
\end{align*}

It follows that $\Xi_\lambda$ induces a diffeomorphism from $\tB^{\gCoul,\tau}_j([x_\lambda], [y_\lambda])$ to $\tB^{\gCoul,\tau}_j([x_\infty], [y_\infty])$. 
\end{proof}

\begin{proof}[Proposition~\ref{prop:Xixy-paths}\eqref{xi:Vgc}]
We would like to show that $\Xi_\lambda$ induces a bundle map 
$$
\xymatrix{
\V^{\gCoul,\tau}_j \ar[r] \ar[d] & \V^{\gCoul,\tau}_j \ar[d] \\
\tB^{\gCoul,\tau}_j(x_\lambda, y_\lambda) \ar[r]^{\Xi_\lambda} & \tB^{\gCoul,\tau}_j(x_\infty, y_\infty).
}
$$
which is a diffeomorphism on the fibers.

Here, we recall that the bundle $\V^{\gCoul,\tau}_j(\rr \times Y)$ over $\tB^{\gCoul,\tau}_j([x],[y])$ comes from quotienting the bundle $\V^{\gCoul,\tau}_j(\rr \times Y)$ over $\tC^{\gCoul,\tau}_j([x],[y])$ by the gauge action.  Therefore, it suffices to show that $\Xi_\lambda$ induces a $\G^{\gCoul,\tau}_{j+1}(\rr \times Y)$-equivariant bundle map
$$
\xymatrix{
\V^{\gCoul,\tau}_j \ar[r]^{\Xi_\lambda}  \ar[d]& \V^{\gCoul,\tau}_j \ar[d] \\
\tC^{\gCoul,\tau}_j(x_\lambda, y_\lambda) \ar[r]^{\Xi_\lambda} & \tC^{\gCoul,\tau}_j(x_\infty, y_\infty),
}
$$
which is a diffeomorphism on the fibers.  
    
Recall that the bundle $\V^{\gCoul,\tau}_j$ over $\tC^{\gCoul,\tau}_j(x,y)$ has fiber over $(a(t) + \alpha(t)dt,s(t),\phi(t))$ consisting of paths in $(b(t), r(t), \psi(t))$ satisfying $\Re \langle \phi(t), \psi(t) \rangle_{L^2(Y)} = 0$ for all $t$.  In other words, if we write $\gamma(t) = (a(t), s(t), \phi(t))$, then the fiber over $(a(t) + \alpha(t)dt,s(t),\phi(t))$ consists of paths $\eta$ such that $\eta(t) \in \T^{\gCoul,\sigma}_{j, \gamma(t)}$ for all $t$.  

Since $\Xi_\lambda : \widetilde{W}^\sigma_j \to \widetilde{W}^\sigma_j$ is a diffeomorphism, we have that for any $x \in \tW^\sigma_j$, we have an induced linear isomorphism 
\begin{equation}\label{eq:Dx-Xi-lambda}
\D_x\Xi_\lambda: \T^{\gCoul,\sigma}_{j,x} \stackrel{\cong}{\to} \T^{\gCoul,\sigma}_{j,\Xi_\lambda(x)}.  
\end{equation}
Therefore, for an element $\eta \in \V^{\gCoul,\tau}_j$ which sits over $(a(t) + \alpha(t)dt,s(t),\phi(t))$, we can define $\Xi_\lambda(\eta)$ by pushforward.  More precisely, $\Xi_\lambda(\eta)$ is the path in $\V^{\gCoul,\tau}_j$ given by $$\bigl(\D_{(a(t),s(t),\phi(t))}  \Xi_\lambda\bigr)(\eta(t)).$$  
By construction $\Xi_\lambda(\eta)$ sits over $\Xi_\lambda(a(t) + \alpha(t) dt, s(t), \phi(t))$, and thus we have a bundle map.  It is not difficult to check that $(\Xi_\lambda)_*$ induces a diffeomorphism of the fiber over $(a(t) + \alpha(t) dt, s(t), \phi(t))$, as $(\Xi^{-1}_\lambda)_*$ provides the inverse. The gauge-equivariance follows from the gauge-equivariance of $\Xi_\lambda$ on $\C^{\gCoul,\tau}$.

The analogous result for $\tB^{\gCoul,\tau}_j([x_\infty], [y_\infty])/\mathbb{R}$ now follows from the above arguments together with the fact that if $[x_\infty] \neq [y_\infty]$, then the $\mathbb{R}$-action is free on $\tB^{\gCoul,\tau}_j([x_\infty], [y_\infty])$ and $\tB^{\gCoul,\tau}_j([x_\lambda], [y_\lambda])$.  
\end{proof}

\chapter{Convergence of approximate trajectories}\label{sec:trajectories1}
In this chapter and the subsequent one, we establish the analogous results of Chapter~\ref{sec:criticalpoints} for flow trajectories instead of stationary points.  In this section, we focus on results related to the convergence of approximate trajectories to honest trajectories as $\lambda \to \infty$.  We now return to the case that $\lambda = \llambda_i \gg 0$; we will often omit this assumption from the discussion.         
 
\section{Convergence downstairs}\label{sec:convergence-downstairs}
We start by discussing trajectories of $\Xqmlgc$ in the blow-down. There are two kinds of convergence results that one expects. The simpler one is $C^{\infty}_{loc}$ convergence of parameterized trajectories, and the more refined one is convergence of unparameterized trajectories to a broken trajectory. We already know that the former kind holds:

\begin{proposition}\label{prop:convergencenoblowup}
Let $I \subseteq \R$ be a closed interval, and $\gamma_n : I \to W$ be a sequence of trajectories of $\Xqmlngc$ contained in $B(2R)$, where $\lambda_n \to \infty$.  Then, there exists a subsequence of $\gamma_n$ for which the restrictions to any subinterval $I' \Subset I$ converge in the $C^{\infty}$ topology of $W(I' \times Y)$ to $\gamma$, a trajectory of $\Xqgc$. 
\end{proposition}
\begin{proof}
As discussed in Section~\ref{sec:verycompact}, $l + c_\q$ satisfies the hypotheses of Proposition~\ref{prop:proposition3perturbed}.  Thus, the result follows from Lemma~\ref{lem:convergencenoblowup}.  
\end{proof}

In particular, when $I=\R$, the conclusion of Proposition~\ref{prop:convergencenoblowup} is convergence in the $C^{\infty}_{loc}$ topology of $W(\R \times Y)$. We denote the resulting topological space by $W_{loc}(\R \times Y)$.

We now seek to show that a sequence of unparameterized trajectories for the approximate equations converge, in a certain sense, to a broken trajectory of $\Xqgc$. It will be convenient to work in the quotient $W/S^1$. Given $x \in W$, we write $[x]$ for its projection to $W/S^1$. If $x$ is a stationary point (of $\Xqgc$ or $\Xqmlgc$), we will say that $[x]$ is a {\em stationary point class}.   

Furthermore, given an interval $I \subset \R$ and a  trajectory $\gamma: I \to W$ of $\Xqgc$ or $\Xqmlgc$, we consider the associated {\em parameterized trajectory class} 
$$[\gamma] \in W(I \times Y)/S^1,$$
where $S^1$ acts by constant gauge transformations. If $I=\R$, we can further divide by reparameterizations (translations in the domain $\R$), and obtain the {\em unparameterized trajectory class}
$$[\cgamma] \in W(\R \times Y)/(S^1 \times \R).$$ 

When $I=\R$, Lemma 16.2.4 in \cite{KMbook}, translated into slicewise Coulomb gauge, shows that any parameterized trajectory class $[\gamma]$ of $\Xqgc$ with finite energy admits limit points $[x]$ and $[y]$ at $\pm \infty$, with $[x]$ and $[y]$ being stationary point classes of $\Xqgc$. Corollary~\ref{cor:endpointsBlowDown} gives the analogous result for trajectory classes of $\Xqmlgc$, provided these come from trajectories contained in $B(2R)$. Since the limit points are unchanged by reparameterizations, we can also talk about the limit points of unparameterized trajectory classes.

\begin{definition}
\label{def:broken}
Fix $[x]$ and $[y]$, stationary point classes of $\Xqgc$.  An {\em (unparameterized) broken trajectory class} of $\Xqgc$  from $[x]$ to $[y]$ consists of 
\begin{itemize}
\item an integer $m \geq 0$, the number of components;
\item an $(m+1)$-tuple of stationary point classes of $\Xqgc$: $[x] = [x_0], \ldots, [x_m] = [y]$ 
\item an unparameterized trajectory class $[\cgamma_i]$ of $\Xqgc$ from $[x_{i-1}]$ to $[x_{i}]$, for every $i = 1,\ldots,m$.  
\end{itemize} 
We will represent broken trajectory classes by the tuple $[\cgammas] = ([\cgamma_1],\ldots,[\cgamma_m])$.  
\end{definition}
  
Next, we want to say what it means for unparameterized trajectory classes of $\Xqmlngc$ to converge to a broken trajectory class of $\Xqgc$, as $\lambda_n \to \infty$. This is done in Definition~\ref{def:brokenconvergencedownstairs} below, which is inspired by the construction of the topology on the space of broken trajectories of $\Xq$, in \cite[Section 16.1]{KMbook}.  Given a trajectory $\gamma: \R \to W$ and $s \in \R$, recall that we write $\tau_s \gamma$ for the translate, $(\tau_s\gamma)(t) = \tau(s+t)$.  

From Proposition~\ref{prop:nearby} and Lemma~\ref{lem:implicitfunctionreducible}, for any stationary point class $[x]$ of $\Xqgc$, there is a corresponding (nearby) stationary point class of $\Xqmlgc$, which we denote by $[x_{\lambda}]$, and similarly in the blow-up.  We also recall from the discussion between Corollaries~\ref{cor:endpointsBlowUp} and \ref{cor:trajectory-Flambda} that for $\lambda \gg 0$, trajectories $[\gamma_\lambda]$ of $\Xqmlagcsigma$ in $(\vml \cap B(2R)^\sigma)/S^1$ connecting stationary points $[x_\lambda]$ and $[y_\lambda]$ are contained in $\B^{\gCoul,\tau}_k([x_\lambda],[y_\lambda])$.  In particular, such trajectories satisfy $\lim_{t \to -\infty} [\tau_t^* \gamma_\lambda] = [\gamma_{x_\lambda}]$ and $\lim_{t \to -\infty} [\tau_t^* \gamma_\lambda] = [\gamma_{y_\lambda}]$ in $\B^{\gCoul,\tau}_{k,loc}(\R \times Y)$ analogous to Definition~\ref{def:sw-moduli-space}.  We have the analogous result in the blow-down as well.  We will use these facts implicitly throughout the next two sections.   


\begin{definition}
\label{def:brokenconvergencedownstairs}
Fix $[x]$ and $[y]$, stationary point classes of $\Xqgc$. For $\lambda_n \to \infty$, consider a sequence of unparameterized trajectory classes $[\cgamma_n]$ of $\Xqmlngc$, coming from trajectories
$$\gamma_n: \R \to W^{\lambda_n} \cap B(2R)$$ 
and such that the endpoints of $[\cgamma_n]$ are the stationary point classes $[x_{\lambda_n}]$ (respectively $[y_{\lambda_n}]$) of $\Xqmlngc$ that correspond to $[x]$ (respectively $[y]$).

We say that $[\cgamma_n]$ converges to the broken trajectory class $[\cgammas_\infty] = ([\cgamma_{\infty,1}],\ldots,[\cgamma_{\infty,m}])$ if, for each $i=1, \dots, m$, there exist sequences of real numbers $(s_{n, i})_{n \geq 0}$ with
$$ s_{n, 1} < s_{n,2} < \dots < s_{n, m}$$
and
$$ s_{n, i} - s_{n, i-1} \to \infty \ \text{ as } \ n \to \infty,$$
such that the translates $[\tau_{s_{n,i}} \gamma_n]$ converge to some representative $[\gamma_{\infty, i}]$ of $[\cgamma_{\infty, i}]$ in the quotient topology of $W_{loc}(\R \times Y)/S^1$. 
\end{definition}

\begin{proposition}\label{prop:brokenconvergencedownstairs}
Fix $[x]$ and $[y]$, stationary point classes of $\Xqgc$.  Fix a sequence $\lambda_n \to \infty$ and a sequence of unparameterized trajectory classes $[\cgamma_n]$ of $\Xqmlngc$, going from $[x_{\lambda_n}]$ to $[y_{\lambda_n}]$, and such that the representatives $\gamma_n$ of $[\cgamma_n]$ are contained in $ W^{\lambda_n} \cap B(2R)$. Then, there exists a subsequence of $[\cgamma_n]$ that converges to a broken trajectory class $[\cgammas_{\infty}]$ of $\Xqgc$, in the sense of Definition~\ref{def:brokenconvergencedownstairs}.
\end{proposition}

\begin{proof}
The proof is essentially the same as \cite[Proposition 16.2.1]{KMbook}.  We provide an outline. 

Since the energy $\E(\gamma)$ of a trajectory $\gamma$ of $\Xqmlgc$ is unchanged by constant gauge transformations, we can talk about the energy $\E([\gamma])$ of the respective parameterized trajectory class. Similarly, since the functions $F_{\lambda}$ constructed in Chapter~\ref{sec:quasigradient} are $S^1$-invariant, we can talk about the drop in $F_{\lambda}$ for a parameterized trajectory class. In view of Corollary~\ref{cor:trajectory-Flambda}, the energy and the drop in $F_{\lambda}$ are commensurable.

Fix a compact interval $I$, constant $C > 0$, and a neighborhood $U_{[x]}$ of each stationary point class $[x]$ of $\Xqgc$, where $[x]$ is thought of as a constant trajectory class in $W(I \times Y)/S^1$. Using Proposition~\ref{prop:convergencenoblowup}, for any other compact interval $I'$, we can find $\epsilon > 0$ independent of $\lambda$ such that if $[\gamma]$ is a trajectory class of $\Xqmlgc$ with energy at most $C$, and with energy at most $\epsilon$ when restricted to  $I'$, then $[\gamma]|_{I \times Y}$ is contained in some $U_{[x]}$. (Compare with Lemma 16.2.2 in \cite{KMbook}.)

Next, observe that, since the trajectory classes $[\gamma_n]$ go from $[x_{\lambda_n}]$ to $[y_{\lambda_n}]$, and  
$$[x_{\lambda_n}] \to [x], \ \ [y_{\lambda_n}] \to [y], \ \ F_{\lambda_n} \to \Lq \ \text{as} \ n \to \infty,$$
we have that the drop in $F_{\lambda_n}$ along $[\gamma_n]$ is bounded. Hence, the energy of $[\gamma_n]$ is bounded, by a constant $K$ independent of $n$. With $\epsilon > 0$ chosen as in the previous paragraph, we find that, for each $n$, there are at most $2K/\epsilon$ integers $p$ such that 
$$\E( [\tau_p \gamma_n]|_{[-1,1]}) = \E([\gamma_n]|_{[p-1,p+1]}) > \epsilon.$$
Therefore, for all other $p$, we must have $[\tau_p \gamma_n] \in U_{[x]}$ for some stationary point class $[x]$ of $\Xqgc$. 

Starting from here, for each $n \gg 0$, we decompose $\R$ into finitely many intervals $I_i^n = [a_i^n, b_i^n]$ of fixed length, and intervals 
$$J_0^n=(-\infty, a_1^n],  \ J_i^n=[b_i^n, a_{i+1}^n], \ J_m^n=[b_m^n, \infty),$$ with the length of each $J_i^n$ going to infinity as $n \to \infty$, and the number $m$ of intervals being independent of $n$. The restriction of $[\gamma_n]$ to each $J_i^n$ lies near a stationary point class of $\Xqgc$, and these point classes provide the breaking points of the limiting broken trajectory class. By applying  Proposition~\ref{prop:convergencenoblowup} to the restrictions of $[\gamma_n]$ to $I_i^n$, we can arrange that they are convergent in $C^{\infty}_{loc}$. This gives the required convergence to a broken trajectory class.  
 \end{proof}

\section{Convergence of parameterized trajectories in the blow-up}
We now move to the blow-up $W^{\sigma}$. We are interested in showing that the trajectories of $\Xqmlgcsigma$ are close to those of $\Xqgcsigma$, given appropriate control on the $\lambda$-spinorial energy.  The goal of this subsection is to establish the following convergence result for parameterized trajectories.  Before doing so, a quick notational remark.  Sometimes we will be interested in studying the image of a path in $W^\sigma/S^1$ in the blow-down.  When doing so, we will use the notation $\gamma^\tau$ for the path upstairs and $\gamma$ for the blow-down.  

\begin{proposition}\label{prop:convergence1}
Fix $\omega > 0$ and a compact interval $I = [t_1,t_2] \subset \R$. Consider a smaller interval $I_{\epsilon}=[t_1 + \epsilon, t_2 - \epsilon]$ for $\epsilon > 0$. Suppose that $$\gamma^\tau_n: I \to (\vmln \cap B(2R))^{\sigma}$$ is a sequence of trajectories of $\Xqmlngcsigma$, where $\lambda_n \to \infty$.  Furthermore, suppose that at the ends of $I_{\epsilon}$ we have
$$\Lambda_{\q^{\lambda_n}}(\gamma^\tau_n(t_1+\epsilon)) \leq \omega, \ \ \ \Lambda_{\q^{\lambda_n}}(\gamma^\tau_n(t_2-\epsilon)) \geq -\omega,$$ for all $n$.  Then, there exists a subsequence of $\gamma^{\tau}_n$ for which the restrictions to any $I' \Subset I_{\epsilon}$ converge in the $C^{\infty}$ topology of $W^\tau(I_{\epsilon} \times Y)$ to $\gamma^\tau$, a trajectory of $\Xqgcsigma$. 
\end{proposition}

\begin{remark}\label{rmk:compactnessdifferences}
The analogous compactness result for trajectories of $\Xqsigma$ is Theorem 10.9.2 in \cite{KMbook}. However, there are a few differences between the statement of Proposition~\ref{prop:convergence1} and   \cite[Theorem 10.9.2]{KMbook}. Of course, the main one is that our result deals with trajectories of the  vector fields $\Xqmlngcsigma$, which approximate the vector field $\Xqgcsigma$ in Coulomb gauge.  Another difference is that in \cite[Theorem 10.9.2]{KMbook} one requires a bound on the energy of trajectories, that is, on the drop of the perturbed Chern-Simons-Dirac functional $\Lq$; this is unnecessary in our setting since our trajectories are assumed to be in $B(2R)^{\sigma}$, which automatically gives bounds on $\Lq$. Also, in \cite[Theorem 10.9.2]{KMbook}, the assumption is that $\q$ is a $k$-tame perturbation, and the conclusion is convergence in $L^2_{k+1}$. In our setting, $\q$ is tame (for all $k$), and we will use bootstrapping and a diagonalization argument for subsequences to get convergence in $C^{\infty}$. Finally, in  \cite[Theorem 10.9.2]{KMbook}, the conclusion is convergence after gauge transformations. In our setting trajectories start in temporal and slicewise global Coulomb gauge; the remaining gauge consists only of constant transformations in $S^1$. Since $S^1$ is compact, there is no need to change trajectories by gauge before they converge. Compare the remark after Corollary 5.1.8 on \cite[p.110]{KMbook}. 
\end{remark}

Before proving Proposition~\ref{prop:convergence1}, we need a couple of lemmas. The first is the following unique continuation result, analogous to \cite[Proposition 10.8.1]{KMbook}:

\begin{lemma}\label{lem:uniquecontinuation}
Let $\gamma(t) = (a(t),\phi(t))$ be a trajectory of $\Xqmlgc$ in $B(2R)$ for some $\lambda \in (0,\infty]$.  If $\phi(t) = 0$ for some $t$, then $\phi(t) \equiv 0$.
\end{lemma}
\begin{proof}
Since $\gamma(t)$ is a trajectory, we have 
\[-\frac{d}{dt}\phi(t) = D \phi(t) + (\pml c_\q)^1(\gamma(t)).\]  
Furthermore, $\gamma(t)$ is continuous and thus $(\pml c_\q )^1(\gamma(t))$ is a continuous path in $L^2_k(Y;\mathbb{S})$ as well.  By \cite[Lemma 7.1.3]{KMbook}, it suffices to show that $\| (\pml c_\q)^1(\gamma(t)) \|_{L^2} \leq C \| \phi(t) \|_{L^2}$ for all $t$.    This follows from Lemma~\ref{lem:linearbounds}. 
\end{proof}

The second result we need is the analogue of \cite[Lemma 10.9.1]{KMbook}:
\begin{lemma}
\label{lemma:1091}
There is a constant $C > 0$ such that, for any $\lambda \gg 0$, any interval $[t_1, t_2] \subseteq \R$, trajectory $\gamma^{\tau}: [t_1, t_2] \to B(2R)^{\sigma}$ of $\Xqmlgcsigma$, and any $t \in [t_1, t_2]$, we have
$$\frac{d}{dt} \Lambda_{\qml}(\gamma^{\tau}(t)) \leq C \cdot \|\Xqmlgc(\gamma(t))\|_{L^2_k(Y)},$$
where $\gamma$ is the projection of $\gamma^{\tau}$ in the blow-down.
\end{lemma}
\begin{proof}
This is similar to the proof of Lemma 10.9.1 in \cite{KMbook}. Note that the bound in \cite{KMbook} involved a function $\zeta(\gamma(t))$; in our case we can take this to be a constant, because we assumed that $\gamma(t) \subset B(2R)$.
\end{proof}

\begin{proof}[Proof of Proposition~\ref{prop:convergence1}]
The argument follows that of Theorem 8.1.1 and Theorem 10.9.2 in \cite{KMbook} nearly verbatim.  We give the argument for completeness.      

Write $\gamma^\tau_n(t) = (a_n(t), s_n(t), \phi_n(t))$.  First, by Proposition~\ref{prop:convergencenoblowup} we can find a subsequence of the blown-down sequence $\gamma_n(t) = (a_n(t), s_n(t) \phi_n(t))$ for which the restriction to any $I' \Subset I$ converges in all $L^2_j(I' \times Y)$ norms.  Choose $I'$ such that $I_\epsilon \Subset I'$.  In particular, we have bounds (and convergence) for all $L^2_j(I' \times Y)$ norms on $a(t)$.  Therefore, we will now focus on the convergence of $s_n(t)$ and $\phi_n(t)$.        

By Lemma~\ref{lem:uniquecontinuation}, each trajectory is either slicewise reducible or irreducible.  After passing to a further subsequence, we can assume that $\gamma^\tau_n$ are all of the same type.  We first consider the case that $\gamma^\tau_n(t)$ are irreducible for all $n$ and $t$.  If the limit of $\gamma_n$ is irreducible, we are done using the result in the blow-down.  Therefore, we assume that the limit is reducible for all $t$.  It will be useful to simultaneously think of $s_n(t) \phi_n(t)$ as a four-dimensional spinor $z_n \Phi_n$, where $z_n = \| s_n \phi_n \|_{L^2(I_\epsilon \times Y)}$ and $\| \Phi_n \|_{L^2(I_\epsilon \times Y)} = 1$, so $z_n \Phi_n = s_n \phi_n$.  (We have $z_n \neq 0$ since $s_n (t) \neq 0$ for {\em every} $t \in I'$.) The subtle change from $I'$ to $I_\epsilon$ when computing $L^2$ norms will be used shortly.     

Since $\gamma_n$ is an approximate trajectory in the blow-down, so is $e^{i \theta} \gamma_n$, and therefore we see that  
$$
- D^+(e^{i\theta} z_n\Phi_n) =  (\pmln c_\q)^1(a_n, e^{i\theta}  z_n\Phi_n),  
$$
for all $\theta$.  (Here $D^+$ denotes the Dirac operator in four-dimensions for the trivial connection induced by $A_0$.)  By differentiating with respect to $\theta$ and evaluating at $\theta = 0$, we obtain 
$$
- D^+(z_n \Phi_n) = \D_{\gamma_n} (\pmln c_\q)^1(0,z_n\Phi_n),
$$
and thus by complex linearity, 
$$
- D^+(\Phi_n) = \D_{\gamma_n} (\pmln c_\q)^1(0,\Phi_n).
$$

By Lemma~\ref{lem:fam}, the $L^2(I' \times Y)$ bounds on $\Phi_n$ and $L^2_k(I' \times Y)$ bounds on $\gamma_n$ give that the right-hand side is $L^2(I' \times Y)$ bounded.  By ellipticity, we see that $\Phi_n$ is $L^2_1$-bounded on any interior cylinder.  We can do further bootstrapping to obtain $L^2_{j}(I_\epsilon \times Y)$ bounds on the four-dimensional spinor $\Phi_n$ for all $j$.  In particular, we can arrange for a further subsequence for which $\Phi_n$ converges in all $L^2_j(I_\epsilon \times Y)$ norms to a spinor $\Phi$ with $\| \Phi \|_{L^2(I_\epsilon \times Y)} = 1$.  

By Lemma~\ref{lemma:1091} and the bounds on $\Lambda_{\qmln}$ at the endpoints of $I_\epsilon$, we see that $\Lambda_{\qmln}(\gamma^\tau_n(t))$ is uniformly bounded in $n$ and $t \in I_\epsilon$.  Since $-\Lambda_{\qmln}(\gamma^\tau_n(t)) = \frac{d}{dt} \log s_n(t)$, it follows that there exists a constant $K > 0$ independent of $n$ and $t \in I_\epsilon$ such that 
$$ s_n(t) \geq K \| s_n \|_{L^2(I_\epsilon)} = K \| s_n \phi_n \|_{L^2(I_\epsilon)} = K z_n.$$
(We used here than $\|\phi_n(t)\|_{L^2(Y)}=1$ for all $t$.) We deduce that 
$$ 
\| \Phi_n(t) \|_{L^2(Y)} \geq K > 0
$$
for all $t$.  Since $\Phi_n$ converges to $\Phi$ in all $L^2_j(I_\epsilon \times Y)$ norms (and thus uniformly in all $L^2_j(Y)$ norms), we see that $\phi_n(t) = \frac{\Phi_n(t)}{\| \Phi_n(t) \|_{L^2(Y)}}$ is bounded in all $L^2_j(I \times Y)$ norms by Lemma~\ref{lem:xi-inversion-global}.  It follows that we have the desired convergence for the $\phi_n$.  

We now focus on the convergence for $s_n$.  Of course, rather than obtaining bounds on $L^2_j(I_\epsilon \times Y)$ norms as above, we could have used any intermediate cylinder with $I_\epsilon \subset I'' \Subset I'$.   We in fact opt for $I''$, since we will need to bootstrap again to get $L^2_j(I_\epsilon)$ bounds on $s_n$.   By \eqref{eq:Xqmlgcsigmaformula} and \eqref{eq:lambda-spinorial}, since both $\gamma_n(t)$ and $\phi_n(t)$ are bounded in all $L^2_j(I'' \times Y)$ norms, we see that $\Lambda_{\qmln}(\gamma^\tau_n(t))$ is bounded in all $L^2_j(I'')$ norms.  Since $\gamma^\tau_n(t)$ is a trajectory of $\Xqmlngcsigma$, \eqref{eq:Xqmlgcsigmaformula} shows that 
\begin{equation}\label{eq:bootstrap-lambdaq}
-\frac{d}{dt} s_n(t) = -\Lambda_{\qmln}(\gamma^\tau_n(t)) s_n(t)).
\end{equation}
Note that for all $n$ and $t \in I'$, $0 < s_n(t) \leq 2R$, so we obtain uniform bounds on $\frac{d}{dt} s_n(t)$ by the bounds on $s_n(t)$ and $\Lambda_{\qmln}(\gamma^\tau_n(t))$.  We may continue to bootstrap using \eqref{eq:bootstrap-lambdaq} together with Sobolev multiplication to obtain bounds on $s_n$ in all $L^2_j(I_\epsilon)$ norms, and consequently convergence in $C^\infty$.  This completes the proof in the case that the $\gamma^\tau_n$ are irreducible.

Now, we consider the case that $\gamma^\tau_n(t)$ is reducible.  As in the irreducible case we have $L^2_1(I')$ bounds on $\Lambda_{\qmln}(\gamma^\tau_n(t))$, by Lemma~\ref{lemma:1091} and the bounds at the endpoints.  In this case, choose $s^*_n(t)$ to solve
$$- \frac{d}{dt} s^*_n(t) = \Lambda_{\qmln}(a_n(t),0,\phi_n(t)) s^*_n(t), $$ 
where we ask that $0 < s^*_n (t) < M$ for all $t \in I$ and $n$.  Since this is a one-dimensional ODE defined on a compact interval, we can easily arrange for such a solution.  It is straightforward to verify that 
$$
-\left(\frac{d}{dt} + \Xqmlngc \right)^1(a_n(t), s^*_n(t) \phi_n(t)) = \D_{(a_n(t),0)} (\pmln c_\q)^1(0,\phi_n(t)).  
$$
We can now repeat the same arguments as above to obtain the desired convergence of $\phi_n(t)$.  
\end{proof}

\begin{corollary}\label{cor:convergence}
Fix $\omega > 0$ and a closed (possibly non-compact) interval $I \subseteq \R$. Suppose that $\gamma^\tau_n: I \to (W^{\lambda_n} \cap B(2R))^{\sigma}$ is a sequence of trajectories of $\Xqmlngcsigma$, where $\lambda_n \to \infty$.  Furthermore, suppose that $|\Lambda_{\q^{\lambda_n}}(\gamma^\tau_n(t))| \leq \omega$ for all $n$ and $t \in I$.  Then, there exists a subsequence of $\gamma^{\tau}_n$ for which the restrictions to any subinterval $I' \Subset I$ converge in the $C^{\infty}$ topology of $W^\tau(I' \times Y)$ to $\gamma^\tau$, a trajectory of $\Xqgcsigma$.  
\end{corollary}

\begin{proof}
This follows from Proposition~\ref{prop:convergence1}, by a diagonalization argument.
\end{proof}

When $I=\R$, we let $W^{\tau}_{loc}(\R \times Y)$ be $W^{\tau}(\R \times Y)$ with the $C^{\infty}_{loc}$ topology. Then, the conclusion of Corollary~\ref{cor:convergence} is convergence in $W^{\tau}_{loc}(\R \times Y)$.

\section{Near-constant approximate trajectories}
This subsection contains several technical results, analogous to the ones in Sections 13.4 and 13.5 of \cite{KMbook}. We will use these to study the moduli spaces of broken trajectories of $\Xqmlagcsigma$, leading up to Propositions~\ref{prop:brokenconvergence} and \ref{prop:L2k} below.

We start by establishing the analogue of Lemma 13.4.4 in \cite{KMbook}, which gives bounds on the distance of a trajectory from a constant trajectory on $[t_1, t_2] \times Y$ in terms of $\L_\q$.  Write $I = [t_1,t_2]$.  We follow the notation of Sections~\ref{sec:Flambda} and \ref{sec:control-near-stationary}.  More precisely, given a stationary point $x_\infty = (a_\infty, \phi_\infty)$ of $\Xqgc$, there is an associated orbit $\O_\lambda$ of stationary points of $\Xqmlgc$.  We define $x_\lambda = (a_\lambda, \phi_\lambda)$ to be the point in $\O_\lambda$ which is $L^2$ closest to $x_\infty$.    
   
\begin{lemma}\label{lem:L21-bounds-Flambda}
Fix a stationary point $x_\infty = (a_\infty, \phi_\infty)$ of $\Xqgc$, which we treat as a (constant) trajectory of $\Xqgc$ on $I \times Y$.  There exists an $S^1$-invariant neighborhood $U$ of $x_\infty$ in $W_{1}(I \times Y)$ and a constant $C$ independent of $\lambda \gg 0$ satisfying the following.  If $\gamma \in U$ is a trajectory of $\Xqmlgc$, there exists a constant gauge transformation $u \in S^1$ such that 
$$
\| u \cdot \gamma - x_\lambda \|_{L^2_1(I \times Y)} \leq C (F_\lambda(\gamma(t_2)) - F_\lambda(\gamma(t_1))). 
$$
\end{lemma}
\begin{proof}
In \cite[Lemma 13.4.4]{KMbook}, Kronheimer and Mrowka choose a gauge transformation of $\gamma$ that moves it into the Coulomb-Neumann slice through $x_\infty$.  In our context, since we prefer to work with $\gamma$ in slicewise global Coulomb gauge, we choose $u \in S^1$ so that 
$$\Re \langle u \cdot \gamma(t_1), (0,i\phi_\lambda) \rangle_{L^2(Y)} = 0.$$
To suppress $u$ from the notation, for the rest of the proof, we assume $\Re \langle \gamma(t_1), (0,i\phi_\lambda) \rangle_{L^2(Y)} = 0$, so that $u = 1$.

We have
\begin{equation}\label{eq:distance-from-constant}
\| \gamma - x_\lambda \|^2_{L^2_1(I \times Y)} = \int^{t_2}_{t_1} \| \gamma - x_\lambda \|^2_{L^2_1(Y)} dt + \int^{t_2}_{t_1} \| \frac{d}{dt} \gamma \|^2_{L^2(Y)} dt.
\end{equation}

In view of \eqref{eq:trajectory-Flambda}, it suffices to bound the right-hand side  by a constant times $\int^{t_2}_{t_1} \| \frac{d}{dt} \gamma \|^2_{L^2(Y)} dt$, or equivalently, $\int^{t_2}_{t_1} \| \Xqmlgc(\gamma(t)) \|_{L^2(Y)}^2 dt$.  Since the second term in \eqref{eq:distance-from-constant} is already of this form, it suffices to bound the first term, i.e., to find a constant $C$ such that 
$$
\int^{t_2}_{t_1} \| \gamma - x_\lambda \|^2_{L^2_1(Y)} dt \leq  C \cdot \int^{t_2}_{t_1} \| \frac{d}{dt} \gamma \|_{L^2(Y)}^2 dt = C \cdot \int^{t_2}_{t_1} \| \Xqmlgc(\gamma(t)) \|_{L^2(Y)}^2 dt.
$$

By our choice of $\q$, we have that $x_\infty$ and $x_\lambda$ are non-degenerate stationary points of $\Xqgc$ and $\Xqmlgc$, respectively.  There exist constants $C'>0$ and $\delta > 0$ (independent of $\lambda$) such that for $\| x - x_\lambda \|_{L^2_1(Y)} < \delta$, we have
\begin{equation}\label{eq:nearby-L21-traj-bounds}
\| x - x_\lambda \|^2_{L^2_1(Y)} \leq C' \left(\|\Xqmlgc(x) \|^2_{L^2(Y)} + |\Re \langle x, (0,i\phi_\lambda) \rangle_{L^2(Y)}|^2\right),  
\end{equation}
by an argument similar to that in Lemma~\ref{lem:approximate-eigenvalue-bounds}, where we do not include the second term on the right-hand side if $x_\lambda$ is reducible.  Indeed, if $x$ is real $L^2$ orthogonal to $(0,i\phi_\lambda)$, then the inequality follows directly from the non-degeneracy of $\Xqmlgc$.  More generally, we use 
\begin{align*}
\| x - x_\lambda \|_{L^2_1(Y)} &\leq \| e^{i\theta}x - x_\lambda \|_{L^2_1(Y)} + \| x - e^{i\theta} x \|_{L^2_1(Y)} \\
&\leq C' \| \Xqmlgc(x) \|_{L^2(Y)} + |(1-e^{i\theta})| \| x \|_{L^2_1(Y)} \\
&\leq C' \| \Xqmlgc(x) \|_{L^2(Y)} + C' | \Re \langle x, (0,i\phi_\lambda) \rangle_{L^2(Y)} |,
\end{align*}
where $\theta$ is the angle between $x$ and $(0,i\phi)$, i.e. $\sin \theta = \frac{\Re \langle x, (0,i\phi_\lambda) \rangle_{L^2(Y)}}{\| x\|_{L^2(Y)} \| i \phi_\lambda \|_{L^2(Y)}}$.  To see the last inequality, we use that $x$ is $L^2_1$ bounded and that for irreducibles, the $\phi_\lambda$ is uniformly $L^2$ bounded above and below. 
  
We therefore choose the neighborhood, $U$, in the statement of the lemma by extending a $\delta$-neighborhood of $x_\lambda$ in $L^2_1(I \times Y)$ to be gauge invariant.  

By \eqref{eq:nearby-L21-traj-bounds}, for a trajectory $\gamma \in U$, we have the desired bounds in the case that $x_\lambda$ is reducible.  If $x_\lambda$ is irreducible, it remains to bound 
$$ 
\int^{t_2}_{t_1} |\Re \langle \gamma(t) , (0,i\phi_\lambda) \rangle_{L^2(Y)}|^2 dt,  
$$
in terms of $\int^{t_2}_{t_1} \| \frac{d}{dt} \gamma \|^2_{L^2(Y)}dt $.  
We write 
\begin{align*}
\int^{t_2}_{t_1} |\Re \langle \gamma(t), (0,i\phi_\lambda) \rangle_{L^2(Y)}|^2 dt  &= \int^{t_2}_{t_1} |\Re \langle \gamma(t) - \gamma(t_1), (0,i\phi_\lambda) \rangle_{L^2(Y)}|^2 dt \\
&\leq C'  \int^{t_2}_{t_1} \| \gamma(t) - \gamma(t_1) \|_{L^2(Y)}^2 dt \\
&\leq C' \int^{t_2}_{t_1} \| \int^{t}_{t_1} (\frac{d}{ds}\gamma) ds \|^2_{L^2(Y)} dt \\
&\leq C' \int^{t_2}_{t_1} (t - t_1) \int^t_{t_1} \| \frac{d}{ds} \gamma \|^2_{L^2(Y)} ds dt \\
& \leq C' (t_2 - t_1) \int^{t_2}_{t_1} \int^{t_2}_{t_1} \| \frac{d}{ds} \gamma \|^2_{L^2(Y)} ds dt \\
& = C' (t_2 - t_1)^2 \int^{t_2}_{t_1} \| \frac{d}{ds} \gamma \|^2_{L^2(Y)} ds,
\end{align*}
where the fourth line follows from Cauchy-Schwarz.  This completes the proof.
\end{proof}

In \cite[Sections 13.4 and 13.5]{KMbook}, Lemma 13.4.4 is the starting point for a sequence of results about trajectories in neighborhoods of stationary points. In our setting, Lemma~\ref{lem:L21-bounds-Flambda} above gives analogous results, by essentially the same arguments. We state the main results below, but omit the proofs. 

First, by bootstrapping, we obtain the following from Lemma~\ref{lem:L21-bounds-Flambda}, analogous to \cite[Proposition 13.4.7]{KMbook}.  

\begin{lemma}\label{lem:L2k-bounds-Flambda}
Fix a stationary point $x_\infty$ of $\Xqgc$, which we treat as a (constant) trajectory of $\Xqgc$ on $I \times Y$.  There exists an $S^1$-invariant neighborhood $U$ of $x_\infty$ in $W_{k}(I \times Y)$ and a constant $C$ independent of $\lambda \gg 0$ satisfying the following.  If $\gamma \in U$ is a trajectory of $\Xqmlgc$, there exists a constant gauge transformation $u \in S^1$ such that 
$$
\| u \cdot \gamma - x_\lambda \|_{L^2_{k+1}(I' \times Y)} \leq C (F_\lambda(\gamma(t_2)) - F_\lambda(\gamma(t_1))), 
$$
for any compact interval $I' \Subset I$.
\end{lemma}

Next, we have two results about trajectories in the blow-up.  These are the analogues of Proposition 13.4.1 and Corollary 13.4.8 in \cite{KMbook} respectively.  

\begin{proposition}\label{prop:13.4.1-analogue}
Let $x_\infty \in B(2R)^\sigma$ be a stationary point of $\Xqgcsigma$, $x_\lambda$ a nearby stationary point of $\Xqmlgcsigma$, and $I' \Subset I = [t_1,t_2]$.  Then, there exists a constant $C$ and a gauge-invariant neighborhood $U$ of the constant trajectory $x_\infty$ in $W^\tau_k(I \times Y)$, independent of $\lambda$, such that for every trajectory $\gamma^\tau: I \to (\vml \cap B(2R))^\sigma$ of $\Xqmlgcsigma$ which belongs to $U$, there is a gauge transformation $u \in S^1$ such that:
\begin{enumerate}[(i)]
\item if $x_\infty$ is irreducible, then $$\| u \cdot \gamma^\tau - x_\lambda \|_{L^2_{k+1}(I' \times Y)} \leq C \left ( F_\lambda(\gamma(t_1)) - F_\lambda(\gamma(t_2))\right),$$
\item if $x_\infty$ is reducible, then $$\| u \cdot \gamma^\tau - x_\lambda \|_{L^2_{k+1}(I' \times Y)} \leq C \left ( \Lambda_{\qml}(\gamma(t_1)) - \Lambda_{\qml}(\gamma(t_2)) + (F_\lambda(\gamma(t_1)) - F_\lambda(\gamma(t_2)))^{\frac{1}{2}} \right ).$$
\end{enumerate}
\end{proposition}

\begin{proposition}\label{prop:13.4.8-analogue}
Let $x_\infty \in B(2R)^\sigma$ be a stationary point of $\Xqgcsigma$ and $x_\lambda$ the  corresponding stationary point of $\Xqmlgcsigma$.  There is a constant $C$ and a gauge-invariant neighborhood $U$ of the constant trajectory in $W_k([-1,1] \times Y)$ obtained from blowing down $x_\infty$, independent of $\lambda \gg 0$, with the following property.  If $\gamma^\tau : [-1,1] \to  (\vml \cap B(2R))^{\sigma}$ is a trajectory of $\Xqmlgcsigma$ whose blow-down $\gamma$ is in $U$, then 
$$
\frac{d}{dt} \Lambda_{\qml}(\gamma^{\tau}(t))\Big|_{t = 0} \leq C (F_\lambda(\gamma(-1)) - F_\lambda(\gamma(1)))^{\frac{1}{2}}.  
$$ 
\end{proposition}

For trajectories that converge to stationary points at the ends, we are able to obtain exponential decay on the value of $F_\lambda$ near stationary points in the blow-up, analogous to that in \cite[Proposition 13.5.1]{KMbook}.  Since $F_\lambda: \vml \cap B(2R) \to \mathbb{R}$ is $S^1$-invariant, we obtain an induced map from $(\vml \cap B(2R))^\sigma/S^1$ to $\mathbb{R}$.  By a slight abuse of notation we will use $F_\lambda$ for this induced map as well.  Recall from Chapter~\ref{sec:criticalpoints}, for every stationary point $[x_\infty]$ of $\Xqagcsigma$ with grading in $[-N,N]$, we have a unique corresponding stationary point $[x_\lambda]$ of $\Xqmlagcsigma$ for $\lambda \gg 0$.  

\begin{proposition}\label{prop:exponential-decay-blowup}
Let $[x_\infty]$ be a stationary point of $\Xqagcsigma$ with grading in $[-N,N]$.  There exists $\delta > 0$ such that for $\lambda \gg 0$, and for every trajectory $[\gamma]: [0,\infty) \to (\vml \cap B(2R))^\sigma/S^1$ of $\Xqmlagcsigma$ with $\lim_{t \to \infty} [\tau^*_t \gamma] = [x_\lambda]$ in $L^2_{k,loc}$, there exists $t_0$ such that for $t \geq t_0$
$$ 
F_\lambda([\gamma(t)]) -  F_\lambda([x_\lambda]) \leq C  e^{-\delta t},
$$ 
where $C = F_\lambda([\gamma(t_0)]) - F_\lambda([x_\lambda])$.
\end{proposition}

Here is a related result in the blow-down, which is the analogue of \cite[Proposition 13.5.2]{KMbook}.  

\begin{proposition}\label{prop:exponential-decay}
Let $x_\infty$ be a stationary point of $\Xqgc$.  There exists a neighborhood $U$ of $[x_\infty]$ in $B(2R)/S^1$ and a constant $\delta > 0$ such that for $\lambda \gg 0$ and any trajectory $\gamma:[t_1,t_2] \to \vml \cap U$ of $\Xqmlgc$ in $L^2_k([t_1,t_2] \times Y)$, we have the inequalities
$$
-C_2 e^{\delta(t-t_2)} \leq F_\lambda(\gamma(t)) - F_\lambda([x_\lambda]) \leq C_1 e^{-\delta(t-t_1)},
$$   
where 
$$ C_1 = | F_\lambda(\gamma(t_1)) - F_\lambda([x_\lambda])|, \ C_2 = | F_\lambda(\gamma(t_2)) - F_\lambda([x_\lambda])|.
$$
\end{proposition}

For a trajectory $\gamma$ of $\Xqmlgcsigma$, let us introduce the quantities
$$
K_\lambda(\gamma) = \int_{\mathbb{R}} \left| \frac{d\Lambda_{\qml}(\gamma)}{dt} \right| dt, \ \ \ \ K_{\lambda,+}(\gamma) = \int_{\mathbb{R}} \left (\frac{d\Lambda_{\qml}(\gamma)}{dt} \right)^+ dt,
$$
which a priori may be infinite.  Here, $f^+$ denotes $\max\{0,f\}$.  Note that $K_\lambda$ is finite if and only if $K_{\lambda,+}$ is finite.  Further, if the result is finite, and $\gamma$ is a trajectory from $x_\lambda$ to $y_\lambda$, then 
\begin{equation}\label{eq:Lambda-K-relation}
\Lambda_{\qml}(x_\lambda) - \Lambda_{\qml}(y_\lambda) = K_\lambda(\gamma) - 2K_{\lambda,+}(\gamma).
\end{equation}
We will also use $K^I_\lambda$ and $K^I_{\lambda,+}$ to restrict the domain of integration to any subinterval $I \subset \mathbb{R}$.

The exponential decay bound from Proposition~\ref{prop:exponential-decay}, combined with Lemmas~\ref{lemma:1091} and \ref{lem:L2k-bounds-Flambda} give the following analogue of \cite[Corollary 13.5.3]{KMbook}.

\begin{corollary}\label{cor:Klambda-J-bounds}
Fix $x_\infty$ a stationary point of $\Xqgc$.  Given a constant $\eta$, there is a gauge-invariant neighborhood $U \subset W_k([-1,1] \times Y)$ of the constant trajectory $x_\infty$ with the following property for $\lambda \gg 0$.  Let $J \subset \mathbb{R}$ be any interval and $J' = J + [-1,1]$.  If we have a trajectory $\gamma^\tau: J' \to (\vml \cap B(2R))^\sigma$ of $\Xqmlgcsigma$ such that $\tau_t \gamma$ are contained in $U$ for all $t \in J$, then $
K^J_{\lambda,+}(\gamma^\tau) \leq \eta$.
\end{corollary}

\section{Convergence of unparameterized trajectories in the blow-up}
\label{sec:UnparameterizedBlowup}
Our goal in this subsection is to prove analogues of the results in Section~\ref{sec:convergence-downstairs} in the blow-up.  Similarly to the stationary point classes in the singular space $W/S^1$ that appeared Section~\ref{sec:convergence-downstairs}, we will now consider stationary points $[x] \in W^\sigma/S^1$ of the vector fields $\Xqagcsigma$ or $\Xqmlagcsigma$.  We have {\em parameterized trajectories} $$[\gamma] \in W^{\tau}(I \times Y)/S^1$$ of $\Xqagcsigma$ or $\Xqmlagcsigma$, and also, if $I=\R$, {\em unparameterized trajectories} $$[\cgamma] \in W^{\tau}(\R \times Y)/(S^1 \times \mathbb{R}).$$  Unlike in the blow-down, we omit the word ``class'' since the relevant objects arise from actual vector fields on the (smooth) quotients by $S^1$. 

We will only consider trajectories of $\Xqagcsigma$ that limit to two stationary points. 
(Lemma 16.3.3 in \cite{KMbook}, translated into slicewise Coulomb gauge, gives conditions for this to happen---but we will not need it here.) For trajectories of $\Xqmlagcsigma$, Corollary~\ref{cor:endpointsBlowUp} shows the existence of limiting points, provided that the trajectories are contained in $(\vml \cap B(2R))^{\sigma}/S^1$.
  
Finally, analogous to Definitions~\ref{def:broken} and \ref{def:brokenconvergencedownstairs}, we have the notions of unparameterized broken trajectories and convergence to them. In particular, convergence to a broken trajectory is defined in terms of the convergence of some parameterized representatives in $W^{\tau}_{loc}(\R \times Y)/S^1$.
  
The main focus of this subsection is the following.
\begin{proposition}\label{prop:brokenconvergence}
Fix $[x]$ and $[y]$, stationary points of $\Xqagcsigma$ with grading in $[-N, N]$.  Fix $\lambda_n \to \infty$, and a sequence of unparameterized trajectories $[\breve{\gamma}_n]$ of $\Xqmlnagcsigma$, going from $[x_{\lambda_n}]$ to $[y_{\lambda_n}]$, and such that the representatives $\gamma_n$ of $[{\cgamma}_n]$ are contained in $(\vmln \cap B(2R))^\sigma$.  Then, there exists a subsequence of $[{\cgamma}_n]$ that converges to an unparameterized broken trajectory $[{\cgammas}_\infty]$ of $\Xqagcsigma$.
\end{proposition}

\begin{lemma}\label{lem:Kbounds}
Fix $x$ and $y$ stationary points of $\Xqgcsigma$.  There exists $C > 0$ such that for all $\lambda \gg 0$ and trajectories $\gamma^{\tau}$ from $x_\lambda$ to $y_\lambda$, we have 
$K_\lambda(\gamma^{\tau}) \leq C$.  
\end{lemma}

\begin{proof}
This is analogous to \cite[Lemma 16.3.1]{KMbook}.  We outline the proof.  We will show that any sequence $\gamma^{\tau}_n$ of trajectories from $x_{\lambda_n}$ to $y_{\lambda_n}$ has $K_{\lambda_n}(\gamma^{\tau}_n)$ bounded, as long as $\lambda_n \gg 0$.  Note that since $x_{\lambda_n}$ (respectively $y_{\lambda_n}$) are $L^2_k$ close to $x$ (respectively $y$), it suffices to obtain uniform bounds on $K_{\lambda_n,+}(\gamma^\tau_n)$ by \eqref{eq:Lambda-K-relation}, since we have bounds on the $\lambda_n$-spinorial energies of $x_{\lambda_n}$ and $y_{\lambda_n}$. Consider the projections $\gamma_n$ of the trajectories $\gamma^{\tau}_n$ to the blow-down. Following the proof of Proposition~\ref{prop:brokenconvergencedownstairs}, we find a decomposition of $\mathbb{R}$ into a finite number, independent of $n$, of intervals of two types: $J^i_n$, on which $\gamma_n$ is close to a constant trajectory as in Corollary~\ref{cor:Klambda-J-bounds}, and $I^i_n$, which have a fixed length.    Corollary~\ref{cor:Klambda-J-bounds} gives uniform bounds on $K^{J^i_n}_{\lambda_n,+}(\gamma^{\tau}_n)$.   For the intervals $I^i_n$ of fixed length, we can apply Lemma~\ref{lemma:1091} to give uniform upper bounds on $\frac{d}{dt} \Lambda_{\qmln}(\gamma^{\tau}_n(t))$.  Since the lengths of the $I^i_n$ are fixed, the values $K^{I^i_n}_{\lambda_n,+}$ are bounded.  Thus, we obtain the desired bounds on $K_{\lambda_n}(\gamma^{\tau}_n)$.  
\end{proof}

Lemma~\ref{lem:Kbounds} in fact gives uniform bounds on $\Lambda_{\qml}$ for approximate trajectories, which we now establish.  

\begin{lemma}\label{lem:spinorialenergybound}
There exists a constant $C>0$ independent of $\lambda \gg 0$ with the following property. If $\gamma$ is a trajectory of $\Xqmlgcsigma$ contained in $(\vml \cap B(2R))^\sigma$, and going between two stationary points in the grading range $[-N,N]$, then $|\Lambda_{\qml}(\gamma(t))| \leq C$ for all $t$.   
\end{lemma}
\begin{proof}
Since there are only finitely many stationary points of $\Xqmlagcsigma$ in the grading range $[-N,N]$, it suffices to establish the result for trajectories $\gamma$ of $\Xqmlgcsigma$ between $x_\lambda$ and $y_\lambda$ for a fixed pair of stationary points.  We have from \eqref{eq:Lambda-K-relation} that
\begin{align*}
|\Lambda_{\qml}(x_\lambda) - \Lambda_{\qml}(\gamma(t))| &= |K_\lambda^{(-\infty,t]}(\gamma) - 2K_{\lambda,+}^{(-\infty,t]}(\gamma)| \\
&\leq 3 K_\lambda(\gamma) \\
&\leq 3C,
\end{align*}  
where $C$ is the constant from Lemma~\ref{lem:Kbounds}.  Since the $x_\lambda$ are all $L^2_k$ close to the $S^1$-orbit of a fixed stationary point $x$ of $\Xqgc$, we have that $|\Lambda_{\qml}(x_\lambda)|$ is bounded independent of $\lambda \gg 0$.  
\end{proof}

The following is the analogue of \cite[Lemma 16.3.2]{KMbook}, and has a similar proof. It is a consequence of the convergence of parameterized trajectories in the blow-up, Corollary~\ref{cor:convergence}.

\begin{lemma}
\label{lem:nearbyblowup}
For each stationary point $x$ of $\Xqgcsigma$, choose a gauge-invariant neighborhood $U_{x} \subset W^\tau_k(I \times Y)$ of the associated constant trajectory.  Further, let $I'$ be any other interval with non-zero length.  Then, there exists $\epsilon > 0$ such that for $\lambda \gg 0$, if $\gamma \in M(y_\lambda, y'_\lambda)$ for stationary points $y_\lambda, y'_\lambda$ of $\Xqmlgcsigma$ in the grading range $[-N,N]$ satisfies $K^{I'}(\gamma) \leq \epsilon$, then $\gamma|_{I \times Y} \in U_x$ for some $x$.       
\end{lemma}

\begin{proof}[Proof of Proposition~\ref{prop:brokenconvergence}]
Lemma~\ref{lem:spinorialenergybound} guarantees that $\lambdanenergy$ is bounded for all trajectories between $[x_{\lambda_n}]$ and $[y_{\lambda_n}]$ contained in $(\vml \cap B(2R))^\sigma$. Therefore, we may repeat the argument of Proposition~\ref{prop:brokenconvergencedownstairs}, where we now apply Corollary~\ref{cor:convergence} and Lemma~\ref{lem:nearbyblowup} in place of Proposition~\ref{prop:convergencenoblowup} and we require that where we see bounds involving terms of the form $\E^I(\gamma)$, we also have analogous bounds on $K^I(\gamma)$.
\end{proof}

\begin{corollary}\label{cor:index1convergence}
Fix $\epsilon > 0$.  For $\lambda \gg 0$, the following is true.  Suppose that $[\gamma_\lambda]$ is a trajectory of $\Xqmlagcsigma$ from $[x_\lambda]$ to $[y_\lambda]$, such that either
\begin{enumerate}
\item $[\gamma_\lambda]$ is not boundary-obstructed and $\gr([x_\lambda],[y_\lambda]) = 1$ or 
\item $[\gamma_\lambda]$ is boundary-obstructed and $\gr([x_\lambda],[y_\lambda]) =0$.  
\end{enumerate} 
Furthermore, suppose that the gradings of $[x_\lambda]$ and $[y_\lambda]$ are in $[-N,N]$.  Then, $[\gamma_\lambda]$ is $\epsilon$-close in $W_{k,loc}(\R \times Y)/S^1$ to $[\gamma]$, a trajectory of $\Xqagcsigma$ with grading in $[-N,N]$.   
\end{corollary}
\begin{proof}
We argue by contradiction.  Suppose there exists a sequence of trajectories $[\gamma_{\lambda_n}]$, where $\lambda_n \to \infty$, which are always at least distance $\epsilon$ from trajectories of $\Xqagcsigma$.  Without loss of generality, none of the trajectories are boundary-obstructed.  A similar argument for the index 0 and boundary-obstructed case follows similarly.  By Proposition~\ref{prop:brokenconvergence}, there exists a subsequence such that the unparameterized trajectories converge to a broken trajectory $[\cgammas_\infty]$ from $[x_\infty]$ to $[y_\infty]$, which is necessarily index 1 by Proposition~\ref{prop:stationarycorrespondence}.  It remains to show that this trajectory is in fact unbroken.  

We suppose that $[\cgammas_\infty] = ([\cgamma_{\infty,1}],\ldots,[\cgamma_{\infty,m}])$ with $m \geq 2$.  By our non-degeneracy assumption, one $[\cgamma_{\infty,i}]$ must be be index 1, while the other trajectories must be index 0.  Note that an index 0 trajectory must be boundary-obstructed (or else the moduli space is necessarily empty), and thus there can be at most one such trajectory.  In particular, we see that $[\cgammas_\infty] = ([\cgamma_{\infty,1}], [\cgamma_{\infty,2}])$.  Since we are assuming that the trajectories $[\gamma_{\lambda_n}]$ are not boundary-obstructed, it follows that one of $[\cgamma_{\infty,1}]$ or $[\cgamma_{\infty,2}]$ goes from boundary-unstable to boundary-stable or one of $[x_\infty]$ or $[y_\infty]$ is irreducible.  In either case, we have that that one of the components of $[\cgamma_{\infty}]$ is an irreducible trajectory (see \cite[Proposition 14.5.7]{KMbook} or the discussion at the end of Section~\ref{sec:AdmPer}).  This implies that $M([x_\infty], [y_\infty])$ must contain an irreducible trajectory.  This contradicts \cite[Corollary 16.5.4]{KMbook}, which implies that if $\breve{M}([\x_\infty], [\y_\infty])$ contains a once-broken trajectory in the compactification and contains an irreducible trajectory, no component of the broken trajectory can be boundary-obstructed.    
\end{proof}

\section{Convergence in $L^2_k$}
In Chapter~\ref{sec:criticalpoints}, we used the inverse function theorem to find approximate stationary points near stationary points of $\Xqgcsigma$.  There we defined identifications $\Xi_{\lambda}$ between $\Crit^\lambda_{\N}$ and $\Crit_{\N}$, which we extended to self-diffeomorphisms of $W^\sigma_j$ in Lemma~\ref{lem:Xixylambda} and then to self-diffeomorphisms of path spaces in Proposition~\ref{prop:Xixy-paths}.  We now use these self-diffeomorphisms to improve the $L^2_{k,loc}$ convergence from Proposition~\ref{prop:brokenconvergence} to $L^2_k$ convergence.  

\begin{proposition}
\label{prop:L2k}
Let $[\gamma_n]: \mathbb{R} \to (\vmln \cap B(2R))^\sigma/S^1$ be a sequence of trajectories of $\Xqmlnagcsigma$ between stationary points $[x_{\lambda_n}]$ and $[y_{\lambda_n}]$.  Further, suppose that the unparameterized trajectory $[\cgamma_n]$ converges to $[\cgamma_{\infty}]$, an unbroken trajectory of $\Xqagcsigma$ from $[x_\infty]$ to $[y_\infty]$.  Then, after possible reparameterization, $\Xi_{\lambda_n}([\gamma_n])$ converges to $[\gamma_{\infty}]$ in $\B_k^{\gCoul,\tau}([x_\infty],[y_\infty])$, where $[\gamma_{\infty}]$ is a representative of $[\cgamma_{\infty}]$.     
\end{proposition}
\begin{proof}
We choose representative trajectories $\gamma_n$ of $\Xqmlngcsigma$ and $\gamma_\infty$ of $\Xqgcsigma$ of the trajectory classes such that $\gamma_n$ converges to $\gamma_\infty$.  (This is easy to arrange, since these are only $S^1$-equivalence classes.)  This means there exists a sequence $s_n \in \mathbb{R}$ such that $\tau_{s_n} \gamma_n \to \gamma_\infty$ in $L^2_{k,loc}$.  Since we may reparameterize $\Xi_{\lambda_n}([\gamma_n])$, we assume that no translations are necessary.  Therefore, we are interested in showing that there exists a sequence of gauge transformations $u_n: \R \to S^1$ such that 
$$
\| u_n \cdot \Xi_{\lambda_n}(\gamma_n) -\gamma_\infty \|_{L^2_k(\mathbb{R} \times Y)} \to 0.
$$
Fix $\epsilon > 0$.  We will show that $\| \Xi_{\lambda_n}(\gamma_n) -\gamma_\infty \|_{L^2_k(\mathbb{R} \times Y)} < \epsilon$ for $n \gg 0$.  Write $x_\infty$ and $y_\infty$ for the limit points of $\gamma_\infty$.  

First, since $\gamma_n \to \gamma_\infty$ in $L^2_{k,loc}(\mathbb{R} \times Y)$, we also have that $\Xi_{\lambda_n}(\gamma_n) \to \gamma_\infty$ in $L^2_{k,loc}(\mathbb{R} \times Y)$ by Corollary~\ref{cor:Xixy-paths}.  In particular, for any $T > 0$, we have for $n \gg 0$, 
$$
\| \Xi_{\lambda_n}(\gamma_n) - \gamma_\infty \|_{L^2_k([-T,T] \times Y)} < \epsilon/2. 
$$  
Therefore, it suffices to show that we can find $T > 0$ such that for $n \gg 0$, 
\begin{equation}\label{eq:L2k-afterxilambda-convergence}
\| \Xi_{\lambda_n}(\gamma_n) - \gamma_\infty \|_{L^2_k([T,\infty) \times Y)} < \epsilon/4, 
\end{equation}
and likewise on $(-\infty,-T] \times Y$.  We will establish the bounds for $[T,\infty) \times Y$; the case of $(-\infty, -T] \times Y$ follows in the same way.   

We first claim that there exists $T > 0$ such that for $n \gg 0$,  
\begin{equation}\label{eq:L2k-endpoints-afterxilambda}
\| \Xi_{\lambda_n}(\gamma_n) - y_\infty \|_{L^2_k([T, \infty) \times Y)} < \epsilon/8.
\end{equation}
By Corollary~\ref{cor:path-distance}, we have that it suffices to establish 
\begin{equation}\label{eq:L2k-endpoints-beforexilambda}
\| \gamma_n - y_{\lambda_n} \|_{L^2_{k}([T,\infty) \times Y)} \to 0, \ n \to \infty.
\end{equation}

This is an analogue of \cite[Theorem 13.3.5]{KMbook}, which implies that there exists a gauge transformation of $\gamma_\infty$ which lands in $\B^{\gCoul,\tau}_k([x_\infty],[y_\infty])$.  However, we will explicitly keep track of the $L^2_{k}$ bounds and the necessary gauge transformation.    First, by Proposition~\ref{prop:13.4.1-analogue}, there exists a constant gauge transformation $u^i_n \in S^1$ and a constant $C$, independent of $i$ and $\lambda$, such that if $y_\infty$ is reducible, then 
\begin{equation}\label{eq:trajectory-constant-bounds}
\| u^i_n \cdot \gamma_n - y_{\lambda_n} \|_{L^2_{k}([i-2,i+2] \times Y)} \leq C \left ( (\Lambda_{\qmln}(i - 3) - \Lambda_{\qmln}(i + 3)) + (F_{\lambda_n}(i-3) - F_{\lambda_n}(i+3))^{\frac{1}{2}}\right),   
\end{equation}
and if $y_\infty$ is irreducible, then  
\begin{equation}\label{eq:trajectory-constant-bounds-irred}
\| u^i_n \cdot \gamma_n - y_{\lambda_n} \|_{L^2_{k}([i-2,i+2] \times Y)} \leq C (F_{\lambda_n}(i-3) - F_{\lambda_n}(i+3)).   
\end{equation}
In Proposition~\ref{prop:exponential-decay} we have established the exponential decay of $F_{\lambda_n}(\gamma_n(t))$ independent of $n$.  Further, by Proposition~\ref{prop:13.4.8-analogue} and Proposition~\ref{prop:exponential-decay}, we obtain uniform bounds on $\Lambda_{\qmln}(i - 3) - \Lambda_{\qmln}(i+3)$ which decay exponentially as $i \to \infty$.   We thus see that for any $T > 0$, 
$$
\sum_{i \geq T} \| u^i_n \cdot \gamma_n - y_{\lambda_n} \|_{L^2_{k}([i-2, i+2] \times Y)} \leq C(T)
$$ 
where the constant $C(T) > 0$ depends only on $T$ (i.e., independent of $n$) and converges to $0$ as $T \to \infty$.    
To get the desired $L^2_{k}([T,\infty) \times Y)$ bounds on $\gamma_n - y_{\lambda_n}$, it suffices to prove that 
$$ 
\sum_{i \geq T} \| u^i_n \cdot  y_{\lambda_n} - y_{\lambda_n} \|_{L^2_{k}([i-1/2,i+1/2] \times Y)} \leq C'(T),
$$   
for some constant $C'(T)$ that converges to 0 as $T \to \infty$.  Without loss of generality, assume $y_\infty$ is reducible (the irreducible case is similar).  By \eqref{eq:trajectory-constant-bounds}, we have 
\begin{align*}
\| (u^i_n)^{-1} \cdot y_{\lambda_n} - & 
(u^{i+1}_n)^{-1} \cdot y_{\lambda_n} \|_{L^2_{k}([i-1/2,i+1/2] \times Y)}  \\
& \leq \|  (u^i_n)^{-1} \cdot y_{\lambda_n} - \gamma_n \|_{L^2_{k}([i-1/2,i+1/2] \times Y)}  + \|(u^{i+1}_n)^{-1} \cdot y_{\lambda_n}  - \gamma_n\|_{L^2_{k}([i-1/2,i+1/2] \times Y)} \\
& = \|  u^i_n \cdot \gamma_n - y_{\lambda_n} \|_{L^2_{k}([i-1/2,i+1/2] \times Y)}  + \| u^{i+1}_n \cdot \gamma_n - y_{\lambda_n}\|_{L^2_{k}([i-1/2,i+1/2] \times Y)} \\
&\leq \|  u^i_n \cdot \gamma_n - y_{\lambda_n} \|_{L^2_{k}([i-2,i+2] \times Y)}  + \|u^{i+1}_n \cdot \gamma_n - y_{\lambda_n} \|_{L^2_{k}([i-1,i+3] \times Y)} \\
& \leq C \Big( (\Lambda_{\qmln}(i - 3) - \Lambda_{\qmln}(i +3))  + ((\Lambda_{\qmln}(i -2) - \Lambda_{\qmln}(i + 4)) \\
& + (F_{\lambda_n}(i-3) - F_{\lambda_n}(i+3))^{\frac{1}{2}} + (F_{\lambda_n}(i-2) - F_{\lambda_n}(i+4))^{\frac{1}{2}} \Big).   
\end{align*}
The bounds from Proposition~\ref{prop:13.4.8-analogue} and exponential decay of Proposition~\ref{prop:exponential-decay} show that 
$$\sum_{i \geq T} \| (u^i_n)^{-1} \cdot  y_{\lambda_n} - (u^{i+1}_n)^{-1} \cdot y_{\lambda_n} \|_{L^2_{k}([i-1/2,i+1/2] \times Y)} \leq C'(T),$$ with $C'(T) \to 0$ as $T \to \infty$.  We would like an analogous inequality without the inverses.  Note that if $u, u'$ are in $S^1$, and $y= (a,s,\phi)$, we have that 
$$u\cdot  y - u'\cdot  y = (0, 0, (u - u')\phi).$$
Since $| u - u'| = |u^{-1} - u'^{-1}|$, we see that 
$$
\| (u^i_n)^{-1} \cdot y_{\lambda_n} - (u^{i+1}_n)^{-1} \cdot y_{\lambda_n}\|_{L^2_{k}([i-1/2,i+1/2] \times Y)} = \| u^i_n \cdot y_{\lambda_n} - u^{i+1}_n \cdot y_{\lambda_n} \|_{L^2_{k}([i-1/2,i+1/2] \times Y)}. 
$$

Therefore, we have 
$$\sum_{i \geq T} \| u^i_n \cdot y_{\lambda_n} - u^{i+1}_n \cdot y_{\lambda_n} \|_{L^2_{k}([i-1/2,i+1/2] \times Y)} \leq C'(T).$$ 
In particular, this implies that as $i \to \infty$, $u^i_n$ converges to some $u_n \in S^1$ (uniformly in $n$), and we see that    
$$ 
\sum_{i \geq T} \| u^i_n \cdot y_{\lambda_n} - u_n \cdot y_{\lambda_n} \|_{L^2_{k}([i-1/2,i+1/2] \times Y)} \leq C'(T).
$$   
Therefore, it remains to establish that $u_n = 1$ for each $n$.  This follows since by the construction of $u^i_n$ in Lemma~\ref{lem:L21-bounds-Flambda}, we have that $\Re \langle u^i_n \cdot \gamma_n(i-2), (0,i\phi_n) \rangle_{L^2(Y)} = 0$ and $\lim_{t \to +\infty} \gamma_n(t) = y_{\lambda_n} = (a_{n},\phi_{n})$.  We have now established \eqref{eq:L2k-endpoints-beforexilambda} and thus \eqref{eq:L2k-endpoints-afterxilambda} as claimed.  

As mentioned above, \cite[Theorem 13.3.5]{KMbook} gives $u : \R \to S^1$ such that $u \cdot \gamma_\infty  - y_\infty$ is in $L^2_k([T,\infty) \times Y)$.  In the above argument, we obtained an analogous result where we did not need the four-dimensional gauge transformation.  By using the analogues of the results of this section for trajectories of $\Xqgcsigma$ in place of $\Xqmlgcsigma$, we can repeat the above arguments to obtain the stronger statement for $\gamma_\infty$ as well, i.e., that $\| \gamma_\infty - y_\infty \|_{L^2_k([T,\infty) \times Y)} < \infty$.  (As explained at the beginning of the proof Lemma~\ref{lem:L21-bounds-Flambda}, we work in a different gauge slice instead of the Coulomb-Neumann slice, which allows the arguments of this section to pin down the gauge transformation precisely.)  Therefore, for $T \gg 0$, we have that
$$
\| \gamma_\infty - y_\infty \|_{L^2_k([T,\infty) \times Y)} < \epsilon/8.
$$
Combining this inequality with \eqref{eq:L2k-endpoints-beforexilambda} we obtain \eqref{eq:L2k-afterxilambda-convergence}, which completes the proof.  
\end{proof}

We can use Proposition~\ref{prop:brokenconvergence} to make more refined statements in the case of an approximate trajectory with small index.  
\begin{lemma}
Fix $\epsilon, N > 0$.  For $\lambda \gg 0$, the following is true.  Let $[\gamma_\lambda]$ be a trajectory from $[x_\lambda]$ to $[y_\lambda]$ such that either
\begin{enumerate}
\item $[\gamma_\lambda]$ is not boundary-obstructed and $\gr([x_\lambda], [y_\lambda]) = 1$ or 
\item $[\gamma_\lambda]$ is boundary-obstructed and $\gr([x_\lambda], [y_\lambda]) =0$.  
\end{enumerate} 
Furthermore, suppose that the gradings of $[x_\lambda]$ and $[y_\lambda]$ are in $[-N,N]$.  Then $\Xi_{\lambda}([\gamma_\lambda])$ is $\epsilon$-close in $L^2_{k}(\R \times Y)$ of $[\gamma]$, a trajectory of $\Xqagcsigma$ from $[x_\infty]$ to $[y_\infty]$ with grading in $[-N,N]$.    
\end{lemma}
\begin{proof}
The argument is the same as for Corollary~\ref{cor:index1convergence}, except after applying Proposition~\ref{prop:brokenconvergence} to extract a convergent subsequence in $L^2_{k,loc}$, we improve this to $L^2_k$ convergence (after composition with $\Xi_{\lambda}$) using Proposition~\ref{prop:L2k}.
\end{proof}

\chapter{Characterization of approximate trajectories}\label{sec:trajectories2}
In this chapter, using the inverse function theorem, we will identify the moduli spaces of isolated approximate trajectories (in a fixed grading range) with those of actual trajectories. We will also produce  similar identifications for the cut moduli spaces used to define the $U$-actions in Morse and Floer homology. Furthermore, we will show that all these identifications can be chosen to preserve orientations. 
 
\section{Stability}
 Let $[x_{\infty}]$ and $[y_{\infty}]$ be two stationary points of $\Xqagcsigma$, with grading in $[-N, N]$; in other words, we have $[x_{\infty}], [y_{\infty}] \in \Crit_{\N}$. We assume that their grading difference is one and that we are not in the boundary-obstructed case. (Recall that boundary-obstructed means that $[x_{\infty}]$ is boundary-stable and $[y_{\infty}]$ is boundary-unstable.) Since $\q$ is an admissible perturbation, the moduli space
 $$ \Mbreve^{\agCoul}([x_{\infty}], [y_{\infty}])= M^{\agCoul}([x_{\infty}], [y_{\infty}]) / \R$$
 is a finite set of points. We can view $ \Mbreve^{\agCoul}([x_{\infty}], [y_{\infty}])$ as the zero set of the section induced by $\Fqgctau$ on the bundle $\V^{\gCoul, \tau}(\R \times Y)$ over $\tB^{\gCoul, \tau}([x_{\infty}], [y_{\infty}])/\R$, restricted to $\B^{\gCoul,\tau}([x_\infty],[y_\infty])/\R$.
 
Consider the nearby stationary points of $\Xqmlagcsigma$
 $$ [x_{\lambda}] = \Xi_{\lambda}^{-1}([x_{\infty}]), \ \ [y_{\lambda}] = \Xi_{\lambda}^{-1}([y_{\infty}]).$$
 
 We also have a moduli space of approximate Seiberg-Witten trajectories in $(B(2R) \cap \vml)^{\sigma}$:
 $$ \Mbreve^{\agCoul}([x_{\lambda}], [y_{\lambda}])=M^{\agCoul}([x_{\lambda}], [y_{\lambda}])/\R.$$

 By Corollary~\ref{cor:GrPres}, the relative grading between $[x_{\lambda}]$ and $[y_{\lambda}]$ is also one. Further, by Proposition~\ref{prop:AllMS} and \ref{prop:MorseSmales}, for $\lambda = \llambda_i$ with $i \gg 0$, $ \Mbreve^{\agCoul}([x_{\lambda}], [y_{\lambda}])$ is regular and consists of finitely many points. 
Our main goal in this subsection is to prove:

\begin{proposition}
\label{prop:EquivalenceTrajectories}
For $\lambda = \llambda_i \gg 0$, there is a one-to-one correspondence between the moduli spaces $ \Mbreve^{\agCoul}([x_{\infty}], [y_{\infty}])$ and $ \Mbreve^{\agCoul}([x_{\lambda}], [y_{\lambda}])$.
 \end{proposition}

The proof has two parts. First, we show that in a fixed neighborhood of any $[\gamma] \in  \Mbreve^{\agCoul}([x_{\infty}], [y_{\infty}])$, there is a unique approximate trajectory. Second, we show that no other approximate trajectories between $[x_{\lambda}]$ and $[y_{\lambda}]$ exist.
 
We start with the first step. This is a stability result, similar to what we established in Proposition~\ref{prop:nearby} for stationary points.
 
\begin{proposition}
\label{prop:stabilitynearby}
Suppose $[x_{\infty}], [y_{\infty}] \in \Crit_{\N}$ have $\gr([x_{\infty}], [y_{\infty}])=1$, and we are not in the boundary-obstructed case. Fix a trajectory $[\gamma_\infty] \in  \Mbreve^{\agCoul}([x_{\infty}], [y_{\infty}])$, and a small neighborhood $U$ of $[\gamma_\infty]$ in $\B^{\gCoul, \tau}_k([x_{\infty}], [y_{\infty}])/\R$. 

Then, for $\lambda \gg 0$, there exists a unique trajectory $[\gamma_\lambda] \in \Mbreve^{\agCoul}([x_{\lambda}], [y_{\lambda}])$ such that $\Xi_{\lambda}([\gamma_\lambda]) \in U$.  
\end{proposition}

\begin{proof}
Recall that, because we do the blow-up construction in each time slice, the space $\B^{\gCoul, \tau}_k([x_{\infty}], [y_{\infty}])$ is not even a Banach manifold with boundary. However, it is a subspace of the Banach manifold $\widetilde{\B}_k^{\gCoul, \tau}([x_{\infty}], [y_{\infty}])$, which consists of (equivalence classes of) paths $(a(t)+\alpha(t)dt, s(t), \phi(t))$, with $s(t)$ allowed to vary in all of $\R$.

By Proposition~\ref{prop:Xixy-paths}\eqref{xi:Vgc}, the map $\Xi_{\lambda}$ gives rise to a bundle map $(\Xi_{\lambda})_*$ that is part of a commutative diagram
$$\xymatrix{
\V_k^{\gCoul, \tau}(\R \times Y) \ar[d] \ar[r]^{(\Xi_{\lambda})_*} & \V_k^{\gCoul, \tau}(\R \times Y)\ar[d] \\
\widetilde{\B}^{\gCoul, \tau}([x_{\lambda}], [y_{\lambda}])/\R \ar[r]^{\Xi_{\lambda}} & \widetilde{\B}^{\gCoul, \tau}_k([x_{\infty}], [y_{\infty}])/\R.
}$$

We define a map
$$ S: \widetilde{\B}^{\gCoul, \tau}_k([x_{\infty}], [y_{\infty}]) \times (-1,1) \to \V^{\gCoul, \tau}_k(\R \times Y) \times \R$$
by the formula 
$$ S([\gamma], r) = \Bigl( (\Xi_{\lambda})_* \circ \Fqmlgctau \circ  \Xi_{\lambda}^{-1} ([\gamma]), \ r \Bigr),$$
where $\Fqmlgctau$ is as in \eqref{eq:Fqmlgctau} and we wrote $\lambda = f^{-1}(|r|)$, with $f: (0,\infty] \to [0,1)$ being the homeomorphism from Section~\ref{sec:StabilityPoints}. 

By Proposition~\ref{prop:Xixy-paths}, we get that the section $S$ is differentiable. Observe also that 
$$S([\gamma], 0) = (\Fqgctau([\gamma]), 0),$$ because $\Xi_{\infty}$ is the identity.
Therefore, the derivative of $S$ at $([\gamma_{\infty}], 0)$ can be written in block form as
$$ 
\begin{pmatrix}
\D^\tau_{[\gamma_{\infty}]} \Fqgctau & * \\
0 & I 
\end{pmatrix}.
$$

By our assumptions on $\q$ and Proposition~\ref{prop:Qgammasurjective}, we get that $\D^\tau_{[\gamma_{\infty}]} \Fqgctau$ is surjective. It is Fredholm index one when viewed as a map with domain $\K^{\tau}_{k, \gamma_{\infty}}$, the tangent space to $\tB^{\gCoul, \tau}_k([x_{\infty}], [y_{\infty}])$ at $[\gamma_\infty]$, and range $\K^{\gCoul,\tau}_{k-1, \gamma_\infty}$. In our setting we divided by $\R$, so it becomes a map of index zero. Since it is surjective, it must be invertible. Given the block form above, it follows that $\D^\tau_{([\gamma_{\infty}], 0)} S$ is also invertible.

We now apply the inverse function theorem. For $r > 0$ small (that is, for $\lambda \gg 0$), we obtain  a unique solution $[\gamma] \in U$ to the equation $S([\gamma], r) = (0, r)$. We then let 
$$ [\gamma_{\lambda}] = \Xi_{\lambda}^{-1} ([\gamma]).$$
Because $\Xi_\lambda$ is a diffeomorphism, we see that $\F^{\gCoul,\tau}_{\qml}([\gamma_\lambda]) = 0$.  Since $[\gamma_{\lambda}]$ has a temporal gauge representative in $\C^{\gCoul,\tau}_k(x_\lambda, y_\lambda)$ for some representatives $x_\lambda, y_\lambda$ (see Remark~\ref{rmk:no-temporal}), we see that $[\gamma_\lambda]$ gives the desired trajectory from $[x_\lambda]$ to $[y_\lambda]$.

Let $\gamma_{\lambda}$ be a temporal gauge representative of $[\gamma_{\lambda}]$. We need to check that it actually takes values in  $(B(2R) \cap \vml)^{\sigma}$.  We can repeat the argument above with the Sobolev coefficient $k+1$ instead of $k$.  The resulting $\gamma_\lambda$ in this case must be the same due to the uniqueness guaranteed by the implicit function theorem.  Since we use the Sobolev coefficient $k+1$, we have that $\Xi_\lambda([\gamma_\lambda])$ must converge to $[\gamma]$ in an $L^2_{k+1}(\rr \times Y)$ neighborhood.  Since any temporal gauge representative of $\gamma$ is contained in $B(R)^\sigma$, we see that $\Xi_\lambda(\gamma_\lambda)$ is contained in $B(3R/2)^\sigma$ for $\lambda \gg 0$.  It is easy to see from the construction of $\Xi_\lambda$ in Chapter~\ref{sec:appendix} that $\gamma_\lambda$ must be contained in $B(2R)^\sigma$.  Further, we see that the values of $\Lambda_{\qml}(\gamma_\lambda)$ are uniformly close to those of $\Lambda_\q(\gamma)$, and thus uniformly bounded.  By Lemma~\ref{lem:blowuptrajectoriesinvml}, we see that $\gamma_\lambda$ is contained in $(\vml)^\sigma$ for $\lambda \gg 0$, which completes the proof.    
\end{proof}

\begin{remark}
If the trajectory $[\gamma_{\infty}]$ from Proposition~\ref{prop:stabilitynearby} is contained in the reducible locus, then so is the nearby approximate trajectory $[\gamma_\lambda]$. This follows from the invariance of the constructions in the proof under the obvious involutions on $\tB^{\gCoul, \tau}([x_{\lambda}], [y_{\lambda}])$ and $\V_k^{\gCoul, \tau}(\R \times Y)$.
\end{remark}

We are now ready to complete the proof of Proposition~\ref{prop:EquivalenceTrajectories} with the second step, which is to show that there are no other approximate trajectories between $[x_\lambda]$ and $[y_\lambda]$.

\begin{proof}[Proof of Proposition~\ref{prop:EquivalenceTrajectories}] 
Let
$$  \Mbreve^{\agCoul}([x_{\infty}], [y_{\infty}]) = \{[\gamma_{\infty}^1], \dots, [\gamma_{\infty}^m]\}.$$
For each $\ell=1, \dots, m$, Proposition~\ref{prop:stabilitynearby} guarantees the existence of a neighborhood $U^\ell$ of $[\gamma_{\infty}^\ell]$ and a unique approximate trajectory $[\gamma_{\lambda}^\ell]$ with $\Xi_{\lambda}([\gamma^\ell_\lambda]) \in U^\ell$.  

We need to check that there are no other approximate trajectories in $\Mbreve^{\agCoul}([x_{\lambda}], [y_{\lambda}])$. Assume we had sequences $\lambda_n \to \infty$ and $[\gamma_n]\in \Mbreve^{\agCoul}([x_{\lambda_n}], [y_{\lambda_n}])$ such that $[\gamma_n]$ is not of the form 
$[\gamma_{\lambda_n}^\ell]$ for any $n$ and $\ell$.  By assumption, each $\lambda_n$ is some $\llambda_i$. By applying Proposition~\ref{prop:L2k} we can find a subsequence of the $[\gamma_n]$ such that $\Xi_{\lambda_n}([\gamma_n])$ converge to some $[\gamma_{\infty}^\ell]$ in $L^2_k$.  (We require that the $\lambda_n$ be of the form $\llambda_i$ since we are invoking results of Chapter~\ref{sec:trajectories1}, which used this assumption throughout.)  This means that the $\Xi_{\lambda_n}([\gamma_n])$ live inside $U^\ell$ for large $n$, and we obtain a contradiction with the uniqueness statement in Proposition~\ref{prop:stabilitynearby}.
\end{proof}

We have an analogous result for the boundary-obstructed case.  

\begin{proposition}
\label{prop:EquivalenceTrajectoriesObstructed}
Let $[x_{\infty}], [y_{\infty}] \in \Crit_{\N}$ be in the boundary-obstructed case, and suppose $\gr([x_{\infty}], [y_{\infty}])=0$. For $\lambda =  \llambda_i \gg 0$, there is a one-to-one correspondence between $ \Mbreve^{\agCoul, \red}([x_{\infty}], [y_{\infty}])$ and $ \Mbreve^{\agCoul, \red}([x_{\lambda}], [y_{\lambda}])$.
 \end{proposition}
\begin{proof}
The proof is similar to that of Proposition~\ref{prop:EquivalenceTrajectories}, but with an application of the inverse function theorem in the subspace of $\widetilde{B}_k^{\gCoul, \tau}([x_{\infty}], [y_{\infty}])/\R$ consisting of reducible paths.  In this subspace, the corresponding linear operator will again be index zero, and in fact invertible.    
\end{proof}

\section{The $U$-maps}\label{subsec:U-maps}
In Section~\ref{sec:mfh} we gave the definition of the $U$-maps in monopole Floer homology, in terms of cut-down moduli spaces $M([x], [y]) \cap \Zs$ and $M^{\red}([x], [y]) \cap \Zs$; cf. Equations \eqref{eqn:mU} and \eqref{eqn:mUred}. Moreover, in Section~\ref{sec:UCoulomb} we identified these cut-down moduli spaces with similar ones in Coulomb gauge:
\begin{equation}
\label{eq:cutdownmoduli}
  M^{\agCoul}([x], [y]) \cap \Zs^{\agCoul} \ \text{and} \ M^{\agCoul, \red}([x], [y]) \cap \Zs^{\agCoul}.
  \end{equation}
Here, $\Zs^{\agCoul}$ is the zero set of a section $\zeta^{\agCoul}$ of the natural complex line bundle $E^{\agCoul, \sigma}$ over $W_{k}^{\sigma}/S^1 \subset W_{k-1/2}^\sigma/S^1$. The section is chosen so that the intersections in \eqref{eq:cutdownmoduli} are transverse.

We would like to further identify the moduli spaces from \eqref{eq:cutdownmoduli} with those consisting of approximate trajectories, that appear in the construction of equivariant Morse homology in finite dimensions; see Section~\ref{subsec:UMorse}. To obtain the cut-down moduli spaces in the approximations $(\vml)^{\sigma}/S^1$, we simply consider the restriction of $\zeta^{\agCoul}$ to those spaces.

Recall from Chapter~\ref{sec:MorseSmale} that we chose the perturbation $\q$ such that, for all $\lambda\in \{\llambda_1, \llambda_2, \dots \}$ sufficiently large, all the moduli spaces of flow lines for $\Xqmlgcsigma$ on $(B(2R) \cap \vml)^{\sigma}/S^1$ are regular. Since there is a countable number of such moduli spaces, we can choose the section $\zeta^{\agCoul}$ such that it intersects all these moduli spaces (including the non-approximate ones) transversely. This implies that we can define the $U$-maps on the corresponding Morse homology groups as in Section~\ref{subsec:UMorse}.

We now focus on the cut-down moduli spaces of the form
$$M^{\agCoul}([x_{\lambda}], [y_{\lambda}]) \cap \Zs^{\agCoul} \ \text{and} \ M^{\agCoul, \red}([x_{\lambda}], [y_{\lambda}]) \cap \Zs^{\agCoul},$$
for $[x_{\infty}], [y_{\infty}] \in \Crit_{\N}$. These moduli spaces suffice to determine the $U$-action in gradings from $-N$ to $N$. 

\begin{proposition}\label{prop:Uagree}
$(a)$ Suppose $[x_{\infty}], [y_{\infty}] \in \Crit_{\N}$ have $\gr([x_{\infty}], [y_{\infty}])=2$ and are not boundary obstructed. For $\lambda = \llambda_i \gg 0$, there is a one-to-one correspondence between $ M^{\agCoul}([x_{\infty}], [y_{\infty}]) \cap \Zs^{\agCoul}$ and $ M^{\agCoul}([x_{\lambda}], [y_{\lambda}]) \cap \Zs^{\agCoul}$.

$(b)$ Let $[x_{\infty}], [y_{\infty}] \in \Crit_{\N}$ be reducible stationary points with $\gr([x_{\infty}], [y_{\infty}])=1$ and boundary obstructed. For $\lambda = \llambda_i \gg 0$, there is a one-to-one correspondence between $ M^{\agCoul, \red}([x_{\infty}], [y_{\infty}]) \cap \Zs^{\agCoul}$ and $M^{\agCoul, \red}([x_{\lambda}], [y_{\lambda}])\cap \Zs^{\agCoul}$.
\end{proposition}

\begin{proof}
Part (a) is the analogue of Proposition~\ref{prop:EquivalenceTrajectories}, and its proof is similar. Since the grading difference between $[x_{\infty}]$ and $[y_{\infty}]$ (and hence also between $[x_{\lambda}]$ and $[y_{\lambda}]$) is two, the cut-down moduli spaces in question are zero-dimensional. Consider a trajectory $$[\gamma] \in M^{\agCoul}([x_{\infty}], [y_{\infty}]) \cap \Zs^{\agCoul}.$$ 
Recall that the intersection is taken at time $t=0$, that is, we have $[\gamma(0)] \in \Zs^{\agCoul}$. The fact that the moduli space $M^{\agCoul}([x_{\infty}], [y_{\infty}]) \cap \Zs^{\agCoul}$ is cut out transversely at $[\gamma]$ can be rephrased in terms of the invertibility of the linear operator
$$ \D^\tau_{[\gamma]}\Fqgctau \oplus \D_{[\gamma(0)]} \zeta^{\agCoul}: \T_{k,[\gamma]}^{\gCoul, \tau}([x_{\infty}], [y_{\infty}]) \to \V^{\gCoul, \tau}_{k-1,[\gamma]}(Z) \oplus E^{\agCoul, \sigma}_{[\gamma(0)]}.$$

An application of the inverse function theorem shows that in a neighborhood of $[\gamma]$ there is a unique element of $M^{\agCoul}([x_{\lambda}], [y_{\lambda}]) \cap \Zs^{\agCoul}$ (after applying the corresponding diffeomorphism $\Xi_{\lambda}$). This produces at least as many elements of $M^{\agCoul}([x_{\lambda}], [y_{\lambda}]) \cap \Zs^{\agCoul}$ as there are in $M^{\agCoul}([x_{\infty}], [y_{\infty}]) \cap \Zs^{\agCoul}.$ To see that there are no more, notice that, given a sequence $[\gamma_n] \in  M^{\agCoul}([x_{\lambda_n}], [y_{\lambda_n}]) \cap \Zs^{\agCoul}$ with $\lambda_n \to \infty$, the sequence $\Xi_{\lambda_n}([\gamma_n])$ admits a convergent subsequence, and the limit must be an element of $M^{\agCoul}([x_{\infty}], [y_{\infty}]) \cap \Zs^{\agCoul}.$

Part (b) is the analogue of Proposition~\ref{prop:EquivalenceTrajectoriesObstructed}, and again the  proof is similar.
\end{proof}

\section{Orientations}
Recall that in Section~\ref{sec:OrientCoulomb} we oriented the moduli spaces of trajectories of $\Xqgcsigma$, using an orientation data system $o^{\gCoul}$ in Coulomb gauge. This was based on trivializing the determinant lines $\det(P_{\gamma}^{\gCoul} )$ for the operators $P_{\gamma}^{\gCoul} $ from \eqref{eq:Pgammagc}. The same procedure, but using operators $P_{\gamma}^{\gCoul} $ defined with the perturbation $\q^{\lambda}$ instead of $\q$, can be used to orient the moduli spaces of trajectories of $\Xqmlgcsigma$ (in infinite dimensions).  
 
We would like to relate the orientations of the moduli spaces of trajectories of $\Xqmlgcsigma$ to the orientations in (finite dimensional) Morse homology and to the orientations of the moduli spaces of trajectories of $\Xqgcsigma$ used in the definition of monopole Floer homology (recast in Coulomb gauge). For the moduli spaces of approximate Seiberg-Witten trajectories in the setting of finite dimensional Morse homology, we use the third construction of orientations discussed in Section~\ref{sec:or1}, in terms of a specialized coherent orientation (cf. Definition~\ref{def:sco}). This is based on trivializing determinant lines $\det(\tL_{\gamma})$, for the operators $\tL_{\gamma}$ defined in \eqref{eq:tLgamma}.

To relate the operators $P_{\gamma}^{\gCoul} $ (with perturbation $\qml$) and $\tL_{\gamma}$, we proceed as in the proof of Proposition~\ref{prop:RelationFinite}. Indeed, $P_{\gamma}^{\gCoul} $ is the analogue of the operator $Q_{\gamma}^{\gCoul} $ considered in that proposition, and $\tL_{\gamma}$ is the analogue of \eqref{eq:HessFinite}; in both cases, what is different here is that we work on compact cylinders and add spectral projections on the boundaries.  Recall from the proof of Proposition~\ref{prop:RelationFinite} that we have an orthogonal splitting $Q_{\gamma}^{\gCoul} = \tL_{\gamma} \oplus F_\gamma$, where $F_\gamma : L^2_j(\R; (\vml)^\perp) \to L^2_{j-1}(\R; (\vml)^\perp)$ is defined by 
$$F_\gamma(b(t), \psi(t)) = (\frac{d}{dt} b(t) + *d(b(t)), \frac{d}{dt} \psi(t) + D\psi(t)  - \langle \phi(t), D \phi(t) \rangle_{L^2} \psi(t)).$$  
Write $\widetilde{F}_\gamma$ for the analogous operator on the compact cylinder coupled with spectral projections.  Note that while the operator $\widetilde{F}_\gamma$ depends on the path, the domain and target do not.  We also write $\widetilde{F}_0$ for the analogous operator without the term $\langle \phi(t), D \phi(t) \rangle_{L^2} \psi(t)$.  

We seek to show that there are trivializations of $\det(\widetilde{F}_\gamma)$ which respect concatenations of paths.  First, note that the spectral decomposition of $(\vml)^\perp$ corresponding to $
\widetilde{F}_\gamma$ is independent of $\gamma$; indeed, since $\phi(t) \in \vml$ and $\psi(t) \in (\vml)^\perp$, we have that $\| \langle \phi(t), D \phi(t) \rangle_{L^2} \psi(t)\|_{L^2(Y)}$ $ < \lambda \| \psi(t) \|_{L^2(Y)}$, while $\|D \psi(t)\|_{L^2(Y)} \geq \lambda \| \psi(t) \|_{L^2(Y)}$ and the claim follows.  Further, we have the same spectral decomposition for any convex combination of $\widetilde{F}_\gamma$ and $\widetilde{F}_0$.  Note that $(1-r)\widetilde{F}_0 + r \widetilde{F}_\gamma$ differs from $\tilde{F}_0$ by a compact operator.  Thus, we have a homotopy between $\tilde{F}_\gamma$ and $\widetilde{F}_0$ through Fredholm operators, which induces an identification between $\det(\widetilde{F}_\gamma)$ and $\det(\widetilde{F}_0)$. This identification can easily be seen to respect concatenations.  Therefore, we have an identification of $\det(\widetilde{F}_\gamma)$, for each $\gamma$,with a fixed determinant line bundle (independent of $\gamma$) which respects concatenation.  Thus, we can find compatible trivializations of $\det(\widetilde{F}_\gamma)$ by simply fixing a trivialization of $\det(\widetilde{F}_0)$.  

The above discussion shows that a trivialization of $\det(P_{\gamma}^{\gCoul} )$ is equivalent to a trivialization of $\det(\tL_{\gamma})$ in a way which respects concatenation of paths.  It follows that an orientation data system in Coulomb gauge gives rise to a specialized coherent orientation on the finite dimensional approximation $(B(2R) \cap \vml)^{\sigma}/S^1$. Once we fix these orientations, we obtain the following.

\begin{proposition}
The bijective correspondences described in Propositions~\ref{prop:EquivalenceTrajectories}, \ref{prop:EquivalenceTrajectoriesObstructed} and \ref{prop:Uagree} are orientation-preserving.
\end{proposition}

\begin{proof}
From the above construction, we see that, on the moduli spaces of approximate Seiberg-Witten trajectories, the signs defined using $P_{\gamma}^{\gCoul} $ (with $\qml$) coincide with those defined using $\tL_{\gamma}$. 

Further, as $\lambda$ varies on an interval $[\lambda_0, \infty]$ (including infinity), the moduli spaces $\Mbreve^{\agCoul}([x_{\lambda}], [y_{\lambda}])$ form a continuous family, carrying determinant line bundles given by $\det(P_{\gamma}^{\gCoul} )$. These bundles vary continuously, so the trivialization at any finite $\lambda$ agrees with the one at $\infty$, under the correspondence from Proposition~\ref{prop:EquivalenceTrajectories}. 

The same argument applies to the other two correspondences.
\end{proof}

\chapter {The equivalence of the homology theories}
\label{sec:equivalence}

We are now ready to prove the results advertised in the Introduction. 

\begin{proof}[Proof of Theorem~\ref{thm:Main}]
Recall that the goal is to establish an absolutely graded isomorphism between $\widetilde{H}^{S^1}_*(\SWF(Y,\spinc))$ and $\hmto(Y,\spinc)$.  First, by Proposition~\ref{prop:perturbedspectrum}, we have that 
\begin{equation}\label{eq:swf-swfq-homology}
\widetilde{H}^{S^1}_*(\SWF(Y,\spinc)) \cong \widetilde{H}^{S^1}_*(\SWF_\q(Y,\spinc)), 
\end{equation}
where $\SWF_\q(Y,\spinc)$, defined in Section~\ref{sec:verycompact}, is the analogous construction of the Floer spectrum using $\Xqmlgc = l + \pml c_\q$ instead of $l + \pml c$.  Here, we require that $\q$ is a very tame, admissible perturbation satisfying the conclusions of Proposition~\ref{prop:ND2}, Proposition~\ref{prop:AllMS}, and Proposition~\ref{prop:MorseSmales}.   Recall that 
$$
\SWF_\q(Y,\spinc) = \Sigma^{-n(Y,\spinc,g)\mathbb{C}} \Sigma^{-W^{(-\lambda,0)}} I^\lambda_\q,
$$
where $I^\lambda_\q$ is the equivariant Conley index constructed from the flow of $l + \pml c_\q$, analogous to $I^\lambda$ constructed in Section~\ref{sec:conleyswf}.

In Propositions~\ref{prop:AllMS} and ~\ref{prop:MorseSmales} we showed that for $\lambda = \llambda_i \gg 0$, $\Xqmlgc$ is a Morse-Smale equivariant quasi-gradient on $\vml \cap B(2R)$. Thus, we can construct a Morse complex $(\check{C}_\lambda, \check{\partial}_\lambda)$ for $\Xqmlagcsigma$ on $(\vml \cap B(2R))^\sigma/S^1$ as defined in Sections~\ref{subsec:morseboundary} and Section~\ref{subsec:CircleMorse}.  By \eqref{eq:EquivConleyMorse}, we have that 
\begin{equation}\label{eq:Ilambda-morse}
\widetilde{H}^{S^1}_i(I^\lambda_\q) \cong H_{i}(\check{C}_\lambda, \check{\partial}_\lambda)  ,
\end{equation}
for $0 \leq i \leq n_\lambda - 1$, where $n_\lambda$ is the connectivity of the pair $\bigl(I^\lambda_\q, \bigl( I^\lambda_\q - (I^\lambda_\q)^{S^1} \bigr) \cup *\bigr)$.  Note that this isomorphism respects the $\mathbb{Z}[U]$-module structures on each of the homologies.  

By shifting gradings, we can rephrase \eqref{eq:Ilambda-morse} as 
\begin{equation}
\label{eq:equalityM}
\widetilde{H}^{S^1}_i(\SWF_\q(Y,\spinc)) \cong H_{i + \dim W^{(-\lambda,0)} + 2n(Y,\spinc,g)}(\check{C}_\lambda, \check{\partial}_\lambda),
\end{equation}
for $- \dim W^{(-\lambda,0)} - 2n(Y,\spinc,g) \leq i \leq n_\lambda - 1 - \dim W^{(-\lambda,0)} - 2n(Y,\spinc,g)$. Set
$$ M_{\lambda} := \min \ \{ \dim W^{(-\lambda,0)} + 2n(Y,\spinc,g), \ n_\lambda - 1 - \dim W^{(-\lambda,0)} - 2n(Y,\spinc,g) \},$$
so that the isomorphism in \eqref{eq:equalityM} holds in the grading range $[-M_{\lambda}, M_{\lambda}]$.

We claim that $M_{\lambda} \to \infty$ as $\lambda = \llambda_i \to \infty$. It is clear that $\dim W^{(-\lambda,0)}  \to \infty$, so what we need to check is that 
\begin{equation}
\label{eq:MlambdaEstimate}
n_\lambda  - \dim W^{(-\lambda,0)} \to \infty.
\end{equation}

Let us investigate how the connectivity of the pair $\bigl(I^\lambda_\q, \bigl( I^\lambda_\q - (I^\lambda_\q)^{S^1} \bigr) \cup *\bigr)$ changes as we increase $\lambda$. Fix a sufficiently large eigenvalue cut-off $\mu = \llambda_{i}$. (Recall that this implies that neither $\mu$ nor $-\mu$ are eigenvalues.) For $\lambda=\llambda_{j} \geq \mu$, it is proved in \cite[Section 7]{Spectrum} (for $\q=0$, but the same argument works for any $\q$) that
\begin{equation}
\label{eq:changeConley}
I^\lambda_\q \simeq I^{\mu}_\q \wedge I(l)^{(\mu, \lambda)}.
\end{equation}
Here, $I(l)^{(\mu, \lambda)}$ is the Conley index associated to the isolated invariant set $\{0\}$ in the linear flow induced by $l$ on the complementary subspace to $W^{\mu}$ in $\vml$, that is, on $W^{(-\lambda, -\mu)} \oplus W^{(\mu, \lambda)}$. Let us decompose this complementary subspace according to the sign of the eigenvalues, and also according to the type of eigenvectors (connections or spinors). Specifically, let
\begin{align*}
a^{\mu, \lambda}_+ &= \dim (W^{(\mu, \lambda)} \cap \ker d^*),& \ \ \ b^{\mu, \lambda}_+ &= \dim (W^{(\mu, \lambda)}\cap \Gamma(\Spin)) \\
a^{\mu, \lambda}_- &= \dim (W^{(-\lambda, -\mu)} \cap \ker d^*),& \ \ \ b^{\mu, \lambda}_- &= \dim (W^{(-\lambda, -\mu)}\cap \Gamma(\Spin)),
\end{align*}
where $\dim$ denotes the real dimension. Then, the Conley index of the linear flow is
$$  I(l)^{(\mu, \lambda)} \simeq D(\rr^{a^{\mu, \lambda}_+})_+ \wedge D(\cc^{b^{\mu, \lambda}_+})_+ \wedge (\rr^{a^{\mu, \lambda}_-})^+ \wedge (\cc^{b^{\mu, \lambda}_-})^+.$$ 

We used here the standard notation in homotopy theory: If $V$ is a vector space, then $D(V)_+$ is the union of the unit disk in $V$ with one disjoint basepoint, and $V^+$ is the one-point compactification of $V$.

Using \eqref{eq:changeConley}, and keeping track of the fixed point sets in each Conley index, we obtain
$$(I^\lambda_\q, (I^\lambda_\q)^{S^1}) \simeq \Bigl( D(\rr^{a^{\mu, \lambda}_+})_+ \wedge D(\cc^{b^{\mu, \lambda}_+})_+ \wedge \Sigma^{a^{\mu, \lambda}_- \rr} \Sigma^{b^{\mu, \lambda}_- \cc} (I^\mu_\q), \ D(\rr^{a^{\mu, \lambda}_+})_+ \wedge \Sigma^{a^{\mu, \lambda}_- \rr}   
(I^\mu_\q)^{S^1}\Bigr).$$

Observe that when we change a pair $(X, Y)$ into $(D(\rr)_+ \wedge X, D(\rr)_+ \wedge Y)$, $(D(\cc)_+  \wedge X, Y)$, $(\Sigma^{\rr} X, \Sigma^{\rr}Y)$, or $(\Sigma^{\cc} X, Y)$, then the connectivity of the pair $\bigl(X, (X-Y) \cup *\bigr)$ increases by $0$, $2$, $1$ and $2$, respectively. Therefore, 
$$ n_{\lambda} = n_{\mu} + 2b^{\mu, \lambda}_+ + a^{\mu, \lambda}_- + 2b^{\mu, \lambda}_-.$$
Since $\dim W^{(-\lambda,0)} = \dim W^{(-\mu,0)} + a^{\mu, \lambda}_- + 2b^{\mu, \lambda}_-$, we see that
$$n_\lambda  - \dim W^{(-\lambda,0)}= n_{\mu} - \dim W^{(-\mu,0)}+2b^{\mu, \lambda}_+  \to \infty \ \ \text{as } \lambda= \llambda_i \to \infty.$$
This proves \eqref{eq:MlambdaEstimate}, and we conclude that $M_{\lambda} \to \infty$. Thus, the isomorphism \eqref{eq:equalityM} holds in a grading range that gets larger as $\lambda = \llambda_i \to \infty$.

While the chain groups of $\check{C}_\lambda$ consist of stationary points of $\Xqmlagcsigma$, the natural gradings from Morse homology are not the absolute gradings on stationary points that we have been working with throughout this monograph.  Recall that the absolute grading $\gr^{\SWF}_\lambda$ on stationary points of $\Xqmlagcsigma$ defined in \eqref{eq:gr-lambda} includes a shift by $\dim W^{(-\lambda,0)} + 2n(Y,\spinc,g)$.  Therefore, we define the complex $\cmto^\lambda(Y,\spinc,\q)$ by taking the Morse complex $(\check{C}_\lambda, \check{\partial}_\lambda)$ and shifting gradings down by $- \dim W^{(-\lambda,0)} + 2n(Y,\spinc,g)$, so stationary points have grading given by $\gr^{\SWF}_\lambda$.  We denote the homology of $\cmto^\lambda$ by $\hmto^\lambda$.  With this definition, we have 
\begin{equation}\label{eq:swf-hmtolambda}
\widetilde{H}^{S^1}_i(\SWF_\q(Y,\spinc)) \cong \hmto^\lambda_{i}(Y,\spinc,\q)
\end{equation}
for $i \in [-M_{\lambda}, M_{\lambda}]$. Therefore, it remains to establish an isomorphism between $\hmto^\lambda_{i}(Y,\spinc,\q)$ and $\hmto_i(Y,\spinc,\q)$.  In other words, we must relate the Morse homology defined in terms of stationary points and trajectories of $\Xqmlagcsigma$ versus the Floer homology defined in terms of stationary points and trajectories of $\Xqsigma$.    

First, by the work of Chapter~\ref{sec:coulombgauge} (more precisely described in Section~\ref{sec:CoulombSummary}), we can compute $\cmto(Y,\spinc,\q)$ using stationary points and trajectories of $\Xqagcsigma$ instead of $\Xqsigma$.  In a fixed grading range $[-N,N]$, Proposition~\ref{prop:stationarycorrespondence} shows that for $\lambda = \llambda_i \gg 0$, there is a grading-preserving identification between the stationary points of $\Xqagcsigma$ and the stationary points of $\Xqmlagcsigma$ in $(\vml \cap B(2R))^\sigma/S^1$, and thus we have an identification between the chain groups $\cmto^\lambda(Y,\spinc,\q)$ and $\cmto(Y,\spinc,\q)$ in the grading range $[-N,N]$.

By the work of Chapter~\ref{sec:trajectories2}, we have established an identification between the signed counts of index one (respectively boundary-obstructed index two) trajectories for $\Xqagcsigma$ and for $\Xqmlagcsigma$, in the given grading range.  Since these are the moduli spaces that are counted in the differentials for $\cmto^\lambda$ and $\cmto$ (defined in \eqref{eq:delcheck} and \eqref{eqn:cmboundary} respectively), we obtain a chain complex isomorphism in this grading range.  Therefore, we have that 
\begin{equation}\label{eq:approximate-monopole-isomorphism}
\hmto^\lambda_i(Y,\spinc) \cong \hmto_i(Y,\spinc)
\end{equation}
for $i \in [-N+1,N-1]$.  

By Proposition~\ref{prop:Uagree}, the $U$-maps on these chain complexes, defined in \eqref{eqn:morsecap} and \eqref{eqn:floercap} for $\cmto^\lambda(Y,\spinc,\q)$ and $\cmto(Y,\spinc,\q)$ respectively, must agree in the given grading range.  Thus, the isomorphism in \eqref{eq:approximate-monopole-isomorphism} respects the $U$-action.  

Therefore, we have established an isomorphism between $\hmto^\lambda_i(Y,\spinc)$ and $\hmto_i(Y,\spinc)$ $\mathbb{Z}[U]$-modules for $i \in [-N+1,N-1]$. Take $\lambda \gg 0$ so that $M_{\lambda} \geq N-1$. Combining this last isomorphism with \eqref{eq:swf-swfq-homology} and \eqref{eq:swf-hmtolambda}, we see that $$\widetilde{H}^{S^1}_i(\SWF(Y,\spinc)) \cong \hmto_i(Y,\spinc),$$ for $i  \in [-N+1,N-1].$  Again, this isomorphism respects the $U$-action.  Since $N$ was arbitrary, we obtain the desired result.  
\end{proof}

\begin{proof}[Proof of Corollary~\ref{cor:MainHat}]
Recall that Bloom's variant of monopole Floer homology, $\widetilde{HM}(Y,\spinc)$, is defined to be the homology of the mapping cone (with some grading shifts) of the $U$-map on $\cmto(Y,\spinc,\mathfrak{q})$.  

If $F$ is a field, the long exact sequence induced by the $U$-map on homology shows that $\hmtilde(Y,\spinc;F)$ is isomorphic to the homology of the mapping cone of the $U$-map on $\hmto(Y,\spinc;F)$, thought of as a complex equipped with the trivial differential.   By Theorem~\ref{thm:Main}, $\hmtilde(Y,\spinc;F)$ is calculated as the mapping cone of the $U$-map on $\tH^{S^1}_*(\SWF(Y,\spinc); F)$.  Applying the following lemma with $\mathbb{Z}/p^k\mathbb{Z}$- and $\mathbb{Q}$-coefficients, together with the universal coefficient theorem, completes the proof.      
\end{proof}

\begin{lemma}
Let $X$ be a based, finite $S^1$-CW complex. Then, for any field $F$, we have a graded isomorphism   
\[
\widetilde{H}^n(X;F) \cong \widetilde{H}^n([H^{*}_{S^1}(X;F)]_1 \xrightarrow{U} \widetilde{H}^{*-2}_{S^1}(X;F)).  
\]  
Here, $[C]_1$ means we are shifting the degree of $C$ by 1.  
\end{lemma}
\begin{proof}
Since the diagonal action of $S^1$ is free on $X \wedge  ES^1_+$, we have that the orbit map $X \wedge ES^1_+ \to X \wedge_{S^1} ES^1_+ := (X \wedge ES^1_+)/S^1$ is a principal $S^1$-bundle (away from the basepoint).  This bundle is isomorphic to the pullback bundle of the map $\pi: X \wedge_{S^1} ES^1_+ \to BS^1$.  There is an associated Gysin sequence: 
\[
\ldots \to \widetilde{H}^n(X \wedge ES^1_+;F) \to \widetilde{H}^{n-1}(X \wedge_{S^1} ES^1_+;F) \xrightarrow{e \cup} \widetilde{H}^{n+1}(X \wedge_{S^1}  ES^1_+;F) \to  \ldots,
\]
where $e$ is the Euler class of this $S^1$-bundle. Since the Euler class agrees with the first Chern class, $e$ is by definition $\pi^*(U)$, where $U$ is the generator of $H^2(BS^1;F)$.  Because $ES^1$ is contractible, we can rewrite the Gysin sequence as 
\[
\ldots \to \widetilde{H}^n(X;F) \to \widetilde{H}_{S^1}^{n-1}(X;F) \xrightarrow{\pi^*(U) \cup} \widetilde{H}_{S^1}^{n+1}(X;F) \to \ldots 
\]    
However, cup product with $\pi^*(U)$ corresponds precisely to the $U$-action in the $F[U]$-module structure on $H_{S^1}^*(X;\mathbb{Z})$. The Gysin sequence now gives the desired isomorphism.\end{proof}

\begin{proof}[Proof of Corollary~\ref{cor:delta}] This follows from Theorem~\ref{thm:Main}, given that $\delta$ and $-h$ are obtained in the same way from $\hmto$ resp. $\tH^{S^1}_*(\swf)$; they are both half of the grading of the lowest nontrivial element in the infinite $U$-tail of the corresponding homology group with $\qq$ coefficients.
\end{proof}

\begin{proof}[Proof of Corollary~\ref{cor:hfswf}] By the Seiberg-\!Witten / Heegaard Floer equivalence \cite{KLT1, CGH1, CGH3}, the minus and hat flavors of Heegard Floer homology are isomorphic to $\hmto$ resp. $\hmtilde$. Thus, the first two isomorphisms in Corollary~\ref{cor:hfswf} are consequences of Theorem~\ref{thm:Main} and Corollary~\ref{cor:MainHat}.

To get the isomorphism between $\HFminus$ and co-Borel homology, we can apply the isomorphism between $\HFplus$ and Borel homology to $Y$ with the orientation reversed. Indeed, there is a duality between $\HFminus(-Y, \s)$ and the cohomology version of $\HFplus(Y, \s)$; see \cite[Proposition 2.5]{HolDiskTwo}. Similarly, co-Borel homology is isomorphic (up to a reversal of degrees) to the Borel cohomology of the dual spectrum, and $\swf(Y, \s)$ is dual to $\swf(-Y, \s)$ by \cite[\S 9, Remark 2]{Spectrum}. Cohomology theories can be related back to homologies by the universal coefficient  theorem.

To get the isomorphism between $\HFinfty$ and Tate homology, note that $\HFinfty$ can be obtained from $\HFminus$ by inverting the variable $U$. Similarly, Tate homology is obtained from co-Borel homology by inverting $U$.
\end{proof}

\backmatter
\bibliography{biblio}
\bibliographystyle{alpha}

\end{document}